\documentclass{article}

\PassOptionsToPackage{numbers, compress}{natbib}
\usepackage[preprint]{neurips_2025}

\usepackage[utf8]{inputenc} % allow utf-8 input
\usepackage[T1]{fontenc}    % use 8-bit T1 fonts
\usepackage{hyperref}       % hyperlinks
\usepackage{url}            % simple URL typesetting
\usepackage{booktabs}       % professional-quality tables
\usepackage{amsfonts}       % blackboard math symbols
\usepackage{nicefrac}       % compact symbols for 1/2, etc.
\usepackage{microtype}      % microtypography
\usepackage{xcolor}         % colors

\title{A stochastic gradient method for trilevel optimization}

\usepackage{array} 
\usepackage{multirow}
\usepackage{float}
\usepackage{caption}
\usepackage{subcaption}
\usepackage{wrapfig}
\usepackage{comment}
\usepackage{graphicx}
\usepackage{amsmath}
\usepackage{amsfonts}
\usepackage{amssymb}
\usepackage{enumerate}
\usepackage{lscape}
\usepackage{longtable}
\usepackage{rotating}
\usepackage{color}
\usepackage{url}
\usepackage{rotating}
\usepackage{algpseudocode}
\usepackage[ruled]{algorithm}
\usepackage{pgf,tikz}
\usepackage[titletoc]{appendix}
\usepackage{mathrsfs}
\usepackage{pgfplots}
\usetikzlibrary{arrows}
\usepackage{color}
\usepackage{bbm}
\usepackage{eqparbox}%

\numberwithin{equation}{section}
\usepackage[symbol]{footmisc}

\newtheorem{theorem}{Theorem}[section]
\newtheorem{proposition}{Proposition}[section]

\newtheorem{lemma}[theorem]{Lemma}

\newtheorem{remark}{Remark}[section]

\newtheorem{assumption}{Assumption}[section]

\newenvironment{proof}[1][Proof]{\noindent \textbf{#1.} }{\hfill$\Box$\par\medskip}
\DeclareMathOperator*{\argmax}{argmax}
\DeclareMathOperator*{\argmin}{argmin}

\DeclareMathOperator{\train}{train}
\DeclareMathOperator{\val}{val}

\newcommand{\beqn}[1]{\begin{equation}\label{#1}}
	\newcommand{\eeqn}{\end{equation}}

\definecolor{darkgreen}{rgb}{0,0.6,0}
\definecolor{aau2}{rgb}{0.0, 0.5, 0.69}
\definecolor{aau3}{rgb}{0.0, 0.53, 0.74}
\definecolor{aau4}{rgb}{0.0, 0.48, 0.65}
\definecolor{aau5}{rgb}{0.0, 0.45, 0.73}
\definecolor{rsap}{RGB}{130, 36, 51}
\definecolor{gsap}{RGB}{112, 164, 137}
\definecolor{tud}{rgb}{0.43,0.73,0.11}
\definecolor{verde}{rgb}{0.33,0.53,0.11}
\definecolor{ttffqq}{rgb}{0.0, 0.48, 0.65}
\definecolor{ffqqqq}{rgb}{0.0, 0.5, 0.69}

\newcommand{\tcg}{\textcolor{darkgreen}}
\author{%
	T. Giovannelli \\
	University of Cincinnati\\
	\texttt{giovanto@ucmail.uc.edu} \\
	\And
	G. D. Kent \\
	Lehigh University \\
	\texttt{gdk220@lehigh.edu} \\
	\And
	L. N. Vicente \\
	Lehigh University \\
	\texttt{lnv@lehigh.edu} 
}

\begin{document}

	\maketitle

	\begin{abstract}
		With the success that the field of bilevel optimization has seen in recent years, similar methodologies have started being applied to solving more difficult applications that arise in trilevel optimization. At the helm of these applications are new machine learning formulations that have been proposed in the trilevel context and, as a result, efficient and theoretically sound stochastic methods are required. In this work, we propose the first-ever stochastic gradient descent method for solving unconstrained trilevel optimization problems and provide a convergence theory that covers all forms of inexactness of the trilevel adjoint gradient, such as the inexact solutions of the middle-level and lower-level problems, inexact computation of the trilevel adjoint formula, and noisy estimates of the gradients, Hessians, Jacobians, and tensors of third-order derivatives involved.
		We also demonstrate the promise of our approach by providing numerical results on both synthetic trilevel problems and trilevel formulations for hyperparameter adversarial tuning.

	\end{abstract}

	% \begin{figure}
		%   \centering
		%   \fbox{\rule[-.5cm]{0cm}{4cm} \rule[-.5cm]{4cm}{0cm}}
		%   \caption{Sample figure caption.}
		% \end{figure}

	\section{Introduction}\label{sec:introduction}
	Multi-level optimization~(MLO) is a general class of problems with the goal of optimizing an upper-level objective while requiring subsets of the considered variables to satisfy optimality principles for some number of nested sub-problems. Hierarchical in nature, these MLO problems have a variety of applications that appear in fields such as defense industry~\cite{BArguello_ESJohnson_JLGearhart_2023, KLai_MIllindala_KSubraminiam_2019, YYao_etal_2007, XWu_AJConejo_2017, YGuo_CGuo_JYang_2023}, signal recovery and power control~\cite{HLiduka_2011, LCCang_APetrusel_2010}, supply chain networks~\cite{XXu_ZMeng_RShen_2013, MRahdar_etal_2018, AFard_MKeshteli_SMirjalili_2018}, and more recently in the field of machine learning~\cite{YJiao_etal_2024, SKeun_etal_2023, MGuo_etal_2020, YJiao_KYang_CJian_2024, HLiu_KSimonyan_YYang_2019, TGiovannelli_GKent_LNVicente_2022, XJin_etal_2019}. Due to the difficulty of these~MLO problems, most of the algorithms have largely only been developed for solving the bilevel case. However, the trilevel case has recently seen further interest by applying similar methodologies that have been utilized in the bilevel case. With this interest comes the aim of developing efficient and theoretically sound first-order stochastic gradient methods for handling large-scale applications of trilevel optimization problems that arise in the field of machine learning. As far as we know, this is the first work that addresses the stochastic setting of a trilevel problem, both theoretically and numerically.

	In this paper, we consider the general trilevel optimization~(TLO) problem formulation
	\begin{equation}\label{prob:trilevel}
		\tag*{TLO}
		\begin{split}
			\min_{x \in \mathbb{R}^n, \, y \in \mathbb{R}^m, \, z \in \mathbb{R}^t} ~~ & f_1(x,y,z) \\
			\mbox{s.t.}~~ & x \in X \\
			& y, \, z \in \argmin_{y \in Y(x), \, z\in\mathbb{R}^t} ~~ f_2(x,y,z) \\
			& \quad\quad\mbox{s.t.}~~  z \in \argmin_{z \in Z(x,y)} ~~ f_3(x,y,z). \\
		\end{split}
	\end{equation}
	The goal of the upper-level~(UL) problem is to determine the optimal value of the~UL function~$f_1: \mathbb{R}^n\times\mathbb{R}^m\times\mathbb{R}^t \to \mathbb{R}$, where the~UL variables~$x$ are subjected to~UL constraints~($x \in X$), the middle-level~(ML) variables~$y$ are subjected to being an optimal solution of the~ML problem, and the lower-level~(LL) variables~$z$ are subjected to being an optimal solution of the~LL problem. In the~ML problem, the~ML function~$f_2: \mathbb{R}^n\times\mathbb{R}^m\times\mathbb{R}^t \to \mathbb{R}$ is optimized in the~ML variables~$y$, subject to the~ML constraints~$y \in Y(x)$. Similarly, in the~LL problem, the~LL function~$f_3: \mathbb{R}^n\times\mathbb{R}^m\times\mathbb{R}^t \to \mathbb{R}$ is optimized in the~LL variables~$z$, subject to the constraints~$z\in Z(x,y)$. In this paper, we will assume that the ML and LL problems are strongly convex (see Subsection~\ref{subsec:assumptions} below) and that the UL problem is possibly nonconvex (see Theorem~\ref{th:TSG_convergence} below).

	\subsection{Trilevel optimization in the literature}\label{lit_rev}
	Trilevel and multi-level optimization has been studied as early as the~1980s (see~\cite{CBlair_1992, JFBard_JEFalk_1982,JFBard_1984,WUepying_WFBialas_1986,HPBenson_1989}), but in many of the aforementioned fields (e.g., defense industry, supply chain networks, etc.), problem-specific formulations typically lack general solution methodologies. We mention here a few notable exceptions that do consider general methodologies. The authors of~\cite{SLTilahun_etal_2012} introduced an evolutionary strategy to update each level sequentially, but without convergence guarantees, as their method failed to account for the hierarchical dependencies present in the MLO problem. In contrast, the authors of~\cite{AShafiei_etal_2024} proposed a proximal gradient method for TLO problems with convex objective functions, offering convergence guarantees but lacking numerical validation, leaving the method's practicality uncertain. For thorough reviews of the development of multi-level optimization, see the surveys~\cite{LNVicente_PHCalamai_1994,JLu_etal_2016,RLiu_JGao_etal_2021,CChen_etal_2023}.

	\textbf{Trilevel optimization for machine learning.} More recently, TLO (also referred to as \emph{trilevel learning} when taking on applications in a machine learning context) and~MLO problems have seen utilization in being applied to solving large-scale hierarchical machine learning problems with applications of hyperparameter tuning, adversarial learning, and federated learning. In~\cite{RSato_etal_2021}, the authors developed a gradient-based method for solving an approximate formulation of the general~MLO problem, as well as presenting convergence guarantees and numeric results for their method in the deterministic case. Such a paper builds on pre-existing methods utilized in~\cite{LFranceschi_etal_2017} for the bilevel case that approximate the solution to each of the lower-level problems with an iterative method. Complimenting this development, the authors of~\cite{SKeun_etal_2023} introduced BETTY, an automatic differentiation library for general multi-level optimization, which has helped facilitate applications like neural architecture search (NAS) with adversarial robustness~\cite{MGuo_etal_2020}. Trilevel optimization has also been further extended to decentralized learning environments in~\cite{YJiao_etal_2024, YJiao_KYang_CJian_2024}, where the authors aim at developing methods with convergence guarantees for federated trilevel learning problems. However, it bears mentioning that all of the aforementioned papers only consider the deterministic setting in their analysis.
	
	%The authors show that the approximate solution converges to the true solution as the number of~UL iterations approaches infinity, and under some assumptions, are further able to demonstrate that the algorithm converges to a stationary point with a rate of~$\mathcal{O} ( 1/\sqrt{k})$, where we use~$k$ to denote the number of iterations when approximating the~LL solutions. 

	\subsection{Contributions of this paper}
	%With the development of TLO problems in large-scale machine learning settings comes the need for efficient stochastic methods for solving these problems. In this paper, we address the need for efficient stochastic gradient first-order methods for handling stochastic TLO problems. To that end, we utilize the true trilevel steepest descent direction defined by the adjoint formula~\eqref{adjoint}. It bears noting that this derivation was obtained by utilizing the first-order necessary optimality conditions of the lower-level and middle-level problems, and is the first time it has been used in the trilevel case as far as we know. We will then propose the trilevel stochastic gradient (TSG) descent algorithm, analyze its convergence rate under appropriate assumptions, and implement the TSG-2 algorithm that utilizes rank-2 approximations of the higher-order derivatives; rank-1 approximations were used in~\cite{TGiovannelli_GKent_LNVicente_2022} successfully and provided efficient and reliable results. The paper will then conclude by providing numerical results of our proposed algorithm on the hyperparameter-adversarial problem that was discussed in subsection~\ref{lit_rev}.
	
	The field of bilevel optimization has seen a rich development of first-order descent methods for solving large-scale problems that arise in the field of machine learning (e.g., see~\cite{TChen_YSun_WYin_2021, TChen_YSun_WYin_2021_closingGap, HLiu_KSimonyan_YYang_2019, TGiovannelli_GKent_LNVicente_2022, TGiovannelli_GKent_LNVicente_2024, XJin_etal_2019}). However, as we have seen in the existing literature, no works have yet begun extending the theory and implementation of stochastic methods to trilevel and higher-level problems. In this paper, we propose TSG, the first stochastic gradient method for solving trilevel optimization problems, along with an extensive convergence analysis with general nonlinear and nonconvex UL functions. This is done by extending the concepts and methodologies developed for first-order bilevel optimization methods that utilize the so-called adjoint gradient (or hyper-gradient) via implicit differentiation, and adapting them to the trilevel setting. To address the significant difficulties imposed by the presence of second-order and third-order derivatives in handling these problems, we also propose practical and efficient strategies for implementing our TSG method and demonstrate its performance on a series of trilevel problems through numerical results.

	\section{Trilevel optimization}\label{sec:trilevel_TSG}
	In this paper, we will only focus on the unconstrained~ML and~LL cases of problem~\ref{prob:trilevel}, i.e., $Y(x) = \mathbb{R}^m$ and~$Z(x,y) = \mathbb{R}^t$.
	Since our goal is to propose and analyze a general optimization methodology for a stochastic~TLO, the~LL problem is assumed to be well-defined, in the sense of having a unique solution~$z(x,y)$ for all~$x\in\mathbb{R}^n$ and~$y\in\mathbb{R}^m$. Thus, problem~\ref{prob:trilevel} is equivalent to the following bilevel optimization~(BLO) problem, which is defined solely in terms of the~UL and~ML variables:
	\begin{equation} \label{bi_reduced}
		\tag*{BLO}
		\begin{split}
			\min_{x \in \mathbb{R}^n, \, y \in \mathbb{R}^m} ~~ & f_1(x,y,z(x,y)) \\
			\mbox{s.t.}~~ & y \in \argmin_{y \in \mathbb{R}^m} ~~ \Bar{f}(x,y) := f_2(x,y,z(x,y)).
		\end{split}
	\end{equation}
	
	Similarly, problem~\ref{bi_reduced} can be even further reduced to a single-level optimization problem under the assumption that the lower-level problem in~\ref{bi_reduced} also has a unique solution~$y(x)$. In this way, since~$y(x)$ is solely determined by~$x$, it is clear that the unique solution~$z(x,y(x))$ is solely determined by~$x$ as well, which we denote simply as $z(x)$. Thus, problem~\ref{prob:trilevel} ultimately reduces to the single-level optimization problem given by
	\begin{equation} \label{reduced}
		\min_{x \in \mathbb{R}^n} \; f(x)=f_1(x,y(x),z(x,y(x))) \quad
		\mbox{s.t.} \quad x \in X.
	\end{equation}
	We define the trilevel adjoint gradient of~$f$ at~$x$ as
	\begin{equation}\label{adjoint}
		\nabla f \; = \; ( \nabla_x f_1 - \nabla_{xz}^2 f_3 \nabla_{zz}^2 f_3^{-1} \nabla_z f_1 ) - \nabla_{xy}^2\Bar{f} \nabla_{yy}^2 \Bar{f}^{-1} ( \nabla_y f_1  - \nabla_{yz}^2 f_3 \nabla_{zz}^2 f_3^{-1} \nabla_z f_1 ),
	\end{equation}
	where all of the gradient and Hessian terms involved are evaluated at the point $(x,y(x),z(x))$. Notice that this is essentially a classical adjoint gradient calculation applied to problem~\ref{bi_reduced}. The complete statement, along with all term definitions and full derivation, is given by Proposition~\ref{prop:adjoint_grad} in Appendix~\ref{app:prop:adjoint_grad}.

	\subsection{The trilevel stochastic gradient method}
	
	The stochastic algorithm developed in this paper proceeds by iteratively updating the~LL variables first, followed by the~ML variables, and lastly the~UL variables. The iterations corresponding to the~UL, ML, and~LL problems are denoted by~$i$, $j$, and~$k$, respectively, with the total number of iterations denoted as~$I$, $J$, and~$K$, respectively. Let~$\{\xi^{i}\}$, $\{\xi^{i,j}\}$, and~$\{\xi^{i,j,k}\}$ denote sequences of random variables defined in a probability space (with probability measure independent from~$x$, $y$, and~$z$) such that~i.i.d.~samples can be observed or generated. Such random variables are introduced for gradient, Jacobian, and Hessian evaluations, and their realizations can be interpreted as a single sample or a batch of samples for a mini-batch stochastic gradient~(SG). For simplicity, we also adopt the following terminology throughout this paper: $z^{i,j} = z^{i,j,0}$, $z^{i,j+1} = z^{i,j+1,0} = z^{i,j,K}$, $z^{i} = z^{i,0,0}$, and $z^{i+1} = z^{i+1,0,0} = z^{i,J,K}$ for the LL iterations and $y^{i} = y^{i,0}$ and~$y^{i+1} = y^{i+1,0} = y^{i,J}$ for the ML iterations. Most of this terminology is merely notation; however, by letting $z^{i+1}=z^{i,J,K}$, $z^{i,j+1}=z^{i,j,K}$, and $y^{i+1} = y^{i,J}$, we are saying that the initial iterates for new cycles are the last ones of the previous corresponding cycles.
	
	Given the current iterate~$(x^i,y^{i,j},z^{i,j,k})$, the update direction that is used for the~LL problem is simply the stochastic gradient of the~LL objective function~$f_3$, denoted as~$g_{f_3}^{i,j,k}$ and given by $g_{f_3}^{i,j,k} \; = \; \nabla_z f_3(x^i,y^{i,j},z^{i,j,k};\xi^{i,j,k})$.
	Letting $\gamma_i\in(0,1]$ denote the step size for the~LL problem at the~UL iteration~$i$, the update of the~LL variables is given by $z^{i,j,k+1} \; = \; z^{i,j,k} - \gamma_i g_{f_3}^{i,j,k}$. The SG algorithm used to obtain the approximate solution $z^{i,j+1} \approx z(x^i,y^{i,j})$ is stated by Algorithm~\ref{alg:TSG_LLP}.

	The exact gradient for the~ML problem is computed via the following standard adjoint gradient (by combining equations~\eqref{eq:exact_grad_MLP_y} and~\eqref{eq:jacobian_z_ret_y} in Appendix~\ref{app:prop:adjoint_grad}):
	\begin{equation}\label{eq:adjoint_ML}
		\nabla_y \Bar{f}(x,y) \; = \; \nabla_y f_2 - \nabla_{yz}^2 f_3 \nabla_{zz}^2 f_3^{-1}  \nabla_z f_2,
	\end{equation}
	where all gradients and Hessians are evaluated at the point $(x,y,z(x,y))$. However, since we solve the~LL problem inexactly to obtain an approximate solution~$z^{i,j+1} \approx z(x^i,y^{i,j})$, the~ML adjoint gradient~\eqref{eq:adjoint_ML} now becomes ``inexact''. Thus, given the current iterate~$(x^i,y^{i,j},z^{i,j+1})$, the update direction that is used for the~ML problem is the inexact stochastic gradient of the function~$\bar{f}$, denoted as~$\tilde g_{f_2}^{i,j}$ and given by
	\begin{alignat}{2}
		\tilde g_{f_2}^{i,j} \;=\; \nabla_y \Bar{f}(x^i,y^{i,j},z^{i,j+1};\xi^{i,j}) \;=\; \nabla_y f_2 - \nabla_{yz}^2 f_3 \nabla_{zz}^2 f_3^{-1}  \nabla_z f_2,\label{eq:TSG_ML_stoch_dir}
	\end{alignat}
	where all gradients and Hessians are evaluated at the point $(x^i,y^{i,j},z^{i,j+1};\xi^{i,j})$. %\tcg{Further, we can denote the ``exact'' stochastic~ML gradient simply as~$g_{f_2}^{i,j}$ (without the tilde), given by~$g_{f_2}^{i,j} = \nabla_y \Bar{f}(x^i,y^{i,j},z(x^i, y^{i,j});\xi^{i,j})$.}
	We highlight this slight abuse of notation, since $\bar f$ is a function of $(x,y)$ and not $(x,y,z)$, as we are utilizing the approximation $z^{i,j+1}\;\approx\; z(x,y^{i,j})$ in computing the gradient $\nabla_y\bar f$. It is for this reason that we adopt the notation $\tilde g_2$ to denote an ``inexact'' SG (as opposed to simply $g_2$, which would denote the ``exact'' SG $\nabla \Bar{f}(x^i,y^{i,j})$).
	Letting $\beta_i\in~(0,1]$ denote the step size for the~ML problem at the~UL iteration~$i$, the update of the~ML variables is given by $y^{i,j+1} \; = \; y^{i,j} - \beta_i \tilde g_{f_2}^{i,j}$. The bilevel SG algorithm that is used to obtain the approximate solution $y^{i+1} \approx y(x^i)$ is stated by Algorithm~\ref{alg:TSG_MLP}. It bears mentioning that after every~ML iteration, we will perform another~LL update to obtain an approximation~$z^{i+1}$ to~$z(x^i,y^{i+1})$.
	
	\vspace{-0.7cm}
	\begin{figure}[H]
		\centering
		\begin{minipage}[c]{0.49\textwidth}
			\begin{algorithm}[H]
				\caption{SG (LL Problem)}\label{alg:TSG_LLP}
				\begin{algorithmic}[1]
					\item[] {\bf Input:} Initial $z^{i,j,0}$, \ $\gamma_i\in(0,1]$.
					\item[] {\bf For $k = 0, 1, 2, \ldots, K-1$ \bf do}
					\item[] \quad\quad {\bf 1.} Compute an~SG~$g_{f_3}^{i,j,k}$.
					\item[] \quad\quad {\bf 2.} Update $z^{i,j,k+1} \;=\; z^{i,j,k} - \gamma_i \, g_{f_3}^{i,j,k}$.
					\item[] {\bf Return} $z^{i,j+1} \;=\; z^{i,j,K}$.
				\end{algorithmic}
			\end{algorithm}
		\end{minipage}\hspace{0.01\textwidth}
		\begin{minipage}[c]{0.49\textwidth}
			\begin{algorithm}[H]
				\caption{Bilevel SG (ML Problem)}\label{alg:TSG_MLP}
				\begin{algorithmic}[1]
					\item[] {\bf Input:} Initial $y^{i,0}$, \ $\beta_i \in (0,1]$, \ $\gamma_i \in(0,1]$.
					\item[] {\bf For $j = 0, 1, 2, \ldots, J-1$ \bf do}
					\item[] \quad\quad {\bf 1.} Compute $z^{i,j+1}$ via Algorithm~\ref{alg:TSG_LLP} %Utilize Algorithm~\ref{alg:TSG_LLP} to compute $z^{i,j+1}$.
					\item[] \quad\quad {\bf 2.} Compute an approximation $\tilde{g}_{f_2}^{i,j}$.
					\item[] \quad\quad {\bf 3.} Update $y^{i,j+1} \;=\; y^{i,j} - \beta_i \, \tilde{g}_{f_2}^{i,j}$.
					\item[] {\bf Return} $(y^{i+1} \;=\; y^{i,J},\; z^{i,J})$.
				\end{algorithmic}
			\end{algorithm}
		\end{minipage}
	\end{figure}

	\vspace{-0.3cm}
	
	% \noindent
	% \begin{minipage}{0.35\textwidth}
		%     \hspace*{17pt} Now, recall that the exact gradient for the~UL problem is computed via the trilevel adjoint gradient given by equation~\eqref{adjoint}. Since we only solve the~ML problem inexactly to obtain an approximate solution~$y^{i+1} \approx y(x^i)$, the trilevel adjoint gradient~\eqref{adjoint} also becomes ``inexact''. Notice that the inexactness here comes from two sources: one related to the inexactness of
		% \end{minipage}
	% \hfill
	% \begin{minipage}{0.63\textwidth}
		%     \begin{algorithm}[H]
			%     	\caption{Trilevel Stochastic Gradient (TSG)}\label{alg:TSG}
			%     	\begin{algorithmic}[1]
				%     		\medskip
				%     		\item[] {\bf Input:} Initial $(x^0, y^{0,0}, z^{0,0,0})$, \ $\alpha_i \in (0,1]$, \ $\beta_i \in (0,1]$, \ $\gamma_i \in(0,1]$.
				%     		\medskip
				%     		\item[] {\bf For $i = 0, 1, 2, \ldots, I-1$ \bf do}
				%             \item[] \quad\quad {\bf 1.} Compute $y^{i+1}\;=\;y^{i,J}$ and $z^{i,J,0}$ via Algorithm~\ref{alg:TSG_MLP}.
				%             \item[] \quad\quad {\bf 2.} Compute $z^{i+1}\;=\;z^{i,J,K}$ via Algorithm~\ref{alg:TSG_LLP}.
				%     		\item[] \quad\quad {\bf 3.} Compute an approximation $\tilde{g}_{f_1}^{i}$.
				%     		\item[] \quad\quad {\bf 4.} Update $x^{i+1} \;=\; x^{i} - \alpha_i \, \tilde{g}_{f_1}^{i}$.
				%     		\item[] {\bf Return} $x^{I}$.
				%     	\end{algorithmic}
			%     \end{algorithm}
		% \end{minipage}
	
	\begin{wrapfigure}{r}{0.6\textwidth}  % 'r' = right side; width = 60% of textwidth
		\begin{minipage}{0.98\linewidth}
			\begin{algorithm}[H]
				\caption{Trilevel Stochastic Gradient (TSG)}\label{alg:TSG}
				\begin{algorithmic}[1]
					\medskip
					\item[] {\bf Input:} Initial $(x^0, y^{0,0}, z^{0,0,0})$, \ $\alpha_i \in (0,1]$, \ $\beta_i \in (0,1]$, \ $\gamma_i \in(0,1]$.
					\medskip
					\item[] {\bf For $i = 0, 1, 2, \ldots, I-1$ \bf do}
					\item[] \quad\quad {\bf 1.} Compute $y^{i+1}\;=\;y^{i,J}$ and $z^{i,J,0}$ via Algorithm~\ref{alg:TSG_MLP}.
					\item[] \quad\quad {\bf 2.} Compute $z^{i+1}\;=\;z^{i,J,K}$ via Algorithm~\ref{alg:TSG_LLP}.
					\item[] \quad\quad {\bf 3.} Compute an approximation $\tilde{g}_{f_1}^{i}$.
					\item[] \quad\quad {\bf 4.} Update $x^{i+1} \;=\; x^{i} - \alpha_i \, \tilde{g}_{f_1}^{i}$.
					\item[] {\bf Return} $x^{I}$.
				\end{algorithmic}
			\end{algorithm}
		\end{minipage}
	\end{wrapfigure}
	
	Now, recall that the exact gradient for the~UL problem is computed via the trilevel adjoint gradient given by equation~\eqref{adjoint}. Since we only solve the~ML problem inexactly to obtain an approximate solution~$y^{i+1} \approx y(x^i)$, the trilevel adjoint gradient~\eqref{adjoint} also becomes ``inexact''. Notice that the inexactness here comes from two sources: one related to the inexactness of the~LL variables and the other to the inexactness of the~ML variables. The first source of inexactness arises from the two Hessian terms of the true~ML problem, i.e., $\nabla_{xy}^2\Bar{f}$ and~$\nabla_{yy}^2 \Bar{f}^{-1}$, due to them being evaluated at the approximate solution~$z^{i+1}$ instead of~$z(x^i,y(x^i))$. The second source of inexactness comes from all of the terms involved being evaluated at the approximate solution~$y^{i+1}$ instead of~$y(x^i)$.
	Thus, given the current iterate~$(x^i,y^{i+1},z^{i+1})$, the update direction that is used for the~UL problem is the inexact stochastic gradient of~$f$, denoted as~$\tilde g_{f_1}^{i}$ and given by
	\begin{equation}\label{eq:TSG_UL_stoch_dir}
		\tilde g_{f_1}^{i} \;=\; \nabla f (x^i,y^{i+1},z^{i+1};\xi^i).
	\end{equation}
	We again highlight this slight abuse of notation, since $f$ is a function of $(x)$ and not $(x,y,z)$, as we are utilizing the approximations $y^{i+1}\approx y(x^i)$ and $z^{i+1}\approx z(x^i,y(x^i))$ in computing the gradient $\bar f$. It is again for this reason that we adopt the notation $\tilde g_1$ to denote an ``inexact'' SG (as opposed to simply $g_1$, which would denote the ``exact'' SG $\nabla f(x^i)$).
	Letting~$\alpha_i\in(0,1]$ denote the step size for the~UL problem in the~UL iteration~$i$, the update of the~UL variables is given by $x^{i+1} \;=\; x^{i} - \alpha_i \tilde g_{f_1}^{i}$. Finally, the schema of the resulting trilevel stochastic gradient~(TSG) algorithm developed in this paper is given by Algorithm~\ref{alg:TSG}.

	\section{Convergence analysis of the~TSG method}\label{sec:assumptions_convergence_rate_TSG_method}
	Throughout this section, to simplify notation when there are no ambiguities, we will write functions, gradients, Jacobians, and Hessians by omitting their arguments~$(x,y,z)$. When dealing with stochastic estimates, we will replace the arguments~$(x,y,z;\xi)$ with an~$\xi$-superscript. For example, we denote~$\nabla_{zz}^2 f_3^\xi = \nabla_{zz}^2 f_3(x,y,z;\xi)$. It also bears mentioning that in the following assumptions, we will omit the iterates~$(i,j,k)$ for the evaluated point~$(x,y,z)$ and the iterate~$i$ for the step sizes~$\alpha$, $\beta$, and~$\gamma$, as the results are required to hold true for any iterate. For convenience throughout the analysis, we utilize the following composite step-size:
	\begin{equation}\label{eq:theta}
		\theta_i \; := \; \alpha_{i}\beta_{i}\gamma_{i} \quad (\text{or } \theta \; := \; \alpha\beta\gamma \text{ in the general case}).
	\end{equation}

	Further, we define the expectations to be taken over $\sigma$-algebras generated by the sets of the relevant random variables. For simplicity, we define a general $\sigma$-algebra~$\mathcal{F}_\xi$ that includes all the events up to the generation of a general point~$(x,y,z)$, before observing a realization of~$\xi$. Further, $\mathbb{E}[\cdot|\mathcal{F}_\xi]$ denotes the expectation taken with respect to the probability distribution of~$\xi$ given~$\mathcal{F}_\xi$. We will also use~$\mathbb{E}[\cdot]$ to denote the \textit{total expectation}, i.e., the expected value with respect to the joint distribution of all the random variables. For a full description of all $\sigma$-algebras used in the analysis, see Section~\ref{sec:def_sig_algebra} of Appendix~\ref{app:conv_theory_discussion}.

	\subsection{Assumptions on the trilevel problem}\label{subsec:assumptions}
	
	We now provide all of the assumptions that are required for the convergence analysis of Algorithm~\ref{alg:TSG}. It bears mentioning that throughout this paper, we use $\|\cdot\|$ to denote the $\ell_2$-Euclidean norm when dealing with vectors and the spectral norm when dealing with matrices. We begin by imposing Assumption~\ref{as:tri_lip_cont} below which ensures that the functions of interest are differentiable and satisfy appropriate smoothness requirements on the functions, gradients, Jacobians, Hessians, and tensors of third-order derivatives involved in problem~\ref{prob:trilevel}.
	\begin{assumption}[Differentiability and Lipschitz continuity]\label{as:tri_lip_cont}
		The function $f_1$ is once continuously differentiable, $f_2$ is twice continuously differentiable, and $f_3$ is thrice continuously differentiable. Further, the functions $f_1$, $\nabla f_1$, $f_2$, $\nabla f_2$, $\nabla^2 f_2$, $\nabla f_3$, $\nabla^2 f_3$, and $\nabla^3 f_3$ are Lipschitz continuous with constants $L_{f_1}$, $L_{\nabla f_1}$, $L_{f_2}$, $L_{\nabla f_2}$, $L_{\nabla^2 f_2}$, $L_{\nabla f_3}$, $L_{\nabla^2 f_3}$, and $L_{\nabla^3 f_3}$, respectively.
	\end{assumption}
	
	To ensure that problem~\ref{prob:trilevel} is well-defined, Assumptions~\ref{as:strong_conv_f3_z}--\ref{as:strong_conv_fbar_y} below require that the LL function~$f_3$ as well as the true ML function~$\bar f$ are strongly convex. These kind of assumptions are standard in the stochastic approximation literature (e.g., see~\cite{SGhadimi_MWang_2018}) and will guarantee the existence and uniqueness of the ML and LL optimal solutions $y(x)$ and $z(x)$, respectively, for any fixed value of $x$. Further, the constants $\mu_z$ and $\mu_y$ defined in these assumptions are positive.
	\begin{assumption}[Strong convexity of $f_3$ in $z$]\label{as:strong_conv_f3_z}
		For any fixed $x$ and $y$, $f_3$ is $\mu_z$-strongly convex in $z$, i.e., $f_3(x,y,z_1) \; \geq \; f_3(x,y,z_2) + \nabla_z f_3(x,y,z_2)^\top(z_1-z_2) + \frac{\mu_z}{2}\|z_1-z_2\|^2$, for all $(z_1,z_2)$.
	\end{assumption}
	\begin{assumption}[Strong convexity of $\bar f$ in $y$]\label{as:strong_conv_fbar_y}
		For any fixed $x$, $\bar f$ is $\mu_y$-strongly convex in $y$, i.e., $\bar f(x,y_1) \; \geq \; \bar f(x,y_2) + \nabla_y \bar f(x,y_2)^\top(y_1-y_2) + \frac{\mu_y}{2}\|y_1-y_2\|^2$, for all $(y_1,y_2)$. %Notice, that this is equivalent to assuming that $\nabla^2\bar f$ is uniformly bounded away from singularity by $\mu_y$, i.e., $\|\nabla^2\bar f\|\geq \mu_y$, and further that $\|[\nabla^2\bar f]^{-1}\|\leq 1/\mu_y$. This further implies that inverse of the principal sub-matrix $\nabla^2_{yy}\bar f$ is also uniformly upper-bounded by $1/\mu_y$, i.e., $\|[\nabla^2_{yy}\bar f]^{-1}\|\leq 1/\mu_y$.
	\end{assumption}
	In practice, $\bar{f}$ will be strongly convex when~$f_2$ is strongly convex in~$(y,z)$ and~$z(x,y)$ is an affine function in~$(x,y)$. Hence, assuming strong convexity of~$\bar{f}$ covers cases where the~LL problem is a~QP problem or even certain special cases of polynomial functions of even order, such as the squared norm of a quadratic function (see~\eqref{prob:trilevel_synthetic_ll_2} in Subsection~\ref{sec:num_results_synth_trilevel_probs}).
	
	Next, as is standard in the stochastic approximation literature, we require that all stochastic estimates be unbiased with bounded variances and that all random variables that are sampled are independent and identically distributed, stated in Assumption~\ref{as:TSG_unbiased_estimators} below. This ensures that the stochastic terms that are used to approximate the gradients, Hessians, Jacobians, and third-order tensors are reliable approximations of their corresponding deterministic counter-parts. In applications of empirical risk minimization like machine learning, such an assumption can easily be satisfied in practice by taking larger sample sizes when approximating these terms.
	\begin{assumption}[Stochastic estimates]\label{as:TSG_unbiased_estimators}
		The stochastic derivatives~$\nabla f_1^\xi$, $\nabla f_2^\xi$, $\nabla^2 f_2^\xi$, $\nabla f_3^\xi$, $\nabla^2 f_3^\xi$, and~$\nabla^3 f_3^\xi$ are unbiased estimators of $\nabla f_1$, $\nabla f_2$, $\nabla^2 f_2$, $\nabla f_3$, $\nabla^2 f_3$, and $\nabla^3 f_3$, respectively. Further, the variances of the stochastic derivatives are bounded by constants $\sigma_{\nabla f_1}^2$, $\sigma_{\nabla f_2}^2$, $\sigma_{\nabla^2 f_2}^2$, $\sigma_{\nabla f_3}^2$, $\sigma_{\nabla^2 f_3}^2$, and $\sigma_{\nabla^3 f_3}^2$, respectively. Further, all of the random variables $\xi$ that are sampled are independent and idententically distributed (i.i.d.).
	\end{assumption}
	
	Although Assumptions~\ref{as:strong_conv_f3_z}--\ref{as:strong_conv_fbar_y} ensure that the Hessian sub-matrices $\nabla_{zz}^2 f_3$ and $\nabla_{yy}^2 \Bar{f}$ are bounded away from singularity, we also require that their stochastic estimates be bounded away from singularity, stated as Assumption~\ref{as:bound_on_hess_inv_true_MLP} below, which ensures that these estimates provide a robust measure of the curvature of the functions $f_3$ and $\bar f$.
	\begin{assumption}[Uniform bound on inverted stochastic Hessians]\label{as:bound_on_hess_inv_true_MLP}
		The stochastic principal sub-matrices $[ \nabla_{zz}^2 f_3^\xi ]^{-1}$ and $[ \nabla_{yy}^2 \Bar{f}^\xi ]^{-1}$ are upper-bounded in norm at all points by the positive constants $b_{zz}$ and $b_{yy}$, respectively. 
	\end{assumption}
	In the stochastic gradient literature concerning second-order derivatives, it is common to assume that a Hessian matrix, stochastic or not, is uniformly bounded below~\cite{RBollapragada_etal_2018}, implying that its inverse is uniformly bounded above. The motivation is that, if the Hessian matrix is not uniformly bounded below, a regularization term can be added to such a matrix to ensure it is non-singular.
	
	Lastly, Assumption~\ref{as:TSG_bounded_var} below is imposed to ensure that the bias of the inverted stochastic estimates~$[\nabla_{zz}^2 f_3^\xi]^{-1}$ and $[ \nabla_{yy}^2 \Bar{f}^{\xi} ]^{-1}$ approach zero on the order $\mathcal{O}(\theta)$. It is known that such an assumption can be satisfied in practice, e.g., by utilizing a truncated-Neumann series (see~\cite{SGhadimi_MWang_2018}) and incrementally increasing the number of samples used when approximating the terms $\nabla_{zz}^2 f_3$ and $\nabla_{yy}^2 \Bar{f}$ (the authors in~\cite{TChen_YSun_WYin_2021_closingGap} utilize such a property to establish a similar bound; though they do not state it as an assumption, but instead leave the number of samples as a parameter in their analysis that they choose to yield their desired convergence result). 
	\begin{assumption}[Bounded bias of inverted stochastic Hessians]\label{as:TSG_bounded_var}
		The stochastic principal sub-matrices $[\nabla_{zz}^2 f_3^\xi]^{-1}$ and $[ \nabla_{yy}^2 \Bar{f}^{\xi} ]^{-1}$ are estimators of $[\nabla_{zz}^2 f_3]^{-1}$ and $[ \nabla_{yy}^2 \Bar{f} ]^{-1}$, respectively, with biases that are bounded on the order of $\mathcal{O}( \theta )$, i.e., there exist positive constants $W_{zz}$ and $W_{yy}$ such that $\|[ \nabla_{zz}^2 f_3 ]^{-1} - \mathbb{E}[ [ \nabla_{zz}^2 f_3^\xi ]^{-1} \vert \mathcal{F}_\xi ] \| \leq W_{zz}\theta$ and $\| [ \nabla_{yy}^2 \Bar{f}]^{-1} - \mathbb{E}[ [ \nabla_{yy}^2 \Bar{f}^{\xi}]^{-1} \vert \mathcal{F}_\xi ] \| \leq W_{yy}\theta$, respectively.
	\end{assumption}
	
	%Lastly, due to the complex nature and the presence of several inverted matrices in the curvature of the true ML function $\bar f$ (see~\eqref{hess_yy_fbar} in Appendix~\ref{app:prop:adjoint_grad}), we require Assumption~\ref{as:TSG_bounded_var_inv_hess_f_bar} below, which ensures that the bias of the inverted stochastic sub-matrix $\nabla_{yy}^2 \Bar{f}^{\xi}$ is bounded on the order of $\mathcal{O}( \theta )$. 
	%\begin{assumption}[Bounded bias of inverted Hessian of $\bar f$]\label{as:TSG_bounded_var_inv_hess_f_bar}
	%    The stochastic principal sub-matrix $[ \nabla_{yy}^2 \Bar{f}^{\xi} ]^{-1}$ is an estimator of $[ \nabla_{yy}^2 \Bar{f} ]^{-1}$, with a bias that is bounded on the order of $\mathcal{O}( \theta )$, i.e., there exists a positive constant $W_{yy}$ such that $\| \mathbb{E}[ [ \nabla_{yy}^2 \Bar{f}^{\xi}]^{-1} \vert \mathcal{F}_\xi ] - [ \nabla_{yy}^2 \Bar{f}]^{-1} \| \leq W_{yy}\theta$.
	%\end{assumption}

	\subsection{Convergence of the TSG method}\label{sec:TSG_convergence}
	
	We now present the convergence result of Algorithm~\ref{alg:TSG} in Theorem~\ref{th:TSG_convergence} below, in which we consider the general case where the true UL function $f$ is possibly nonconvex. For a full discussion of the analysis, see Appendix~\ref{app:conv_theory_discussion}. Further, for the full proof of this theorem, see Appendix~\ref{app:th:TSG_convergence}.
	\begin{theorem}[Convergence of TSG -- Nonconvex $f$]\label{th:TSG_convergence}
		Under Assumptions~\ref{as:tri_lip_cont}--\ref{as:TSG_bounded_var}, choose the step-sizes $\alpha_i = 1/\sqrt{I}$, $\beta_i = (1/\sqrt{J}) \alpha_i$, and $\gamma_i = (1/(\sqrt{J}\sqrt{K}))\alpha_i$.
		%\[\alpha_i \;=\; \frac{1}{\sqrt{I}}, \qquad\qquad \beta_i = \frac{1}{\sqrt{J}} \alpha_i, \qquad\qquad \gamma_i = \frac{1}{\sqrt{J}\sqrt{K}}\alpha_i.\]
		Then, the iterates $\{x^i\}_{i\geq0}$ generated by Algorithm~\ref{alg:TSG} satisfy $\frac{1}{I}\sum_{i=0}^{I-1} \mathbb{E}[ \|\nabla f(x^i)\|^2 ] = \mathcal{O}( J/\sqrt{I} )$,
		%\[
		%\frac{1}{I}\sum_{i=0}^{I-1} \mathbb{E}[ \|\nabla f(x^i)\|^2 ] \; = \; \mathcal{O}\left( \frac{J}{\sqrt{I}} \right),
		%\]
		when choosing any $I\in\mathbb{N}_+$, $J\in\mathbb{N}_+$, and $K\in\mathbb{N}_+$ such that $\varsigma\leq J$, $\varpi\leq I$, and $\Xi(I,J)=\mathcal{O}(J^3I)\leq K$,
		%\[\varsigma\leq J, \quad\quad \varpi\leq I, \quad\quad \Xi(I,J)=\mathcal{O}(J^3I)\leq K,\]
		where $\varsigma\in\mathbb{R}_+$ is defined by~\eqref{eq:varsigma_choice}, $\varpi\in\mathbb{R}_+$ is defined by~\eqref{eq:varpi_def}, and $\Xi(I,J):\mathbb{N}_+\times\mathbb{N}_+\rightarrow \mathbb{R}_+$ is defined by~\eqref{eq:Xi_def}, all in Appendix~\ref{app:th:TSG_convergence}.
	\end{theorem}
	We state this theorem as our primary convergence result due to the intuitive choice of its step-sizes and the appeal to its direct implementability. It bears mentioning that a tighter rate, matching the best that has been derived for general nonconvex bilevel optimization (see~\cite{TChen_YSun_WYin_2021_closingGap}), can be derived by choosing more complex step-sizes dependent on unknown Lipschitz constants, as stated in Remark~\ref{rem:alternate_conv_result} below.
	\begin{remark}\label{rem:alternate_conv_result}
		Under Assumptions~\ref{as:tri_lip_cont}--\ref{as:TSG_bounded_var}, when choosing the step-sizes $\alpha_i$, $\beta_i$, and $\gamma_i$ to incorporate more problem-specific information, a stronger convergence rate of $\frac{1}{I}\sum_{i=0}^{I-1} \mathbb{E}[ \|\nabla f(x^i)\|^2 ] = \mathcal{O}( 1/\sqrt{I} )$ can be obtained. Such a result also does not require lower-bounds on the UL or ML variables $I$ and $J$, but requires $K\geq\mathcal{O}(J^4I)$. The formal statement of this result is given by Theorem~B.2.5 in the PhD thesis~\cite{GKent_2025}.
	\end{remark}
	
	Notice that both of these results share a common constraint: the LL iterations $K$ must scale linearly in $I$ and polynomially in $J$. We argue that such a requirement follows intuitively, as the accuracy of the LL solution directly impacts the inexactness of the bilevel adjoint gradient for the ML problem. Further, this constraint reveals the hierarchical interplay within trilevel problems, i.e., more LL iterations are required to obtain a higher accuracy in the ML problem than in the UL problem. This implies that the trilevel adjoint gradient $\nabla f$ tolerates more inexactness from the ML problem than the bilevel adjoint gradient $\nabla_y \bar f$ does from the LL problem. Such a relationship underscores how errors in the LL propagate upward through the levels: greater accuracy at any sub-upper level necessitates significantly higher precision in the LL solution. Whether this pattern extends to all sup-upper levels in general multi-level problems or entirely shifts the computational burden to the lowest level remains an open question for future research. Lastly, we highlight that the extra $J$ present in the iteration complexity on $K$ in Remark~\ref{rem:alternate_conv_result} (i.e., $K\geq\mathcal{O}((J\times J^3)I)$ can be thought of as the $J$ that is present in the numerator of the convergence bound $\mathcal{O}(J/\sqrt{I})$ from Theorem~\ref{th:TSG_convergence}.%; thus, by choosing more specific step-sizes, we are able to shift the burden of guaranteeing convergence from requiring bounds on both $I$ and $K$ to simply a bound solely on $K$.

	%\section{A heuristic approach}
	
	%\section{Handling third-order derivatives}
	
	\section{Numerical experiments}
	
	The experimental results were obtained on a desktop workstation with~128GB of~RAM, an Intel(R) Core(TM) i9-13950HX processor (24 cores, 32 threads) running at~2200 MHz under Windows 11. 
	% and an~NVIDIA~RTX~5000 Ada Generation Laptop~GPU. 
	Our code is available at~\url{https://github.com/GdKent/TSG}.
	% Our code is available at~\url{https://github.com/} \tcr{TO DO}.
	
	%%%%%%%%%%%%%%%%%%%%%%%%%%%%%%%%%%%%%%%%%%%%%%%%%%%%%%%%%%%%%%%%%%%%%%%%%%%%%%%%%%%%%%%%%
	\subsection{Our practical~TSG methods}
	%%%%%%%%%%%%%%%%%%%%%%%%%%%%%%%%%%%%%%%%%%%%%%%%%%%%%%%%%%%%%%%%%%%%%%%%%%%%%%%%%%%%%%%%%
	
	A major difficulty in the adjoint gradient~\eqref{adjoint} is the need for second-order derivatives of~$\bar{f}$ (a challenge that also arises in the adjoint gradient of a~BLO problem), and, in particular, the presence of third-order derivative tensors in~$\nabla_{yx} \bar{f}$ and~$\nabla_{yy} \bar{f}$ in~\eqref{hess_yx_fbar} and~\eqref{hess_yy_fbar}, due to~\eqref{eq:partial_x} and~\eqref{eq:partial_y}, respectively. We consider two approaches to address this issue (see Appendix~\ref{subsec:computing_TSG_adj_grad_inexactly}), leading to two practical versions of the~TSG method, referred to as~TSG-N-FD and~TSG-AD. In the numerical experiments, we are mainly interested in testing these two practical implementations (Algorithms~\ref{alg:TSG-N-FD} and~\ref{alg:TSG-AD} in Appendices~\ref{subsec:TSG-N-FD} and~\ref{subsec:TSG-AD}, respectively) rather than the method we refer to as~TSG-H, which uses the true Hessians and third-order derivative tensors (Algorithm~\ref{alg:TSG} in Section~\ref{sec:trilevel_TSG}).
	
	The first algorithm we propose, TSG-N-FD, is based on the adjoint equation approach and involves solving any adjoint system arising in~\eqref{adjoint} and~\eqref{eq:nablay_barf} by using the linear~CG method, where each Hessian-vector product is approximated via a finite-difference~(FD) scheme. 
	When using~TSG-H, we will apply the linear~CG method to solve any adjoint system arising in~\eqref{adjoint} and~\eqref{eq:nablay_barf} until non-positive curvature is detected.
	The second algorithm we propose, TSG-AD, is based on the truncated Neumann series approach and consists of approximating each Hessian-vector product by using automatic differentiation~(AD).
	Note that~TSG-H is not suited for practical optimization problems, but we include it in the experiments for completeness. For very large problems, one must use~TSG-N-FD or~TSG-AD. 
	
	% In the~TSG-N-FD, TSG-AD, and~TSG-H versions of the TSG method, \tcg{we follow the schema in Algorithm~\ref{alg:TSG}.} 
	% \tcr{f}or the solution of the~LL~problem, we consider an increasing accuracy strategy, which consists of increasing~$K$ by~1 every time the difference of the~ML~objective function~$f_2$ between two consecutive~ML iterations is less than a given threshold. Similarly, for the solution of the~ML problem, the increasing accuracy strategy consists of increasing~$J$ by~1 every time the difference of the~UL~objective function~$f_1$ between two consecutive~UL iterations is less than a given threshold.
	
	% To determine the number of iterations~$J$ and~$K$ \tcr{used by the algorithms} for the ML and LL problems, respectively, we adopted an increasing accuracy strategy, similar to the one used for bilevel optimization in~\cite{TGiovannelli_GKent_LNVicente_2022}. Specifically, the number of ML iterations increases by one whenever the difference in~$f_1$ values between two consecutive~UL iterations falls below a given threshold~(i.e., $10^{-2}$). Likewise, the number of~LL iterations increases by one whenever the difference in~$f_2$ values between two consecutive~ML iterations is below a different threshold~(i.e., $10^{-1}$).
	
	To determine the~ML and~LL iteration iterations~$J$ and~$K$, we used an increasing accuracy strategy inspired by~\cite{TGiovannelli_GKent_LNVicente_2022}: the number of~ML iterations increases by one when the change in~$f_1$ between two consecutive~UL iterations drops below~$10^{-2}$, and the number of~LL iterations increases similarly when the change in~$f_2$ between two consecutive~ML iterations drops below~$10^{-1}$.
	
	%%%%%%%%%%%%%%%%%%%%%%%%%%%%%%%%%%%%%%%%%%%%%%%%%%%%%%%%%%%%%%%%%%%%%%%%%%%%%%%%%%%%%%%%%
	\subsection{Numerical results for synthetic trilevel problems}\label{sec:num_results_synth_trilevel_probs}
	%%%%%%%%%%%%%%%%%%%%%%%%%%%%%%%%%%%%%%%%%%%%%%%%%%%%%%%%%%%%%%%%%%%%%%%%%%%%%%%%%%%%%%%%%
	
	We first report results for two synthetic trilevel problems that differ in their~LL problem formulations (see Appendix~\ref{subsec:synthetic_trilevel_problems}). In the first, all levels have quadratic objective functions, leading to a quadratic trilevel problem (with zero third-order derivatives). In the second, the~UL and~ML objective functions are quadratic, while the~LL objective is quartic (resulting in non-zero third-order derivatives). For simplicity, we refer to the second trilevel problem as quartic.
	% , even though only the last level involves a quartic objective. 
	
	% For both the quadratic and quartic problems, we present figures for \tcr{the} deterministic and stochastic cases.
	
	Figures~\ref{fig:quad_prob_det}, \ref{fig:quart_prob_det}, \ref{fig:quad_prob_stoch}, and~\ref{fig:quart_prob_stoch} compare the sequences of~$f(x^i)$ values obtained by~TSG-H, TSG-N-FD, and~TSG-AD over~UL iterations and running time. 
	% In the deterministic case, we also include Figures~\ref{fig:quad_prob_det_breakdown} and~\ref{fig:quart_prob_det_breakdown} \tcr{(see Appendix~\ref{subsubsec:synthetic_trilevel_problems_additional_figs})}, which break down the behavior of each algorithm at the~UL, ML, and~LL levels. 
	In the stochastic case, we computed the stochastic gradients and Hessians by adding Gaussian noise with mean zero to the corresponding deterministic quantities. We did not add noise to the third-order tensors, as these are not used in the practical algorithms~TSG-N-FD and TSG-AD.
	All figures involving stochasticity include 95\% confidence intervals computed using the t-distribution over~10 runs.
	
	\begin{figure}[t]
		\centering
		\begin{minipage}[b]{0.48\linewidth}
			\centering
			\includegraphics[width=0.49\linewidth]{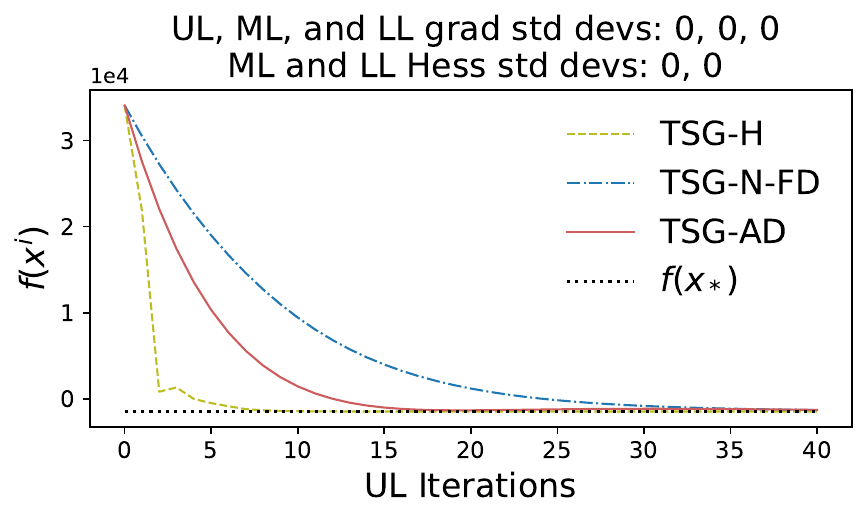}
			\includegraphics[width=0.49\linewidth]{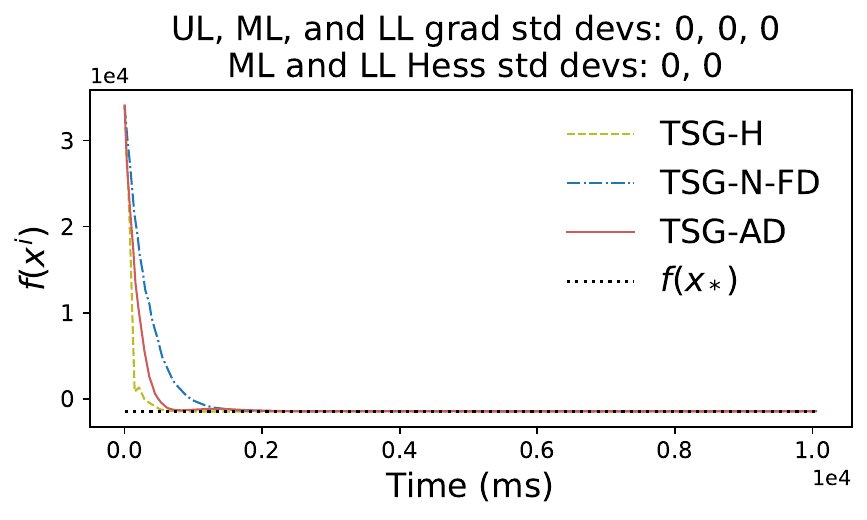}      
			\captionof{figure}{Quadratic problem, deterministic case.}
			\label{fig:quad_prob_det}
		\end{minipage}
		\hfill
		\begin{minipage}[b]{0.48\linewidth}
			\centering
			\includegraphics[width=0.49\linewidth]{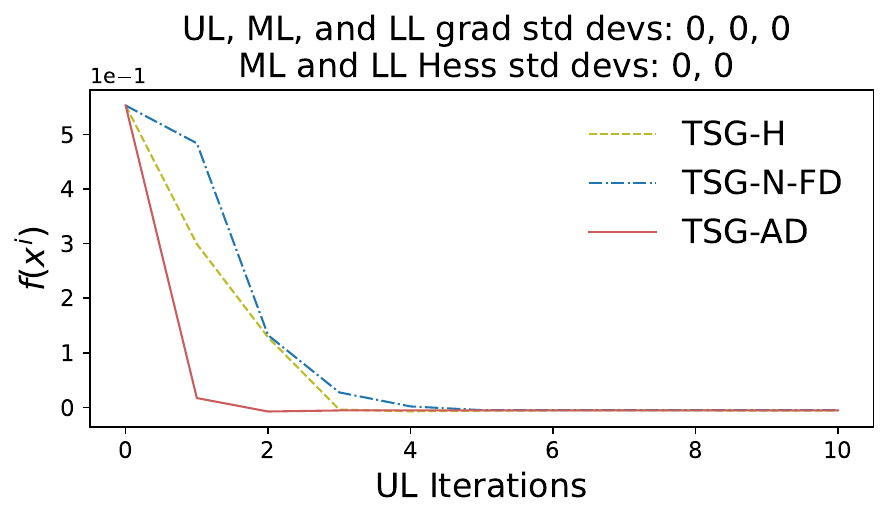}
			\includegraphics[width=0.49\linewidth]{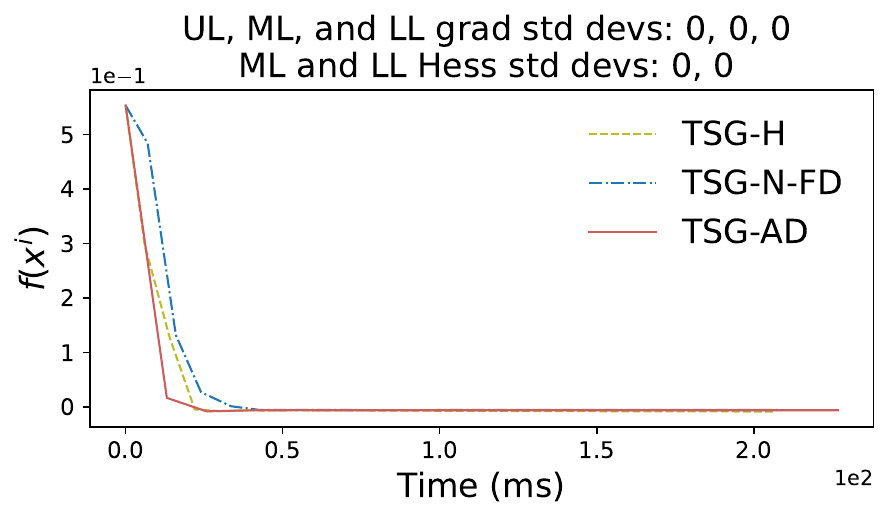}
			\captionof{figure}{Quartic problem, deterministic case.}
			\label{fig:quart_prob_det}
		\end{minipage}
	\end{figure}
	
	% \begin{figure}
		% \centering
		%       \includegraphics[scale=0.23]{Figures/Quadratic2_dim505050_det_stepsize111-comparison-iter.png}
		%       \includegraphics[scale=0.23]{Figures/Quadratic2_dim505050_det_stepsize111-comparison-time.png}      
		%     \caption{Quadratic problem, deterministic case.}\label{fig:quad_prob_det}
		% \end{figure}
	
	% \begin{figure}
		% \centering
		%     \includegraphics[scale=0.25]{Figures/Quadratic2_dim505050_det_stepsize111-TSGH-func.png}
		%     \includegraphics[scale=0.25]{Figures/Quadratic2_dim505050_det_stepsize111-TSGNFD-func.png}
		%     \includegraphics[scale=0.25]{Figures/Quadratic2_dim505050_det_stepsize111-TSGAD-func.png}                    
		%     \caption{Breakdown of the algorithms, quadratic problem, deterministic case.}\label{fig:quad_prob_det_breakdown}
		% \end{figure}
	
	\begin{figure}
		\centering
		\includegraphics[scale=0.23]{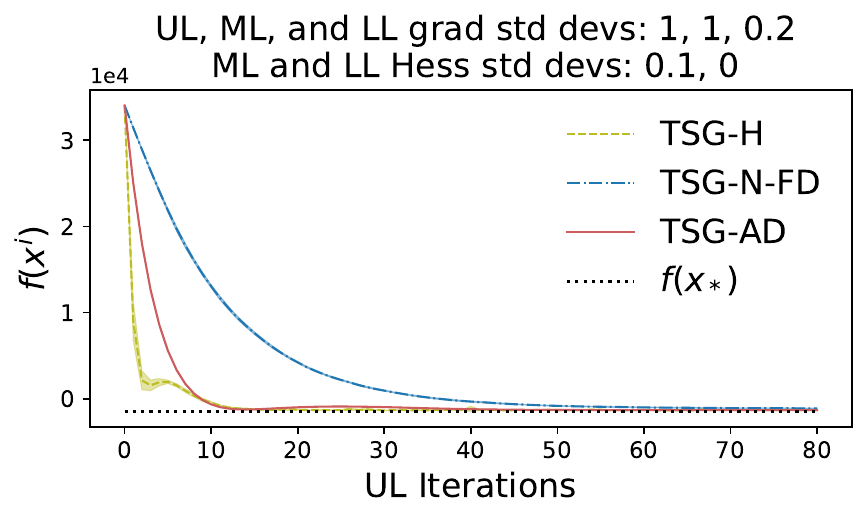}
		\includegraphics[scale=0.23]{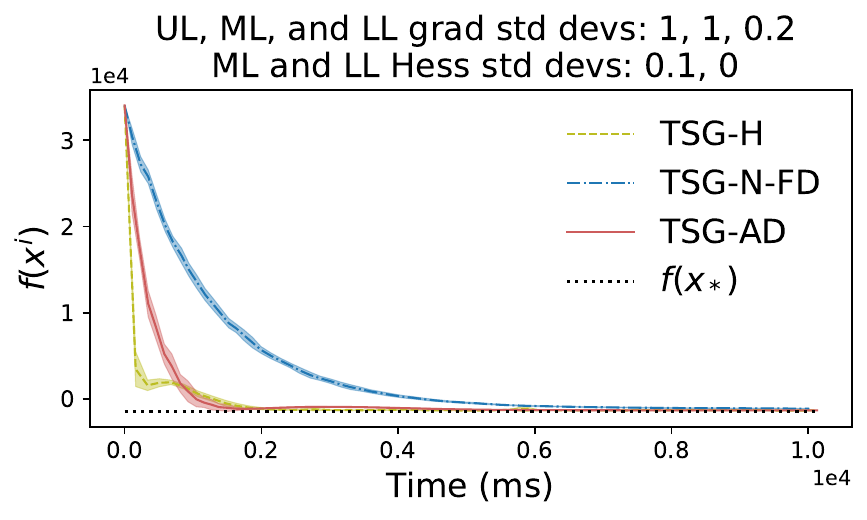}
		\vspace{0.2cm}
		\includegraphics[scale=0.23]{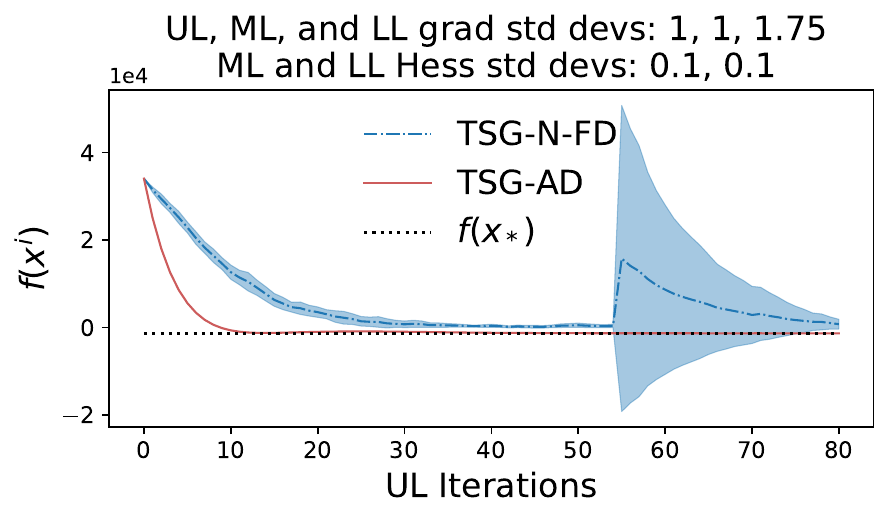}
		\includegraphics[scale=0.23]{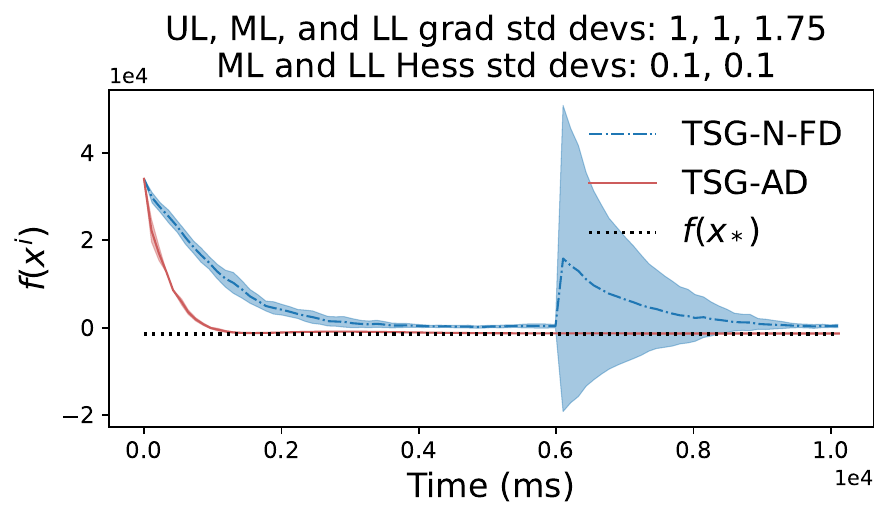}                   
		\caption{Quadratic problem, stochastic case (low noise: two left plots; high noise: two right plots).}\label{fig:quad_prob_stoch}
	\end{figure}
	
	For the quadratic problem, Figure~\ref{fig:quad_prob_det} shows that~TSG-H, which uses Hessians and third-order tensors, outperforms~TSG-N-FD and~TSG-AD in terms of both~UL iterations and time in the deterministic case. Figure~\ref{fig:quad_prob_stoch} shows the plots for the stochastic case. Note that TSG-N-FD and TSG-AD are not affected by the noise in the Hessians of~$f_2$ and~$f_3$, as they rely only on first-order derivatives. TSG-H is highly sensitive to the standard deviation of the Hessian of~$f_3$ (which appears in the trilevel adjoint gradient~\eqref{adjoint}), and its performance deteriorates significantly when this value exceeds~$0.1$. Such behavior aligns with the well-known fact that stochastic Hessians require lower noise levels (i.e., larger mini-batch sizes when noise arises from sampling finite-sum Hessians in~SG contexts) than stochastic gradients to perform well~\cite[Section 6.1.1]{LBottou_FECurtis_JNocedal_2018}.
	% \tcr{(see Appendix~\ref{subsubsec:synthetic_trilevel_problems_additional_figs} for a discussion)}. 
	For this reason, we omit~TSG-H from the two right plots. As noise levels increase, the performance of~TSG-N-FD deteriorates, whereas~TSG-AD remains more robust. The most critical source of noise for~TSG-N-FD is that added to~$\nabla f_3$, which is used to approximate the matrix-vector products involving the Hessian of~$f_3$ via the~FD scheme in~\eqref{eq:TSGNFD_FD0}. Note that such an~FD scheme affects the computation of both~\eqref{adjoint_NFD} and~\eqref{eq:nablay_barf_NFD}.
	
	For the quartic problem, in the deterministic case, Figure~\ref{fig:quart_prob_det} shows that~TSG-H is the least competitive algorithm in terms of time, as the computation of third-order tensors slows it down. In the stochastic case, shown in Figure~\ref{fig:quart_prob_stoch}, increasing noise levels lead to performance deterioration for both~TSG-N-FD and~TSG-H, whereas~TSG-AD remains the most robust. We can conclude that when third-order derivatives are non-zero, the~FD approximations used in~TSG-N-FD become less accurate.  
	
	% \begin{figure}
		% \centering
		%   \includegraphics[scale=0.23]{Figures/Quartic2_dim551_det_stepsize111-comparison-iter.png}
		%   \includegraphics[scale=0.23]{Figures/Quartic2_dim551_det_stepsize111-comparison-time.png}                    
		%     \caption{Quartic problem, deterministic case.}\label{fig:quart_prob_det}
		% \end{figure}
	
	% \begin{figure}
		% \centering
		%     \includegraphics[scale=0.25]{Figures/Quartic2_dim551_det_stepsize111-TSGH-func.png}
		%     \includegraphics[scale=0.25]{Figures/Quartic2_dim551_det_stepsize111-TSGNFD-func.png}
		%     \includegraphics[scale=0.25]{Figures/Quartic2_dim551_det_stepsize111-TSGAD-func.png}                    
		%     \caption{Breakdown of the algorithms, quartic problem, deterministic case.}\label{fig:quart_prob_det_breakdown}
		% \end{figure}
	
	\begin{figure}
		\centering
		\includegraphics[scale=0.22]{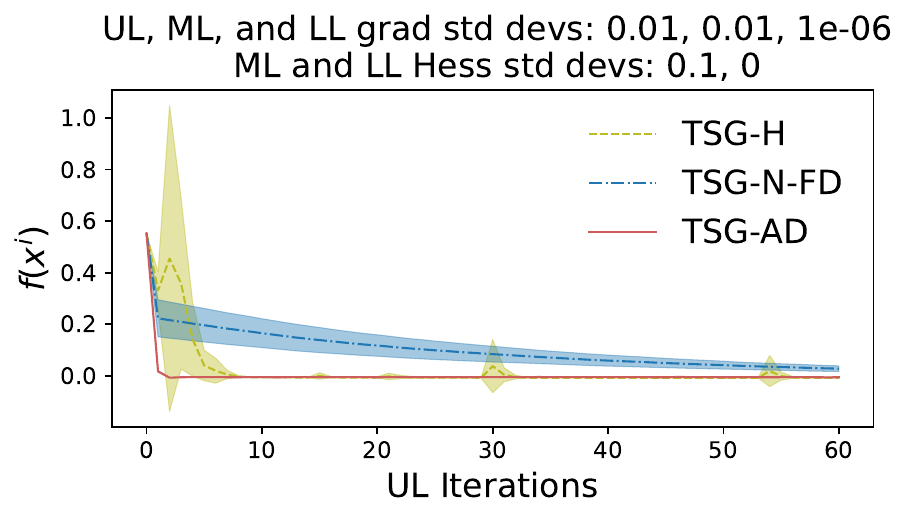}
		\includegraphics[scale=0.22]{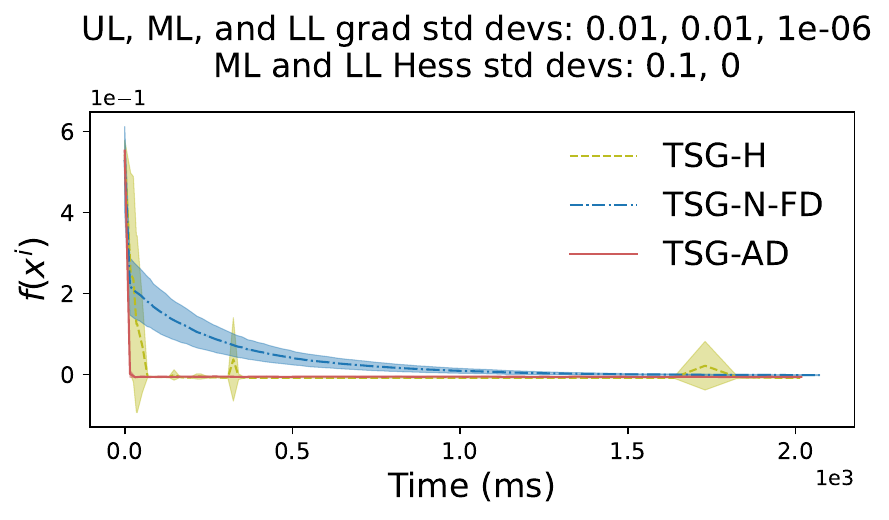}
		\vspace{0.2cm}
		\includegraphics[scale=0.22]{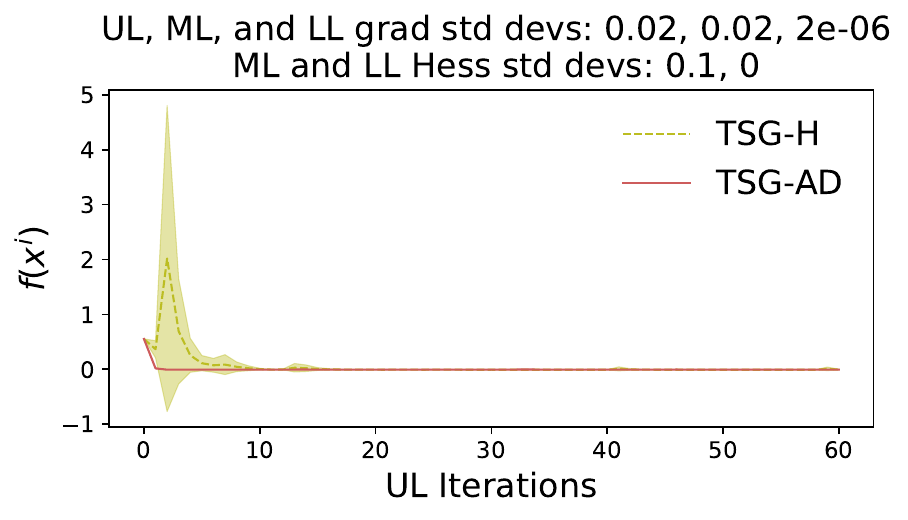}
		\includegraphics[scale=0.22]{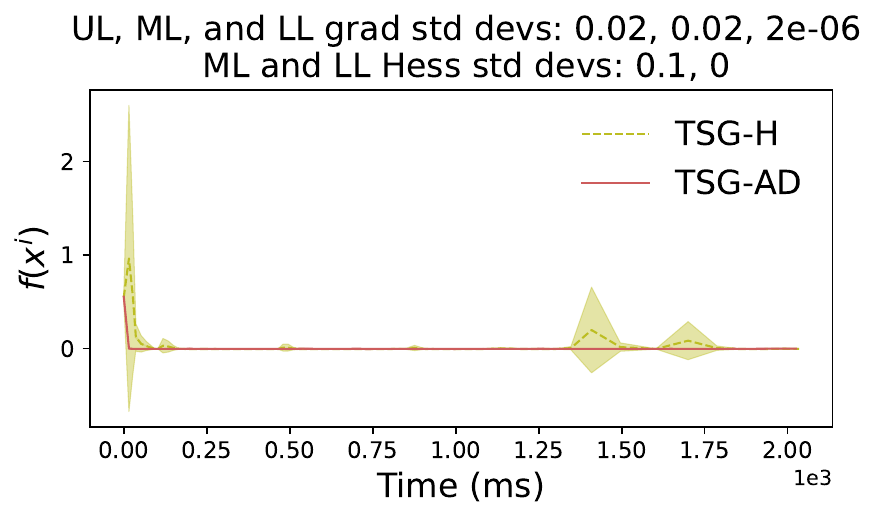}                   
		\caption{Quartic problem, stochastic case (low noise: two left plots; high noise: two right plots).}\label{fig:quart_prob_stoch}
	\end{figure}

	\subsection{Numerical results for trilevel adversarial hyperparameter tuning}
	
	In the~TLO formulation we propose for adversarial hyperparameter tuning (see problem~\ref{prob:trilevel_adv_train} in Appendix~\ref{subsec:trilevel_adversarial_problems} for the rigorous formulation), the~UL problem aims to minimize the validation loss over a regularization parameter used in the training loss, the~ML problem minimizes the training loss over the model parameters, and the~LL problem is posed on the variables that perturb the data in a worst-case fashion. In the formulation proposed in~\cite{RSato_etal_2021}, the~ML and~LL problems are swapped compared to our formulation in problem~\eqref{prob:trilevel_adv_train}. We adopt~\eqref{prob:trilevel_adv_train} because it more accurately reflects the original minimax formulation for adversarial training~\eqref{prob:adv_prob_minmax}, and indeed leads to improved performance (see Appendix~\ref{subsubsec:trilevel_adversarial_additional_figs}).
	We will also evaluate~BLO formulations obtained by removing either the~UL or~LL problem from~\eqref{prob:trilevel_adv_train}. Removing the~UL problem yields a~BLO problem similar in spirit to the original minimax formulation of adversarial learning, while removing the~LL problem results in a~BLO problem for hyperparameter tuning without adversarial learning.
	
	The~BLO problems obtained from~\eqref{prob:trilevel_adv_train} are solved using corresponding bilevel algorithms (denoted as~BSG-AD) derived from~TSG-AD. Such algorithms are essentially equivalent to the well-known~StocBiO~\cite{KJi_JYang_YLiang_2020}. In this section, we will not test~TSG-H, as it requires second and third-order derivatives, which are impractical to compute in applications involving large-scale datasets. Similarly, we will not test the trilevel algorithm proposed in~\cite{RSato_etal_2021}, as it is designed specifically for the deterministic setting. When using~\eqref{prob:trilevel_adv_train}, TSG-N-FD does not perform well and is therefore excluded from further analysis (see Appendix~\ref{subsubsec:trilevel_adversarial_additional_figs} for a discussion). 
	
	For the experiments, we consider three popular tabular datasets: the red and white wine quality datasets~\cite{PCortez_ALCerdeira_2009} and the California housing dataset~\cite{RKPace_RBarry_1997}. To assess the performance of the algorithms and formulations on these datasets, we compute the test~MSE after adding Gaussian noise (with a standard deviation of~5) to the features of the test data, averaged over~100 realizations of the noise. The optimal solution obtained from the trilevel formulation~\eqref{prob:trilevel_adv_train} is expected to yield a model robust to such noise. 
	% than those obtained from formulations that do not incorporate adversarial learning.

	The results for our~TLO formulation in~\eqref{prob:trilevel_adv_train}, along with those for the~BLO formulations obtained by removing the~UL and~LL problems from~\eqref{prob:trilevel_adv_train}, are shown in Figures~\ref{fig:nonSato_redwine}--\ref{fig:nonSato_california}. %When the standard deviation of the Gaussian noise added to the test data is~0, all three formulations yield similar test~MSE values.
	% , comparable to those for the formulation in~\cite{RSato_etal_2021} (see the top plots in Figure~\ref{fig:Sato_redwine}).
	The TLO formulation in~\eqref{prob:trilevel_adv_train} proves to be the most consistently effective for adversarial hyperparameter tuning, with the BLO variants demonstrating competitive runtime but greater sensitivity to the nature of the dataset, reflected in the contrasting dependencies observed across the datasets. 
	In fact, the superior performance of BSG-AD (without LL) over BSG-AD (without UL) on the red and white wine datasets is an indication of the reliance of these datasets on hyperparameter tuning, whereas the inverted performance of the BSG algorithms on the California housing dataset is a symptom of this dataset's dependence on adversarial learning.
	%In fact, the superior performance of BSG-AD (without LL) over BSG-AD (without UL) on the red and white wine dataset is an indication of its reliance on hyperparameter tuning, whereas the opposite performance on the California housing dataset is a symptom of its dependence on adversarial learning. 
	%Specifically, the red and white wine datasets seem to rely more heavily on hyperparameter tuning, explaining the comparable performance of BSG-AD (without LL) with TSG-AD, while the California housing dataset seems to be more reliant on adversarial training, leading to the reversed pattern between the datasets. 
	Overall, TSG-AD, which leverages both adversarial and hyperparameter tuning components during model training, consistently yields the most robust performance across all the tested datasets and will likely deliver further performance improvements in settings where both components are jointly critical.

	\begin{figure}[t]
		\centering
		\begin{minipage}[b]{0.48\linewidth}
			\centering
			\includegraphics[width=0.49\linewidth]{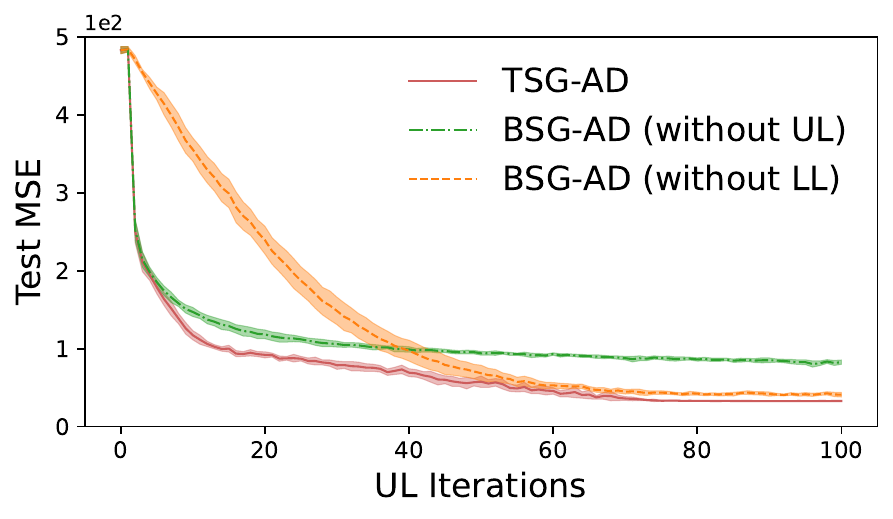}
			\includegraphics[width=0.49\linewidth]{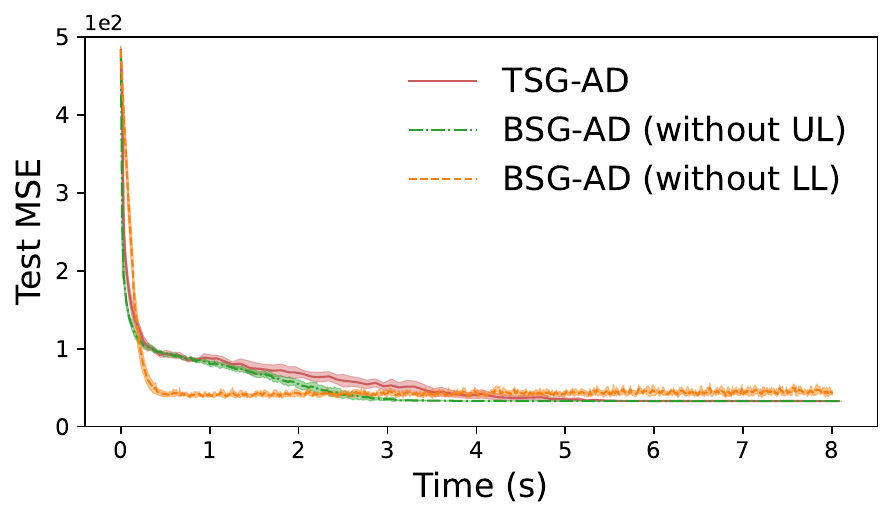}
			\caption{Adversarial learning formulation~\eqref{prob:trilevel_adv_train}, red wine quality dataset.}\label{fig:nonSato_redwine}
		\end{minipage}
		\hfill
		\begin{minipage}[b]{0.48\linewidth}
			\centering
			\includegraphics[width=0.49\linewidth]{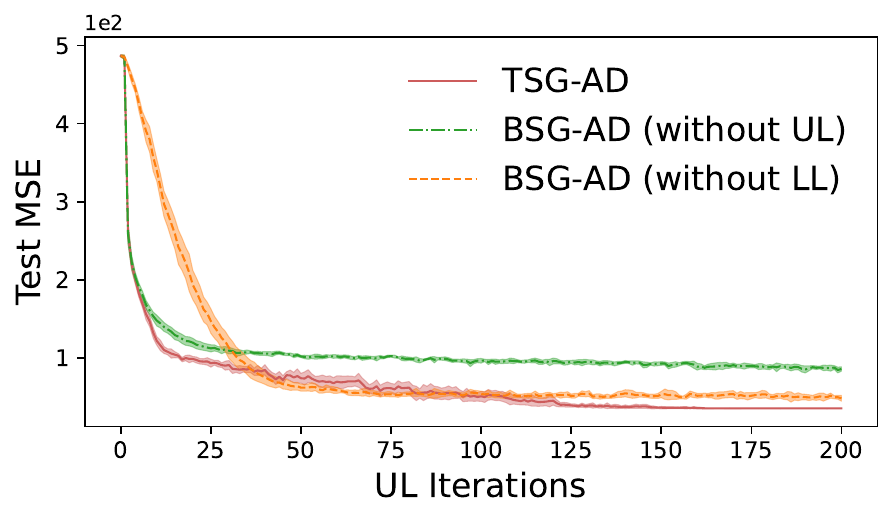}
			\includegraphics[width=0.49\linewidth]{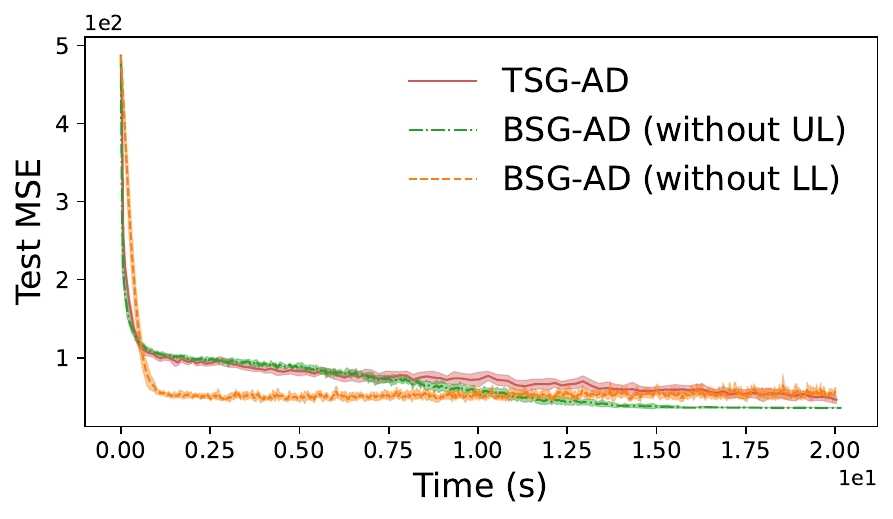}                 
			\caption{Adversarial learning formulation~\eqref{prob:trilevel_adv_train}, white wine quality dataset.}\label{fig:nonSato_whitewine}
		\end{minipage}
	\end{figure}

	\begin{wrapfigure}{r}{0.50\textwidth}
		\centering
		\includegraphics[width=0.49\linewidth]{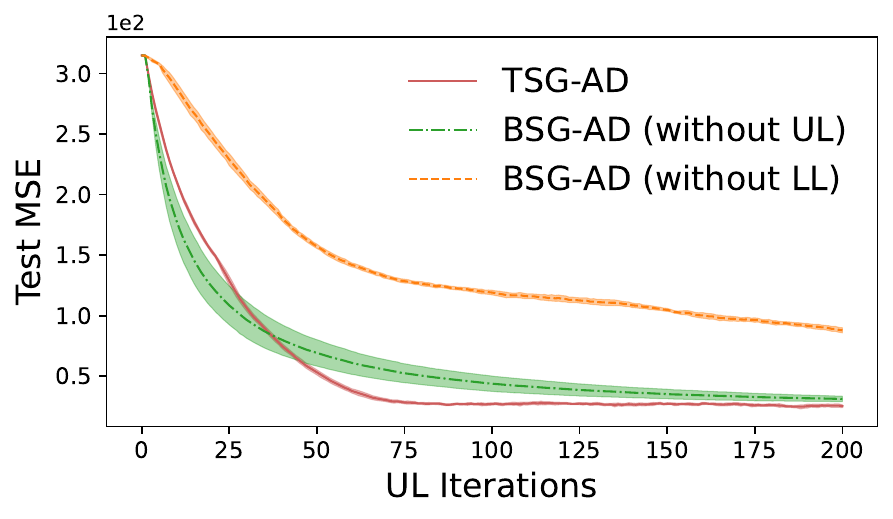}
		\includegraphics[width=0.49\linewidth]{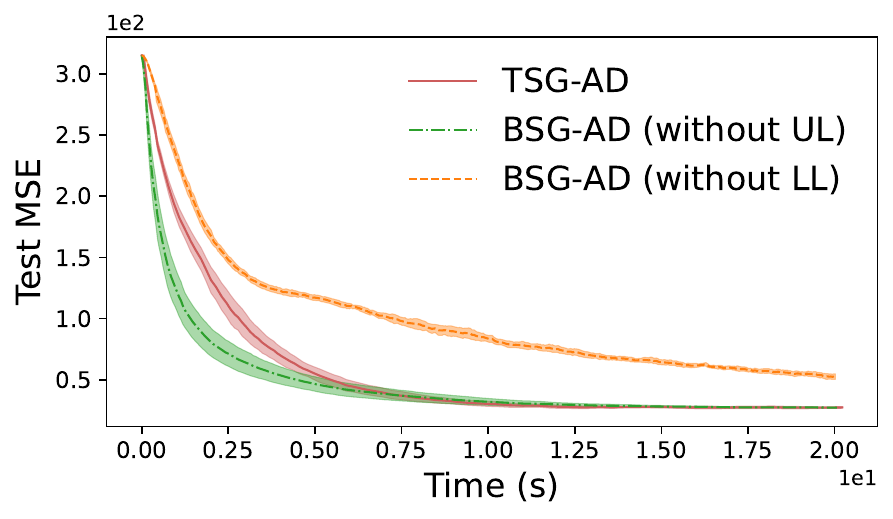}              
		\caption{Adversarial learning formulation~\eqref{prob:trilevel_adv_train}, California housing dataset.}
		\label{fig:nonSato_california}
	\end{wrapfigure}

	\section{Conclusion}
	
	In this paper, we proposed the first stochastic first-order method for trilevel optimization along with a rigorous convergence theory for the nonconvex setting. The proposed theory also covers all forms of inexactness that arise within the trilevel adjoint gradient, such as the inexact solutions of the middle and lower-level problems, inexact computation of the trilevel adjoint formula, and noisy estimates of the gradients, Hessians, Jacobians, and tensors of third-order derivatives involved. 
	Our experiments demonstrate that the proposed~TLO formulation can be more robust than the~BLO formulations corresponding to its~UL and~ML (i.e., hyperparameter tuning without adversarial learning), or its~ML and~LL (i.e., the original minimax adversarial training), as well as the~TLO formulation in~\cite{RSato_etal_2021}, where the~ML and~LL are swapped compared to ours.
	A natural direction left for future research lies in thoroughly exploring how the accuracy at any given intermediate level relates to the precision required at lower levels within general multi-level optimization problems. Specifically, such an investigation would seek to clarify whether increasing the accuracy at a particular level necessitates higher precision at all subsequent lower levels, or if the computational burden entirely shifts to the lowest level.
	
	Potential limitations of our work include the strong convexity assumptions on the~ML and~LL objective functions, which constrain the applicability of our~TSG method. Following directions similar to those emerging in the~BLO literature, one avenue to relax such assumptions would be to explore penalization techniques that allow for non-convex objectives at lower levels. Furthermore, although our experiments on trilevel adversarial hyperparameter tuning demonstrate that a~TLO formulation can outperform a~BLO formulation in terms of iterations, BLO formulations remain competitive in terms of running time. Finally, we only evaluated the algorithms on regression tasks with tabular data, but we expect similar performance on other tasks, such as image classification.
	
	% Nevertheless, we believe our work lays the groundwork for developing new, more efficientTLO formulations and algorithms for machine learning.

	\section*{Acknowledgments}
	This work is partially supported by the U.S. Air Force Office of Scientific Research (AFOSR) award FA9550-23-1-0217 and the U.S. Office of Naval Research~(ONR) award~N000142412656.
	
	\newpage
	
	%\bibliography{ref-trilevel,ref-BSG}

	%%%%%%%%%%%%%%%%%%%%%%%%%%%%%%%%%%%%%%%%%%%%%%%%%%%%%%%%%%%%

	%%%%%%%%%%%%%%%%%%%%%%%%%%%%%%%%%%%%%%%%%%%%%%%%%%%%%%%%%%%%
	
	\newpage
	
	\textbf{\large Technical Appendices}
	% \textbf{\large Technical Appendices and Supplementary Material}
	
	\appendix
	
	\section{Derivation of the trilevel adjoint gradient}\label{app:prop:adjoint_grad}
	This appendix contains the formal statement and derivation of the trilevel adjoint gradient given by equation~\eqref{adjoint}.
	
	\begin{proposition}[Trilevel adjoint gradient]\label{prop:adjoint_grad}
		Under assumptions that will ensure all terms are well-defined (specifically, Assumptions \ref{as:tri_lip_cont}--\ref{as:strong_conv_fbar_y}), we define the adjoint gradient of~$f$ as (referenced as~\eqref{adjoint})
		\begin{equation}
			\begin{split}
				\nabla f \; = \; ( \nabla_x f_1 - \nabla_{xz}^2 f_3 \nabla_{zz}^2 f_3^{-1} \nabla_z f_1 ) - \nabla_{xy}^2\Bar{f} \nabla_{yy}^2 \Bar{f}^{-1} ( \nabla_y f_1  - \nabla_{yz}^2 f_3 \nabla_{zz}^2 f_3^{-1} \nabla_z f_1 ), \nonumber%- \nabla_{xz}^2 f_3 \nabla_{zz}^2 f_3^{-1} \nabla_z f_1.
			\end{split}
		\end{equation}
		where all of the gradient and Hessian terms of $f_1$ and $f_3$ on the right-hand side are evaluated at~$(x,y(x),z(x))$. Further, the $\bar{f}$ terms are evaluated at~$(x,y(x))$ where
		\begin{alignat}{2}
			&\nabla_{yx}^2 \Bar{f}(x,y) &\; = \; \nabla_{yx}^2 f_2 + \nabla_{yz}^2 f_2 \nabla_x z^\top + \frac{\partial}{\partial x}\left[ \nabla_y z \nabla_z f_2 \right],\label{hess_yx_fbar}\\
			&\nabla_{yy}^2 \Bar{f}(x,y) &\; = \; \nabla_{yy}^2 f_2 + \nabla_{yz}^2 f_2 \nabla_y z^\top + \frac{\partial}{\partial y}\left[ \nabla_y z \nabla_z f_2 \right],\label{hess_yy_fbar}
		\end{alignat}
		with
		\begin{alignat}{2}
			\nabla_x z(x,y)^\top &&\; = \; -\nabla_{zz}^2 f_3^{-1} \nabla_{zx}^2 f_3,\label{eq:jacobian_z_ret_x}\\
			\nabla_y z(x,y)^\top &&\; = \; -\nabla_{zz}^2 f_3^{-1} \nabla_{zy}^2 f_3,\label{eq:jacobian_z_ret_y}
		\end{alignat}
		\begin{alignat}{2}
			\frac{\partial}{\partial x}[ \nabla_y z \nabla_z f_2 ] \; &= \; -[ \nabla_{yzx}^3 f_3 + \nabla_{yzz}^3 f_3 \nabla_x z^\top - \nabla_{yz}^2f_3 \nabla_{zz}^2 f_3^{-1} ( \nabla_{zzx}^3 f_3 + \nabla_{zzz}^3 f_3 \nabla_x z^\top ) ]\nabla_{zz}^2 f_3^{-1} \nabla_z f_2 \nonumber \\
			&\quad - \nabla_{yz}^2 f_3 \nabla_{zz}^2f_3^{-1} ( \nabla_{zx}^2 f_2 + \nabla_{zz}^2 f_2 \nabla_x z^\top ), \label{eq:partial_x}\\
			\frac{\partial}{\partial y}[ \nabla_y z \nabla_z f_2 ] \; &= \; -[ \nabla_{yzy}^3 f_3 + \nabla_{yzz}^3 f_3 \nabla_y z^\top - \nabla_{yz}^2f_3 \nabla_{zz}^2 f_3^{-1} ( \nabla_{zzy}^3 f_3 + \nabla_{zzz}^3 f_3 \nabla_y z^\top ) ]\nabla_{zz}^2 f_3^{-1} \nabla_z f_2 \nonumber \\
			&\quad - \nabla_{yz}^2 f_3 \nabla_{zz}^2f_3^{-1} ( \nabla_{zy}^2 f_2 + \nabla_{zz}^2 f_2 \nabla_y z^\top ). \label{eq:partial_y}
		\end{alignat}
		Notice that all of the gradients and Hessians of $f_2$ and the gradients, Hessians, and tensors of third-order derivatives (which we denote by $\nabla^3$)\footnote{To clarify the notation for third-order derivatives, consider the following example: given an~$m \times t \times n$ tensor~$\nabla^3_{yzx} f_3$ and a~$t \times t$ matrix~$\nabla^2_{zz}f_3^{-1}$, the product~$\nabla^3_{yzx} f_3 \nabla^2_{zz}f_3^{-1}$ yields an~$m \times t \times n$ matrix. Left-multiplying a~$t$-dimensional vector~$\nabla_z f_2$ by~$\nabla^3_{yzx} f_3 \nabla^2_{zz}f_3^{-1}$ results in an~$m \times 1 \times n$ matrix (or~$m \times n$, for brevity).} of~$f_3$ in~\eqref{hess_yx_fbar}--\eqref{eq:partial_y} are evaluated at the point $(x,y,z(x,y))$ and all of the $\nabla z$ terms are evaluated at the point $(x,y)$.
	\end{proposition}
	\begin{proof}
		One arrives at the adjoint formula~\eqref{adjoint} by first applying the multivariate chain rule \linebreak to $f_1(x,y(x),z(x,y(x)))$ in the following manner:
		\begin{alignat}{2}
			\nabla f \; &= \; \frac{d}{dx} f_1(x,y(x),z(x,y(x))) \; = \; \frac{\partial f_1}{\partial x} + \frac{d y}{d x} \frac{\partial f_1}{\partial y} +  \frac{d}{d x} z(x,y(x)) \frac{\partial f_1}{\partial z} \nonumber,
		\end{alignat}
		where
		\[\frac{d}{d x} z(x,y(x)) \; = \; \frac{\partial z}{\partial x} + \frac{d y}{d x} \frac{\partial z}{\partial y}.\]
		Thus, we have
		\begin{alignat}{2}
			\nabla f \; &= \; \frac{\partial f_1}{\partial x} + \frac{d y}{d x} \frac{\partial f_1}{\partial y}  +  ( \frac{\partial z}{\partial x} + \frac{d y}{d x} \frac{\partial z}{\partial y}  ) \frac{\partial f_1}{\partial z} \nonumber \\
			&= \; \nabla_x f_1 + \nabla y \nabla_y f_1 + \left( \nabla_x z + \nabla y \nabla_y z \right) \nabla_z f_1 \nonumber \\
			& = \; \nabla_x f_1 + \nabla_x z \nabla_z f_1 + \nabla y \left( \nabla_y f_1  + \nabla_y z \nabla_z f_1 \right). \label{adjoint_jacobian}
		\end{alignat}
		
		The Jacobian of~$y(x)$, i.e.,~$\nabla y(x)^{\top}\in\mathbb{R}^{m\times n}$,  can be computed from the first-order necessary optimality conditions of the ML problem, defined by~$\nabla_y \Bar{f}(x,y(x)) = 0$. In particular, taking the derivative of both sides with respect to~$x$, utilizing the chain rule and the implicit function theorem (which ensures $y(\cdot)$ to be continuously differentiable~\cite{Ruding_1953}), we obtain~$\nabla_{yx}^2 \Bar{f} + \nabla_{yy}^2 \Bar{f}\nabla y(x)^\top = 0$ (where all Hessians are evaluated at~$(x,y(x))$), which yields
		\begin{equation}\label{eq:jacobian_y}
			\nabla y(x)^\top \; = \; - \nabla_{yy}^2 \Bar{f}(x,y(x))^{-1} \nabla_{yx}^2 \Bar{f}(x,y(x)).
		\end{equation}
		
		Since
		\begin{equation}\label{eq:exact_grad_MLP_y}
			\nabla_y \Bar{f}(x,y) = \nabla_y f_2(x,y,z(x,y)) + \nabla_y z \nabla_z f_2(x,y,z(x,y)),
		\end{equation}
		taking the derivative of both sides with respect to~$x$ and~$y$ and utilizing the chain rule, we obtain the expressions for~$\nabla_{yx}^2 \Bar{f}(x,y)$ and~$\nabla_{yy}^2 \Bar{f}(x,y)$ in~\eqref{hess_yx_fbar} and~\eqref{hess_yy_fbar}, respectively.
		% \[\nabla_{yx}^2 \Bar{f}(x,y) \; = \; \nabla_{yx}^2 f_2(x,y,z(x,y)) + \nabla_{yz}^2 f_2(x,y,z(x,y)) \nabla_x z^\top + \frac{\partial}{\partial x}\left[ \nabla_y z \nabla_z f_2(x,y,z(x,y)) \right],\]
		% and (equation~\eqref{hess_yy_fbar})
		% \[\nabla_{yy}^2 \Bar{f}(x,y) \; = \; \nabla_{yy}^2 f_2(x,y,z(x,y)) + \nabla_{yz}^2 f_2(x,y,z(x,y)) \nabla_y z^\top + \frac{\partial}{\partial y}\left[ \nabla_y z \nabla_z f_2(x,y,z(x,y)) \right].\]
		
		Similarly, we can derive expressions for both of the Jacobians of~$z(x,y)$, i.e., $\nabla_x z(x,y)^\top\in\mathbb{R}^{t\times n}$ and~$\nabla_y z(x,y)^\top\in\mathbb{R}^{t\times m}$, respectively, from the first-order necessary optimality conditions of the LL problem, defined by~$\nabla_z f_3(x,y,z(x,y)) = 0$. In particular, taking derivatives of both sides with respect to~$x$ will yield $\nabla_{zx}^2 f_3(x,y,z(x,y)) + \nabla_{zz}^2 f_3(x,y,z(x,y)) \nabla_x z(x,y)^\top  = 0$, whereas taking derivatives with respect to~$y$ will yield~$\nabla_{zy}^2 f_3(x,y,z(x,y)) + \nabla_{zz}^2 f_3(x,y,z(x,y)) \nabla_y z(x,y)^\top = 0$. Solving these two equations for both~$\nabla_x z(x,y)^\top$ and~$\nabla_y z(x,y)^\top$, respectively, we obtain the expressions for~$\nabla_x z(x,y)^\top$ and~$\nabla_y z(x,y)^\top$ in~\eqref{eq:jacobian_z_ret_x} and~\eqref{eq:jacobian_z_ret_y}, respectively. Now, substituting~\eqref{eq:jacobian_y}, \eqref{eq:jacobian_z_ret_x}, and~\eqref{eq:jacobian_z_ret_y} into~\eqref{adjoint_jacobian}, we obtain the adjoint gradient defined by~\eqref{adjoint}.
		% \[\nabla_x z(x,y)^\top \; = \; -\nabla_{zz}^2 f_3(x,y,z(x,y))^{-1} \nabla_{zx}^2 f_3(x,y,z(x,y)),\]
		% and (equation~\eqref{eq:jacobian_z_ret_y})
		% \[\nabla_y z(x,y)^\top \; = \; -\nabla_{zz}^2 f_3(x,y,z(x,y))^{-1} \nabla_{zy}^2 f_3(x,y,z(x,y)).\]
		
		It remains to derive~\eqref{eq:partial_x} and~\eqref{eq:partial_y}. Using the property that the derivative of the inverse of a matrix~$K\left(g(x)\right)$ with respect to~$x$, where~$g$ is a vector-valued function of~$x$, is given by
		\[
		\frac{\partial}{\partial x}K\left(g(x)\right)^{-1} = -K\left(g(x)\right)^{-1} \left[\frac{\partial}{\partial x}K\left(g(x)\right)\right] K\left(g(x)\right)^{-1},
		\]
		and applying the product rule twice, it follows that the last term in~\eqref{hess_yx_fbar} can be written as
		\begin{alignat}{2}
			\frac{\partial}{\partial x}\left[ \nabla_y z \nabla_z f_2 \right] \; &= \; \frac{\partial}{\partial x}\left[-\nabla_{yz}^2 f_3 \nabla_{zz}^2 f_3^{-1} \nabla_z f_2 \right] \nonumber \\
			& = - \left(\frac{\partial}{\partial x} \nabla_{yz}^2 f_3\right) \nabla_{zz}^2 f_3^{-1} \nabla_z f_2 - \nabla_{yz}^2 f_3 \left( \frac{\partial}{\partial x}\nabla_{zz}^2 f_3^{-1} \nabla_z f_2 \right) \nonumber \\
			& = - \left(\frac{\partial}{\partial x} \nabla_{yz}^2 f_3 \right)\nabla_{zz}^2 f_3^{-1} \nabla_z f_2 - \nabla_{yz}^2 f_3 \left[ \left( \frac{\partial}{\partial x}\nabla_{zz}^2 f_3^{-1} \right)\nabla_z f_2 + \nabla_{zz}^2 f_3^{-1}\left( \frac{\partial}{\partial x}\nabla_z f_2 \right) \right], \label{eq:intermediate}
		\end{alignat}
		where
		\begin{alignat}{7}
			&\frac{\partial}{\partial x} \nabla_{yz}^2 f_3 \; &&= \; \nabla_{yzx}^3 f_3 + \nabla_{yzz}^3 f_3 \nabla_x z(x,y)^\top, \quad&&\frac{\partial}{\partial x}\nabla_z f_2 \; &&= \; \nabla_{zx}^2 f_2 + \nabla_{zz}^2 f_2 \nabla_x z(x,y)^\top,\nonumber\\
			&\frac{\partial}{\partial x}\nabla_{zz}^2 f_3^{-1} \; &&= \; -\nabla_{zz}^2 f_3^{-1} \left( \frac{\partial}{\partial x}\nabla_{zz}^2 f_3 \right) \nabla_{zz}^2 f_3^{-1}, \qquad&&\frac{\partial}{\partial x} \nabla_{zz}^2 f_3 \; &&= \; \nabla_{zzx}^3 f_3 + \nabla_{zzz}^3 f_3 \nabla_x z(x,y)^\top.
		\end{alignat}
		Substituting these equations into \eqref{eq:intermediate} and simplifying, we obtain~\eqref{eq:partial_x}.
		% \begin{alignat*}{2}
			%     \frac{\partial}{\partial x}\left[ \nabla_y z \nabla_z f_2 \right] \; &= \; -\left[ \nabla_{yzx}^3 f_3 + \nabla_{yzz}^3 f_3 \nabla_x z^\top - \nabla_{yz}^2f_3 \nabla_{zz}^2 f_3^{-1} \left( \nabla_{zzx}^3 f_3 + \nabla_{zzz}^3 f_3 \nabla_x z^\top \right) \right]\nabla_{zz}^2 f_3^{-1} \nabla_z f_2 \nonumber \\
			%     &\quad - \nabla_{yz}^2 f_3 \nabla_{zz}^2f_3^{-1} \left( \nabla_{zx}^2 f_2 + \nabla_{zz}^2 f_2 \nabla_x z^\top \right).
			% \end{alignat*}
		Through a similar process, we obtain that the right-most term of \eqref{hess_yy_fbar} is given by~\eqref{eq:partial_y}.
		% \begin{alignat*}{2}
			%     \frac{\partial}{\partial y}\left[ \nabla_y z \nabla_z f_2 \right] \; &= \; -\left[ \nabla_{yzy}^3 f_3 + \nabla_{yzz}^3 f_3 \nabla_y z^\top - \nabla_{yz}^2f_3 \nabla_{zz}^2 f_3^{-1} \left( \nabla_{zzy}^3 f_3 + \nabla_{zzz}^3 f_3 \nabla_y z^\top \right) \right]\nabla_{zz}^2 f_3^{-1} \nabla_z f_2 \nonumber \\
			%     &\quad - \nabla_{yz}^2 f_3 \nabla_{zz}^2f_3^{-1} \left( \nabla_{zy}^2 f_2 + \nabla_{zz}^2 f_2 \nabla_y z^\top \right).
			% \end{alignat*}
		%Now that we have fully defined the two matrices in~\eqref{eq:jacobian_y}, we can substitute~\eqref{eq:jacobian_y}, \eqref{eq:jacobian_z_ret_x}, and~\eqref{eq:jacobian_z_ret_y} into~\eqref{adjoint_jacobian} to obtain the formula for the adjoint gradient defined by~\eqref{adjoint}.
	\end{proof}

	\section{Discussion on the convergence analysis of the TSG method}\label{app:conv_theory_discussion}
	In this appendix, we outline the convergence analysis for the TSG method (Algorithm~\ref{alg:TSG}) and highlight all of the relevant results and notation involved. For simplicity of the convergence analysis, we utilize the following Lyapunov function:
	\begin{equation}\label{eq:lyapunov}
		\mathbb{V}^i := f(x^i) + \| y^{i} - y( x^{i} ) \|^2 + \| z^{i} - z( x^{i} ) \|^2 + \| z^{i} - z( x^{i}, y^{i} ) \|^2.
	\end{equation}
	There is no particular property that is required from Lyapunov functions for our analysis. Rather, \eqref{eq:lyapunov} is defined to allow for appropriate telescoping cancellations in the proof of Theorem~\ref{th:TSG_convergence} (which is an extension of the methodology utilized in~\cite{TChen_YSun_WYin_2021_closingGap} for bilevel problems). Further, the difference between two consecutive Lyapunov evaluations can by quantified as
	\begin{alignat}{2}
		&\mathbb{V}^{i+1} - \mathbb{V}^i \nonumber\\
		& = \underbrace{f(x^{i+1}) - f(x^i) }_{\text{Lemma~\ref{lem:TSG_UL_descent}}} + \underbrace{\| y^{i+1} - y( x^{i+1} ) \|^2 - \| y^{i} - y( x^{i} ) \|^2}_{\text{Lemma~\ref{lem:TSG_ML_descent}}}  \nonumber\\
		&\quad+  \underbrace{\| z^{i+1} - z( x^{i+1} ) \|^2 - \| z^{i} - z( x^{i} ) \|^2}_{\text{Lemma~\ref{lem:TSG_LL_descent}}}  + 
		\underbrace{\| z^{i+1} - z( x^{i+1}, y^{i+1} ) \|^2  - \| z^{i} - z( x^{i}, y^{i} ) \|^2}_{\text{Lemma~\ref{lem:TSG_aux_LL_descent}}}.\label{eq:lyapunov_diff}
	\end{alignat}
	Notice that this consists of four differences: the first difference measures the amount of descent that is achieved in the~UL problem, the second and third differences correspond to the error present in the~ML and~LL problems, respectively, and the fourth difference is an auxiliary term that corresponds to the inexact~LL error relative to the~ML variables. Further, it bears mentioning that Appendix~\ref{app:conv_theory_proofs} contains the proofs of Lemmas~\ref{lem:TSG_UL_descent}--\ref{lem:TSG_ML_descent} and Theorem~\ref{th:TSG_convergence}, and Appendix~\ref{app:conv_theory_proofsi_var_inex} contains the statements and proofs of intermediary results that are required for the arguments used in Appendix~\ref{app:conv_theory_proofs}. Lastly, Appendix~\ref{app:lipschitz_properties} includes auxiliary lemmas proving Lipschitz continuity properties for the following functions, gradients, and Jacobians:~$z(x)$, $z(x,y)$, $y(x)$, $\nabla_y \bar f$, $\nabla_{xy}^2\Bar{f}$, $\nabla_{yy}^2\Bar{f}$, $\nabla f$, $\nabla z$, and~$\nabla y$. For ease of reference, Table~\ref{tab:reference_table_of_constants} below compiles all the relevant constants utilized throughout the theory which are not defined in Lemmas~\ref{lem:TSG_UL_descent}--\ref{lem:TSG_ML_descent}.

	\begin{table}[h!]
		\caption{Reference table of constants associated with derived bounds.}
		\label{tab:reference_table_of_constants}
		\centering
		\begin{tabular}{ccc}
			\toprule
			Descriptions & Constants & References \\
			\midrule
			Bounds on bias \& variance  &  $U_x$, $U_y$, $U_{xy}$, $U_{yy}$, $V_{xy}$, $V_{yy}$  &  Lemmas~\ref{lem:TSG_derived_unbiasedness} -- \ref{lem:TSG_derived_variance_bounds_of_f_bar} \\
			\hline
			Bounds on UL inexactness & $\omega$, $\tau$, $\zeta$ & Lemmas~\ref{lem:TSG_new_inexact_variance_bound} -- \ref{lem:TSG_size_inexact_g} \\
			\hline
			Bounds on ML inexactness & $\hat\omega$, $\hat\tau$, $\Upsilon$ & Lemmas~\ref{lem:TSG_var_bound_ML_dir} -- \ref{lem:TSG_size_inexact_g_ML} \\
			\hline
			Derived Lipschitz properties & \parbox[c]{6cm}{ $L_z$, $L_{z_{xy}}$, $L_{z_y}$, $L_y$, $L_{\nabla_z}$, $L_{\bar F}$, $L_{\bar F_y}$ \\ $L_{\bar F_z}$, $L_{\nabla^2_{yx} \bar f}$, $L_{\nabla^2_{yy} \bar f}$, $L_F$, $L_{F_{yz}}$, $L_{\nabla y}$ }  & Equations~\ref{eq:lip_prop_1} -- \ref{eq:lip_prop_13}\\
			\bottomrule
		\end{tabular}
	\end{table}

	\subsection{Descriptions of $\sigma$-algebras}\label{sec:def_sig_algebra}
	We denote three auxiliary sets $\Sigma_i$, $\Sigma_{i,j}$, and $\Sigma_{i,j,k}$, each corresponding to the set of iterates generated by Algorithm~\ref{alg:TSG} for the UL update, ML update, and LL update, respectively. We define these sets explicitly in the following way:
	\[
	\Sigma_i \;:=\; \{ x^{\hat i}, y^{\hat i}, z^{\hat i} ~\vert~ \forall \hat i\in\{0,1,...,i\} \},
	\]
	\[
	\Sigma_{i,j} \;:=\; \{ x^{\hat i}, y^{\hat i, \hat j}, z^{\hat i, \hat j} ~\vert~ \forall \hat i\in\{0,1,...,i\} \text{ and } \forall \hat j\in\{0,1,...,j\} \},
	\]
	\[
	\Sigma_{i,j,k} \;:=\; \{ x^{\hat i}, y^{\hat i, \hat j}, z^{\hat i, \hat j, \hat k} ~\vert~ \forall \hat i\in\{0,1,...,i\} \text{ and } \forall \hat j\in\{0,1,...,j\} \text{ and } \forall \hat k\in\{0,1,...,k\} \}.
	\]
	Now, we define the corresponding~$\sigma$-algebras generated as~$\mathcal{F}_i := \sigma\left( \Sigma_i \cup \{y^{i+1},z^{i+1}\} \right)$, \linebreak$\mathcal{F}_{i,j} := \sigma\left( \Sigma_{i,j} \cup \{z^{i,j+1}\} \right)$, and~$\mathcal{F}_{i,j,k} := \sigma\left( \Sigma_{i,j,k} \right)$, respectively. Further, we will use the expressions~$\mathbb{E}\left[ \cdot\vert \mathcal{F}_i \right]$, $\mathbb{E}\left[ \cdot\vert \mathcal{F}_{i,j} \right]$, and~$\mathbb{E}\left[ \cdot\vert \mathcal{F}_{i,j,k} \right]$ to denote the conditional expectations taken with respect to the probability distributions of~$\xi^i$, $\xi^{i,j}$, and~$\xi^{i,j,k}$ given~$\mathcal{F}_i$, $\mathcal{F}_{i,j}$, and~$\mathcal{F}_{i,j,k}$, respectively. Recalling from the beginning of Section~\ref{sec:assumptions_convergence_rate_TSG_method}, we also define a general sigma-algebra~$\mathcal{F}_\xi$ that includes all the events up to the generation of a general point~$(x,y,z)$, before observing a realization of~$\xi$; similarly, $\mathbb{E}[\cdot|\mathcal{F}_\xi]$ denotes the expectation taken with respect to the probability distribution of~$\xi$ given~$\mathcal{F}_\xi$. We also use~$\mathbb{E}[\cdot]$ to denote the \textit{total expectation}, i.e., the expected value with respect to the joint distribution of all the random variables.

	\subsection{Statements of descent and error bound results}\label{sec:statements_of_conv_results}
	We now provide the statements of Lemmas~\ref{lem:TSG_UL_descent}--\ref{lem:TSG_ML_descent} below, that bound the terms in the Lyapunov difference given by~\eqref{eq:lyapunov_diff}, and which are ultimately required to prove the main convergence result of Algorithm~\ref{alg:TSG}, presented in Theorem~\ref{th:TSG_convergence}. The proofs of such lemmas and the theorem are provided in Appendix~\ref{app:conv_theory_proofs}. They required a non-trivial adaptation of the proofs in~\cite{TChen_YSun_WYin_2021_closingGap}, which were specific for bilevel problems. 
	
	\begin{lemma}[Descent of the true trilevel UL problem]\label{lem:TSG_UL_descent}
		Recalling~$\bar g_{f_1}^{i} = \mathbb{E} [ \tilde g_{f_1}^{i} \vert \mathcal{F}_{i} ]$, under Assumptions~\ref{as:tri_lip_cont}--\ref{as:TSG_bounded_var}, the sequence of iterates $\{ x^i \}_{i\geq0}$ generated by Algorithm~\ref{alg:TSG} satisfies
		\begin{alignat}{2}
			\mathbb{E}[ f(x^{i+1}) ] - \mathbb{E}[ f(x^i) ] \;&\leq\; -\frac{\alpha_i}{2} \mathbb{E}[ \| \nabla f(x^i) \|^2 ] - \left( \frac{\alpha_i}{2} - \frac{L_F\alpha_i^2}{2} \right) \mathbb{E}[ \| \bar g_{f_1}^{i} \|^2 ] + \tilde\omega\alpha_i^2\nonumber\\
			&\quad+  \alpha_iL_{F_{yz}}^2( \mathbb{E}[ \| y(x^i) - y^{i+1} \|^2 ] + \mathbb{E}[ \| z(x^i) - z^{i+1} \|^2 ] ),
		\end{alignat}
		where~$\tilde{\omega}$ is given by~\eqref{eq:1020} in Appendix~\ref{app:lem:TSG_UL_descent}.
	\end{lemma}

	\begin{lemma}[Error bounds of the trilevel LL problem]\label{lem:TSG_LL_descent}
		Suppose that Assumptions~\ref{as:tri_lip_cont}--\ref{as:TSG_bounded_var} hold. Then, choosing the~LL step-size~$\gamma_i$ such that~$\gamma_i\leq \frac{1}{\mu_z + L_{\nabla f_3}}$, there exists the positive constant~$\rho_{f_3}$, given by~\eqref{eq:1060} in Appendix~\ref{app:lem:TSG_LL_descent}, and %positive constants~$\rho_{f_3}$, $\tau$, and~$\zeta$, and 
		a positive quantity $\kappa_i$, such that
		\begin{alignat}{3}
			&\mathbb{E}[ \| z^{i,j+1} - z(x^i)\|^2 ] \;&&\leq\; \left( 1 - \gamma_i \rho_{f_3} \right)^K\mathbb{E}[ \| z^{i,j} - z(x^i) \|^2 ] + K\gamma_i^2\sigma_{\nabla f_3}^2,\label{eq:TSG_LLP_bound_1}\\
			&\mathbb{E}[\| z^{i+1} - z(x^i)\|^2] \;&&\leq\; \left( 1 - \gamma_i \rho_{f_3} \right)^{JK}\mathbb{E}[\| z^{i} - z(x^i) \|^2] + J K\gamma_i^2\sigma_{\nabla f_3}^2,\label{eq:TSG_LLP_bound_2}\\
			&\mathbb{E}[\| z^{i+1}-z(x^{i+1})\|^2] \;&&\leq\; \left( 1 + 2\kappa_i + \frac{L_{\nabla z}\alpha_i^2\zeta}{2} \right) \mathbb{E}[ \| z^{i+1} - z(x^i) \|^2]\nonumber\\
			&&&\quad+ \left( 2L_z^2 + \frac{L_z^2}{2\kappa_i} + \frac{L_{\nabla z}}{2} \right)\alpha_i^2 \mathbb{E}[\| \bar g_{f_1}^i \|^2] + \left( 2L_z^2 + \frac{ L_{\nabla z}}{2} \right)\tau\alpha_i^2.\label{eq:TSG_LLP_bound_3}
		\end{alignat}
		%Specifically, $\rho_{f_3}$ is given by~\eqref{eq:1060} in Appendix~\ref{app:lem:TSG_LL_descent}, $\tau$ was introduced in Lemma~\ref{lem:TSG_new_inexact_variance_bound}, and~$\zeta$ was introduced in Lemma~\ref{lem:TSG_size_inexact_g}.
	\end{lemma}

	\begin{lemma}[Auxiliary error bounds of the trilevel LL problem]\label{lem:TSG_aux_LL_descent}
		Suppose that Assumptions~\ref{as:tri_lip_cont}--\ref{as:TSG_bounded_var} hold. Then, choosing the LL step-size $\gamma_i$ such that $\gamma_i\leq \frac{1}{\mu_z + L_{\nabla f_3}}$, there exist positive quantities~$\eta_i$ and~$\hat{\eta}_i$ such that
		\begin{alignat}{3}
			&\mathbb{E}[ \| z^{i,j+1} - z(x^i,y^{i,j+1}) \|^2 ] \;&&\leq\; (\left( 1 - \gamma_i \rho_{f_3} \right)^K + \eta_i) \mathbb{E}[\| z^{i,j} - z(x^i,y^{i,j})\|^2] + \hat{\eta_i}L_{z_y}^2\Upsilon\beta_i^2 + K\gamma_i^2\sigma_{\nabla f_3}^2,\label{eq:TSG_aux_LL_bound_1}\\
			&\mathbb{E}[ \| z^{i,j+1} - z(x^i,y^{i}) \|^2 ] \;&&\leq\; \left( 1 - \gamma_i \rho_{f_3} \right)^{K}\mathbb{E}[\| z^{i} - z(x^i,y^i) \|^2] + K\gamma_i^2\sigma_{\nabla f_3}^2,\label{eq:TSG_aux_LL_bound_2}\\
			&\mathbb{E}[ \| z^{i+1} - z(x^i,y^{i}) \|^2 ] \;&&\leq\; \left( 1 - \gamma_i \rho_{f_3} \right)^{J K}\mathbb{E}[\| z^{i} - z(x^i,y^i) \|^2] + J K\gamma_i^2\sigma_{\nabla f_3}^2,\label{eq:TSG_aux_LL_bound_3}\\
			&\mathbb{E}[ \| z^{i+1} - z( x^{i+1}, y^{i+1} ) \|^2 ]
			\;&&\leq\; 2\mathbb{E}[ \| z^{i+1} - z(x^i,y^{i}) \|^2 ] + 4L_{z_{xy}}^2 \alpha_i^2 ( \mathbb{E}[\|\bar g_{f_1}^{i}\|^2] + \tau) + 2J^2\Upsilon L_{z_{xy}}^2\beta_i^2,\label{eq:TSG_aux_LL_bound_4}\\
			&\mathbb{E}[ \| z^{i,j+1} - z(x^i,y^{i,j}) \|^2 ] \;&&\leq\; \left( 1 - \gamma_i \rho_{f_3} \right)^{K}\mathbb{E}[\| z^{i,j} - z(x^i,y^{i,j}) \|^2] + K\gamma_i^2\sigma_{\nabla f_3}^2.\label{eq:TSG_aux_LL_bound_5}
		\end{alignat}
	\end{lemma}

	\begin{lemma}[Error bounds of the trilevel ML problem]\label{lem:TSG_ML_descent}
		Suppose that Assumptions~\ref{as:tri_lip_cont}--\ref{as:TSG_bounded_var} hold. Then, choosing the ML step-size~$\beta_i$ such that~$\beta_i \leq \frac{1}{\mu_y + L_{\nabla \bar f}}$ and~$\beta_i \leq \frac{\rho}{2\hat\omega^2+1}$ as well as choosing the LL step-size~$\gamma_i$ such that $\gamma_i\leq\frac{1}{\mu_z + L_{\nabla f_3}}$, there are positive quantities~$\psi_i$ and~$\phi_i$ such that
		\begin{alignat}{3}
			&\mathbb{E}[ \| y^{i+1} - y(x^i)\|^2 ] &&\;\leq\; \left( 1 - \psi_i\beta_i \right)^J \mathbb{E}[ \| y^{i} - y(x^i) \|^2 ]  + \left( 1 + \frac{1}{2}(J-1)\hat{\eta_i} L_{z_y}^2 \right)J\Upsilon \beta_i^2 \nonumber\\
			&&&\quad + \left( 1 - \gamma_i \rho_{f_3} \right)^K J\mathbb{E}[ \| z^{i} - z(x^i,y^{i}) \|^2 ] + \frac{J+1}{2} JK\gamma_i^2\sigma_{\nabla f_3}^2,\label{eq:TSG_ML_bound_1}\\
			&\mathbb{E}[\| y^{i+1}-y(x^{i+1})\|^2] && \;\leq\; \left( 1 + 2\phi_i + \frac{L_{\nabla y}\alpha_i^2 \zeta}{2}\right)\mathbb{E}[ \| y^{i+1} - y(x^i) \|^2 ]\nonumber \\
			&&&\quad+ \left( 2L_y^2 + \frac{ L_y^2}{2 \phi_i} + \frac{L_{\nabla y}}{2} \right)\alpha_i^2\mathbb{E}[ \| \bar g_{f_1}^{i} \|^2 ] + \left( 2L_y^2 + \frac{ L_{\nabla y}}{2} \right)\tau\alpha_i^2,\label{eq:TSG_ML_bound_2}
		\end{alignat}
		%where~$\hat{\omega}$ was introduced in Lemma~\ref{lem:TSG_var_bound_ML_dir},
		where $\rho$ is given by~\eqref{eq:2020} in Appendix~\ref{app:lem:TSG_ML_descent}.
		Specifically, $\psi_i$ is a function of~$\theta_i$ given by~\eqref{eq:2020} in Appendix~\ref{app:lem:TSG_ML_descent}.
	\end{lemma}

	\section{Convergence theory proofs}\label{app:conv_theory_proofs}
	
	This appendix contains the proofs of Lemmas~\ref{lem:TSG_UL_descent}--\ref{lem:TSG_ML_descent} (which are utilized to bound the terms in the Lyapunov function given by~\eqref{eq:lyapunov_diff}) as well as the proof of Theorem~\ref{th:TSG_convergence}.% and~\ref{th:TSG_convergence_v2}.

	\subsection{Proof of Lemma~\ref{lem:TSG_UL_descent}}\label{app:lem:TSG_UL_descent}
	\begin{proof}
		From the Lipschitz property of $\nabla f$ (equation~\eqref{eq:lip_prop_11}), taking expectation conditioned on $\mathcal{F}_i$, and letting $\bar g_{f_1}^{i}=\mathbb{E}[ \tilde g_{f_1}^{i} \vert \mathcal{F}_i]$, we have
		\begin{alignat*}{2}
			\mathbb{E}[ f(x^{i+1}) \vert \mathcal{F}_i] - \mathbb{E}[ f(x^i) \vert \mathcal{F}_i] & \leq \mathbb{E}[ \nabla f(x^i)^\top (x^{i+1} - x^i) \vert \mathcal{F}_i] + \frac{L_F}{2}\mathbb{E}[ \| x^{i+1} - x^i \|^2 \vert \mathcal{F}_i] \\
			& = \mathbb{E}[ \nabla f(x^i)^\top (x^{i} - \alpha_i \tilde g_{f_1}^{i} - x^i) \vert \mathcal{F}_i] + \frac{L_F}{2}\mathbb{E}[ \| x^{i} - \alpha_i \tilde g_{f_1}^{i} - x^i \|^2 \vert \mathcal{F}_i ] \\
			& = -\alpha_i \nabla f(x^i)^\top \bar g_{f_1}^{i} + \frac{L_F}{2}\alpha_i^2\mathbb{E}[ \| \tilde g_{f_1}^{i} \|^2 \vert \mathcal{F}_i ],
		\end{alignat*}
		Using the fact that $2a^\top b = \|a\|^2 + \|b\|^2 - \|a-b\|^2$ twice, with~$a$ and~$b$ real-valued vectors, yields
		\begin{alignat*}{2}
			\mathbb{E}[ f(x^{i+1}) \vert \mathcal{F}_i] - \mathbb{E}[ f(x^i) \vert \mathcal{F}_i] & \leq -\frac{\alpha_i}{2} \| \nabla f(x^i) \|^2 - \frac{\alpha_i}{2} \| \bar g_{f_1}^{i} \|^2 + \frac{\alpha_i}{2} \| \nabla f(x^i) - \bar g_{f_1}^{i} \|^2 + \frac{L_F}{2}\alpha_i^2\mathbb{E}[ \| \tilde g_{f_1}^{i} \|^2 \vert \mathcal{F}_i ]\\
			& = -\frac{\alpha_i}{2} \| \nabla f(x^i) \|^2 - \frac{\alpha_i}{2} \| \bar g_{f_1}^{i} \|^2 + \frac{\alpha_i}{2} \| \nabla f(x^i) - \bar g_{f_1}^{i} \|^2\\
			&\quad+ \frac{L_F\alpha_i^2}{2}\mathbb{E}[ 2 \left( \tilde g_{f_1}^{i} \right)^\top \bar g_{f_1}^{i}  - \| \bar g_{f_1}^{i} \|^2 + \| \tilde g_{f_1}^{i} - \bar g_{f_1}^{i} \|^2 \vert \mathcal{F}_i]\\
			%& = -\frac{\alpha_i}{2} \| \nabla f(x^i) \|^2 - \frac{\alpha_i}{2} \| \bar g_{f_1}^{i} \|^2 + \frac{\alpha_i}{2} \| \nabla f(x^i) - \bar g_{f_1}^{i} \|^2\\
			%&\quad+ \frac{L_F\alpha_i^2}{2}\mathbb{E}[ 2 \left( \tilde g_{f_1}^{i} \right)^\top \bar g_{f_1}^{i} \vert \mathcal{F}_i]  - \frac{L_F\alpha_i^2}{2}\mathbb{E}[ \| \bar g_{f_1}^{i} \|^2 \vert \mathcal{F}_i] + \frac{L_F\alpha_i^2}{2}\mathbb{E}[ \| \tilde g_{f_1}^{i} - \bar g_{f_1}^{i} \|^2 \vert \mathcal{F}_i]\\
			& = -\frac{\alpha_i}{2} \| \nabla f(x^i) \|^2 - \frac{\alpha_i}{2} \| \bar g_{f_1}^{i} \|^2 + \frac{\alpha_i}{2} \| \nabla f(x^i) - \bar g_{f_1}^{i} \|^2\\
			&\quad+  \frac{L_F\alpha_i^2}{2}\mathbb{E}[ \| \bar g_{f_1}^{i} \|^2 \vert \mathcal{F}_i] + \frac{L_F\alpha_i^2}{2}\mathbb{E}[ \| \tilde g_{f_1}^{i} - \bar g_{f_1}^{i} \|^2 \vert \mathcal{F}_i].
		\end{alignat*}
		Utilizing Lemma~\ref{lem:TSG_new_inexact_variance_bound} and realizing that $\mathbb{E}[ \| \bar g_{f_1}^{i} \|^2 \vert \mathcal{F}_i ] = \| \bar g_{f_1}^{i} \|^2$, we have
		\[\mathbb{E}[ f(x^{i+1}) \vert \mathcal{F}_i ] - \mathbb{E}[ f(x^i) \vert \mathcal{F}_i ] \leq -\frac{\alpha_i}{2} \| \nabla f(x^i) \|^2 - \left( \frac{\alpha_i}{2} - \frac{L_F\alpha_i^2}{2} \right) \| \bar g_{f_1}^{i} \|^2 + \frac{\alpha_i}{2} \| \nabla f(x^i) - \bar g_{f_1}^{i} \|^2 + \frac{\tau L_F\alpha_i^2}{2}.\]
		Further, we decompose the gradient bias term by adding and subtracting $\nabla f (x^i,y^{i+1},z^{i+1})$, using the fact that $\|a+b\|^2 \leq 2\|a\|^2 + 2\|b\|^2$, with~$a$ and~$b$ real-valued vectors, yielding
		\begin{alignat*}{2}
			\| \nabla f(x^i) - \bar g_{f_1}^{i} \|^2 %& = \| \nabla f(x^i) - \nabla f (x^i,y^{i+1},z^{i+1}) + \nabla f (x^i,y^{i+1},z^{i+1}) - \bar g_{f_1}^{i} \|^2\\
			& \leq 2\| \nabla f(x^i,y(x^i),z(x^i)) - \nabla f (x^i,y^{i+1},z^{i+1}) \|^2 + 2\| \nabla f (x^i,y^{i+1},z^{i+1}) - \bar g_{f_1}^{i} \|^2\\
			& \leq 2L_{F_{yz}}^2 \| ( y(x^i),z(x^i) ) - ( y^{i+1}, z^{i+1} ) \|^2 + 2\omega^2\theta_i^2\\
			& \leq 2L_{F_{yz}}^2( \| y(x^i) - y^{i+1} \|^2 + \| z(x^i) - z^{i+1} \|^2 ) + 2\omega^2\alpha_i,
		\end{alignat*}
		where the second inequality follows from~\eqref{eq:lip_prop_12} and Lemma~\ref{lem:TSG_new_inexact_variance_bound}, and the last inequality follows from the fact that $\theta_i=\alpha_i\beta_i\gamma_i\leq\alpha_i$ and $0<\alpha_i^2\leq\alpha_i\leq1$. Putting this all together, we have
		\begin{alignat}{2}
			&\mathbb{E}[ f(x^{i+1}) \vert \mathcal{F}_i] - \mathbb{E}[ f(x^i) \vert \mathcal{F}_i ]\nonumber\\ %&\leq -\frac{\alpha_i}{2} \| \nabla f(x^i) \|^2 - \left( \frac{\alpha_i}{2} - \frac{L_F\alpha_i^2}{2} \right) \| \bar g_{f_1}^{i} \|^2\\
			%&\quad+  \alpha_iL_{F_{yz}}^2\left( \| y(x^i) - y^{i+1} \|^2 + \| z(x^i) - z^{i+1} \|^2 \right) + \omega^2\alpha_i^2 + \frac{\tau L_F\alpha_i^2}{2}\\
			& \leq -\frac{\alpha_i}{2} \| \nabla f(x^i) \|^2 - \left( \frac{\alpha_i}{2} - \frac{L_F\alpha_i^2}{2} \right) \| \bar g_{f_1}^{i} \|^2 + \alpha_iL_{F_{yz}}^2( \| y(x^i) - y^{i+1} \|^2 + \| z(x^i) - z^{i+1} \|^2 ) + \tilde\omega\alpha_i^2,\nonumber\\
			&\quad\text{where}\quad \tilde{\omega} \; := \; \left(\omega^2 + \frac{\tau L_F}{2} \right).\label{eq:1020}
		\end{alignat}
		%where
		%\begin{equation}\label{eq:1020}
		%\tilde{\omega} \; := \; \left(\omega^2 + \frac{\tau L_F}{2} \right).
		%\end{equation}
		Taking total expectation, we obtain the final bound, completing the proof.
	\end{proof}

	\subsection{Proof of Lemma~\ref{lem:TSG_LL_descent}}\label{app:lem:TSG_LL_descent}
	\begin{proof}
		To derive the error bound defined by~\eqref{eq:TSG_LLP_bound_3}, we start by decomposing the error of the LL variables by adding and subtracting $z(x^i)$ in the following way:
		\begin{alignat}{2}
			\mathbb{E}[\| z^{i+1}-z(x^{i+1})\|^2] &= \underbrace{\mathbb{E}[\| z^{i+1} - z(x^i)\|^2]}_{A^{(1)}_1} +  \underbrace{\mathbb{E}[\|z(x^i) - z(x^{i+1})\|^2]}_{A^{(1)}_2}\nonumber\\
			&\quad+ 2\underbrace{\mathbb{E}[( z^{i+1} - z(x^i) )^\top ( z(x^i) - z(x^{i+1}) )]}_{A^{(1)}_3}.\label{eq:TSG_LLP_error_intermediate_1}
		\end{alignat}
		
		\textbf{$(\text{Analysis of }A^{(1)}_1)$:} To derive an upper-bound on $A^{(1)}_1$ in~\eqref{eq:TSG_LLP_error_intermediate_1}, recall that $z^{i+1} = z^{i+1,0,0} = z^{i,J,K}$ and $g_{f_3}^{i,j,k} = \nabla_z f_3(x^i,y^{i,j},z^{i,j,k};\xi^{i,j,k})$. Further, notice that there will be a total of $J K$ updates to the LL variables starting from $z^i$ to obtain $z^{i+1}$. Thus, in general, taking expectation conditioned on~$\mathcal{F}_{i,j,k}$, we have
		\begin{alignat*}{2}
			\mathbb{E}[ \| z^{i,j,k+1} - z(x^i)\|^2 \vert \mathcal{F}_{i,j,k} ] & \;=\; \mathbb{E}[ \| z^{i,j,k} - \gamma_i g_{f_3}^{i,j,k} - z(x^i)\|^2 \vert \mathcal{F}_{i,j,k} ]\\
			%&\;=\; \| z^{i,j,k} - z(x^i) \|^2 - 2\gamma_i\left( z^{i,j,k} - z(x^i) \right)^\top \mathbb{E}[ g_{f_3}^{i,j,k} \vert \mathcal{F}_{i,j,k} ] + \gamma_i^2 \mathbb{E}[ \|g_{f_3}^{i,j,k}\|^2 \vert \mathcal{F}_{i,j,k} ]\\
			&\;=\; \| z^{i,j,k} - z(x^i) \|^2 - 2\gamma_i( z^{i,j,k} - z(x^i) )^\top \nabla_z f_3^{i,j,k} + \gamma_i^2 \mathbb{E}[ \|g_{f_3}^{i,j,k}\|^2 \vert \mathcal{F}_{i,j,k} ],
		\end{alignat*}
		where the last equality follows from the unbiasedness of the stochastic estimates (Assumption~\ref{as:TSG_unbiased_estimators}). Using the fact that Var$[X|Y] = \mathbb{E}[X^2|Y] - \mathbb{E}[ X |Y]^2$, where~$X$ and~$Y$ are random variables, along with Assumption~\ref{as:TSG_unbiased_estimators}, we have
		\begin{alignat*}{2}
			\mathbb{E}[ \| z^{i,j,k+1} - z(x^i)\|^2 \vert \mathcal{F}_{i,j,k} ] %& \;=\; \| z^{i,j,k} - z(x^i) \|^2 - 2\gamma_i\left( z^{i,j,k} - z(x^i) \right)^\top \nabla_z f_3^{i,j,k}\\
			%& \quad\quad+ \gamma_i^2 \| \mathbb{E}[ g_{f_3}^{i,j,k} \vert \mathcal{F}_{i,j,k} ]\|^2 + \gamma_i^2\text{Var}[g_{f_3}^{i,j,k} \vert \mathcal{F}_{i,j,k}]\\
			& \;\leq \;\| z^{i,j,k} - z(x^i) \|^2 - 2\gamma_i( z^{i,j,k} - z(x^i) )^\top \nabla_z f_3^{i,j,k} + \gamma_i^2  \|\nabla_z f_3^{i,j,k} \|^2 + \gamma_i^2\sigma_{\nabla f_3}^2.
		\end{alignat*}
		Now, utilizing~\cite[Theorem 2.1.12]{YNesterov_2018},
		% Now, utilizing Theorem 2.1.12 
		% from Nesterov's book \textcolor{red}{(cite Nesterov's book)} 
		which follows from the strong convexity and Lipschitz continuity of $f_3$ (Assumptions~\ref{as:tri_lip_cont} and~\ref{as:strong_conv_f3_z}, respectively), we have
		\begin{alignat*}{2}
			&\mathbb{E}[ \| z^{i,j,k+1} - z(x^i)\|^2 \vert \mathcal{F}_{i,j,k} ] \\
			& \;\leq\; \| z^{i,j,k} - z(x^i) \|^2 - 2\gamma_i\left( \frac{\mu_z L_{\nabla f_3}}{\mu_z + L_{\nabla f_3}}\| z^{i,j,k} - z(x^i) \|^2 + \frac{1}{\mu_z + L_{\nabla f_3}}\|\nabla_z f_3^{i,j,k} \|^2 \right) + \gamma_i^2  \|\nabla_z f_3^{i,j,k} \|^2 + \gamma_i^2\sigma_{\nabla f_3}^2\\
			& \;= \; \left( 1 - \frac{2\gamma_i\mu_z L_{\nabla f_3}}{\mu_z + L_{\nabla f_3}} \right)\| z^{i,j,k} - z(x^i) \|^2 + \gamma_i\left( \gamma_i - \frac{2}{\mu_z + L_{\nabla f_3}} \right)\|\nabla_z f_3^{i,j,k} \|^2 + \gamma_i^2\sigma_{\nabla f_3}^2\\
			& \;\leq\; \left( 1 - \gamma_i \rho_{f_3} \right)\| z^{i,j,k} - z(x^i) \|^2 + \gamma_i^2\sigma_{\nabla f_3}^2,
		\end{alignat*}
		where the last inequality follows from the assumption that~$\gamma_i\leq \frac{1}{\mu_z + L_{\nabla f_3}}$ and by letting 
		\begin{equation}\label{eq:1060}
			\rho_{f_3} \; := \; \frac{2\mu_z L_{\nabla f_3}}{\mu_z + L_{\nabla f_3}}.
		\end{equation}
		Using induction over~$K$ and taking total expectation, we obtain the bound~\eqref{eq:TSG_LLP_bound_1}.
		%\[\mathbb{E}[ \| z^{i,j+1} - z(x^i)\|^2 ] \;\leq\; \left( 1 - \gamma_i \rho_{f_3} \right)^K\mathbb{E}[ \| z^{i,j} - z(x^i) \|^2 ] + K\gamma_i^2\sigma_{\nabla f_3}^2.\]
		
		At this point, there would be an update in the ML variables $y$, i.e., \linebreak$(x^i,y^{i,j}, z^{i,j,K})\rightarrow(x^i,y^{i,j+1}, z^{i,j,K})$. However, since this upper-bound is not dependent on $y$, we can use induction over all $J$ iterations (each consisting of $K$ iterations), which yields the bound~\eqref{eq:TSG_LLP_bound_2}.
		%\[A^{(1)}_1 = \mathbb{E}[\| z^{i+1} - z(x^i)\|^2] \;\leq\; \left( 1 - \gamma_i \rho_{f_3} \right)^{J K}\mathbb{E}[\| z^{i} - z(x^i) \|^2] + J K\gamma_i^2\sigma_{\nabla f_3}^2.\]
		These results follow by ensuring that $0\leq 1 - \gamma_i \rho_{f_3}\leq 1$, which is satisfied by the assumption~$\gamma_i\leq \frac{1}{\mu_z + L_{\nabla f_3}}$ and recalling that~$\gamma_i$ and~$\rho_{f_3}$ are positive.
		
		\textbf{$(\text{Analysis of }A^{(1)}_2)$:} Taking expectation conditioned on~$\mathcal{F}_{i}$ and applying~\eqref{eq:lip_prop_1} yields
		\begin{alignat*}{2}
			\mathbb{E}[ \|z(x^i) - z(x^{i+1})\|^2 \vert \mathcal{F}_{i} ]& \;\leq\; L_z^2\mathbb{E}[ \|x^i - x^{i+1}\|^2 \vert \mathcal{F}_{i} ]
			%& = L_z^2\mathbb{E}[ \|x^i - x^i + \alpha_i\tilde g_{f_1}^i\|^2 \vert \mathcal{F}_{i}  ]\\
			\;=\; L_z^2\alpha_i^2\mathbb{E}[ \| \tilde g_{f_1}^i\|^2 \vert \mathcal{F}_{i} ].
		\end{alignat*}
		Adding and subtracting $\bar g_{f_1}^i = \mathbb{E}[\tilde g_{f_1}^i \vert \mathcal{F}_{i}]$ followed by using the fact that $\|a+b\|^2 \leq2\left(\|a\|^2+\|b\|^2\right)$, with~$a$ and~$b$ real-valued vectors, along with Lemma~\ref{lem:TSG_new_inexact_variance_bound}, we have
		\begin{alignat*}{2}
			\mathbb{E}[ \|z(x^i) - z(x^{i+1})\|^2 \vert \mathcal{F}_{i} ] & \;\leq\; L_z^2\alpha_i^2\mathbb{E}[ \| \tilde g_{f_1}^i - \bar g_{f_1}^i + \bar g_{f_1}^i \|^2 \vert \mathcal{F}_{i} ]
			%& \;\leq\;  2L_z^2\alpha_i^2\left(\mathbb{E}[ \| \bar g_{f_1}^i \|^2 \vert \mathcal{F}_{i} ] + \mathbb{E}[ \| \tilde g_{f_1}^i - \bar g_{f_1}^i \|^2 \vert \mathcal{F}_{i} ] \right)\\
			\;\leq\;  2L_z^2\alpha_i^2\left(\mathbb{E}[ \| \bar g_{f_1}^i \|^2 \vert \mathcal{F}_{i} ] + \tau \right).
		\end{alignat*}
		Lastly, taking total expectation, we obtain the bound
		\[\mathbb{E}[ \|z(x^i) - z(x^{i+1})\|^2 ] \;\leq\; 2L_z^2\alpha_i^2(\mathbb{E}[ \| \bar g_{f_1}^i \|^2 ] + \tau ).\]
		
		\textbf{$(\text{Analysis of }A^{(1)}_3)$:} Taking expectation conditioned on~$\mathcal{F}_{i}$ followed by adding and subtracting $\nabla_xz\left(x^i\right)^\top\left(x^{i+1} - x^i\right)$ in the following way:
		\begin{alignat}{2}
			&\mathbb{E}[ ( z^{i+1} - z(x^i) )^\top ( z(x^i) - z(x^{i+1}) ) \vert \mathcal{F}_{i} ]\nonumber\\
			%& \;= \; -\mathbb{E}[ \left( z^{i+1} - z\left(x^i\right) \right)^\top \left( z\left(x^{i+1}\right) - z\left(x^i\right) \right) \vert \mathcal{F}_{i} ]\nonumber\\
			& \;=\; -\mathbb{E}[ ( z^{i+1} - z(x^i) )^\top ( \nabla_xz(x^i)^\top(x^{i+1} - x^i) + z(x^{i+1}) - z(x^i) - \nabla_xz(x^i)^\top(x^{i+1} - x^i) ) \vert \mathcal{F}_{i} ]\nonumber\\
			& \;=\; \underbrace{-\mathbb{E}[ ( z^{i+1} - z(x^i) )^\top ( \nabla_xz(x^i)^\top(x^{i+1} - x^i) ) \vert \mathcal{F}_{i} ]}_{B^{(1)}_1} \nonumber\\
			& \quad\quad \underbrace{- \mathbb{E}[ ( z^{i+1} - z(x^i) )^\top ( z(x^{i+1}) - z(x^i) - \nabla_xz(x^i)^\top(x^{i+1} - x^i) ) \vert \mathcal{F}_{i} ]}_{B^{(1)}_2}.\label{eq:TSG_LLP_error_intermediate_2}
		\end{alignat}
		
		\textbf{$(\text{Analysis of }B^{(1)}_1)$:} Utilizing the update $x^{i+1} = x^i - \alpha_ig_{f_1}^i$, the fact that $\mathbb{E}[X]=\mathbb{E}[\mathbb{E}[X|Y]]$, along with the Cauchy-Schwarz inequality, yields
		\begin{alignat*}{2}
			%& -\mathbb{E}[ \left( z^{i+1} - z\left(x^i\right) \right)^\top ( \nabla_xz\left(x^i\right)^\top\left(x^{i+1} - x^i\right) ) \vert \mathcal{F}_{i} ]\\
			%& \;=\; -\mathbb{E}[ \left( z^{i+1} - z\left(x^i\right) \right)^\top ( \nabla_xz \left(x^i\right)^\top \mathbb{E}[\left(x^{i+1} - x^i\right) \vert \mathcal{F}_{i}] ) \vert \mathcal{F}_{i} ]\\
			%& \;=\; \alpha_i\mathbb{E}[ \left( z^{i+1} - z\left(x^i\right) \right)^\top ( \nabla_xz\left(x^i\right)^\top \bar g_{f_1}^i ) \vert \mathcal{F}_{i}]\\
			B^{(1)}_1 & \;\leq\; \alpha_i\mathbb{E}[ \| z^{i+1} - z(x^i) \|  \| \nabla_xz(x^i)^\top \bar g_{f_1}^i \| \vert \mathcal{F}_{i}]\\
			& \;\leq\; \alpha_iL_z\mathbb{E}[ \| z^{i+1} - z(x^i) \|  \| \bar g_{f_1}^i \| \vert \mathcal{F}_{i}]\\
			& \;\leq\; \kappa_i\mathbb{E}[ \| z^{i+1} - z(x^i) \|^2 \vert \mathcal{F}_{i}] + \frac{\alpha_i^2L_z^2}{4\kappa_i} \mathbb{E}[\| \bar g_{f_1}^i \|^2 \vert \mathcal{F}_{i}],
		\end{alignat*}
		where the second inequality comes from~\eqref{eq:lip_prop_1}, and the last inequality comes from using Young's inequality (i.e., $ab\leq \frac{\epsilon a^2}{2} + \frac{b^2}{2\epsilon}$ for $\epsilon>0$), where $\epsilon = 2\kappa_i$ for some $\kappa_i>0$, and where $a=\| z^{i+1} - z(x^i) \|$ and $b=\alpha_iL_z\| \bar g_{f_1}^i \|$.
		
		\textbf{$(\text{Analysis of }B^{(1)}_2)$:} Now, we can bound the term $B^{(1)}_2$ in~\eqref{eq:TSG_LLP_error_intermediate_2} by using the Cauchy-Schwarz inequality and applying the Lipschitz property of~\eqref{eq:lip_prop_5} (i.e, $z(x^{i+1}) - z(x^i) - \nabla z(x^i)(x^{i+1} - x^i) \leq \frac{L_{\nabla z}}{2}\| x^{i+1} - x^i \|^2$) to obtain:
		\begin{alignat*}{2}
			& - \mathbb{E}[ ( z^{i+1} - z(x^i) )^\top ( z(x^{i+1}) - z(x^i) - \nabla_xz(x^i)^\top(x^{i+1} - x^i) ) \vert \mathcal{F}_{i} ]\\
			%& \;\leq\; \mathbb{E}[ \| z^{i+1} - z\left(x^i\right) \| \| z\left(x^{i+1}\right) - z\left(x^i\right) - \nabla_xz\left(x^i\right)^\top\left(x^{i+1} - x^i\right) \| \vert \mathcal{F}_{i} ]\\
			%& \;\leq\; \frac{L_{\nabla z}}{2}\mathbb{E}[ \| z^{i+1} - z\left(x^i\right) \| \| x^{i+1} - x^i \|^2 \vert \mathcal{F}_{i} ]\\
			& \;\leq\; \frac{L_{\nabla z}}{2}\mathbb{E}[ \| z^{i+1} - z(x^i) \| \| x^{i+1} - x^i \| \| x^{i+1} - x^i \| \vert \mathcal{F}_{i}].
		\end{alignat*}
		Further, using Young's inequality with $a=\| z^{i+1} - z(x^i) \| \| x^{i+1} - x^i \|$ and $b=\| x^{i+1} - x^i \|$ such that $ab\leq \frac{a^2}{2} + \frac{b^2}{2}$, along with the update $x^{i+1} = x^i - \alpha_i \tilde g_{f_1}^i$ and the fact that $\mathbb{E}[X]=\mathbb{E}[\mathbb{E}[X|Y]]$, we have
		\begin{alignat*}{2}
			& - \mathbb{E}[ ( z^{i+1} - z(x^i) )^\top ( z(x^{i+1}) - z(x^i) - \nabla_xz(x^i)^\top(x^{i+1} - x^i) ) \vert \mathcal{F}_{i} ]\\
			& \;\leq\; \frac{L_{\nabla z}}{2}\left(\frac{1}{2}\mathbb{E}[ \| z^{i+1} - z(x^i) \|^2 \| x^{i+1} - x^i \|^2 \vert \mathcal{F}_{i}] + \frac{1}{2}\mathbb{E}[\| x^{i+1} - x^i \|^2 \vert \mathcal{F}_{i}] \right)\\
			%& \;=\; \frac{L_{\nabla z}\alpha_i^2}{4}\mathbb{E}[ \| z^{i+1} - z\left(x^i\right) \|^2  \mathbb{E}[\| \tilde g_{f_1}^i \|^2 \vert \mathcal{F}_{i}] \vert \mathcal{F}_{i}] + \frac{L_{\nabla z}\alpha_i^2}{4}\mathbb{E}[\| \tilde g_{f_1}^i \|^2 \vert \mathcal{F}_{i}]\\
			& \;\leq\; \frac{L_{\nabla z}\alpha_i^2\zeta}{4}\mathbb{E}[ \| z^{i+1} - z(x^i) \|^2 \vert \mathcal{F}_{i}] + \frac{L_{\nabla z}\alpha_i^2}{4}\mathbb{E}[ \| \tilde g_{f_1}^i \|^2 \vert \mathcal{F}_{i} ]\\
			& \leq \frac{L_{\nabla z}\alpha_i^2\zeta}{4}\mathbb{E}[ \| z^{i+1} - z(x^i) \|^2 \vert \mathcal{F}_{i}] + \frac{L_{\nabla z}\alpha_i^2}{4}( \mathbb{E}[ \| \bar g_{f_1}^i \|^2 \vert \mathcal{F}_{i} ] + \tau ).
		\end{alignat*}
		where the second inequality follows by applying Lemma~\ref{lem:TSG_size_inexact_g} and the last follows by applying the definition of variance along with Lemma~\ref{lem:TSG_new_inexact_variance_bound}. % Using the fact that Var$[X|Y] = \mathbb{E}[X^2|Y] - \mathbb{E}[ X |Y]^2$, where~$X$ and~$Y$ are random variables, along with Lemma~\ref{lem:TSG_new_inexact_variance_bound} and the fact that $\| \bar g_{f_1}^i \|^2 = \mathbb{E}[ \| \bar g_{f_1}^i \|^2 \vert \mathcal{F}_{i} ]$ (since $\| \bar g_{f_1}^i \|^2$ is a constant), we have
		%\begin{alignat*}{2}
		%    & - \mathbb{E}[ \left( z^{i+1} - z\left(x^i\right) \right)^\top \left( z\left(x^{i+1}\right) - z\left(x^i\right) - \nabla_xz\left(x^i\right)^\top\left(x^{i+1} - x^i\right) \right) \vert \mathcal{F}_{i} ]\\
		%    & \leq \frac{L_{\nabla z}\alpha_i^2\zeta}{4}\mathbb{E}[ \| z^{i+1} - z\left(x^i\right) \|^2 \vert \mathcal{F}_{i}] + \frac{L_{\nabla z}\alpha_i^2}{4}\left( \| \bar g_{f_1}^i \|^2 + \text{Var}[ \tilde g_{f_1}^i \vert \mathcal{F}_{i} ] \right)\\
		%    & \leq \frac{L_{\nabla z}\alpha_i^2\zeta}{4}\mathbb{E}[ \| z^{i+1} - z\left(x^i\right) \|^2 \vert \mathcal{F}_{i}] + \frac{L_{\nabla z}\alpha_i^2}{4}\left( \mathbb{E}[ \| \bar g_{f_1}^i \|^2 \vert \mathcal{F}_{i} ] + \tau \right).
		%\end{alignat*}
		
		Substituting these bounds for $B^{(1)}_1$ and $B^{(1)}_2$ back into~\eqref{eq:TSG_LLP_error_intermediate_2} and taking total expectation, we obtain the bound on the term $A^{(1)}_3$ as
		\begin{alignat*}{2}
			&\mathbb{E}[ ( z^{i+1} - z(x^i) )^\top ( z(x^i) - z(x^{i+1}) ) ]\\
			%& \leq \kappa_i\mathbb{E}[ \| z^{i+1} - z\left(x^i\right) \|^2] + \frac{\alpha_i^2L_z^2}{4\kappa_i} \mathbb{E}[\| \bar g_{f_1}^i \|^2]\\
			%&\quad+ \frac{L_{\nabla z}\alpha_i^2\zeta}{4}\mathbb{E}[ \| z^{i+1} - z\left(x^i\right) \|^2] + \frac{L_{\nabla z}\alpha_i^2}{4}\left( \mathbb{E}[ \| \bar g_{f_1}^i \|^2 ] + \tau \right)\\
			& \leq \left( \kappa_i + \frac{L_{\nabla z}\alpha_i^2\zeta}{4} \right) \mathbb{E}[ \| z^{i+1} - z(x^i) \|^2] + \left( \frac{\alpha_i^2L_z^2}{4\kappa_i} + \frac{L_{\nabla z}\alpha_i^2}{4} \right) \mathbb{E}[\| \bar g_{f_1}^i \|^2] + \frac{\tau L_{\nabla z}}{4}\alpha_i^2.
		\end{alignat*}
		Finally, substituting these bounds for $A^{(1)}_1$, $A^{(1)}_2$, and $A^{(1)}_3$ back into~\eqref{eq:TSG_LLP_error_intermediate_1}, we obtain the desired upper-bound on $\mathbb{E}[\| z^{i+1}-z(x^{i+1})\|^2]$, completing the proof.
	\end{proof}

	\subsection{Proof of Lemma~\ref{lem:TSG_aux_LL_descent}}\label{app:lem:TSG_aux_LL_descent}
	\begin{proof}
		To derive the error bound defined by~\eqref{eq:TSG_aux_LL_bound_1}, recall that $z^{i+1} = z^{i+1,0,0} = z^{i,J,K}$ and $g_{f_3}^{i,j,k} = \nabla_z f_3(x^i,y^{i,j},z^{i,j,k};\xi^{i,j,k})$ and notice that there will be a total of $K$ updates to the LL variables starting from $z^{i,j}$ to obtain $z^{i,j+1}$. Further, following the exact same steps utilized in Lemma~\ref{lem:TSG_LL_descent} to derive bound~\eqref{eq:TSG_LLP_bound_1} (only with $z(x^i)$ replaced with $z(x^i,y^{i,j+1})$), we have
		\[\mathbb{E}[ \| z^{i,j+1} - z(x^i,y^{i,j+1})\|^2 ] \;\leq\; \left( 1 - \gamma_i \rho_{f_3} \right)^K \| z^{i,j} - z(x^i,y^{i,j+1}) \|^2 + K\gamma_i^2\sigma_{\nabla f_3}^2.\]
		Now, adding and subtracting $z(x^i,y^{i,j})$ in the norm, followed by using the fact that $\|a+b\|^2 = \|a\|^2 + \|b\|^2 + 2a^\top b$, with~$a$ and~$b$ real-valued vectors, the Cauchy-Schwarz inequality, and the fact that $\left( 1 - \gamma_i \rho_{f_3} \right)^K\leq 1$ which is satisfied by our choice of~$\gamma_i\leq \frac{1}{\mu_z + L_{\nabla f_3}}$, we have
		\begin{alignat*}{2}
			&\mathbb{E}[ \| z^{i,j+1} - z(x^i,y^{i,j+1})\|^2 ]\\
			%&\tcr{\;\leq\; \left( 1 - \gamma_i \rho_{f_3} \right)^K \| z^{i,j} - z(x^i,y^{i,j})\|^2 + \left( 1 - \gamma_i \rho_{f_3} \right)^K\|z(x^i,y^{i,j}) - z(x^i,y^{i,j+1}) \|^2 + K\gamma_i^2\sigma_{\nabla f_3}^2}\\
			%&\quad \tcr{+ 2(z^{i,j} - z(x^i,y^{i,j}))^\top (z(x^i,y^{i,j}) - z(x^i,y^{i,j+1}))}\\
			&\;\leq\; \left( 1 - \gamma_i \rho_{f_3} \right)^K \| z^{i,j} - z(x^i,y^{i,j})\|^2 + \left( 1 - \gamma_i \rho_{f_3} \right)^K\|z(x^i,y^{i,j}) - z(x^i,y^{i,j+1}) \|^2 + K\gamma_i^2\sigma_{\nabla f_3}^2\\
			&\quad + 2\|z^{i,j} - z(x^i,y^{i,j})\|  \|z(x^i,y^{i,j}) - z(x^i,y^{i,j+1})\|\\
			&\;\leq\; \left( 1 - \gamma_i \rho_{f_3} \right)^K \| z^{i,j} - z(x^i,y^{i,j})\|^2 + \left( 1 - \gamma_i \rho_{f_3} \right)^K\|z(x^i,y^{i,j}) - z(x^i,y^{i,j+1}) \|^2 + K\gamma_i^2\sigma_{\nabla f_3}^2\\
			&\quad + \eta_i\|z^{i,j} - z(x^i,y^{i,j})\|^2 + \frac{1}{\eta_i}\|z(x^i,y^{i,j}) - z(x^i,y^{i,j+1})\|^2\\
			%&\tcr{\;=\; \left(\left( 1 - \gamma_i \rho_{f_3} \right)^K + \hat\eta_i\right) \| z^{i,j} - z(x^i,y^{i,j})\|^2 + \left(\left( 1 - \gamma_i \rho_{f_3} \right)^K +\frac{1}{\hat\eta_i}\right)\|z(x^i,y^{i,j}) - z(x^i,y^{i,j+1}) \|^2 + K\gamma_i^2\sigma_{\nabla f_3}^2}\\
			&\;\leq\; (\left( 1 - \gamma_i \rho_{f_3} \right)^K + \eta_i) \| z^{i,j} - z(x^i,y^{i,j})\|^2 + \left(\left( 1 - \gamma_i \rho_{f_3} \right)^K +\frac{1}{\eta_i}\right)L_{z_y}^2\|y^{i,j+1} - y^{i,j} \|^2 + K\gamma_i^2\sigma_{\nabla f_3}^2\\
			&\;=\; (\left( 1 - \gamma_i \rho_{f_3} \right)^K + \eta_i) \| z^{i,j} - z(x^i,y^{i,j})\|^2 + \left(\left( 1 - \gamma_i \rho_{f_3} \right)^K +\frac{1}{\eta_i}\right)L_{z_y}^2\|y^{i,j} - \beta_i \tilde g_{f_2}^{i,j} - y^{i,j} \|^2 \\
			&\quad + K\gamma_i^2\sigma_{\nabla f_3}^2\\
			%&\tcr{\;=\; \left(\left( 1 - \gamma_i \rho_{f_3} \right)^K + \hat\eta_i\right) \| z^{i,j} - z(x^i,y^{i,j})\|^2 + \left(\left( 1 - \gamma_i \rho_{f_3} \right)^K +\frac{1}{\hat\eta_i}\right)L_{z_y}^2\beta_i^2 \|\tilde g_{f_2}^{i,j}\|^2 + K\gamma_i^2\sigma_{\nabla f_3}^2}\\
			%&\tcr{\;\leq\; \left(\left( 1 - \gamma_i \rho_{f_3} \right)^K + \hat\eta_i\right) \| z^{i,j} - z(x^i,y^{i,j})\|^2 + \left(1 +\frac{1}{\hat\eta_i}\right)L_{z_y}^2\beta_i^2 \|\tilde g_{f_2}^{i,j}\|^2 + K\gamma_i^2\sigma_{\nabla f_3}^2}\\
			&\;\leq\; (\left( 1 - \gamma_i \rho_{f_3} \right)^K + \eta_i) \| z^{i,j} - z(x^i,y^{i,j})\|^2 + \hat{\eta_i}L_{z_y}^2\beta_i^2 \|\tilde g_{f_2}^{i,j}\|^2 + K\gamma_i^2\sigma_{\nabla f_3}^2,
		\end{alignat*}
		where the second inequality follows from applying Young's inequality (i.e., $ab\leq\frac{\epsilon a^2}{2}+\frac{b^2}{2\epsilon}$ for $\epsilon>0$) with $\epsilon=\eta_i$ for some $\eta_i>0$ (notice that $a=\|z^{i,j} - z(x^i,y^{i,j})\|$ and $b=\|z(x^i,y^{i,j}) - z(x^i,y^{i,j+1})\|$ here), the third inequality follows from applying~\eqref{eq:lip_prop_3}, and the last inequality follows from the fact that $0\leq 1 - \gamma_i \rho_{f_3}\leq1$ (where we define $\hat{\eta_i}:=1 +\frac{1}{\eta_i}$). Lastly, taking total expectation and using the fact that $\mathbb{E}[X]=\mathbb{E}[\mathbb{E}[X|Y]]$, we will obtain the bound~\eqref{eq:TSG_aux_LL_bound_1}
		\begin{alignat*}{2}
			&\mathbb{E}[ \| z^{i,j+1} - z(x^i,y^{i,j+1})\|^2 ] \\
			%& \;\leq\; \left(\left( 1 - \gamma_i \rho_{f_3} \right)^K + \hat\eta_i\right) \mathbb{E}[\| z^{i,j} - z(x^i,y^{i,j})\|^2] + \hat{\hat{\eta_i}}L_{z_y}^2\beta_i^2 \mathbb{E}[\|\tilde g_{f_2}^{i,j}\|^2] + K\gamma_i^2\sigma_{\nabla f_3}^2\\
			&\;\leq\; (\left( 1 - \gamma_i \rho_{f_3} \right)^K + \eta_i) \mathbb{E}[\| z^{i,j} - z(x^i,y^{i,j})\|^2] + \hat{\eta_i}L_{z_y}^2\beta_i^2 \mathbb{E}[\mathbb{E}[\|\tilde g_{f_2}^{i,j}\|^2 | \mathcal{F}_{i,j}]] + K\gamma_i^2\sigma_{\nabla f_3}^2\\
			&\;\leq\; (\left( 1 - \gamma_i \rho_{f_3} \right)^K + \eta_i) \mathbb{E}[\| z^{i,j} - z(x^i,y^{i,j})\|^2] + \hat{\eta_i}L_{z_y}^2\beta_i^2 \Upsilon + K\gamma_i^2\sigma_{\nabla f_3}^2,
		\end{alignat*}
		where the last inequality follows by applying Lemma~\ref{lem:TSG_size_inexact_g_ML}.

		Now, to derive results~\eqref{eq:TSG_aux_LL_bound_2}, \eqref{eq:TSG_aux_LL_bound_3}, and~\eqref{eq:TSG_aux_LL_bound_4}, we start by decomposing the expected error of the LL variables by adding and subtracting $z(x^i,y^i)$ followed by using the fact that $\|a+b\|^2\leq2(\|a\|^2+\|b\|^2)$ with $a$ and $b$ real-valued vectors:
		\begin{alignat}{2}
			\mathbb{E}[\| z^{i+1}-z(x^{i+1},y^{i+1})\|^2] &\leq 2\underbrace{\mathbb{E}[\| z^{i+1} - z(x^i,y^i)\|^2]}_{A^{(2)}_1} + 2\underbrace{\mathbb{E}[\|z(x^i,y^i) - z(x^{i+1},y^{i+1})\|^2]}_{A^{(2)}_2}.\label{eq:TST_intermidiate_result_for_LL_proof}
			%&\quad+ 2\underbrace{\mathbb{E}[\left( z^{i+1} - z\left(x^i,y^i\right) \right)^\top \left( z\left(x^i,y^i\right) - z\left(x^{i+1},y^{i+1}\right) \right)]}_{A^{(2)}_3}.\label{eq:TST_intermidiate_result_for_LL_proof}
		\end{alignat}
		
		\textbf{$(\text{Analysis of }A^{(2)}_1)$:} To derive an upper-bound on $A^{(2)}_1$ in~\eqref{eq:TST_intermidiate_result_for_LL_proof}, we can follow the exact same steps that were utilized in Lemma~\ref{lem:TSG_LL_descent} to derive bound~\eqref{eq:TSG_LLP_bound_1} (only with $z(x^i)$ replaced with $z(x^i,y^{i})$), which will yield the bound~\eqref{eq:TSG_aux_LL_bound_2}.
		%\[\mathbb{E}[ \| z^{i,j+1} - z(x^i,y^{i})\|^2 ] \;\leq\; \left( 1 - \gamma_i \rho_{f_3} \right)^K\mathbb{E}[ \| z^{i,j} - z(x^i,y^{i}) \|^2 ] + K\gamma_i^2\sigma_{\nabla f_3}^2.\]
		Further, using induction over $J$ (each consisting of $K$ iterations) will yield the following bound on $A^{(2)}_1$ in~\eqref{eq:TST_intermidiate_result_for_LL_proof} (which is the bound~\eqref{eq:TSG_aux_LL_bound_3}).
		%\[\mathbb{E}[\| z^{i+1} - z(x^{i},y^{i})\|^2] \;\leq\; \left( 1 - \gamma_i \rho_{f_3} \right)^{J K}\mathbb{E}[\| z^{i} - z(x^i,y^i) \|^2] + J K\gamma_i^2\sigma_{\nabla f_3}^2.\]
		Notice that this induction result again follows by ensuring that $0\leq 1 - \gamma_i \rho_{f_3}\leq 1$, which is satisfied by the assumption~$\gamma_i\leq \frac{1}{\mu_z + L_{\nabla f_3}}$ and recalling that~$\gamma_i$ and~$\rho_{f_3}$ are positive.
		
		\textbf{$(\text{Analysis of }A^{(2)}_2)$:} Now, the upper-bound on $A^{(2)}_2$ in~\eqref{eq:TST_intermidiate_result_for_LL_proof} can be derived by taking total expectation, using the fact that $\mathbb{E}[X]=\mathbb{E}[\mathbb{E}[X|Y]]$, applying~\eqref{eq:lip_prop_2}, and recursively using the fact that $y^{i,j+1} = y^{i,j} - \beta_i\tilde g_{f_2}^{i,j}$ (while recalling that $y^{i+1}=y^{i,J}$ and $x^{i+1}=x^{i}-\alpha_i\tilde g_{f_1}^{i}$):
		\begin{alignat*}{2}
			\mathbb{E}[ \|z(x^i,y^i) - z(x^{i+1},y^{i+1})\|^2 ] %& \;\leq\; L_{z_{xy}}^2 \mathbb{E}[\| x^i - x^{i+1}\|^2 ] + L_{z_{xy}}^2\mathbb{E}[\| y^i - y^{i+1} \|^2 ]\\
			& \leq L_{z_{xy}}^2 \mathbb{E}[\| x^i - x^{i+1}\|^2 ] + L_{z_{xy}}^2\mathbb{E}[\mathbb{E}[\| y^i - y^{i+1} \|^2 \vert \mathcal{F}_{i,j}]]\\
			%& = L_{z_{xy}}^2 \mathbb{E}[\| x^i - \left(x^{i} - \alpha_i \tilde g_{f_1}^{i}\right)\|^2] + L_{z_{xy}}^2\mathbb{E}[\mathbb{E}[\| y^{i,J} - y^i \|^2 \vert \mathcal{F}_{i,j}]]\\
			& = L_{z_{xy}}^2 \alpha_i^2\mathbb{E}[\|\tilde g_{f_1}^{i}\|^2] + L_{z_{xy}}^2\mathbb{E}[\mathbb{E}[\| y^{i} - \sum_{j=0}^{J-1} \beta_i\tilde g_{f_2}^{i,j} - y^i \|^2 \vert \mathcal{F}_{i,j}]]\\
			%&= L_{z_{xy}}^2 \alpha_i^2\mathbb{E}[\|\tilde g_{f_1}^{i}\|^2] + L_{z_{xy}}^2\beta_i^2\mathbb{E}[\mathbb{E}[\| \sum_{j=0}^{J-1} \tilde g_{f_2}^{i,j} \|^2 \vert \mathcal{F}_{i,j}]]\\
			&\;\leq\; L_{z_{xy}}^2 \alpha_i^2\mathbb{E}[\|\tilde g_{f_1}^{i}\|^2] + JL_{z_{xy}}^2\beta_i^2\sum_{j=0}^{J-1}\mathbb{E}[\mathbb{E}[\| \tilde g_{f_2}^{i,j} \|^2 \vert \mathcal{F}_{i,j}]]\\
			&\;\leq\; L_{z_{xy}}^2 \alpha_i^2\mathbb{E}[\|\tilde g_{f_1}^{i}\|^2] + J^2\Upsilon L_{z_{xy}}^2\beta_i^2,
		\end{alignat*}
		where the second inequality follows from using the fact that $\|\sum_{i=1}^N a_i\|^2 \leq N\sum_{i=1}^N\|a_i\|^2$ (for some $a\in\mathbb{R}^N$) and the last inequality follows from applying Lemma~\ref{lem:TSG_size_inexact_g_ML}. Now, using the fact that $\mathbb{E}[X]=\mathbb{E}[\mathbb{E}[X|Y]]$, adding and subtracting $\bar g_{f_1}^{i}$ in the norm, followed by using the fact that $\|a+b\|^2 \leq 2\left(\|a\|^2+\|b\|^2\right)$, and applying Lemma~\ref{lem:TSG_new_inexact_variance_bound}, we have
		\begin{alignat*}{2}
			\mathbb{E}[ \|z(x^i,y^i) - z(x^{i+1},y^{i+1})\|^2 ] & = L_{z_{xy}}^2 \alpha_i^2\mathbb{E}[\mathbb{E}[\|\tilde g_{f_1}^{i}\|^2 \vert \mathcal{F}_{i}]] + J^2\Upsilon L_{z_{xy}}^2\beta_i^2\\
			%& \;\leq\; L_{z_{xy}}^2 \alpha_i^2\mathbb{E}[\mathbb{E}[\|\tilde g_{f_1}^{i} - \bar g_{f_1}^{i} + \bar g_{f_1}^{i}\|^2 \vert \mathcal{F}_{i}]] + J^2\Upsilon L_{z_{xy}}^2\beta_i^2\\
			& \;\leq\; 2L_{z_{xy}}^2 \alpha_i^2(\mathbb{E}[\|\bar g_{f_1}^{i}\|^2] + \tau) + J^2\Upsilon L_{z_{xy}}^2\beta_i^2.
		\end{alignat*}
		Notice in the inequality that $\mathbb{E}[\|\bar g_{f_1}^{i}\|^2] = \mathbb{E}[\mathbb{E}[\|\bar g_{f_1}^{i}\|^2 \vert \mathcal{F}_{i}]] = \|\bar g_{f_1}^{i}\|^2$ since $\bar g_{f_1}^{i}$ is deterministic.
		Finally, substituting these bounds for $A^{(2)}_1$ and $A^{(2)}_2$ back into~\eqref{eq:TST_intermidiate_result_for_LL_proof}, we can obtain the desired upper-bound on $\mathbb{E}[ \| z^{i+1} - z\left( x^{i+1}, y^{i+1} \right) \|^2 ]$ defined by~\eqref{eq:TSG_aux_LL_bound_4}.
		%\begin{alignat*}{2}
		%    &\mathbb{E}[ \| z^{i+1} - z\left( x^{i+1}, y^{i+1} \right) \|^2 ]\\
		%    &\leq \mathbb{E}[\| z^{i+1} - z(x^i,y^i)\|^2] + 2L_{z_{xy}}^2 \alpha_i^2\mathbb{E}[\|\bar g_{f_1}^{i}\|^2] + 2L_{z_{xy}}^2 \alpha_i^2\tau + J^2\Upsilon L_{z_{xy}}^2\beta_i^2\\
		%    &\quad+ 2\left( \eta_i + \upsilon_i + \frac{L_{z_{xy}}\alpha_i^2 \zeta}{4} + \frac{J^2\Upsilon L_{z_{xy}}\beta_i^2}{4}\right)\mathbb{E}[\| z^{i+1} - z\left(x^i,y^i\right) \|^2]\\
		%    &\quad + 2\left( \frac{\alpha_i^2L_{z_{xy}}^2}{4\eta_i} + \frac{L_{z_{xy}}\alpha_i^2}{4}\right)\mathbb{E}[\| \bar g_{f_1}^i \|^2] + 2\frac{L_{z_{xy}}\alpha_i^2}{4}\tau + 2\left( \frac{L_{z_{xy}}}{\upsilon_i} + 1\right)\frac{J^2\Upsilon L_{z_{xy}}\beta_i^2}{4}\\
		%    & = \left( 1 + 2\left(\eta_i + \upsilon_i\right) + \frac{L_{z_{xy}}\alpha_i^2 \zeta}{2} + \frac{J^2\Upsilon L_{z_{xy}}\beta_i^2}{2}\right)\mathbb{E}[\| z^{i+1} - z\left(x^i,y^i\right) \|^2]\\
		%    &\quad+ \left( 2L_{z_{xy}}^2 + \frac{L_{z_{xy}}^2}{2\eta_i} + \frac{L_{z_{xy}}}{2} \right) \alpha_i^2 \mathbb{E}[\|\bar g_{f_1}^{i}\|^2] + \left(2L_{z_{xy}}^2 + \frac{L_{z_{xy}}}{2} \right) \alpha_i^2\tau + \left( 2L_{z_{xy}} + \frac{L_{z_{xy}}}{\upsilon_i} + 1\right)\frac{J^2\Upsilon L_{z_{xy}}\beta_i^2}{2}.
		%\end{alignat*}
		
		Lastly, to derive the upper-bound~\eqref{eq:TSG_aux_LL_bound_5}, we can follow the exact same steps that were utilized in Lemma~\ref{lem:TSG_LL_descent} to derive bound~\eqref{eq:TSG_LLP_bound_1} (only with $z(x^i)$ replaced with $z(x^i,y^{i,j})$).
		%\[\mathbb{E}[ \| z^{i,j+1} - z(x^i,y^{i,j}) \|^2 ] \leq \left( 1 - \gamma_i \rho_{f_3} \right)^{K}\mathbb{E}[\| z^{i,j} - z(x^i,y^{i,j}) \|^2] + K\gamma_i^2\sigma_{\nabla f_3}^2.\]
		Notice that this induction result again follows by ensuring that $0\leq 1 - \gamma_i \rho_{f_3}\leq 1$, which is satisfied by the assumption of~$\gamma_i\leq \frac{1}{\mu_z + L_{\nabla f_3}}$ and recalling that~$\gamma_i$ and~$\rho_{f_3}$ are positive.% This completes the proof.
	\end{proof}

	\subsection{Proof of Lemma~\ref{lem:TSG_ML_descent}}\label{app:lem:TSG_ML_descent}
	\begin{proof}
		To derive the error bound defined by~\eqref{eq:TSG_ML_bound_2}, we start by decomposing the expected error of the~LL variables by adding and subtracting $y(x^i)$ in the following way:
		\begin{alignat}{2}
			\mathbb{E}[\| y^{i+1}-y(x^{i+1})\|^2] &= \underbrace{\mathbb{E}[\| y^{i+1} - y(x^i)\|^2]}_{A^{(3)}_1} + \underbrace{\mathbb{E}[\|y(x^i) - y(x^{i+1})\|^2]}_{A^{(3)}_2}\nonumber\\
			&\quad+ 2\underbrace{\mathbb{E}[( y^{i+1} - y(x^i) )^\top ( y(x^i) - y(x^{i+1}) )]}_{A^{(3)}_3}.\label{eq:TSG_MLP_error_intermediate_1}
		\end{alignat}
		
		\textbf{$(\text{Analysis of }A^{(3)}_1)$:} To derive an upper-bound on~$A^{(3)}_1$ in~\eqref{eq:TSG_MLP_error_intermediate_1}, recall that~$y^{i+1} = y^{i,J}$ and $\tilde g_{f_2}^{i,j} = \nabla_y \Bar{f}(x^i,y^{i,j},z^{i,j+1};\xi^{i,j})$. Further, notice that there will be a total of~$J$ updates to the~ML variables starting from~$y^i$ to obtain~$y^{i+1}$. Thus, in general, taking expectation conditioned on~$\mathcal{F}_{i,j}$ and applying Lemma~\ref{lem:TSG_size_inexact_g_ML}, we have
		\begin{alignat*}{2}
			\mathbb{E}[ \| y^{i,j+1} - y(x^i)\|^2 \vert \mathcal{F}_{i,j} ] & \;=\; \mathbb{E}[ \| y^{i,j} - \beta_i \tilde g_{f_2}^{i,j} - y(x^i)\|^2 \vert \mathcal{F}_{i,j} ]\\
			%&\;=\; \| y^{i,j} - y(x^i) \|^2 - 2\beta_i\left( y^{i,j} - y(x^i) \right)^\top \mathbb{E}[ \tilde g_{f_2}^{i,j} \vert \mathcal{F}_{i,j} ] + \beta_i^2 \mathbb{E}[ \| \tilde g_{f_2}^{i,j}\|^2 \vert \mathcal{F}_{i,j} ]\\
			&\;\leq\; \| y^{i,j} - y(x^i) \|^2 - 2\beta_i( y^{i,j} - y(x^i) )^\top \bar g_{f_2}^{i,j} + \Upsilon \beta_i^2\\
			& \;=\; \| y^{i,j} - y(x^i) \|^2 + \Upsilon \beta_i^2 - 2\beta_i( y^{i,j} - y(x^i) )^\top \nabla_y \Bar{f}(x^i,y^{i,j})\\
			&\quad - 2\beta_i( y^{i,j} - y(x^i) )^\top (\bar g_{f_2}^{i,j} - \nabla_y \Bar{f}(x^i,y^{i,j}) ),
		\end{alignat*}
		where the last equality follows from adding and subtracting $\nabla_y \Bar{f}(x^i,y^{i,j})$ to the $\bar g_{f_2}^{i,j}$ term in the cross-product.
		%\begin{alignat*}{2}
		%    \mathbb{E}[ \| y^{i,j+1} - y(x^i)\|^2 \vert \mathcal{F}_{i,j} ] & \;\leq\; \| y^{i,j} - y(x^i) \|^2 + \Upsilon \beta_i^2 \\
		%    &\quad - 2\beta_i\left( y^{i,j} - y(x^i) \right)^\top \left(\bar g_{f_2}^{i,j} - \nabla_y \Bar{f}(x^i,y^{i,j}) + \nabla_y \Bar{f}(x^i,y^{i,j}) \right)\\
		%    & \;=\; \| y^{i,j} - y(x^i) \|^2 + \Upsilon \beta_i^2 - 2\beta_i\left( y^{i,j} - y(x^i) \right)^\top \nabla_y \Bar{f}(x^i,y^{i,j})\\
		%    &\quad - 2\beta_i\left( y^{i,j} - y(x^i) \right)^\top \left(\bar g_{f_2}^{i,j} - \nabla_y \Bar{f}(x^i,y^{i,j}) \right).
		%\end{alignat*}
		Now, under the strong convexity of~$\bar f$ (Assumption~\ref{as:strong_conv_fbar_y}) and the Lipschitz continuity of~$\nabla_y\bar f$ in~$y$ (equation~\eqref{eq:lip_prop_7}), we can utilize~\cite[Theorem 2.1.12]{YNesterov_2018}, 
		% the result from Nesterov's book (Theorem 2.1.12) \tcr{(cite the book)}, 
		yielding
		\begin{alignat*}{2}
			\mathbb{E}[ \| y^{i,j+1} - y(x^i)\|^2 \vert \mathcal{F}_{i,j} ] %& \;\leq\; \| y^{i,j} - y(x^i) \|^2 + \Upsilon \beta_i^2 \\
			%&\quad- 2\beta_i \left( \frac{\mu_y L_{\nabla \bar f}}{\mu_y + L_{\nabla \bar f}}\| y^{i,j} - y(x^i) \|^2 + \frac{1}{\mu_y + L_{\nabla \bar f}}\| \nabla_y \Bar{f}(x^i,y^{i,j}) \|^2 \right)\\
			%&\quad - 2\beta_i\left( y^{i,j} - y(x^i) \right)^\top \left(\bar g_{f_2}^{i,j} - \nabla_y \Bar{f}(x^i,y^{i,j}) \right)\\
			& \;\leq\; \| y^{i,j} - y(x^i) \|^2 + \Upsilon \beta_i^2 \\
			&\quad- 2\beta_i \left( \frac{\mu_y L_{\nabla \bar f}}{\mu_y + L_{\nabla \bar f}}\| y^{i,j} - y(x^i) \|^2 + \frac{1}{\mu_y + L_{\nabla \bar f}}\| \nabla_y \Bar{f}(x^i,y^{i,j}) \|^2 \right)\\
			&\quad + 2\beta_i\| y^{i,j} - y(x^i) \|^2\| \bar g_{f_2}^{i,j} - \nabla_y \Bar{f}(x^i,y^{i,j},z^{i,j+1}) \|^2\\
			&\quad + 2\beta_i\| y^{i,j} - y(x^i) \| \| \nabla_y \Bar{f}(x^i,y^{i,j},z^{i,j+1}) - \nabla_y \Bar{f}(x^i,y^{i,j}) \|,
		\end{alignat*}
		where the last two added terms come from adding and subtracting $\nabla_y \Bar{f}(x^i,y^{i,j},z^{i,j+1})$ to the $\bar g_{f_2}^{i,j} - \nabla_y \Bar{f}(x^i,y^{i,j})$ term in the cross product followed by applying the Cauchy Schwarz inequality.
		Now, utilizing the Lipschitz continuity of $\nabla_y \Bar{f}$ in $z$ (equation~\eqref{eq:lip_prop_8}), the bound on the biasedness of $\tilde g_{f_2}$ (Lemma~\ref{lem:TSG_var_bound_ML_dir}), and the fact that $\frac{2\beta_i}{\mu_y + L_{\nabla \bar f}}\| \nabla_y \Bar{f}(x^i,y^{i,j}) \|^2$ is non-negative, we have
		\begin{alignat*}{2}
			&\mathbb{E}[ \| y^{i,j+1} - y(x^i)\|^2 \vert \mathcal{F}_{i,j} ]\\ %& \;\leq\; \| y^{i,j} - y(x^i) \|^2 + \Upsilon \beta_i^2 \\
			%&\quad- 2\beta_i \left( \frac{\mu_y L_{\nabla \bar f}}{\mu_y + L_{\nabla \bar f}}\| y^{i,j} - y(x^i) \|^2 + \frac{1}{\mu_y + L_{\nabla \bar f}}\| \nabla_y \Bar{f}(x^i,y^{i,j}) \|^2 \right)\\
			%&\quad + 2\hat\omega^2\beta_i\theta_i^2\| y^{i,j} - y(x^i) \|^2 + 2\beta_i\| y^{i,j} - y(x^i) \| \| z^{i,j+1} - z(x^i,y^{i,j}) \|\\
			%& \;\leq\;  \left( 1 - \frac{2\beta_i\mu_y L_{\nabla \bar f}}{\mu_y + L_{\nabla \bar f}} + 2\hat\omega^2\beta_i\theta_i^2 \right)\| y^{i,j} - y(x^i) \|^2 + \Upsilon \beta_i^2 \\
			%&\quad  + 2\beta_i\| y^{i,j} - y(x^i) \| \| z^{i,j+1} - z(x^i,y^{i,j}) \|\\
			& \;\leq\;  \left( 1 - \beta_i\left(\frac{2\mu_y L_{\nabla \bar f}}{\mu_y + L_{\nabla \bar f}} - 2\hat\omega^2\theta_i^2\right) \right)\| y^{i,j} - y(x^i) \|^2 + 2\beta_i\| y^{i,j} - y(x^i) \| \| z^{i,j+1} - z(x^i,y^{i,j}) \| + \Upsilon \beta_i^2\\
			%\mathbb{E}[ \| y^{i,j+1} - y(x^i)\|^2 \vert \mathcal{F}_{i,j} ] & \;\leq\; \left( 1 - \beta_i\left(\frac{2\mu_y L_{\nabla \bar f}}{\mu_y + L_{\nabla \bar f}} - 2\hat\omega^2\theta_i^2\right) \right)\| y^{i,j} - y(x^i) \|^2 + \beta_i^2\| y^{i,j} - y(x^i) \|^2\\
			%&\quad + \| z^{i,j+1} - z(x^i,y^{i,j}) \|^2 + \Upsilon \beta_i^2\\
			& \leq \left( 1 - \beta_i\left(\frac{2\mu_y L_{\nabla \bar f}}{\mu_y + L_{\nabla \bar f}} - 2\hat\omega^2\theta_i^2 - \beta_i \right) \right)\| y^{i,j} - y(x^i) \|^2 + \| z^{i,j+1} - z(x^i,y^{i,j}) \|^2 + \Upsilon \beta_i^2\\
			& = \left( 1 - \psi_i\beta_i \right)\| y^{i,j} - y(x^i) \|^2 + \| z^{i,j+1} - z(x^i,y^{i,j}) \|^2 + \Upsilon \beta_i^2,
		\end{alignat*}
		where the last inequality follows from the fact that~$2ab \leq a^2 + b^2$ ($a$ and~$b$ positive scalars) where
		\begin{equation}\label{eq:2020}
			\psi_i \;:=\; \rho - 2\hat\omega^2\theta_i^2 - \beta_i \quad \text{ and } \quad \rho \;:=\; \frac{2\mu_y L_{\nabla \bar f}}{\mu_y + L_{\nabla \bar f}}.
		\end{equation}
		Taking total expectation and using bound~\eqref{eq:TSG_aux_LL_bound_5} from Lemma~\ref{lem:TSG_aux_LL_descent}, we have
		\begin{alignat}{2}
			&\mathbb{E}[ \| y^{i,j+1} - y(x^i)\|^2 ]\\ %& \leq \left( 1 - \psi_i\beta_i \right)\mathbb{E}[\| y^{i,j} - y(x^i) \|^2] + \mathbb{E}[\| z^{i,j+1} - z(x^i,y^{i,j}) \|^2] + \Upsilon \beta_i^2\\
			& \leq \left( 1 - \psi_i\beta_i \right)\mathbb{E}[ \| y^{i,j} - y(x^i) \|^2 ] + \Upsilon \beta_i^2 + \left( 1 - \gamma_i \rho_{f_3} \right)^K\mathbb{E}[ \| z^{i,j} - z(x^i,y^{i,j}) \|^2 ] + K\gamma_i^2\sigma_{\nabla f_3}^2\nonumber\\
			& \;\leq\; \left( 1 - \psi_i\beta_i \right)^J \mathbb{E}[ \| y^{i} - y(x^i) \|^2 ] + \left( 1 - \gamma_i \rho_{f_3} \right)^K\sum_{j=0}^{J-1}\mathbb{E}[ \| z^{i,j} - z(x^i,y^{i,j}) \|^2 ] + J\Upsilon \beta_i^2 + JK\gamma_i^2\sigma_{\nabla f_3}^2,\label{eq:TSG_intermidiate_result_for_ML_proof_1}
		\end{alignat}
		where the last inequality follows by using induction over $J$. Notice that this result follows by ensuring that~$0\leq 1 - \psi_i\beta_i\leq1$, which holds when choosing~$\beta_i$ such that $\beta_i \leq \frac{1}{\mu_y + L_{\nabla \bar f}}$ and~$\beta_i \leq \frac{\rho}{2\hat\omega^2+1}$. In other words, to show that~$0\leq 1 - \psi_i\beta_i$, we have
		\[
		\psi_i\beta_i = \beta_i( \rho - 2\hat\omega^2\theta_i^2 - \beta_i ) < \beta_i \rho \leq \frac{2\mu_y L_{\nabla \bar f}}{(\mu_y + L_{\nabla \bar f})^2} \leq 1,
		\]
		where the first inequality follows by observing that $- 2\hat\omega^2\theta_i^2\beta_i - \beta_i^2 < 0$, the second inequality follows by choosing $\beta_i \leq \frac{1}{\mu_y + L_{\nabla \bar f}}$ along with the definition of $\rho$, and the third inequality follows from the fact that $2ab \le (a+b)^2$, with~$a$ and~$b$ positive scalars. Notice that showing that $1 - \psi_i\beta_i \leq 1$ is equivalent to showing that $\psi_i \geq0$, i.e., using the fact that $0<\theta_i^2\leq\theta_i\leq 1$ along with $\theta_i = \alpha_i\beta_i\gamma_i \leq \beta_i$, we have
		\[
		\rho - 2\hat\omega^2\theta_i^2 - \beta_i \geq 0 \quad \Rightarrow \quad 2\hat\omega^2\beta_i + \beta_i \leq \rho \quad \Rightarrow \quad \beta_i \leq \frac{\rho}{2\hat\omega^2+1}. \]
		Now, looking at the $\sum_{j=0}^{J-1}\mathbb{E}[ \| z^{i,j} - z(x^i,y^{i}) \|^2 ]$ term in~\eqref{eq:TSG_intermidiate_result_for_ML_proof_1} and defining $\Theta_i:=\hat{\eta_i}L_{z_y}^2\Upsilon\beta_i^2~+~K\gamma_i^2\sigma_{\nabla f_3}^2$, we have
		\begin{alignat}{2}
			&\sum_{j=0}^{J-1}\mathbb{E}[ \| z^{i,j} - z(x^i,y^{i,j}) \|^2 ]\nonumber\\
			& = \mathbb{E}[ \| z^{i,0} - z(x^i,y^{i,0}) \|^2 ] + \mathbb{E}[ \| z^{i,1} - z(x^i,y^{i,1}) \|^2 ] + \cdots + \mathbb{E}[ \| z^{i,J-1} - z(x^i,y^{i,J-1}) \|^2 ]\nonumber\\
			& = \mathbb{E}[ \| z^{i} - z(x^i,y^{i}) \|^2 ]\nonumber\\
			&\quad + \mathbb{E}[ \| z^{i,1} - z(x^i,y^{i,1}) \|^2 ] \;\;\;\;\;\;\;\;\;\;\longrightarrow\;\;\; \left( \leq (\left( 1 - \gamma_i \rho_{f_3} \right)^K + \eta_i) \mathbb{E}[\| z^{i} - z(x^i,y^{i})\|^2] + \Theta_i \right)\nonumber\\
			&\quad + \mathbb{E}[ \| z^{i,2} - z(x^i,y^{i,2}) \|^2 ] \;\;\;\;\;\;\;\;\;\;\longrightarrow\;\;\; \left( \leq (\left( 1 - \gamma_i \rho_{f_3} \right)^K + \eta_i)^2 \mathbb{E}[\| z^{i} - z(x^i,y^{i})\|^2] + 2\Theta_i \right)\nonumber\\
			&\quad \;\vdots\nonumber\\
			&\quad + \mathbb{E}[ \| z^{i,J-1} - z(x^i,y^{i,J-1}) \|^2 ] \;\;\;\longrightarrow\;\;\; \left( \leq (\left( 1 - \gamma_i \rho_{f_3} \right)^K + \eta_i)^{J-1} \mathbb{E}[\| z^{i} - z(x^i,y^{i})\|^2] + (J-1)\Theta_i \right)\nonumber\\
			&\leq \mathbb{E}[ \| z^{i} - z(x^i,y^{i}) \|^2 ] + \sum_{j=1}^{J-1}(\left( 1 - \gamma_i \rho_{f_3} \right)^K + \eta_i)^j\mathbb{E}[ \| z^{i} - z(x^i,y^{i}) \|^2 ] + \Theta_i\sum_{j=1}^{J-1}j,\label{eq:TSG_intermidiate_result_for_ML_proof_2}
		\end{alignat}
		where the intermediate inequalities follow from applying equation~\eqref{eq:TSG_aux_LL_bound_1} from Lemma~\ref{lem:TSG_aux_LL_descent} repeatedly while choosing $\eta_i$ such that $\eta_i\leq 1 - \left( 1 - \gamma_i \rho_{f_3} \right)^K$ (which will ensure that \linebreak$0\leq \left( 1 - \gamma_i \rho_{f_3} \right)^K + \eta_i \leq 1$ when considering the fact that $0\leq \left( 1 - \gamma_i \rho_{f_3} \right)^K\leq 1$ which is satisfied by our choice of~$\gamma_i\leq \frac{1}{\mu_z + L_{\nabla f_3}}$ and recalling that~$\gamma_i$, $\rho_{f_3}$, and~$\eta_i$ are positive).
		Now, looking at the $\sum_{j=1}^{J-1}(\left( 1 - \gamma_i \rho_{f_3} \right)^K + \eta_i)^j$ term in~\eqref{eq:TSG_intermidiate_result_for_ML_proof_2}, we have
		\begin{alignat*}{2}
			\sum_{j=1}^{J-1}(\left( 1 - \gamma_i \rho_{f_3} \right)^K + \eta_i)^j &= \left( \frac{(\left( 1 - \gamma_i \rho_{f_3} \right)^K + \eta_i) -(\left( 1 - \gamma_i \rho_{f_3} \right)^K + \eta_i)^{J}}{1-(\left( 1 - \gamma_i \rho_{f_3} \right)^K + \eta_i)} \right) = \left( \frac{\vartheta_i -\vartheta_i^{J}}{1-\vartheta_i} \right),
		\end{alignat*}
		where the last equality follows by using the geometric series $\sum_{j=1}^{J-1}a^j=\frac{a-a^{J}}{1-a}$ when~$a\in[0,1]$ and defining $\vartheta_i:=\left( 1 - \gamma_i \rho_{f_3} \right)^K + \eta_i$ for ease of notation. Now, using the partial sum $\sum_{j=1}^{J-1}j = \frac{J(J-1)}{2}$, we can see that the bound~\eqref{eq:TSG_intermidiate_result_for_ML_proof_2} on the expression $\sum_{j=0}^{J-1}\mathbb{E}[ \| z^{i,j} - z(x^i,y^{i,j}) \|^2 ]$ is given by
		\begin{equation}
			\sum_{j=0}^{J-1}\mathbb{E}[ \| z^{i,j} - z(x^i,y^{i,j}) \|^2 ] \leq \left( 1 + \left( \frac{\vartheta_i -\vartheta_i^{J}}{1-\vartheta_i} \right) \right)\mathbb{E}[ \| z^{i} - z(x^i,y^{i}) \|^2 ] + \frac{J(J-1)}{2} \Theta_i. \label{eq:intermediary_bound_on_sum_w_vartheta_1}
		\end{equation}
		
		Now, we wish to analyze the limiting behavior of the term $\frac{\vartheta_i -\vartheta_i^{J}}{1-\vartheta_i}$ as $\vartheta\rightarrow0$ and $\vartheta\rightarrow1$ in order to obtain an upper-bound. Starting by analyzing the limiting behavior as $\vartheta\rightarrow0$, we have
		\[\lim_{\vartheta_i\rightarrow0} \frac{\vartheta_i(1 -\vartheta_i^{J-1})}{1-\vartheta_i} = \frac{0\cdot1}{1} = 0.\]
		Further, when $\vartheta_i\rightarrow1$, we can analyze the limiting behavior via L'Hopital's rule to obtain
		\[\lim_{\vartheta_i\rightarrow1} \frac{\vartheta_i -\vartheta_i^{J}}{1-\vartheta_i} = \lim_{\vartheta_i\rightarrow1} \frac{\frac{d}{d\vartheta_i}(\vartheta_i -\vartheta_i^{J})}{\frac{d}{d\vartheta_i}(1-\vartheta_i)} = \lim_{\vartheta_i\rightarrow1} -(1 - J\vartheta_i^{J-1}) = J-1.\]
		Therefore, we can see that (since $1\leq J\in\mathbb{N}$)
		\begin{equation}
			0 \leq \frac{\vartheta_i -\vartheta_i^{J}}{1-\vartheta_i} \leq J-1.\label{eq:vartheta_term_UB}
		\end{equation}
		Utilizing the upper-bound of~\eqref{eq:vartheta_term_UB} in~\eqref{eq:intermediary_bound_on_sum_w_vartheta_1} yields
		\begin{equation}
			\sum_{j=0}^{J-1}\mathbb{E}[ \| z^{i,j} - z(x^i,y^{i,j}) \|^2 ] \leq J\mathbb{E}[ \| z^{i} - z(x^i,y^{i}) \|^2 ] + \frac{J(J-1)}{2} \Theta_i. \label{eq:intermediary_bound_on_sum_w_vartheta_2}
		\end{equation}
		
		Now, substituting~\eqref{eq:intermediary_bound_on_sum_w_vartheta_2} back into equation~\eqref{eq:TSG_intermidiate_result_for_ML_proof_1} and using the fact that $0\leq1 - \gamma_i \rho_{f_3}\leq1$, which is satisfied by our choice of~$\gamma_i\leq\frac{1}{\mu_z + L_{\nabla f_3}}$ and recalling that~$\gamma_i$ and~$\rho_{f_3}$ are positive, yields
		\begin{alignat*}{2}
			\mathbb{E}[ \| y^{i+1} - y(x^i)\|^2 ] %& \leq \left( 1 - \psi_i\beta_i \right)^J \mathbb{E}[ \| y^{i} - y(x^i) \|^2 ]  + J\Upsilon \beta_i^2 + JK\gamma_i^2\sigma_{\nabla f_3}^2 + \frac{J(J-1)}{2} \left( 1 - \gamma_i \rho_{f_3} \right)^K\Theta_i\\
			%&\quad + \left( 1 - \gamma_i \rho_{f_3} \right)^K\left( 1 + \tcr{\left( \frac{\vartheta_i -\vartheta_i^{J}}{1-\vartheta_i} \right)} \right)\mathbb{E}[ \| z^{i} - z(x^i,y^{i}) \|^2 ]\\
			& \leq \left( 1 - \psi_i\beta_i \right)^J \mathbb{E}[ \| y^{i} - y(x^i) \|^2 ]  + J\Upsilon \beta_i^2 + JK\gamma_i^2\sigma_{\nabla f_3}^2 + \frac{J(J-1)}{2}\Theta_i\\
			&\quad + \left( 1 - \gamma_i \rho_{f_3} \right)^K J\mathbb{E}[ \| z^{i} - z(x^i,y^{i}) \|^2 ],
		\end{alignat*}
		Further simplifying this expression, we obtain the bound~\eqref{eq:TSG_ML_bound_1}.
		%\begin{alignat*}{2}
		%    \mathbb{E}[ \| y^{i+1} - y(x^i)\|^2 ]&\leq \left( 1 - \psi_i\beta_i \right)^J \mathbb{E}[ \| y^{i} - y(x^i) \|^2 ]  + \tcr{\left( 1 + \frac{1}{2}(J-1)\hat{\hat{\eta_i}} L_{z_y}^2 \right)}J\Upsilon \beta_i^2 + \frac{J+1}{2} JK\gamma_i^2\sigma_{\nabla f_3}^2\\
		%    &\quad + \left( 1 - \gamma_i \rho_{f_3} \right)^K\left( 1 + \tcr{\left( \frac{\vartheta_i -\vartheta_i^{J}}{1-\vartheta_i} \right)} \right)\mathbb{E}[ \| z^{i} - z(x^i,y^{i}) \|^2 ].
		%\end{alignat*}

		\textbf{$(\text{Analysis of }A^{(3)}_2)$:} The derivation of the upper-bound on $A^{(3)}_2$ in~\eqref{eq:TSG_MLP_error_intermediate_1} follows the exact same steps that were used to derive the upper-bound on the term $A^{(1)}_2$ in Lemma~\ref{lem:TSG_LL_descent} (only with using~\eqref{eq:lip_prop_4} instead of~\eqref{eq:lip_prop_1}), from which we have
		\[\mathbb{E}[ \| y(x^i) - y(x^{i+1}) \|^2 ] \leq 2L_y^2\alpha_i^2( \mathbb{E}[ \| \bar g_{f_1}^{i} \|^2 ] + \tau ).\]
		
		\textbf{$(\text{Analysis of }A^{(3)}_3)$:} The term $A^{(3)}_3$ in~\eqref{eq:TSG_MLP_error_intermediate_1} can be bounded by taking expectation conditioned on~$\mathcal{F}_i$ followed by adding and subtracting $\nabla y(x^i)(x^{i+1} - x^i)$ in the following way:
		\begin{alignat}{2}
			&\mathbb{E}[ ( y^{i+1} - y(x^i) )^\top ( y(x^i) - y(x^{i+1}) ) \vert \mathcal{F}_i ] \nonumber\\ %= - \mathbb{E}[ \left( y^{i+1} - y\left(x^i\right) \right)^\top \left( y\left(x^{i+1}\right) - y\left(x^i\right) \right) \vert \mathcal{F}_i ]\nonumber\\
			& = - \mathbb{E}[ ( y^{i+1} - y(x^i) )^\top ( \nabla y(x^i)(x^{i+1} - x^i) + y(x^{i+1}) - y(x^i) - \nabla y(x^i)(x^{i+1} - x^i) ) \vert \mathcal{F}_i]\nonumber\\
			& = \underbrace{- \mathbb{E}[ ( y^{i+1} - y(x^i) )^\top ( \nabla y(x^i)(x^{i+1} - x^i) ) \vert \mathcal{F}_i ]}_{B^{(3)}_1}\nonumber\\
			&\quad \underbrace{- \mathbb{E}[ ( y^{i+1} - y(x^i) )^\top ( y(x^{i+1}) - y(x^i) - \nabla y(x^i)(x^{i+1} - x^i) ) \vert \mathcal{F}_i ]}_{B^{(3)}_2}. \label{eq:TSG_intermidiate_result_for_ML_proof_3}
		\end{alignat}
		
		\textbf{$(\text{Analysis of }B^{(3)}_1)$:} The derivation of the upper-bound on $B^{(3)}_1$ in~\eqref{eq:TSG_intermidiate_result_for_ML_proof_3} follows the exact same steps that were used to derive the upper-bound on the term $B^{(1)}_1$ in Lemma~\ref{lem:TSG_LL_descent} (only with using~\eqref{eq:lip_prop_4} instead of~\eqref{eq:lip_prop_1}), from which, for some $\phi_i>0$, we have
		\[- \mathbb{E}[ ( y^{i+1} - y(x^i) )^\top ( \nabla y(x^i)(x^{i+1} - x^i) ) \vert \mathcal{F}_i ] \leq \phi_i\mathbb{E}[ \| y^{i+1} - y(x^i) \|^2 \vert \mathcal{F}_i ] + \frac{\alpha_i^2 L_y^2}{4 \phi_i}\mathbb{E}[ \| \bar g_{f_1}^{i} \|^2 \vert \mathcal{F}_i ].\]
		
		\textbf{$(\text{Analysis of }B^{(3)}_2)$:} The derivation of the upper-bound on $B^{(3)}_2$ in~\eqref{eq:TSG_intermidiate_result_for_ML_proof_3} follows the exact same steps that were used to derive the upper-bound on the term $B^{(1)}_2$ in Lemma~\ref{lem:TSG_LL_descent} (only with using~\eqref{eq:lip_prop_13} instead of~\eqref{eq:lip_prop_5}), from which we have
		\begin{alignat*}{2}
			&- \mathbb{E}[ ( y^{i+1} - y(x^i) )^\top ( y(x^{i+1}) - y(x^i) - \nabla y(x^i)(x^{i+1} - x^i) ) \vert \mathcal{F}_i ]\\
			& \leq \frac{L_{\nabla y}\alpha_i^2 \zeta}{4} \mathbb{E}[ \| y^{i+1} - y(x^i) \|^2 \vert \mathcal{F}_i ] + \frac{L_{\nabla y}\alpha_i^2}{4} ( \mathbb{E}[ \| \bar g_{f_1}^{i} \|^2 \vert \mathcal{F}_i ] + \tau ).
		\end{alignat*}

		Finally, substituting these bounds for $B^{(3)}_1$ and $B^{(3)}_2$ back into~\eqref{eq:TSG_intermidiate_result_for_ML_proof_3} and taking total expectation, we obtain the bound on the term $A^{(3)}_3$ as
		\begin{alignat*}{2}
			\mathbb{E}[ ( y^{i+1} - y(x^i) )^\top ( y(x^i) - y(x^{i+1}) ) ]
			%& \leq \phi_i\mathbb{E}[ \| y^{i+1} - y\left(x^i\right) \|^2 ] + \frac{\alpha_i^2 L_y^2}{4 \phi_i}\mathbb{E}[ \| \bar g_{f_1}^{i} \|^2 ]\\
			%&\quad+ \frac{L_{\nabla y}\alpha_i^2 \zeta}{4} \mathbb{E}[ \| y^{i+1} - y\left(x^i\right) \|^2 ] + \frac{L_{\nabla y}\alpha_i^2}{4} \left( \mathbb{E}[ \| \bar g_{f_1}^{i} \|^2 ] + \tau \right)\\
			&\leq \left( \phi_i + \frac{L_{\nabla y}\alpha_i^2 \zeta}{4}\right)\mathbb{E}[ \| y^{i+1} - y\left(x^i\right) \|^2 ] \\
			&\quad+ \left( \frac{\alpha_i^2 L_y^2}{4 \phi_i} + \frac{L_{\nabla y}\alpha_i^2}{4} \right)\mathbb{E}[ \| \bar g_{f_1}^{i} \|^2 ] + \frac{\tau L_{\nabla y}}{4}\alpha_i^2.
		\end{alignat*}
		Finally, substituting these bounds for $A^{(3)}_1$, $A^{(3)}_2$, and $A^{(3)}_3$ back into~\eqref{eq:TSG_MLP_error_intermediate_1}, we obtain the desired upper-bound on $\mathbb{E}[\| y^{i+1}-y(x^{i+1})\|^2]$, completing the proof.
	\end{proof}

	\subsection{Proof of Theorem~\ref{th:TSG_convergence}}\label{app:th:TSG_convergence}
	\begin{proof}
		To begin, using Lemmas~\ref{lem:TSG_UL_descent}, \ref{lem:TSG_LL_descent}, \ref{lem:TSG_aux_LL_descent}, and~\ref{lem:TSG_ML_descent}, we can bound the two Lyapunov difference terms (defined in~\eqref{eq:lyapunov_diff}) by taking total expectation in the following way:
		\begin{alignat*}{2}
			&\mathbb{E}[ \mathbb{V}^{i+1} ] - \mathbb{E}[ \mathbb{V}^i ]\\
			& = \underbrace{\mathbb{E}[ f(x^{i+1}) ] - \mathbb{E}[ f(x^i) ]}_{\text{Lemma~\ref{lem:TSG_UL_descent}}} + \underbrace{\mathbb{E}[ \| y^{i+1} - y( x^{i+1} ) \|^2 ]}_{\text{Lemma~\ref{lem:TSG_ML_descent}}} - \mathbb{E}[ \| y^{i} - y( x^{i} ) \|^2 ] \\
			&\quad+  \underbrace{\mathbb{E}[ \| z^{i+1} - z( x^{i+1} ) \|^2 ]}_{\text{Lemma~\ref{lem:TSG_LL_descent}}} - \mathbb{E}[ \| z^{i} - z( x^{i} ) \|^2 ] + 
			\underbrace{\mathbb{E}[ \| z^{i+1} - z( x^{i+1}, y^{i+1} ) \|^2 ]}_{\text{Lemma~\ref{lem:TSG_aux_LL_descent}}} - \mathbb{E}[ \| z^{i} - z( x^{i}, y^{i} ) \|^2 ]\\
			& \leq -\frac{\alpha_i}{2} \mathbb{E}[ \| \nabla f(x^i) \|^2 ] - \left( \frac{\alpha_i}{2} - \frac{L_F\alpha_i^2}{2} \right) \mathbb{E}[ \| \bar g_{f_1}^{i} \|^2 ] + \tilde\omega\alpha_i^2\\
			&\quad+  \alpha_iL_{F_{yz}}^2\mathbb{E}[ \| y(x^i) - y^{i+1} \|^2 ] + \alpha_iL_{F_{yz}}^2\mathbb{E}[ \| z(x^i) - z^{i+1} \|^2 ]\\
			&\quad+ \left( 1 + 2\phi_i + \frac{L_{\nabla y}\alpha_i^2 \zeta}{2}\right)\mathbb{E}[ \| y^{i+1} - y(x^i) \|^2 ] \\
			&\quad+ \left( 2L_y^2 + \frac{ L_y^2}{2 \phi_i} + \frac{L_{\nabla y}}{2} \right)\alpha_i^2\mathbb{E}[ \| \bar g_{f_1}^{i} \|^2 ] + \left( 2L_y^2 + \frac{ L_{\nabla y}}{2} \right)\tau\alpha_i^2\\
			&\quad+ \left( 1 + 2\kappa_i + \frac{L_{\nabla z}\alpha_i^2\zeta}{2} \right) \mathbb{E}[ \| z^{i+1} - z(x^i) \|^2]\\
			&\quad+ \left( 2L_z^2 + \frac{L_z^2}{2\kappa_i} + \frac{L_{\nabla z}}{2} \right)\alpha_i^2 \mathbb{E}[\| \bar g_{f_1}^i \|^2] + \left( 2L_z^2 + \frac{ L_{\nabla z}}{2} \right)\tau\alpha_i^2\\
			&\quad+ 2\mathbb{E}[ \| z^{i+1} - z(x^i,y^{i}) \|^2 ] + 4L_{z_{xy}}^2 \alpha_i^2 \mathbb{E}[\|\bar g_{f_1}^{i}\|^2] + 4L_{z_{xy}}^2 \alpha_i^2 \tau + 2J^2\Upsilon L_{z_{xy}}^2\beta_i^2\\
			&\quad - \mathbb{E}[ \| y^{i} - y( x^{i} ) \|^2 ] - \mathbb{E}[ \| z^{i} - z( x^{i} ) \|^2 ] - \mathbb{E}[ \| z^{i} - z( x^{i}, y^{i} ) \|^2 ].
		\end{alignat*}
		Simplifying, we have
		\begin{alignat}{2}
			\mathbb{E}[ \mathbb{V}^{i+1} ] - \mathbb{E}[ \mathbb{V}^i ] & \leq -\frac{\alpha_i}{2} \mathbb{E}[ \| \nabla f(x^i) \|^2 ] - \left( \frac{\alpha_i}{2} - \frac{L_F\alpha_i^2}{2} \right) \mathbb{E}[ \| \bar g_{f_1}^{i} \|^2 ]\nonumber\\
			&\quad+ \left( 1 + \alpha_iL_{F_{yz}}^2 + 2\phi_i + \frac{L_{\nabla y}\alpha_i^2 \zeta}{2}\right) \underbrace{\mathbb{E}[ \| y^{i+1} - y(x^i) \|^2 ]}_{\text{Lemma~\ref{lem:TSG_ML_descent}}} \label{eq:7010}\\
			&\quad+ \left( 1 + \alpha_iL_{F_{yz}}^2 + 2\kappa_i + \frac{L_{\nabla z}\alpha_i^2\zeta}{2} \right) \underbrace{\mathbb{E}[ \| z^{i+1} - z(x^i) \|^2]}_{\text{Lemma~\ref{lem:TSG_LL_descent}}}\label{eq:7020}\\
			&\quad+ 2 \underbrace{\mathbb{E}[\| z^{i+1} - z(x^i,y^i) \|^2]}_{\text{Lemma~\ref{lem:TSG_aux_LL_descent}}}\nonumber\\
			&\quad+ \left( 2L_y^2 + \frac{ L_y^2}{2 \phi_i} + \frac{L_{\nabla y}}{2} + 2L_z^2 + \frac{L_z^2}{2\kappa_i} + \frac{L_{\nabla z}}{2} + 4L_{z_{xy}}^2 \right)\alpha_i^2\mathbb{E}[ \| \bar g_{f_1}^{i} \|^2 ]\label{eq:7040}\\
			&\quad+ \left(\left( 2L_y^2 + \frac{ L_{\nabla y}}{2} + 2L_z^2 + \frac{ L_{\nabla z}}{2} + 4L_{z_{xy}}^2 \right)\tau + \tilde\omega \right)\alpha_i^2\label{eq:7050}\\
			&\quad+ 2J^2\Upsilon L_{z_{xy}}^2\beta_i^2\nonumber\\
			&\quad - \mathbb{E}[ \| y^{i} - y( x^{i} ) \|^2 ] - \mathbb{E}[ \| z^{i} - z( x^{i} ) \|^2 ] - \mathbb{E}[ \| z^{i} - z( x^{i}, y^{i} ) \|^2 ].\nonumber
		\end{alignat}
		Now, for ease of notation, we denote the coefficients in~\eqref{eq:7010}--\eqref{eq:7050} as follows:
		\begin{alignat}{3}
			&G_1^i \;&&:=\; \left( 1 + \alpha_iL_{F_{yz}}^2 + 2\phi_i + \frac{L_{\nabla y}\alpha_i^2 \zeta}{2}\right),\label{eq:TSG_def_G1i}\\
			&G_2^i\;&&:=\; \left( 1 + \alpha_iL_{F_{yz}}^2 + 2\kappa_i + \frac{L_{\nabla z}\alpha_i^2\zeta}{2} \right),\label{eq:TSG_def_G2i}\\
			%&G_3^i\;&&:=\; \tcb{2},\label{eq:TSG_def_G3i}\\
			&G_3^i \;&&:=\; \left( 2L_y^2 + \frac{ L_y^2}{2 \phi_i} + \frac{L_{\nabla y}}{2} + 2L_z^2 + \frac{L_z^2}{2\kappa_i} + \frac{L_{\nabla z}}{2} + 4L_{z_{xy}}^2 \right),\label{eq:TSG_def_G4i}\\
			&\Phi \;&&:=\; \left(\left( 2L_y^2 + \frac{ L_{\nabla y}}{2} + 2L_z^2 + \frac{ L_{\nabla z}}{2} + 4L_{z_{xy}}^2 \right)\tau + \tilde\omega \right).\label{eq:TSG_def_Phi}
		\end{alignat}
		Then, using these definitions and applying Lemmas~\ref{lem:TSG_LL_descent}, \ref{lem:TSG_aux_LL_descent}, and~\ref{lem:TSG_ML_descent}, we have
		\begin{alignat*}{2}
			\mathbb{E}[ \mathbb{V}^{i+1} ] - \mathbb{E}[ \mathbb{V}^i ] & \leq -\frac{\alpha_i}{2} \mathbb{E}[ \| \nabla f(x^i) \|^2 ] - \left( \frac{\alpha_i}{2} - \frac{L_F\alpha_i^2}{2} - G_3^i \alpha_i^2 \right) \mathbb{E}[ \| \bar g_{f_1}^{i} \|^2 ] + \Phi \alpha_i^2\\
			&\quad+ G_1^i\left( 1 - \psi_i\beta_i \right)^J \mathbb{E}[ \| y^{i} - y(x^i) \|^2 ]  + G_1^i\left( 1 + \frac{1}{2}(J-1)\hat{\eta_i} L_{z_y}^2 \right)J\Upsilon\beta_i^2 \\
			&\quad + G_1^i\frac{J+1}{2}JK\gamma_i^2\sigma_{\nabla f_3}^2 + G_1^i\left( 1 - \gamma_i \rho_{f_3} \right)^K J\mathbb{E}[ \| z^{i} - z(x^i,y^{i}) \|^2 ]\\
			&\quad+ G_2^i \left( 1 - \gamma_i \rho_{f_3} \right)^{JK}\mathbb{E}[\| z^{i} - z(x^i) \|^2] + G_2^iJ K\gamma_i^2\sigma_{\nabla f_3}^2\\
			&\quad+ 2 \left( 1 - \gamma_i \rho_{f_3} \right)^{J K}\mathbb{E}[\| z^{i} - z(x^i,y^i) \|^2] + 2J K\gamma_i^2\sigma_{\nabla f_3}^2 +  2J^2\Upsilon L_{z_{xy}}^2\beta_i^2\\
			&\quad - \mathbb{E}[ \| y^{i} - y( x^{i} ) \|^2 ] - \mathbb{E}[ \| z^{i} - z( x^{i} ) \|^2 ] - \mathbb{E}[ \| z^{i} - z( x^{i}, y^{i} ) \|^2 ].
		\end{alignat*}
		Simplifying once again while using the fact that~$\left( 1 - \gamma_i \rho_{f_3} \right)^{J K}\leq\left( 1 - \gamma_i \rho_{f_3} \right)^{K}$ (recalling that~$\rho_{f_3}$ from Lemma~\ref{lem:TSG_LL_descent} and~$\gamma_i$ are positive) as well as $J-1\leq J$, we have
		\begin{alignat}{2}
			&\mathbb{E}[ \mathbb{V}^{i+1} ] - \mathbb{E}[ \mathbb{V}^i ]\nonumber\\
			& \leq -\frac{\alpha_i}{2} \mathbb{E}[ \| \nabla f(x^i) \|^2 ] + \Phi \alpha_i^2 - \underbrace{\left( \frac{\alpha_i}{2} - \frac{L_F\alpha_i^2}{2} - G_3^i \alpha_i^2 \right)}_{A_1} \mathbb{E}[ \| \bar g_{f_1}^{i} \|^2 ] \nonumber\\
			&\quad+ \underbrace{((G_1^i J + 2)\left( 1 - \gamma_i \rho_{f_3} \right)^K - 1)}_{A_2}\mathbb{E}[ \| z^{i} - z(x^i,y^{i}) \|^2 ]\nonumber\\
			&\quad+ \underbrace{(G_1^i \left( 1 - \psi_i\beta_i \right)^J - 1)}_{A_3} \mathbb{E}[ \| y^{i} - y(x^i) \|^2 ] + \underbrace{(G_2^i \left( 1 - \gamma_i \rho_{f_3} \right)^{JK} - 1)}_{A_4}\mathbb{E}[\| z^{i} - z(x^i) \|^2]\nonumber\\
			&\quad+ \left( 2JL_{z_{xy}}^2 + \left( 1 + \frac{1}{2}J\hat{\eta_i} L_{z_y}^2 \right)G_1^i \right)J\Upsilon\beta_i^2 + \left(G_1^i\frac{J+1}{2} + G_2^i + 2\right) JK\gamma_i^2\sigma_{\nabla f_3}^2.\label{eq:TSG_final_theorem_intermediate_1}
		\end{alignat}
		
		\textbf{$(\text{Choice of step sizes})$:} In the proof of this theorem, we choose the UL, ML, and LL step sizes to be the following:
		\begin{equation}\label{eq:alpha_i_choice}
			\alpha_i := \frac{1}{\sqrt{I}},
		\end{equation}
		\begin{equation}\label{eq:beta_i_choice}
			\beta_i := \frac{1}{\sqrt{J}}\alpha_i = \frac{1}{\sqrt{I}\sqrt{J}},
		\end{equation}
		\begin{equation}\label{eq:gamma_i_choice}
			\gamma_i := \frac{1}{\sqrt{J}\sqrt{K}}\alpha_i = \frac{1}{\sqrt{I}\sqrt{J}\sqrt{K}}.
		\end{equation}

		\textbf{$(\text{Analysis of }A_1)$:} Now, consider the coefficient $A_1$ of the $\mathbb{E}[ \| \bar g_{f_1}^{i} \|^2 ]$ term in~\eqref{eq:TSG_final_theorem_intermediate_1}. We wish to determine an appropriate bound on $\alpha_i$ (in terms of $I$) such that this term in non-negative. To that end, we wish to ensure that $A_1\geq0$, which is true if
		%\[
		%\frac{\alpha_i}{2} - \frac{L_F\alpha_i^2}{2} - G_4^i \alpha_i^2 \geq 0,
		%\]
		%\[
		%\alpha_i\left( \frac{1}{2} - \frac{L_F\alpha_i}{2} - \left( 2L_y^2 + \frac{ L_y^2}{2 \phi_i} + \frac{L_{\nabla y}}{2} + 2L_z^2 + \frac{L_z^2}{2\kappa_i} + \frac{L_{\nabla z}}{2} + \tcb{4L_{z_{xy}}^2} \right) \alpha_i \right) \geq0,
		%\]    
		
		\[
		\frac{1}{2} - \frac{L_y^2\alpha_i}{2 \phi_i} - \frac{L_z^2\alpha_i}{2\kappa_i} - \left( \frac{L_F}{2} + 2L_y^2 + \frac{L_{\nabla y}}{2} + 2L_z^2 + \frac{L_{\nabla z}}{2} + 4L_{z_{xy}}^2 \right) \alpha_i \geq 0.
		\]
		Now, choosing 
		\begin{equation}\label{eq:70500}
			\phi_i = 4L_y^2\alpha_i \quad \text{and} \quad \kappa_i = 4L_z^2\alpha_i,% \text{ and } \eta_i = 4L_{z_{xy}}^2\alpha_i, 
		\end{equation}
		we have
		%\[
		%\frac{1}{2} - \frac{1}{8} - \frac{1}{8} - \frac{1}{8} \geq \left( \frac{L_F}{2} + 2L_y^2 + \frac{L_{\nabla y}}{2} + 2L_z^2 + \frac{L_{\nabla z}}{2} + 2L_{z_{xy}}^2 + \frac{L_{z_{xy}}}{2} \right) \alpha_i,
		%\]
		\begin{equation}\label{eq:TSG_UB_alpha}
			\alpha_i \leq \frac{\frac{1}{4}}{\frac{L_F}{2} + 2L_y^2 + \frac{L_{\nabla y}}{2} + 2L_z^2 + \frac{L_{\nabla z}}{2} + 4L_{z_{xy}}^2} = \frac{1}{2(L_F + 4L_y^2 + L_{\nabla y} + 4L_z^2 + L_{\nabla z} + 8L_{z_{xy}}^2)}.
		\end{equation}
		Now, recalling our choice for $\alpha_i$ given by~\eqref{eq:alpha_i_choice}, then from~\eqref{eq:TSG_UB_alpha} we see that we must choose $I\in\mathbb{N}$ such that
		\begin{equation}\label{eq:I_LB}
			4(L_F + 4L_y^2 + L_{\nabla y} + 4L_z^2 + L_{\nabla z} + 8L_{z_{xy}}^2)^{2} \leq I.
		\end{equation}
		Therefore, when choosing $I$ such that the inequality~\eqref{eq:I_LB} is satisfied, the coefficient~$A_1$ of the~$\mathbb{E}[ \| \bar g_{f_1}^{i} \|^2 ]$ term in~\eqref{eq:TSG_final_theorem_intermediate_1} will be non-negative.
		
		\textbf{$(\text{Analysis of }A_2)$:} Now, consider the coefficient~$A_2$ of the~$\mathbb{E}[ \| z^{i} - z(x^i,y^{i}) \|^2 ]$ term in~\eqref{eq:TSG_final_theorem_intermediate_1}. We wish to determine an appropriate bound on $\gamma_i$ (in terms of $I$, $J$, and $K$) such that this term is non-positive. Now, recall from Lemma~\ref{lem:TSG_LL_descent} that~$\rho_{f_3} = \frac{2\mu_z L_{\nabla f_3}}{\mu_z + L_{\nabla f_3}}$ (see~\eqref{eq:1060} in Appendix~\ref{app:lem:TSG_LL_descent}) as well as the assumed bound (imposed in Lemmas~\ref{lem:TSG_LL_descent}, \ref{lem:TSG_aux_LL_descent}, and~\ref{lem:TSG_ML_descent})
		\begin{equation}\label{eq:TSG_UB1_gamma}
			\gamma_i\leq\frac{1}{\mu_z + L_{\nabla f_3}}.
		\end{equation}
		Utilizing our choice of $\gamma_i$ given by~\eqref{eq:gamma_i_choice}, this can be satisfied by choosing $I$, $J$, and $K$ such that
		\begin{equation}\label{eq:IJK_LB}
			(\mu_z + L_{\nabla f_3})^2 \leq IJK,
		\end{equation}
		
		Recall the fact that~$\gamma_i$ and~$\rho_{f_3}$ are positive, along with~\eqref{eq:TSG_UB1_gamma}, which ensures that \linebreak$0\leq 1 - \gamma_i\rho_{f_3} \leq 1$. %Further, recall that $0\leq \left( 1 - \gamma_i \rho_{f_3} \right)^K + \eta_i \leq 1$ from the assumed bound (imposed in Lemma~\ref{lem:TSG_ML_descent}) of
		%\begin{equation}\label{eq:UB_hat_eta_i}
		%    \eta_i\leq 1 - \left( 1 - \gamma_i \rho_{f_3} \right)^K,
		%\end{equation}
		%on the positive quantity $\eta_i>0$.
		With this, to guarantee that~$A_2$ is non-positive, we wish to ensure that
		\begin{equation}\label{eq:TSG_eq_to_bound_1}
			( G_1^i J + 2)\left( 1 - \gamma_i \rho_{f_3} \right)^K \leq 1,
		\end{equation}
		Now, recall the fact that $1+a \leq e^a$ for any $a\in\mathbb{R}$. Multiplying both sides of this equation by the quantity $\left(1-\frac{a}{K}\right)^{K}$, we can see that
		\begin{equation}\label{eq:8040}
			(1+a)\left(1-\frac{a}{K}\right)^{K} \leq e^a\left(1-\frac{a}{K}\right)^{K} \leq e^a\left(e^{-\frac{a}{K}}\right)^{K} = e^ae^{-a} = 1.
		\end{equation}
		%Notice that we need~$0\leq\frac{a}{K}\leq1$, which is guaranteed when choosing~$a=K\gamma_i \rho_{f_3}$ under the upper-bound defined by~\eqref{eq:TSG_UB1_gamma}, as follows:
		%\[
		%\frac{a}{K} = \frac{K\gamma_i \rho_{f_3}}{K} = \gamma_i \rho_{f_3} \leq \frac{2\mu_z L_{\nabla f_3}}{\left(\mu_z + L_{\nabla f_3}\right)^2} \leq 1,
		%\]
		%where the first inequality is due to the definition of~$\rho_{f_3}$ and~\eqref{eq:TSG_UB1_gamma}, while the second inequality follows by utilizing the fact that~$2ab \le (a+b)^2$ (with~$a$ and~$b$ positive scalars).
		
		Now, to ensure that~\eqref{eq:TSG_eq_to_bound_1} holds, applying~\eqref{eq:8040} with~$a = K \gamma_i \rho_{f_3}$, yields the new inequality we wish to satisfy given by
		\begin{alignat}{2}
			G_1^i J + 2 \;&\leq\; 1 + K\gamma_i \rho_{f_3}\label{eq:8010}.
		\end{alignat}
		Now, using the fact that $\alpha_i\leq 1$ along with the choice $\phi_i=4L_y^2\alpha_i$, we can upper-bound $G_1^i$ as
		\begin{equation}
			G_1^i = 1 + \alpha_iL_{F_{yz}}^2 + 2\phi_i + \frac{L_{\nabla y}\alpha_i^2 \zeta}{2} \leq 1 + L_{F_{yz}}^2 + 8L_y^2 + \frac{L_{\nabla y} \zeta}{2} := g_1.\label{eq:G_1^i_UB}
		\end{equation}
		%Then in a similar fashion, using the choice of $\beta_i$ given by~\eqref{eq:beta_i_choice} and the fact that $\alpha_i\leq 1$ and $1/I\leq 1$ along with the choices $\eta_i = 4L_{z_{xy}}^2\alpha_i$ and $\upsilon_i:=1$, we can upper-bound $G_3^i$ as
		%\begin{equation}
		%    G_3^i = 1 + 2\left(\eta_i + \upsilon_i\right) + \frac{L_{z_{xy}}\alpha_i^2 \zeta}{2} + \frac{J^2\Upsilon L_{z_{xy}}\beta_i^2}{2} \leq g_3 + \tilde{g_3}J,\label{eq:G_3^i_UB}
		%\end{equation}
		%where $g_3:= 1 + 2(4L_{z_{xy}}^2 + 1) + \frac{L_{z_{xy}} \zeta}{2}$ and $\tilde{g_3}:=\frac{\Upsilon L_{z_{xy}}}{2}$.
		Thus, utilizing~\eqref{eq:G_1^i_UB}, we can guarantee~\eqref{eq:8010} if $Jg_1 + 1 \leq K\gamma_i \rho_{f_3}$ is satisfied. Now, utilizing the choice of $\gamma_i$ given by~\eqref{eq:gamma_i_choice}, we have
		%\[\frac{J(g_1 + \tilde{g_3}) + g_3}{\rho_{f_3}} \leq \frac{K}{\sqrt{I}\sqrt{J}\sqrt{K}},\]
		\begin{equation}
			\frac{Jg_1 + 1}{\rho_{f_3}} \leq \frac{K}{\sqrt{I}\sqrt{J}\sqrt{K}} \quad \Rightarrow \quad \frac{IJ(Jg_1 + 1)^2}{\rho_{f_3}^2} \leq K.\label{eq:K_LB_1}
		\end{equation}
		Therefore, when choosing $I$, $J$, and $K$ such that the inequality~\eqref{eq:K_LB_1} is satisfied, the coefficient~$A_2$ of the~$\mathbb{E}[ \| z^{i} - z(x^i,y^{i}) \|^2 ]$ term in~\eqref{eq:TSG_final_theorem_intermediate_1} will be non-positive.

		\textbf{$(\text{Analysis of }A_3)$:} Now, consider the coefficient~$A_3$ of the~$\mathbb{E}[ \| y^{i} - y(x^i) \|^2 ]$ term in~\eqref{eq:TSG_final_theorem_intermediate_1}. We wish to determine an appropriate bound on~$\beta_i$ (in terms of $I$ and $J$) such that this term is non-positive. Recall from the proof of Lemma~\ref{lem:TSG_ML_descent} that~$\rho=\frac{2\mu_y L_{\nabla \bar f}}{\mu_y + L_{\nabla \bar f}}$ (see~\eqref{eq:2020} in Appendix~\ref{app:lem:TSG_ML_descent}) and that
		\begin{equation}\label{eq:TSG_UB_beta}
			\beta_i \;\leq\; \frac{1}{\mu_y + L_{\nabla \bar{f}}}, \quad\quad \beta_i \;\leq\; \frac{\rho}{2\hat\omega^2+1}.
		\end{equation}
		%\begin{equation}\label{eq:TSG_UB2_beta}
		%    \beta_i \;\leq\; \frac{\rho}{2\hat\omega^2+1}.
		%\end{equation}
		Utilizing our choice of $\beta_i$ given by~\eqref{eq:beta_i_choice}, this can be satisfied by choosing $I$ and $J$ such that
		\begin{equation}\label{eq:IJ_LB}
			\max\left\{ \mu_y + L_{\nabla \bar{f}}, \frac{2\hat\omega^2+1}{\rho} \right\}^2 \leq IJ.
		\end{equation}
		
		Further, recall from Lemma~\ref{lem:TSG_ML_descent} that these two upper-bounds ensure that $0\leq 1 - \psi_i\beta_i \leq1$, where $\psi_i = \rho - 2\hat\omega^2\theta_i^2 - \beta_i$. With this, we wish to ensure that $G_1^i \left( 1 - \psi_i\beta_i \right)^J \leq 1$. Now, once again using the fact that $(1+a)\left(1-\frac{a}{J}\right)^{J}\leq 1$ as discussed for the analysis of~$A_2$, we need to choose an $a$ such that $0\leq\frac{a}{J}\leq 1$. Choosing $a=J\psi_i\beta_i$, we have $\frac{a}{J} = \frac{J\psi_i\beta_i}{J} = \psi_i\beta_i$, which from Lemma~\ref{lem:TSG_ML_descent}, we know that $0\leq \psi_i\beta_i\leq 1$, and by extension that $0\leq \frac{a}{J}\leq1$. Thus, we have
		%\[G_1^i \leq 1 + J\psi_i\beta_i,\]
		\begin{equation}\label{eq:TSG_LB1_beta}
			G_1^i \leq 1 + J\psi_i\beta_i \quad \Rightarrow \quad \alpha_iL_{F_{yz}}^2 + 2\phi_i + \frac{L_{\nabla y}\alpha_i^2 \zeta}{2} \leq J\psi_i\beta_i.
		\end{equation}
		Notice that from equation~\eqref{eq:TSG_UB_beta}, $\beta_i$ is upper-bounded by the constant $\bar\beta_1$ (defined as the largest value that $\beta_i$ can take) given by
		\begin{equation}\label{eq:TSG_bar_beta_1}
			\bar\beta_1 := \min\left\{ 1, \frac{1}{\mu_y + L_{\nabla \bar{f}}}, \frac{\rho}{2\hat\omega^2+1} \right\}.
		\end{equation}
		Using the fact that $\theta_i = \alpha_i\beta_i\gamma_i\leq\beta_i$ (by $\alpha_i\leq1$ and $\gamma_i\leq1$) and~\eqref{eq:TSG_bar_beta_1} in the definition of $\psi_i=\rho - 2\hat\omega^2\theta_i^2 - \beta_i$, we can define the new lower-bounding constant $\Gamma$ as
		\begin{equation}\label{eq:4010_v2}
			\Gamma \;:=\; \rho - 2\hat\omega^2\bar\beta_1^2 - \bar\beta_1.
		\end{equation}
		Notice that $0\leq\Gamma\leq\psi_i$ for all feasible values of $\theta_i$ and $\beta_i$ in $\psi_i$. Now, using this definition of $\Gamma$, the fact that $\alpha_i\leq 1$, and the choice $\phi_i=4L_y^2\alpha_i$, we have that the following implies~\eqref{eq:TSG_LB1_beta}:
		\[\alpha_i\left(L_{F_{yz}}^2 + 8L_y^2 + \frac{L_{\nabla y} \zeta}{2}\right) \leq J\Gamma\beta_i.\]
		Utilizing the choices for $\alpha_i$ and $\beta_i$ given by~\eqref{eq:alpha_i_choice} and~\eqref{eq:beta_i_choice}, respectively, it follows that the bound
		\begin{equation}\label{eq:J_LB_1}
			\frac{\left(L_{F_{yz}}^2 + 8L_y^2 + \frac{L_{\nabla y} \zeta}{2}\right)^2}{\Gamma^2} \leq J,
		\end{equation}
		implies that~\eqref{eq:TSG_LB1_beta} is satisfied. Therefore, when choosing $J$ such that the inequality~\eqref{eq:J_LB_1} is satisfied, the coefficient $A_3$ of the $\mathbb{E}[ \| y^{i} - y(x^i) \|^2 ]$ term in~\eqref{eq:TSG_final_theorem_intermediate_1} will be non-positive.

		\textbf{$(\text{Analysis of }A_4)$:} Now, consider the coefficient $A_4$ of the $\mathbb{E}[\| z^{i} - z(x^i) \|^2]$ term in~\eqref{eq:TSG_final_theorem_intermediate_1}. We wish to determine an appropriate bound on~$\gamma_i$ (in terms of $I$, $J$, and $K$) such that this term is non-positive. That is, we wish to show $G_2^i\left(1 - \gamma_i\rho_{f_3}  \right)^{JK} \leq 1$. Now, recall that equation~\eqref{eq:TSG_UB1_gamma} ensures $0\leq \left(1 - \gamma_i\rho_{f_3}  \right)^{JK} \leq1$. With this, and using the same reasoning that was used earlier, we need to show that
		\begin{equation}\label{eq:TSG_LB_gamma}
			G_2^i \leq 1 + JK\gamma_i\rho_{f_3} \quad \Rightarrow \quad \alpha_iL_{F_{yz}}^2 + 2\kappa_i + \frac{L_{\nabla z}\alpha_i^2\zeta}{2} \leq JK\gamma_i\rho_{f_3}.
		\end{equation}
		%which will yield the bound
		%\[\alpha_iL_{F_{yz}}^2 + 2\kappa_i + \frac{L_{\nabla z}\alpha_i^2\zeta}{2} \leq JK\gamma_i\rho_{f_3}.\]
		Now, using the fact that $\alpha_i \leq 1$ along with the choice $\kappa_i = 4L_z^2\alpha_i$, we can see that~\eqref{eq:TSG_LB_gamma} is satisfied if
		\[\alpha_i\left(L_{F_{yz}}^2 + 8L_z^2 + \frac{L_{\nabla z}\zeta}{2}\right) \leq JK\gamma_i\rho_{f_3}.\]
		Utilizing the choices for $\alpha_i$ and $\gamma_i$ given by~\eqref{eq:alpha_i_choice} and~\eqref{eq:gamma_i_choice}, respectively, it follows that the bound
		\begin{equation}\label{eq:JK_LB_1}
			\frac{\left(L_{F_{yz}}^2 + 8L_z^2 + \frac{L_{\nabla z}\zeta}{2}\right)^2}{\rho_{f_3}^2} \leq JK,
		\end{equation}
		implies~\eqref{eq:TSG_LB_gamma}. Therefore, when choosing $J$ and $K$ such that the inequality~\eqref{eq:JK_LB_1} is satisfied, the coefficient $A_4$ of the $\mathbb{E}[\| z^{i} - z(x^i) \|^2]$ term in~\eqref{eq:TSG_final_theorem_intermediate_1} will be non-positive.

		\textbf{$(\text{Upper-bounding }\hat{\eta_i})$:} We need an upper-bound on the positive quantity $\hat{\eta_i}$ in the second to last term of~\eqref{eq:TSG_final_theorem_intermediate_1}. Specifically, we wish to upper bound the term given by
		\begin{equation}\label{eq:eta_hat_hat_intermediate}
			\hat{\eta_i} = 1 + \frac{1}{\eta_i}.
		\end{equation}
		Further, recall that $0\leq \left( 1 - \gamma_i \rho_{f_3} \right)^K + \eta_i \leq 1$ from the assumed bound (imposed in Lemma~\ref{lem:TSG_ML_descent})
		\begin{equation}\label{eq:UB_hat_eta_i}
			\eta_i\leq 1 - \left( 1 - \gamma_i \rho_{f_3} \right)^K,
		\end{equation}
		on the positive quantity $\eta_i>0$. To ensure that bound~\eqref{eq:UB_hat_eta_i} is always satisfied, we can start by choosing $\eta_i$ to be
		\begin{equation}\label{eq:hat_eta_i_choice}
			\eta_i:= \mathcal{E}(1-(1-\gamma_i\rho_{f_3})^K),
		\end{equation}
		for some constant $0<\mathcal{E}<1$. When utilizing the choice of $\gamma_i$ given by~\eqref{eq:gamma_i_choice}, we have
		\[\eta_i:= \mathcal{E}\left(1-\left(1-\frac{\rho_{f_3}}{\sqrt{I}\sqrt{J}\sqrt{K}}\right)^K\right).\]
		Thus, we want to derive an upper-bound on the term $1/\eta_i$. Recall that $1+a\leq e^a$ for all $a\in\mathbb{R}$. Letting $\hat{a}:=\frac{\rho_{f_3}}{\sqrt{I}\sqrt{J}\sqrt{K}}$, we have that $(1-\hat{a})^K \leq e^{-K\hat{a}}$. For simplification, let $\bar{a} = K\hat{a}=\frac{\sqrt{K}}{\sqrt{I}\sqrt{J}}\rho_{f_3}$. Further, multiplying both sides of the inequality by $-1$ and adding $1$ to both sides, %adding $-1$ to both sides of the inequality followed by multiplying by $-1$,
		we obtain $1 - (1-\hat{a})^K \geq 1 - e^{-\bar{a}}$. Lastly, multiplying by $\mathcal{E}$ and inverting, we obtain the inequality
		\begin{equation}\label{eq:invs_hat_eta_i_UB_1}
			\eta_i = \mathcal{E}(1 - (1-\hat{a})^K) \geq \mathcal{E}(1 - e^{-\bar{a}}) \quad\quad \Longrightarrow \quad\quad \frac{1}{\eta_i} \leq \frac{1}{\mathcal{E}(1 - e^{-\bar{a}})}.
		\end{equation}
		It is clear that as $\bar{a}\rightarrow\infty$ (i.e., $\sqrt{K}$ approaches infinity faster than $\sqrt{I}\sqrt{J}$) then $\lim_{\bar a\rightarrow\infty}e^{-\bar{a}}=0$, leading to the lower-bounding limit of
		\[\lim_{K\rightarrow\infty}\frac{1}{\mathcal{E}(1 - e^{-\bar{a}})} = \frac{1}{\mathcal{E}} \quad\quad \Longrightarrow \quad\quad \frac{1}{\mathcal{E}} \leq \frac{1}{\eta_i} \leq \frac{1}{\mathcal{E}(1 - e^{-\bar{a}})}.\]
		%That is,
		%\[\frac{1}{\mathcal{E}} \leq \frac{1}{\hat{\eta_i}} \leq \frac{1}{\mathcal{E}\left(1 - e^{-\bar{a}}\right)}.\]
		Now, notice that the expression $\frac{1}{\mathcal{E}(1 - e^{-\bar{a}})}$ grows toward infinity as $\bar{a}\rightarrow0^+$ (which will occur when $\sqrt{I}\sqrt{J}$ approaches infinity faster than $\sqrt{K}$), since $\lim_{\bar a\rightarrow0^+}e^{-\bar a}=1$. Therefore, to prevent the term $\bar{a}$ from approaching $0$, we can impose the bound
		\begin{equation}
			IJ\leq K.\label{eq:K_LB_2}
		\end{equation}
		Thus, when imposing bound~\eqref{eq:K_LB_2} and considering that $I\geq1$, $J\geq1$, and $K\geq1$, we can see that $\bar{a} = \frac{\sqrt{K}}{\sqrt{I}\sqrt{J}}\rho_{f_3}$ is bounded by
		\begin{equation}\label{eq:bounds_on_bar_a}
			\rho_{f_3} \leq \bar a.
		\end{equation}
		Therefore, utilizing the lower-bound in~\eqref{eq:bounds_on_bar_a} will yield the desired upper-bound on $1/\eta_i$ of
		\begin{equation}\label{eq:invs_hat_eta_i_UB_final}
			\frac{1}{\eta_i} \leq \frac{1}{\mathcal{E}(1 - e^{-\rho_{f_3}})}.
		\end{equation}

		\textbf{(Consolidation of bounds):} To summarize, we choose the step-sizes~$\alpha_i$, $\beta_i$, and~$\gamma_i$ according to~\eqref{eq:alpha_i_choice}, \eqref{eq:beta_i_choice}, and~\eqref{eq:gamma_i_choice}, respectively, as well as impose the following bounds on $I$, $J$, and $K$ (defined by~\eqref{eq:I_LB}, \eqref{eq:IJK_LB}, \eqref{eq:K_LB_1}, \eqref{eq:IJ_LB}, \eqref{eq:J_LB_1}, \eqref{eq:JK_LB_1}, and lastly~\eqref{eq:K_LB_2}, respectively), restated here for convenience: 
		\[4(L_F + 4L_y^2 + L_{\nabla y} + 4L_z^2 + L_{\nabla z} + 8L_{z_{xy}}^2)^{2} \leq I,\]
		\[\left(\mu_z + L_{\nabla f_3}\right)^2 \leq IJK, \quad\quad\quad \frac{IJ(Jg_1 + 1)^2}{\rho_{f_3}^2} \leq K, \quad\quad\quad \max\left\{ \mu_y + L_{\nabla \bar{f}}, \frac{2\hat\omega^2+1}{\rho} \right\}^2 \leq IJ,\]
		\[\frac{\left(L_{F_{yz}}^2 + 8L_y^2 + \frac{L_{\nabla y} \zeta}{2}\right)^2}{\Gamma^2} \leq J,  \quad\quad\quad \frac{\left(L_{F_{yz}}^2 + 8L_z^2 + \frac{L_{\nabla z}\zeta}{2}\right)^2}{\rho_{f_3}^2} \leq JK, \quad\quad\quad IJ\leq K.\]
		We can denote the constant lower-bound on $J$ given by~\eqref{eq:J_LB_1} as
		\begin{equation}\label{eq:varsigma_choice}
			J\geq\varsigma := \frac{\left(L_{F_{yz}}^2 + 8L_y^2 + \frac{L_{\nabla y} \zeta}{2}\right)^2}{\Gamma^2}.
		\end{equation}
		Using~\eqref{eq:varsigma_choice}, the bounds~\eqref{eq:I_LB} and~\eqref{eq:IJ_LB} are implied by the consolidated bound
		\begin{equation}\label{eq:consolidated_LB_I}
			\varpi \leq I,
		\end{equation}
		where the constant $\varpi$ is defined as
		\begin{equation}\label{eq:varpi_def}
			\varpi := \max\left\{ 4(L_F + 4L_y^2 + L_{\nabla y} + 4L_z^2 + L_{\nabla z} + 8L_{z_{xy}}^2)^{2}, \frac{\max\left\{ \mu_y + L_{\nabla \bar{f}}, \frac{2\hat\omega^2+1}{\rho} \right\}^2}{\varsigma}\right\}.
		\end{equation}
		Similarly, using~\eqref{eq:varsigma_choice} and~\eqref{eq:varpi_def}, we can see that the bounds~\eqref{eq:IJK_LB}, \eqref{eq:K_LB_1}, \eqref{eq:JK_LB_1}, and~\eqref{eq:K_LB_2} are implied by the following consolidated bound
		\begin{equation}\label{eq:consolidated_LB_K}
			\Xi(I,J)\leq K,
		\end{equation}
		where the function $\Xi:\mathbb{N}_+\times \mathbb{N}_+\rightarrow\mathbb{R}_+$ is defined as
		\begin{equation}\label{eq:Xi_def}
			\Xi(I,J) := \max\left\{\frac{(\mu_z + L_{\nabla f_3})^2}{\varpi\varsigma}, \frac{IJ(Jg_1 + 1)^2}{\rho_{f_3}^2}, \frac{\left(L_{F_{yz}}^2 + 8L_z^2 + \frac{L_{\nabla z}\zeta}{2}\right)^2}{\varsigma\rho_{f_3}^2}, IJ \right\},
		\end{equation}
		from which it is immediately clear that $K\geq\mathcal{O}(J^3I)$.

		\textbf{$(\text{Upper-bounding the remaining terms in }\eqref{eq:TSG_final_theorem_intermediate_1})$:} When choosing the step-sizes~$\alpha_i$, $\beta_i$, and~$\gamma_i$ according to~\eqref{eq:alpha_i_choice}, \eqref{eq:beta_i_choice}, and~\eqref{eq:gamma_i_choice}, respectively, as well as choosing $I$, $J$, and $K$ according to~\eqref{eq:consolidated_LB_I}, \eqref{eq:varsigma_choice}, and~\eqref{eq:consolidated_LB_K}, respectively, it follows that~$A_1$ is non-negative while~$A_2$, $A_3$, and~$A_4$ are non-positive in~\eqref{eq:TSG_final_theorem_intermediate_1}. Thus, we can simplify inequality~\eqref{eq:TSG_final_theorem_intermediate_1} to
		\begin{alignat}{2}
			\mathbb{E}[ \mathbb{V}^{i+1} ] - \mathbb{E}[ \mathbb{V}^i ] & \leq -\frac{\alpha_i}{2} \mathbb{E}[ \| \nabla f(x^i) \|^2 ] + \Phi \alpha_i^2 + \left(G_1^i\frac{J+1}{2} + G_2^i + 2\right) JK\gamma_i^2\sigma_{\nabla f_3}^2\nonumber\\
			&\quad + \left(2JL_{z_{xy}}^2 + \left( 1 + J\left( \frac{1}{\mathcal{E}\left(1 - e^{-\rho_{f_3}}\right)} \right) \frac{L_{z_y}^2}{2} \right)G_1^i \right) J\Upsilon\beta_i^2\nonumber \\
			& \leq -\frac{\alpha_i}{2} \mathbb{E}[ \| \nabla f(x^i) \|^2 ] + \left(\Phi + c_1 + c_2J\right)\alpha_i^2,\label{eq:TSG_final_theorem_intermediate_2}
		\end{alignat}
		where the last inequality follows from utilizing the step-sizes~$\alpha_i$, $\beta_i$, and~$\gamma_i$ according to~\eqref{eq:alpha_i_choice}, \eqref{eq:beta_i_choice}, and~\eqref{eq:gamma_i_choice}, respectively, as well as the inequality~\eqref{eq:G_1^i_UB}, recalling that $g_1 = 1 + L_{F_{yz}}^2 + 8L_y^2 + \frac{L_{\nabla y} \zeta}{2}$, and defining the upper-bound on $G_2^i$ of $G_2^i\leq 1 + L_{F_{yz}}^2 + 8L_z^2 + \frac{L_{\nabla z}\zeta}{2} := g_2$ (obtained from~\eqref{eq:TSG_def_G2i} by using $\alpha_i\leq1$ and $\kappa_i=4L_z^2\alpha_i$). Further, the constants $c_1$ and $c_2$ are defined as
		\[c_1 := \sigma_{\nabla_{f_3}}^2\left( \frac{g_1}{2} + g_2 + 2 \right) + g_1\Upsilon, \quad\quad c_2 := 2L_{z_{xy}}^2\Upsilon + \frac{g_1 \sigma_{\nabla_{f_3}}^2}{2} + \frac{g_1 L_{z_y}^2\Upsilon}{2}\left( \frac{1}{\mathcal{E}\left(1 - e^{-\rho_{f_3}}\right)} \right).\]
		%\begin{alignat}{2}
		%    \mathbb{E}[ \mathbb{V}^{i+1} ] - \mathbb{E}[ \mathbb{V}^i ] &\leq -\frac{\alpha_i}{2} \mathbb{E}[ \| \nabla f(x^i) \|^2 ] + \Phi \alpha_i^2 + c_1J\alpha_i^2 + c_2\alpha_i^2\nonumber\\
		%    & = -\frac{\alpha_i}{2} \mathbb{E}[ \| \nabla f(x^i) \|^2 ] + \left(\Phi + c_1J + c_2\right)\alpha_i^2.\label{eq:TSG_final_theorem_intermediate_2}
		%\end{alignat}

		\textbf{$(\text{Telescoping})$:} Now, rearranging~\eqref{eq:TSG_final_theorem_intermediate_2} and telescoping over $i=0,1,...,I-1$ leads to
		\begin{equation}\label{eq:9010}
			\frac{1}{2}\sum_{i=0}^{I-1}\alpha_i \mathbb{E}[ \| \nabla f(x^i) \|^2 ] \leq \mathbb{V}^0 - \mathbb{V}^I + \sum_{i=0}^{I-1}(\Phi + c_1 + c_2J)\alpha_i^2.
		\end{equation}
		Note that $\alpha_i$ is a constant that does not depend on~$i$ given by~\eqref{eq:alpha_i_choice}. Thus, dividing both sides of~\eqref{eq:9010} by~$\frac{1}{2}I\alpha_i$, while noting that $\sum_{i=0}^{I-1}\alpha_i = I\alpha_i$, and considering that~$0\leq\mathbb{V}^i$ for all~$i\in\{0,1,...,I-1\}$, we have
		\[\frac{1}{I}\sum_{i=0}^{I-1} \mathbb{E}[ \| \nabla f(x^i) \|^2 ] \;\leq\; \frac{\mathbb{V}^0 + (\Phi + c_1 + c_2J)\sum_{i=0}^{I-1}\alpha_i^2}{\frac{1}{2}I\alpha_i} \;=\; \frac{2\mathbb{V}^0 + 2(\Phi + c_1 + c_2J)}{\sqrt{I}}.% = \mathcal{O}\left( \frac{J}{\sqrt{I}} \right).
		%&= \frac{\mathbb{V}^0 + \left(\Phi + c_1J + c_2\right)I\alpha_i^2}{\frac{1}{2}I\alpha_i}\\
		%&= \frac{2\mathbb{V}^0 + 2\left(\Phi + c_1J + c_2\right)}{\sqrt{I}}\\
		%& = \mathcal{O}\left( \frac{J}{\sqrt{I}} \right).
		\]
		Therefore, we have obtained the desired convergence result, completing the proof. 
	\end{proof}

	\section{Bounds on bias, variance, and inexactness}\label{app:conv_theory_proofsi_var_inex}
	
	This appendix contains derivations of results that yield bounds on the biasedness and variance of stochastic terms as well as bounds on the sizes, inexactness, and variances of the UL and ML search directions. For ease of notation, since all expectations that are present in the proofs of Lemmas~\ref{lem:TSG_derived_unbiasedness}, \ref{lem:TSG_derived_variance_bounds_of_f_bar}, and \ref{lem:TSG_new_inexact_variance_bound} are conditioned on~$\mathcal{F}_\xi$, we utilize the short-hand notation of $\mathbb{E}[\cdot]:=\mathbb{E}[\cdot\vert\mathcal{F}_\xi]$, unless stated otherwise.

	\begin{lemma}[Bounds on bias of $\nabla z$ and $\nabla^2\Bar{f}$]\label{lem:TSG_derived_unbiasedness}
		Under Assumptions~\ref{as:tri_lip_cont}, \ref{as:strong_conv_f3_z}, \ref{as:TSG_unbiased_estimators}, \ref{as:bound_on_hess_inv_true_MLP}, \linebreak and~\ref{as:TSG_bounded_var}, the stochastic terms $\nabla_x z^\xi$, $\nabla_y z^\xi$, $\nabla_{xy}^2\Bar{f}^\xi$, and $\nabla_{yy}^2\Bar{f}^\xi$ estimate $\nabla_x z$, $\nabla_y z$, $\nabla_{xy}^2\Bar{f}$, and $\nabla_{yy}^2\Bar{f}$, respectively, with biases that are bounded on the order of~$\mathcal{O}(\theta)$, i.e., there exist positive constants~$U_x$, $U_y$, $U_{xy}$, and~$U_{yy}$ such that
		\begin{alignat}{3}
			&\| \nabla_x z(x,y)^\top - \mathbb{E}[ \nabla_x z(x,y;\xi)^\top \vert \mathcal{F}_\xi ] \| \;&&\leq\; U_x\theta,\label{eq:TSG_derived_unbiasedness_x}\\
			&\| \nabla_y z(x,y)^\top - \mathbb{E}[ \nabla_y z(x,y;\xi)^\top \vert \mathcal{F}_\xi] \| \;&&\leq\; U_y\theta,\label{eq:TSG_derived_unbiasedness_y}\\
			&\| \nabla_{xy}^2\Bar{f}(x,y,z) - \mathbb{E}[ \nabla_{xy}^2\Bar{f}(x,y,z;\xi) \vert \mathcal{F}_\xi] \| \;&&\leq\; U_{xy}\theta,\label{eq:TSG_derived_unbiasedness_xy} \\
			&\| \nabla_{yy}^2\Bar{f}(x,y,z) - \mathbb{E}[ \nabla_{yy}^2\Bar{f}(x,y,z;\xi) \vert \mathcal{F}_\xi] \| \;&&\leq\; U_{yy}\theta.\label{eq:TSG_derived_unbiasedness_yy}
		\end{alignat}
		%Specifically, $U_x$ and~$U_{xy}$ are given by~\eqref{eq:1010} and~\eqref{eq:0030}, respectively, in Appendix~\ref{app:lem:TSG_derived_unbiasedness}. Similar expressions can be derived for~$U_y$ and~$U_{yy}$.
	\end{lemma}
	\begin{proof}
		For this proof, we will omit the point $(x,y,z)$ that the terms are evaluated at; we will simply use a $\xi$-superscript as short-hand to indicate any random terms. %Further, for ease of notation, since all expectations that are utilized are conditioned on~$\mathcal{F}_\xi$, we will use the temporary short-hand notation of $\mathbb{E}[\cdot]:=\mathbb{E}[\cdot\vert\mathcal{F}_\xi]$ only for the duration of this proof.
		We can obtain the bound on the biasedness of the estimator $\nabla_x z(x,y;\xi)$ in~equation~\eqref{eq:TSG_derived_unbiasedness_x} by utilizing the consistency of norms along with~\eqref{eq:jacobian_z_ret_x} and Assumption~\ref{as:TSG_unbiased_estimators} to obtain
		\begin{alignat}{2}
			\| \nabla_x z(x,y)^\top - \mathbb{E}[ \nabla_x z(x,y;\xi)^\top ] \| & = \| [ \nabla_{zz}^2 f_3 ]^{-1} \nabla_{zx}^2 f_3 - \mathbb{E}[ [ \nabla_{zz}^2 f_3^\xi  ]^{-1} ] \nabla_{zx}^2 f_3 \| \nonumber\\
			& \leq \| \nabla_{zx}^2 f_3 \|  \| [ \nabla_{zz}^2 f_3 ]^{-1} - \mathbb{E}[ [ \nabla_{zz}^2 f_3^\xi  ]^{-1} ] \|\nonumber\\
			& \leq L_{\nabla f_3}W_{zz} \theta := U_x \theta, \label{eq:1010}
		\end{alignat}
		where the last inequality follows from applying Assumptions~\ref{as:tri_lip_cont} and~\ref{as:TSG_bounded_var}.
		The proof of biasedness for the estimator $\nabla_y z(x,y;\xi)$ in equation~\eqref{eq:TSG_derived_unbiasedness_y} can be established following identical arguments.
		
		Now, to prove the biasedness of the estimator $\nabla_{xy}^2\Bar{f}(x,y,z;\xi)$, referencing equations~\eqref{hess_yx_fbar} and~\eqref{eq:partial_x}, utilizing Assumption~\ref{as:TSG_unbiased_estimators}, applying the triangle inequality along with the consistency of matrix norms, we have
		%\begin{alignat}{2}
		%    &\| \nabla_{xy}^2\Bar{f} - \mathbb{E}[ \nabla_{xy}^2\Bar{f}^\xi ] \|\nonumber\\
		%    & \leq \| \nabla_{yzx}^3 f_3 [\nabla_{zz}^2 f_3]^{-1} \nabla_z f_2 - \mathbb{E}[\nabla_{yzx}^3 f_3^\xi  [ \nabla_{zz}^2 f_3^\xi ]^{-1}  \nabla_z f_2^\xi ] \|\nonumber\\
		%    &\quad+ \| \nabla_{yzz}^3 f_3 \nabla_x z^\top [\nabla_{zz}^2 f_3]^{-1} \nabla_z f_2 - \mathbb{E}[ \nabla_{yzz}^3 f_3^\xi \nabla_x z^{^\xi\top}  [ \nabla_{zz}^2 f_3^\xi ]^{-1}  \nabla_z f_2^\xi ] \|\nonumber\\
		%    &\quad+ \| \nabla_{yz}^2f_3 [ \nabla_{zz}^2 f_3 ]^{-1}  \nabla_{zzx}^3 f_3  [ \nabla_{zz}^2 f_3 ]^{-1} \nabla_z f_2 - \mathbb{E}[ \nabla_{yz}^2f_3^\xi  [ \nabla_{zz}^2 f_3^\xi ]^{-1}  \nabla_{zzx}^3 f_3^\xi [ \nabla_{zz}^2 f_3^\xi ]^{-1}  \nabla_z f_2^\xi ]\|\nonumber\\
		%    &\quad+ \| \nabla_{yz}^2f_3 [ \nabla_{zz}^2 f_3 ]^{-1} \nabla_{zzz}^3 f_3 \nabla_x z^\top [ \nabla_{zz}^2 f_3 ]^{-1} \nabla_z f_2 - \mathbb{E}[ \nabla_{yz}^2f_3^\xi [ \nabla_{zz}^2 f_3^\xi ]^{-1} \nabla_{zzz}^3 f_3^\xi \nabla_x z^{^\xi\top} [ \nabla_{zz}^2 f_3^\xi ]^{-1}  \nabla_z f_2^\xi ] \|\nonumber\\
		%    &\quad + \| \nabla_{yz}^2 f_3 [ \nabla_{zz}^2f_3 ]^{-1} \nabla_{zx}^2 f_2 - \mathbb{E}[ \nabla_{yz}^2 f_3^\xi [ \nabla_{zz}^2f_3^\xi ]^{-1}  \nabla_{zx}^2 f_2^\xi ]\|\nonumber\\
		%    &\quad+ \| \nabla_{yz}^2 f_3 [ \nabla_{zz}^2f_3 ]^{-1} \nabla_{zz}^2 f_2 \nabla_x z^\top - \mathbb{E}[ \nabla_{yz}^2 f_3^\xi [ \nabla_{zz}^2f_3^\xi ]^{-1} \nabla_{zz}^2 f_2^\xi \nabla_x z^{^\xi\top} ] \|\nonumber.
		%\end{alignat}
		\begin{alignat}{2}
			&\| \nabla_{xy}^2\Bar{f} - \mathbb{E}[ \nabla_{xy}^2\Bar{f}^\xi ] \|\nonumber\\
			& \leq \| \nabla_{yzx}^3 f_3 \| \| \nabla_z f_2 \|  \|  [\nabla_{zz}^2 f_3]^{-1}  - \mathbb{E}[ [ \nabla_{zz}^2 f_3^\xi ]^{-1} ] \| \label{eq:TSG_biased_1}\\
			&\quad+ \| \nabla_{yzz}^3 f_3 \| \| \nabla_z f_2 \| \| \nabla_x z^\top [\nabla_{zz}^2 f_3]^{-1} - \mathbb{E}[ \nabla_x z^{^\xi\top}  [ \nabla_{zz}^2 f_3^\xi ]^{-1} ] \| \label{eq:TSG_biased_2}\\
			&\quad+ \| \nabla_{yz}^2f_3 \| \| \nabla_z f_2 \| \|  [ \nabla_{zz}^2 f_3 ]^{-1}  \nabla_{zzx}^3 f_3  [ \nabla_{zz}^2 f_3 ]^{-1} - \mathbb{E}[ [ \nabla_{zz}^2 f_3^\xi ]^{-1}  \nabla_{zzx}^3 f_3^\xi [ \nabla_{zz}^2 f_3^\xi ]^{-1} ]\| \label{eq:TSG_biased_3}\\
			&\quad+ \| \nabla_{yz}^2f_3 \| \| \nabla_z f_2 \| \| [ \nabla_{zz}^2 f_3 ]^{-1} \nabla_{zzz}^3 f_3 \nabla_x z^\top [ \nabla_{zz}^2 f_3 ]^{-1} - \mathbb{E}[ [ \nabla_{zz}^2 f_3^\xi ]^{-1} \nabla_{zzz}^3 f_3^\xi \nabla_x z^{^\xi\top} [ \nabla_{zz}^2 f_3^\xi ]^{-1} ] \| \label{eq:TSG_biased_4}\\
			&\quad+ \| \nabla_{yz}^2 f_3  \| \| \nabla_{zx}^2 f_2 \| \| [ \nabla_{zz}^2f_3 ]^{-1} - \mathbb{E}[ [ \nabla_{zz}^2f_3^\xi ]^{-1} ]\| \label{eq:TSG_biased_5}\\
			&\quad+ \| \nabla_{yz}^2 f_3 \| \| [ \nabla_{zz}^2f_3 ]^{-1} \nabla_{zz}^2 f_2 \nabla_x z^\top - \mathbb{E}[ [ \nabla_{zz}^2f_3^\xi ]^{-1} \nabla_{zz}^2 f_2^\xi \nabla_x z^{^\xi\top} ] \|. \label{eq:TSG_biased_6}
		\end{alignat}
		
		Notice that there are six difference terms here. Applying Assumption~\ref{as:tri_lip_cont} and~\ref{as:TSG_bounded_var}, we can bound equations~\eqref{eq:TSG_biased_1} and~\eqref{eq:TSG_biased_5} in the following way:
		\begin{equation}\label{eq:TSG_biased_1_bound}
			\| \nabla_{yzx}^3 f_3 \| \| \nabla_z f_2 \|  \|  [\nabla_{zz}^2 f_3]^{-1}  - \mathbb{E}[ [ \nabla_{zz}^2 f_3^\xi ]^{-1} ] \| \leq L_{\nabla^2 f_3} L_{f_2} W_{zz} \theta,
		\end{equation}
		\begin{equation}\label{eq:TSG_biased_5_bound}
			\| \nabla_{yz}^2 f_3  \| \| \nabla_{zx}^2 f_2 \| \| [ \nabla_{zz}^2f_3 ]^{-1} - \mathbb{E}[ [ \nabla_{zz}^2f_3^\xi ]^{-1} ]\| \leq L_{\nabla f_3} L_{\nabla f_2} W_{zz} \theta.
		\end{equation}
		
		Now, looking at equation~\eqref{eq:TSG_biased_2}, applying Assumption~\ref{as:tri_lip_cont}, adding and subtracting \linebreak$\nabla_x z^\top \mathbb{E}[  [ \nabla_{zz}^2 f_3^\xi ]^{-1} ]$, and applying the triangle inequality, we have
		\begin{alignat}{2}
			&\| \nabla_{yzz}^3 f_3 \| \| \nabla_z f_2 \| \| \nabla_x z^\top [\nabla_{zz}^2 f_3]^{-1} - \mathbb{E}[ \nabla_x z^{^\xi\top}  [ \nabla_{zz}^2 f_3^\xi ]^{-1} ] \|\nonumber\\
			& \leq L_{\nabla^2 f_3} L_{f_2} \| \nabla_x z^\top [\nabla_{zz}^2 f_3]^{-1} - \mathbb{E}[ \nabla_x z^{^\xi\top} ] \mathbb{E}[  [ \nabla_{zz}^2 f_3^\xi ]^{-1} ] \|\nonumber \\
			& \leq L_{\nabla^2 f_3} L_{f_2} \| \nabla_x z^\top [\nabla_{zz}^2 f_3]^{-1} - \nabla_x z^\top \mathbb{E}[  [ \nabla_{zz}^2 f_3^\xi ]^{-1} ] \|\nonumber\\
			&\quad+ L_{\nabla^2 f_3} L_{f_2}\| \nabla_x z^\top \mathbb{E}[  [ \nabla_{zz}^2 f_3^\xi ]^{-1} ] - \mathbb{E}[ \nabla_x z^{^\xi\top} ] \mathbb{E}[  [ \nabla_{zz}^2 f_3^\xi ]^{-1} ] \|\nonumber\\
			& \leq L_{\nabla^2 f_3} L_{f_2} \| \nabla_x z^\top \|  \| [\nabla_{zz}^2 f_3]^{-1} - \mathbb{E}[  [ \nabla_{zz}^2 f_3^\xi ]^{-1} ] \| + L_{\nabla^2 f_3} L_{f_2}\| \mathbb{E}[  [ \nabla_{zz}^2 f_3^\xi ]^{-1} ] \|  \| \nabla_x z^\top - \mathbb{E}[ \nabla_x z^{^\xi\top} ] \|\nonumber\\
			&\leq L_{\nabla^2 f_3} L_{f_2} \frac{L_{\nabla f_3}}{\mu_z} W_{zz} \theta + L_{\nabla^2 f_3} L_{f_2} b_{zz} U_x \theta \;=\; L_{\nabla^2 f_3} L_{f_2}\left( \frac{W_{zz} L_{\nabla f_3}}{\mu_z} + b_{zz} U_x\right)  \theta,\label{eq:TSG_biased_2_bound}
		\end{alignat}
		where the second-to-last inequality follows from the consistency of norms, and the last inequality follows from applying the derived bound~\eqref{eq:1010}, equation~\eqref{eq:aux_lip_z_result}, and Assumptions~\ref{as:bound_on_hess_inv_true_MLP} and~\ref{as:TSG_bounded_var}.
		
		Now, looking at~\eqref{eq:TSG_biased_6}, applying Assumptions~\ref{as:tri_lip_cont} and~\ref{as:TSG_unbiased_estimators}, adding and subtracting the term~$[ \nabla_{zz}^2f_3 ]^{-1} \nabla_{zz}^2 f_2 \mathbb{E}[ \nabla_x z^{^\xi\top} ]$, applying the triangle inequality, and using the consistency of matrix norms, we have
		\begin{alignat*}{2}
			&\| \nabla_{yz}^2 f_3 \| \| [ \nabla_{zz}^2f_3 ]^{-1} \nabla_{zz}^2 f_2 \nabla_x z^\top - \mathbb{E}[ [ \nabla_{zz}^2f_3^\xi ]^{-1}  \nabla_{zz}^2 f_2^\xi \nabla_x z^{^\xi\top} ] \|\\
			&\leq L_{\nabla f_3} \| [ \nabla_{zz}^2f_3 ]^{-1} \nabla_{zz}^2 f_2 \nabla_x z^\top - [ \nabla_{zz}^2f_3 ]^{-1} \nabla_{zz}^2 f_2 \mathbb{E}[ \nabla_x z^{^\xi\top} ] \|\\
			&\quad+ L_{\nabla f_3} \| [ \nabla_{zz}^2f_3 ]^{-1} \nabla_{zz}^2 f_2 \mathbb{E}[ \nabla_x z^{^\xi\top} ] - \mathbb{E}[ [ \nabla_{zz}^2f_3^\xi ]^{-1} ] \nabla_{zz}^2 f_2 \mathbb{E}[ \nabla_x z^{^\xi\top} ] \|\\
			& \leq L_{\nabla f_3} \| [ \nabla_{zz}^2f_3 ]^{-1} \|  \| \nabla_{zz}^2 f_2 \|  \| \nabla_x z^\top - \mathbb{E}[ \nabla_x z^{^\xi\top} ] \|\\
			&\quad+ L_{\nabla f_3} \| \nabla_{zz}^2 f_2 \|  \| \mathbb{E}[ \nabla_x z^{^\xi\top} ] \|  \| [ \nabla_{zz}^2f_3 ]^{-1} - \mathbb{E}[ [ \nabla_{zz}^2f_3^\xi ]^{-1} ] \|.
		\end{alignat*}
		Now, applying Assumptions~\ref{as:tri_lip_cont}, \ref{as:strong_conv_f3_z}, and~\ref{as:TSG_bounded_var}, along with the derived bound~\eqref{eq:1010}, we have
		\begin{alignat}{2}
			&\| \nabla_{yz}^2 f_3 \| \| [ \nabla_{zz}^2f_3 ]^{-1} \nabla_{zz}^2 f_2 \nabla_x z^\top - \mathbb{E}[ [ \nabla_{zz}^2f_3^\xi ]^{-1}  \nabla_{zz}^2 f_2^\xi \nabla_x z^{^\xi\top} ] \|\nonumber\\
			&\leq \frac{L_{\nabla f_3} L_{\nabla f_2}U_x}{\mu_z} \theta + L_{\nabla f_3} L_{\nabla f_2} b_{zz} L_{\nabla f_3} W_{zz} \theta \;=\; L_{\nabla f_3} L_{\nabla f_2}\left(\frac{U_x}{\mu_z} + b_{zz} L_{\nabla f_3}W_{zz} \right) \theta,\label{eq:TSG_biased_6_bound}
		\end{alignat}
		where the last inequality follows from $\| \mathbb{E}[ \nabla_x z^{^\xi\top} ] \| = \| -\mathbb{E}[[ \nabla_{zz}^2 f_3^\xi ]^{-1}]  \nabla_{zx}^2 f_3 \| \leq b_{zz} L_{\nabla f_3}$ (from Assumptions~\ref{as:tri_lip_cont}, \ref{as:TSG_unbiased_estimators}, and~\ref{as:bound_on_hess_inv_true_MLP}).
		
		Now, looking at~\eqref{eq:TSG_biased_3}, applying Assumptions~\ref{as:tri_lip_cont} and~\ref{as:TSG_unbiased_estimators}, adding and subtracting the term $[ \nabla_{zz}^2 f_3 ]^{-1}  \nabla_{zzx}^3 f_3 \mathbb{E}[[ \nabla_{zz}^2 f_3^\xi ]^{-1} ]$, applying the triangle inequality, and using the consistency of matrix norms, we have
		\begin{alignat}{2}
			&\| \nabla_{yz}^2f_3 \| \| \nabla_z f_2 \| \|  [ \nabla_{zz}^2 f_3 ]^{-1}  \nabla_{zzx}^3 f_3  [ \nabla_{zz}^2 f_3 ]^{-1} - \mathbb{E}[ [ \nabla_{zz}^2 f_3^\xi ]^{-1}  \nabla_{zzx}^3 f_3^\xi [ \nabla_{zz}^2 f_3^\xi ]^{-1} ]\|\nonumber\\
			& \leq L_{\nabla f_3} L_{f_2} \|  [ \nabla_{zz}^2 f_3 ]^{-1}  \nabla_{zzx}^3 f_3  [ \nabla_{zz}^2 f_3 ]^{-1} - [ \nabla_{zz}^2 f_3 ]^{-1}  \nabla_{zzx}^3 f_3 \mathbb{E}[[ \nabla_{zz}^2 f_3^\xi ]^{-1} ]\|\nonumber\\
			&\quad+ L_{\nabla f_3} L_{f_2} \| [ \nabla_{zz}^2 f_3 ]^{-1}  \nabla_{zzx}^3 f_3 \mathbb{E}[[ \nabla_{zz}^2 f_3^\xi ]^{-1} ] - \mathbb{E}[ [ \nabla_{zz}^2 f_3^\xi ]^{-1}]  \nabla_{zzx}^3 f_3 \mathbb{E}[[ \nabla_{zz}^2 f_3^\xi ]^{-1} ]\|\nonumber\\
			&\leq L_{\nabla f_3} L_{f_2} \| [ \nabla_{zz}^2 f_3 ]^{-1} \|  \| \nabla_{zzx}^3 f_3 \|  \| [ \nabla_{zz}^2 f_3 ]^{-1} -  \mathbb{E}[[ \nabla_{zz}^2 f_3^\xi ]^{-1} ] \|\nonumber\\
			&\quad+ L_{\nabla f_3} L_{f_2} \| \nabla_{zzx}^3 f_3 \|  \| \mathbb{E}[[ \nabla_{zz}^2 f_3^\xi ]^{-1} ] \|  \| [ \nabla_{zz}^2 f_3 ]^{-1} - \mathbb{E}[ [ \nabla_{zz}^2 f_3^\xi ]^{-1}] \|\nonumber\\
			& \leq L_{\nabla f_3} L_{f_2} \frac{1}{\mu_z} L_{\nabla^2 f_3} W_{zz} \theta + L_{\nabla f_3} L_{f_2} L_{\nabla^2 f_3} b_{zz} W_{zz} \theta \;=\; L_{\nabla f_3} L_{f_2} L_{\nabla^2 f_3} W_{zz}\left( \frac{1}{\mu_z} + b_{zz}\right)  \theta,\label{eq:TSG_biased_3_bound}
		\end{alignat}
		where the last inequality follows from applying Assumptions~\ref{as:tri_lip_cont}, \ref{as:strong_conv_f3_z}, \ref{as:bound_on_hess_inv_true_MLP}, and~\ref{as:TSG_bounded_var}.
		
		Now, looking at~\eqref{eq:TSG_biased_4}, applying Assumptions~\ref{as:tri_lip_cont} and~\ref{as:TSG_unbiased_estimators}, adding and subtracting the term $[ \nabla_{zz}^2 f_3 ]^{-1} \nabla_{zzz}^3 f_3 \nabla_x z^\top \mathbb{E}[ [ \nabla_{zz}^2 f_3^\xi ]^{-1} ]$, applying the triangle inequality, and using the consistency of matrix norms, we have
		\begin{alignat*}{2}
			&\| \nabla_{yz}^2f_3 \| \| \nabla_z f_2 \| \| [ \nabla_{zz}^2 f_3 ]^{-1} \nabla_{zzz}^3 f_3 \nabla_x z^\top [ \nabla_{zz}^2 f_3 ]^{-1} - \mathbb{E}[ [ \nabla_{zz}^2 f_3^\xi ]^{-1} \nabla_{zzz}^3 f_3^\xi \nabla_x z^{^\xi\top} [ \nabla_{zz}^2 f_3^\xi ]^{-1} ] \|\\
			& \leq L_{\nabla f_3} L_{f_2} \| [ \nabla_{zz}^2 f_3 ]^{-1} \nabla_{zzz}^3 f_3 \nabla_x z^\top [ \nabla_{zz}^2 f_3 ]^{-1} - [ \nabla_{zz}^2 f_3 ]^{-1} \nabla_{zzz}^3 f_3 \nabla_x z^\top \mathbb{E}[ [ \nabla_{zz}^2 f_3^\xi ]^{-1} ] \|\\
			&\quad+ L_{\nabla f_3} L_{f_2} \| [ \nabla_{zz}^2 f_3 ]^{-1} \nabla_{zzz}^3 f_3 \nabla_x z^\top \mathbb{E}[ [ \nabla_{zz}^2 f_3^\xi ]^{-1} ] - \mathbb{E}[ [ \nabla_{zz}^2 f_3^\xi ]^{-1} ] \nabla_{zzz}^3 f_3 \mathbb{E}[\nabla_x z^{^\xi\top}] \mathbb{E}[ [ \nabla_{zz}^2 f_3^\xi ]^{-1} ] \|\\
			& \leq L_{\nabla f_3} L_{f_2} \| [ \nabla_{zz}^2 f_3 ]^{-1} \|  \| \nabla_{zzz}^3 f_3 \|  \| \nabla_x z^\top \|  \| [ \nabla_{zz}^2 f_3 ]^{-1} -  \mathbb{E}[ [ \nabla_{zz}^2 f_3^\xi ]^{-1} ] \|\\
			&\quad+ L_{\nabla f_3} L_{f_2} \| \mathbb{E}[ [ \nabla_{zz}^2 f_3^\xi ]^{-1} ] \|  \| [ \nabla_{zz}^2 f_3 ]^{-1} \nabla_{zzz}^3 f_3 \nabla_x z^\top  - \mathbb{E}[ [ \nabla_{zz}^2 f_3^\xi ]^{-1} ] \nabla_{zzz}^3 f_3 \mathbb{E}[\nabla_x z^{^\xi\top}] \|\\
			&\leq \frac{L_{\nabla f_3}^2 L_{f_2} L_{\nabla^2 f_3}}{\mu_z^2} W_{zz} \theta + L_{\nabla f_3} L_{f_2} b_{zz} \| [ \nabla_{zz}^2 f_3 ]^{-1} \nabla_{zzz}^3 f_3 \nabla_x z^\top  - \mathbb{E}[ [ \nabla_{zz}^2 f_3^\xi ]^{-1} ] \nabla_{zzz}^3 f_3 \mathbb{E}[\nabla_x z^{^\xi\top}] \|,
		\end{alignat*}
		where the last inequality follows from applying Assumptions~\ref{as:tri_lip_cont}, \ref{as:strong_conv_f3_z}, \ref{as:bound_on_hess_inv_true_MLP}, and~\ref{as:TSG_bounded_var}, along with equation~\eqref{eq:aux_lip_z_result}. Now, using nearly identical arguments to those that were used in deriving the bound on~\eqref{eq:TSG_biased_6}, we have
		\begin{alignat}{2}
			&\| \nabla_{yz}^2f_3 \| \| \nabla_z f_2 \| \| [ \nabla_{zz}^2 f_3 ]^{-1} \nabla_{zzz}^3 f_3 \nabla_x z^\top [ \nabla_{zz}^2 f_3 ]^{-1} - \mathbb{E}[ [ \nabla_{zz}^2 f_3^\xi ]^{-1} \nabla_{zzz}^3 f_3^\xi \nabla_x z^{^\xi\top} [ \nabla_{zz}^2 f_3^\xi ]^{-1} ] \|\nonumber\\
			& \leq \frac{L_{\nabla f_3}^2 L_{f_2} L_{\nabla^2 f_3}}{\mu_z^2} W_{zz} \theta + L_{\nabla f_3} L_{f_2} b_{zz} \left( L_{\nabla^2 f_3}\left(\frac{U_x}{\mu_z} + b_{zz} L_{\nabla f_3}W_{zz} \right) \theta \right)\nonumber\\
			& = L_{\nabla f_3}L_{f_2}L_{\nabla^2 f_3}\left(\frac{L_{\nabla f_3}}{\mu_z^2} W_{zz} +  b_{zz} \left(\frac{U_x}{\mu_z} + b_{zz} L_{\nabla f_3}W_{zz} \right) \right) \theta.\label{eq:TSG_biased_4_bound}
		\end{alignat}
		Finally, substituting the newly-derived upper-bounds~\eqref{eq:TSG_biased_1_bound}--\eqref{eq:TSG_biased_4_bound} in for the terms~\eqref{eq:TSG_biased_1}--\eqref{eq:TSG_biased_6}, we have the desired upper bound~\eqref{eq:TSG_derived_unbiasedness_xy} as
		\begin{alignat*}{2}
			\| \nabla_{xy}^2\Bar{f} - \mathbb{E}[ \nabla_{xy}^2\Bar{f}^\xi ] \| &\leq L_{\nabla^2 f_3} L_{f_2} W_{zz} \theta + L_{\nabla^2 f_3} L_{f_2}\left( \frac{W_{zz} L_{\nabla f_3}}{\mu_z} + b_{zz} U_x\right) \theta + L_{\nabla f_3} L_{\nabla f_2} W_{zz} \theta\\
			&\quad + L_{\nabla f_3}L_{f_2}L_{\nabla^2 f_3}\left(\frac{L_{\nabla f_3}}{\mu_z^2} W_{zz} +  b_{zz} \left(\frac{U_x}{\mu_z} + b_{zz} L_{\nabla f_3}W_{zz} \right) \right) \theta\\
			&\quad+  L_{\nabla f_3} L_{\nabla f_2}\left(\frac{U_x}{\mu_z} + b_{zz} L_{\nabla f_3}W_{zz} \right) \theta  + L_{\nabla f_3} L_{f_2} L_{\nabla^2 f_3} W_{zz}\left( \frac{1}{\mu_z} + b_{zz}\right) \theta\\
			&\;=\; U_{xy} \theta,
		\end{alignat*}
		where
		\begin{equation}\label{eq:0030}
			\begin{split}
				U_{xy} & \;:=\; L_{\nabla^2 f_3} L_{f_2} \left( W_{zz}  + b_{zz} U_x + L_{\nabla f_3} W_{zz}b_{zz} + 2\frac{W_{zz} L_{\nabla f_3}}{\mu_z}\right)\\
				&\quad + L_{\nabla f_3} \left( L_{f_2}L_{\nabla^2 f_3}\left(\frac{L_{\nabla f_3}}{\mu_z^2} W_{zz} +  b_{zz} \left(\frac{U_x}{\mu_z} + b_{zz} L_{\nabla f_3}W_{zz} \right) \right)  +  L_{\nabla f_2}\left(\frac{U_x}{\mu_z} + b_{zz} L_{\nabla f_3}W_{zz} + W_{zz}\right)\right).
			\end{split}
		\end{equation}
		
		The proof of the biasedness errors for the estimator~$\nabla_{xy}^2\Bar{f}(x,y,z;\xi)$ in equation~\eqref{eq:TSG_derived_unbiasedness_yy} can be established following nearly identical arguments.
	\end{proof}

	\begin{lemma}[Bounds on variance of $\nabla^2\Bar{f}$]\label{lem:TSG_derived_variance_bounds_of_f_bar}
		Under Assumptions~\ref{as:tri_lip_cont}, \ref{as:TSG_unbiased_estimators}, and~\ref{as:bound_on_hess_inv_true_MLP}, the variances of the stochastic matrices~$\nabla_{xy}^2\Bar{f}^{\xi}$ and~$\nabla_{yy}^2\Bar{f}^{\xi}$ are bounded, i.e., there exist positive constants~$V_{xy}$ and $V_{yy}$ such that
		\[
		\mathbb{E}[\|\nabla_{xy}^2\Bar{f}(x,y,z;\xi) - \mathbb{E}\left[ \nabla_{xy}^2\Bar{f}(x,y,z;\xi) \vert \mathcal{F}_\xi] \|^2 \vert \mathcal{F}_\xi\right] \;\leq\; V_{xy},
		\]
		\[
		\mathbb{E}[\| \nabla_{yy}^2\Bar{f}(x,y,z;\xi) - \mathbb{E}[ \nabla_{yy}^2\Bar{f}(x,y,z;\xi) \vert \mathcal{F}_\xi ] \|^2 \vert \mathcal{F}_\xi] \;\leq\; V_{yy}.
		\]
		%Specifically, $V_{xy}$ is given by~\eqref{eq:0040} in Appendix~\ref{app:lem:TSG_derived_variance_bounds_of_f_bar}. A similar expression can be derived for~$V_{yy}$.
	\end{lemma}
	\begin{proof}
		For this proof, we will omit the point $(x,y,z)$ that the terms are evaluated at; we will simply use a $\xi$-superscript as short-hand to indicate any random terms. %Further, for ease of notation throughout this proof, since all expectations that are utilized are conditioned on~$\mathcal{F}_\xi$, we will use the temporary short-hand notation of $\mathbb{E}[\cdot]:=\mathbb{E}[\cdot\vert\mathcal{F}_\xi]$ only for the duration of this proof.
		We can obtain the bound on the variance of~$\nabla_{xy}^2\Bar{f}^{\xi}$ by first referencing~\eqref{hess_yx_fbar} and applying the fact that~$\|\sum_{i=1}^N a_i\|^2 \leq N\sum_{i=1}^N\|a_i\|^2$ (for some $a\in\mathbb{R}^N$) to the two initial difference terms as well as all of the resulting terms (leading to a total of~12 terms), Assumption~\ref{as:TSG_unbiased_estimators}, and the consistency of matrix norms, to obtain
		\begin{alignat*}{2}
			&\mathbb{E}[\|\nabla_{xy}^2\Bar{f}^{\xi} - \mathbb{E}[ \nabla_{xy}^2\Bar{f}^{\xi} ] \|^2]\\
			&\leq 12\mathbb{E}[\| \nabla_{yzx}^3 f_3^{\xi}\|^2] \mathbb{E}[\| [\nabla_{zz}^2 f_3^{\xi}]^{-1}\|^2] \mathbb{E}[\| \nabla_z f_2^{\xi} \|^2] \;+\; 12\|\nabla_{yzx}^3 f_3\|^2\|\mathbb{E}[ [ \nabla_{zz}^2 f_3^{\xi} ]^{-1} ]\|^2\| \nabla_z f_2 \|^2\\
			&\quad+ 12\mathbb{E}[\| \nabla_{yzz}^3 f_3^{\xi} \|^2] \mathbb{E}[\| \nabla_x z^{{\xi}\top} \|^2] \mathbb{E}[\| [\nabla_{zz}^2 f_3^{\xi}]^{-1} \|^2] \mathbb{E}[\| \nabla_z f_2^{\xi}\|^2]\\
			&\quad+ 12\|\nabla_{yzz}^3 f_3 \|^2 \| \mathbb{E}[\nabla_x z^{{\xi}\top} ] \|^2 \| \mathbb{E}[[ \nabla_{zz}^2 f_3^{\xi} ]^{-1}] \|^2 \| \nabla_z f_2 \|^2\\
			&\quad+ 12\mathbb{E}[\| \nabla_{yz}^2f_3^{\xi}\|^2] \mathbb{E}[\| [ \nabla_{zz}^2 f_3^{\xi} ]^{-1} \|^2] \mathbb{E}[\| \nabla_{zzx}^3 f_3^{\xi} \|^2] \mathbb{E}[\| [ \nabla_{zz}^2 f_3^{\xi} ]^{-1} \|^2] \mathbb{E}[\| \nabla_z f_2^\xi \|^2]\\
			&\quad+ 12\| \nabla_{yz}^2f_3\|^2 \|\mathbb{E}[  [ \nabla_{zz}^2 f_3^{\xi} ]^{-1} ]\|^2 \| \nabla_{zzx}^3 f_3\|^2 \|\mathbb{E}[ [ \nabla_{zz}^2 f_3^{\xi} ]^{-1} ]\|^2 \| \nabla_z f_2\|^2\\
			&\quad+ 12\mathbb{E}[\| \nabla_{yz}^2f_3^{\xi}\|^2] \mathbb{E}[\| [ \nabla_{zz}^2 f_3^{\xi} ]^{-1} \|^2] \mathbb{E}[\| \nabla_{zzz}^3 f_3^{\xi} \|^2] \mathbb{E}[\| \nabla_x z^{{\xi}\top} \|^2] \mathbb{E}[\| [ \nabla_{zz}^2 f_3^{\xi} ]^{-1} \|^2] \mathbb{E}[\| \nabla_z f_2^{\xi}\|^2]\\
			&\quad+ 12\| \nabla_{yz}^2f_3\|^2 \|\mathbb{E}[ [ \nabla_{zz}^2 f_3^{\xi} ]^{-1} ]\|^2 \| \nabla_{zzz}^3 f_3\|^2 \|\mathbb{E}[ \nabla_x z^{{\xi}\top} ]\|^2 \|\mathbb{E}[ [ \nabla_{zz}^2 f_3^{\xi} ]^{-1} ]\|^2 \| \nabla_z f_2 \|^2\\
			&\quad + 12\mathbb{E}[\| \nabla_{yz}^2 f_3^{\xi}\|^2] \mathbb{E}[\| [ \nabla_{zz}^2f_3^{\xi} ]^{-1} \|^2] \mathbb{E}[\| \nabla_{zx}^2 f_2^{\xi}\|^2] \;+\; 12\| \nabla_{yz}^2 f_3\|^2 \|\mathbb{E}[ [ \nabla_{zz}^2f_3^{\xi} ]^{-1} ]\|^2 \| \nabla_{zx}^2 f_2\|^2\\
			&\quad+ 12\mathbb{E}[\| \nabla_{yz}^2 f_3^{\xi}\|^2] \mathbb{E}[\| [ \nabla_{zz}^2f_3^{\xi} ]^{-1} \|^2] \mathbb{E}[\| \nabla_{zz}^2 f_2^{\xi} \|^2] \mathbb{E}[\| \nabla_x z^{{\xi}\top}\|^2] \\
			&\quad+ 12\| \nabla_{yz}^2 f_3\|^2 \|\mathbb{E}[ [ \nabla_{zz}^2f_3^{\xi} ]^{-1} ]\|^2 \| \nabla_{zz}^2 f_2\|^2 \|\mathbb{E}[ \nabla_x z^{{\xi}\top} ] \|^2.
		\end{alignat*}
		Now, using the result that $\|\mathbb{E}[ \nabla_x z^{{\xi}\top}] \|^2 \leq \|\mathbb{E}[ [ \nabla_{zz}^2 f_3^{\xi}  ]^{-1} ]\|^2 \| \nabla_{zx}^2 f_3\|^2 \leq b_{zz}^2 L_{\nabla f_3}^2$ (from Assumptions~\ref{as:tri_lip_cont}, \ref{as:TSG_unbiased_estimators}, and~\ref{as:bound_on_hess_inv_true_MLP} along with the consistency of matrix norms), the result that $\mathbb{E}[\| \nabla_x z^{{\xi}\top} \|^2] \leq\linebreak \mathbb{E}[\| [ \nabla_{zz}^2 f_3^{\xi}  ]^{-1} \|^2] \mathbb{E}[\| \nabla_{zx}^2 f_3{\xi}\|^2] \leq b_{zz}^2\mathbb{E}[\| \nabla_{zx}^2 f_3{\xi}\|^2]$ (from Assumptions~\ref{as:TSG_unbiased_estimators} and~\ref{as:bound_on_hess_inv_true_MLP} along with the consistency of matrix norms), and applying Assumptions~\ref{as:tri_lip_cont} and~\ref{as:bound_on_hess_inv_true_MLP}, we have
		\begin{alignat*}{2}
			&\mathbb{E}[\|\nabla_{xy}^2\Bar{f}^{\xi} - \mathbb{E}[ \nabla_{xy}^2\Bar{f}^{\xi} ] \|^2]\\
			&\leq 12\mathbb{E}[\| \nabla_{yzx}^3 f_3^{\xi}\|^2] b_{zz}^2 \mathbb{E}[\| \nabla_z f_2^{\xi} \|^2] + 12L_{\nabla^2 f_3} ^2 b_{zz}^2 L_{f_2}^2 + 12\mathbb{E}[\| \nabla_{yzz}^3 f_3^{\xi} \|^2] b_{zz}^2\mathbb{E}[\| \nabla_{zx}^2 f_3^{\xi}\|^2] b_{zz}^2 \mathbb{E}[\| \nabla_z f_2^{\xi}\|^2]\\
			&\quad+ 12L_{\nabla^2 f_3}^2 b_{zz}^2 L_{\nabla f_3}^2 b_{zz}^2 L_{f_2}^2 + 12\mathbb{E}[\| \nabla_{yz}^2f_3^{\xi}\|^2] b_{zz}^2 \mathbb{E}[\| \nabla_{zzx}^3 f_3^{\xi} \|^2] b_{zz}^2 \mathbb{E}[\| \nabla_z f_2^{\xi} \|^2]\\
			&\quad+ 12L_{\nabla f_3}^2 b_{zz}^2 L_{\nabla^2 f_3}^2 b_{zz}^2 L_{f_2}^2 + 12\mathbb{E}[\| \nabla_{yz}^2f_3^{\xi}\|^2] b_{zz}^2 \mathbb{E}[\| \nabla_{zzz}^3 f_3^{\xi} \|^2] b_{zz}^2\mathbb{E}[\| \nabla_{zx}^2 f_3^{\xi}\|^2] b_{zz}^2 \mathbb{E}[\| \nabla_z f_2^{\xi}\|^2]\\
			&\quad+ 12L_{\nabla f_3}^2 b_{zz}^2 L_{\nabla^2 f_3}^2 b_{zz}^2 L_{\nabla f_3}^2 b_{zz}^2 L_{f_2}^2 + 12\mathbb{E}[\| \nabla_{yz}^2 f_3^{\xi}\|^2] b_{zz}^2 \mathbb{E}[\| \nabla_{zx}^2 f_2^{\xi}\|^2] + 12L_{\nabla f_3}^2 b_{zz}^2 L_{\nabla f_2}^2\\
			&\quad+ 12\mathbb{E}[\| \nabla_{yz}^2 f_3^{\xi}\|^2] b_{zz}^2 \mathbb{E}[\| \nabla_{zz}^2 f_2^{\xi} \|^2] b_{zz}^2\mathbb{E}[\| \nabla_{zx}^2 f_3^{\xi}\|^2] + 12L_{\nabla f_3}^2 b_{zz}^2 L_{\nabla f_2}^2 b_{zz}^2 L_{\nabla f_3}^2.
		\end{alignat*}
		Finally, with all of the remaining expectation terms, we can apply the definition of variance (i.e., $\mathbb{E}[ X^2 ] = \text{Var}[ X ] + \mathbb{E}[ X ]^2$) followed by Assumption~\ref{as:TSG_unbiased_estimators} to upper-bound the variance term along with Assumptions~\ref{as:tri_lip_cont} and~\ref{as:TSG_unbiased_estimators} to upper-bound the $\mathbb{E}[ X ]^2$ term. These bounds are given as: 
		\begin{alignat}{7}
			&\mathbb{E}[\| \nabla_{yzx}^3 f_3^{\xi}\|^2] &&\leq \sigma^2_{\nabla^3 f_3} + L_{\nabla^2 f_3}^2, \quad &&\mathbb{E}[\| \nabla_{yzz}^3 f_3^{\xi}\|^2] &&\leq \sigma^2_{\nabla^3 f_3} + L_{\nabla^2 f_3}^2, \nonumber\\
			&\mathbb{E}[\| \nabla_{zzx}^3 f_3^{\xi}\|^2] &&\leq \sigma^2_{\nabla^3 f_3} + L_{\nabla^2 f_3}^2, \quad 
			&&\mathbb{E}[\| \nabla_{zzz}^3 f_3^{\xi}\|^2] &&\leq \sigma^2_{\nabla^3 f_3} + L_{\nabla^2 f_3}^2, \nonumber\\
			&\mathbb{E}[\| \nabla_{zx}^2 f_3^{\xi}\|^2] &&\leq \sigma^2_{\nabla^2 f_3} + L_{\nabla f_3}^2, \quad &&\mathbb{E}[\| \nabla_{yz}^2 f_3^{\xi}\|^2] &&\leq \sigma^2_{\nabla^2 f_3} + L_{\nabla f_3}^2,\nonumber\\ 
			&\mathbb{E}[\| \nabla_{zz}^2 f_3^{\xi}\|^2] &&\leq \sigma^2_{\nabla^2 f_3} + L_{\nabla f_3}^2, \quad &&\mathbb{E}[\| \nabla_z f_2^{\xi} \|^2] &&\leq \sigma^2_{\nabla f_2} + L_{f_2}^2. \nonumber
		\end{alignat}
		Applying these bounds, we will obtain
		\begin{alignat}{2}
			&\mathbb{E}[\|\nabla_{xy}^2\Bar{f}^{\xi} - \mathbb{E}[ \nabla_{xy}^2\Bar{f}^{\xi} ] \|^2]\nonumber\\
			&\leq 12( \sigma^2_{\nabla^3 f_3} + L_{\nabla^2 f_3}^2 ) b_{zz}^2 ( \sigma^2_{\nabla f_2} + L_{f_2}^2 ) + 12L_{\nabla^2 f_3} ^2 b_{zz}^2 L_{f_2}^2\nonumber\\
			&\quad + 12( \sigma^2_{\nabla^3 f_3} + L_{\nabla^2 f_3}^2 ) b_{zz}^2( \sigma^2_{\nabla^2 f_3} + L_{\nabla f_3}^2 ) b_{zz}^2 ( \sigma^2_{\nabla f_2} + L_{f_2}^2 )\nonumber\\
			&\quad+ 12L_{\nabla^2 f_3}^2 b_{zz}^2 L_{\nabla f_3}^2 b_{zz}^2 L_{f_2}^2 + 12( \sigma^2_{\nabla^2 f_3} + L_{\nabla f_3}^2 ) b_{zz}^2 ( \sigma^2_{\nabla^3 f_3} + L_{\nabla^2 f_3}^2 ) b_{zz}^2 ( \sigma^2_{\nabla f_2} + L_{f_2}^2 )\nonumber\\
			&\quad+ 12L_{\nabla f_3}^2 b_{zz}^2 L_{\nabla^2 f_3}^2 b_{zz}^2 L_{f_2}^2 + 12( \sigma^2_{\nabla^2 f_3} + L_{\nabla f_3}^2 ) b_{zz}^2 ( \sigma^2_{\nabla^3 f_3} + L_{\nabla^2 f_3}^2 ) b_{zz}^2( \sigma^2_{\nabla^2 f_3} + L_{\nabla f_3}^2 ) b_{zz}^2 ( \sigma^2_{\nabla f_2} + L_{f_2}^2 )\nonumber\\
			&\quad+ 12L_{\nabla f_3}^2 b_{zz}^2 L_{\nabla^2 f_3}^2 b_{zz}^2 L_{\nabla f_3}^2 b_{zz}^2 L_{f_2}^2 + 12( \sigma^2_{\nabla^2 f_3} + L_{\nabla f_3}^2 ) b_{zz}^2 ( \sigma^2_{\nabla^2 f_3} + L_{\nabla f_3}^2 ) + 12L_{\nabla f_3}^2 b_{zz}^2 L_{\nabla f_2}^2\nonumber\\
			&\quad+ 12( \sigma^2_{\nabla^2 f_3} + L_{\nabla f_3}^2 ) b_{zz}^2 ( \sigma^2_{\nabla^2 f_3} + L_{\nabla f_3}^2 ) b_{zz}^2( \sigma^2_{\nabla^2 f_3} + L_{\nabla f_3}^2 ) + 12L_{\nabla f_3}^2 b_{zz}^2 L_{\nabla f_2}^2 b_{zz}^2 L_{\nabla f_3}^2\nonumber\\
			& := V_{xy}.\label{eq:0040}
		\end{alignat}
		This completes the proof for the variance bound on $\nabla_{xy}^2\Bar{f}^{\xi}$.
		
		The proof of the variance bound on $\nabla_{yy}^2\Bar{f}^{\xi}$ follows nearly identical arguments.
	\end{proof}

	\begin{lemma}[Bounds on bias and variance of UL direction]\label{lem:TSG_new_inexact_variance_bound}
		Recalling the definition of~$\tilde g_{f_1}^{i}$ in equation~\eqref{eq:TSG_UL_stoch_dir}, define~$\bar g_{f_1}^{i} = \mathbb{E} [ \tilde g_{f_1}^{i} \vert \mathcal{F}_{i} ]$. Then, under Assumptions~\ref{as:tri_lip_cont}, \ref{as:strong_conv_f3_z}, \ref{as:TSG_unbiased_estimators}, \ref{as:bound_on_hess_inv_true_MLP}, and~\ref{as:TSG_bounded_var}, there exist positive constants~$\omega$ and~$\tau$  such that
		\begin{alignat}{2}
			\|\nabla f (x^i,y^{i+1},z^{i+1}) - \bar g_{f_1}^{i} \| \;&\leq\; \omega\theta_i \quad\quad \text{and} \quad\quad \mathbb{E}[ \| \tilde g_{f_1}^{i} - \bar g_{f_1}^{i} \|^2 \vert \mathcal{F}_i ] \;&\leq\; \tau.\nonumber
		\end{alignat}
		%Specifically, $\omega$ and~$\tau$ are given by~\eqref{eq:0050} and~\eqref{eq:0060}, respectively, in Appendix~\ref{app:lem:TSG_new_inexact_variance_bound}.
	\end{lemma}
	\begin{proof}
		For this proof, we will omit the point~$(x^i,y^{i+1},z^{i+1})$ that the terms are evaluated at; we will simply use a $\xi^i$-superscript as short-hand to indicate any random terms. Similarly, we will use the notation $(\cdot)^{\xi^i}$ to denote that every term in the parenthesis is a random variable. To prove the upper-bound on the biasedness of~$\tilde g_{f_1}^{i}$, we can begin by referring to~\eqref{adjoint} and applying Assumption~\ref{as:TSG_unbiased_estimators}, yielding
		%\begin{alignat*}{2}
		%    \bar g_{f_1}^{i} = \mathbb{E}[ \tilde g_{f_1}^{i} \vert \mathcal{F}_{i} ] & = \mathbb{E}[\nabla f (x^i,y^{i+1},z^{i+1};\xi^i) \vert \mathcal{F}_{i} ]\\
		%    & =  \nabla_x f_1 - \nabla_{xz}^2 f_3 \mathbb{E}[ [ \nabla_{zz}^2 f_3^{\xi^i} ]^{-1} \vert \mathcal{F}_{i} ] \nabla_z f_1 - \mathbb{E}[ \nabla_{xy}^2\Bar{f}^{\xi^i} \vert \mathcal{F}_{i}] \mathbb{E}[ [ \nabla_{yy}^2 \Bar{f}^{\xi^i} ]^{-1} \vert \mathcal{F}_{i} ] \nabla_y f_1\\
		%    &\quad+ \mathbb{E}[ \nabla_{xy}^2\Bar{f}^{\xi^i} \vert \mathcal{F}_{i}] \mathbb{E}[[ \nabla_{yy}^2 \Bar{f}^{\xi^i} \vert \mathcal{F}_{i}]^{-1}] \nabla_{yz}^2 f_3 \mathbb{E}[ [ \nabla_{zz}^2 f_3^{\xi^i} ]^{-1} \vert \mathcal{F}_{i}] \nabla_z f_1.
		%\end{alignat*}
		\begin{alignat*}{2}
			\bar g_{f_1}^{i} = \mathbb{E}[ \tilde g_{f_1}^{i} ] & = \mathbb{E}[\nabla f (x^i,y^{i+1},z^{i+1};\xi^i) ]\\
			& =  \nabla_x f_1 - \nabla_{xz}^2 f_3 \mathbb{E}[ [ \nabla_{zz}^2 f_3^{\xi^i} ]^{-1} ] \nabla_z f_1 - \mathbb{E}[ \nabla_{xy}^2\Bar{f}^{\xi^i} ] \mathbb{E}[ [ \nabla_{yy}^2 \Bar{f}^{\xi^i} ]^{-1} ] \nabla_y f_1\\
			&\quad+ \mathbb{E}[ \nabla_{xy}^2\Bar{f}^{\xi^i} ] \mathbb{E}[[ \nabla_{yy}^2 \Bar{f}^{\xi^i} ]^{-1}] \nabla_{yz}^2 f_3 \mathbb{E}[ [ \nabla_{zz}^2 f_3^{\xi^i} ]^{-1} ] \nabla_z f_1.
		\end{alignat*}
		%For ease of notation throughout the rest of this proof, since all expectations that are utilized are conditioned on~$\mathcal{F}_{i}$, we will use the temporary short-hand notation of $\mathbb{E}[\cdot]:=\mathbb{E}[\cdot\vert\mathcal{F}_i]$ only for the duration of this proof.
		Now, to derive a bound on the biasedness $\|\nabla f (x^i,y^{i+1},z^{i+1}) - \bar g_{f_1}^{i} \|$, we can begin by utilizing the triangle inequality, the consistency of matrix norms, and Assumption~\ref{as:tri_lip_cont} and~\ref{as:TSG_bounded_var}, yielding
		\begin{alignat}{2}
			&\|\nabla f (x^i,y^{i+1},z^{i+1}) - \bar g_{f_1}^{i} \| \nonumber\\
			&\leq L_{\nabla f_3} L_{f_1} W_{zz} \theta_i + \underbrace{L_{f_1} \| \mathbb{E}[ \nabla_{xy}^2\Bar{f}^{\xi^i}] \mathbb{E}[ [ \nabla_{yy}^2 \Bar{f}^{\xi^i} ]^{-1} ] - \nabla_{xy}^2\Bar{f} [ \nabla_{yy}^2 \Bar{f} ]^{-1} \|}_{T^{(1)}_1}\nonumber\\
			&\quad+ \underbrace{L_{f_1} \| \nabla_{xy}^2\Bar{f} [ \nabla_{yy}^2 \Bar{f} ]^{-1} \nabla_{yz}^2 f_3 [ \nabla_{zz}^2 f_3 ]^{-1} - \mathbb{E}[ \nabla_{xy}^2\Bar{f}^{\xi^i}] \mathbb{E}[[ \nabla_{yy}^2 \Bar{f}^{\xi^i} ]^{-1}] \nabla_{yz}^2 f_3 \mathbb{E}[ [ \nabla_{zz}^2 f_3^{\xi^i} ]^{-1} ] \|}_{T^{(1)}_2}, \label{eq:TSG_g1_biasedness_eq_to_bound_1}
		\end{alignat}
		
		\textbf{$( \text{Analysis of } T^{(1)}_1 )$:} Now, to upper-bound $T^{(1)}_1$ in~\eqref{eq:TSG_g1_biasedness_eq_to_bound_1}, we begin by adding and subtracting the term $\nabla_{xy}^2\Bar{f} \mathbb{E}[ [ \nabla_{yy}^2 \Bar{f}^{\xi^i} ]^{-1} ]$, applying the triangle inequality, and utilizing the consistency of matrix norms to obtain
		\begin{alignat}{2}
			L_{f_1} T^{(1)}_1 &\leq L_{f_1} \| \mathbb{E}[ [ \nabla_{yy}^2 \Bar{f}^{\xi^i} ]^{-1} ] \|  \| \mathbb{E}[ \nabla_{xy}^2\Bar{f}^{\xi^i}] - \nabla_{xy}^2\Bar{f} \| + L_{f_1} \| \nabla_{xy}^2\Bar{f} \|  \| \mathbb{E}[ [ \nabla_{yy}^2 \Bar{f}^{\xi^i} ]^{-1} ] - [ \nabla_{yy}^2 \Bar{f} ]^{-1} \|\nonumber\\
			&\leq \left(b_{yy} U_{xy} + T_{xy} W_{yy}\right)L_{f_1}\theta_i,\label{eq:TSG_bound_on_T_1^{(1)}}
		\end{alignat}
		where the last inequality follows by applying Assumptions~\ref{as:tri_lip_cont}, \ref{as:bound_on_hess_inv_true_MLP}, and~\ref{as:TSG_bounded_var} along with Lemma~\ref{lem:TSG_derived_unbiasedness}, and where $\| \nabla_{xy}^2\Bar{f} \|\leq T_{xy}$, which follows from the following reasoning (applying the triangle inequality, the consistency of matrix norms, along with Assumptions~\ref{as:tri_lip_cont} and \ref{as:strong_conv_f3_z}, and equation~\eqref{eq:aux_lip_z_result}):
		\begin{alignat}{2}
			&\| \nabla_{xy}^2\Bar{f} \|\nonumber\\
			&\leq \|\nabla_{yx}^2 f_2\| + \|\nabla_{yz}^2 f_2 \nabla_x z^\top\| + \| \nabla_{yzx}^3 f_3 \nabla_{zz}^2 f_3^{-1} \nabla_z f_2 \| + \| \nabla_{yzz}^3 f_3 \nabla_x z^\top \nabla_{zz}^2 f_3^{-1} \nabla_z f_2 \| \nonumber\\
			&\quad+ \| \nabla_{yz}^2f_3 [ \nabla_{zz}^2 f_3 ]^{-1}  \nabla_{zzx}^3 f_3  [ \nabla_{zz}^2 f_3 ]^{-1} \nabla_z f_2 \| + \| \nabla_{yz}^2f_3 [ \nabla_{zz}^2 f_3 ]^{-1} \nabla_{zzz}^3 f_3 \nabla_x z^\top [ \nabla_{zz}^2 f_3 ]^{-1} \nabla_z f_2 \|\nonumber\\
			&\quad + \| \nabla_{yz}^2 f_3 [ \nabla_{zz}^2f_3 ]^{-1} \nabla_{zx}^2 f_2\| + \| \nabla_{yz}^2 f_3 [ \nabla_{zz}^2f_3 ]^{-1} \nabla_{zz}^2 f_2 \nabla_x z^\top \|\nonumber\\
			&\leq \left( L_{\nabla f_2} + \frac{L_{\nabla^2 f_3} L_{f_2}}{\mu_z} \right) \left(1 + \frac{2 L_{\nabla f_3}}{\mu_z} + \frac{L_{\nabla f_3}^2}{\mu_z^2} \right) \;:=\; T_{xy}.\label{eq:TSG_UB_on_norm_hess_xy_f_bar}
		\end{alignat}
		
		\textbf{$( \text{Analysis of } T^{(1)}_2 )$:} Now, to upper-bound $T^{(1)}_2$ in~\eqref{eq:TSG_g1_biasedness_eq_to_bound_1}, we begin by adding and subtracting the term $\nabla_{xy}^2\Bar{f} [ \nabla_{yy}^2 \Bar{f} ]^{-1} \nabla_{yz}^2 f_3 \mathbb{E}[ [ \nabla_{zz}^2 f_3^{\xi^i} ]^{-1} ]$, applying the triangle inequality, along with the consistency of matrix norms to obtain
		\begin{alignat}{2}
			L_{f_1} T^{(1)}_2 &\leq L_{f_1} \| \nabla_{xy}^2\Bar{f} \|  \| [ \nabla_{yy}^2 \Bar{f} ]^{-1} \|  \| \nabla_{yz}^2 f_3 \|  \| [ \nabla_{zz}^2 f_3 ]^{-1} - \mathbb{E}[ [ \nabla_{zz}^2 f_3^{\xi^i} ]^{-1} ] \| \nonumber\\
			&\quad+ L_{f_1} \| \nabla_{yz}^2 f_3 \|  \| \mathbb{E}[ [ \nabla_{zz}^2 f_3^{\xi^i} ]^{-1} ] \|  \| \nabla_{xy}^2\Bar{f} [ \nabla_{yy}^2 \Bar{f} ]^{-1} - \mathbb{E}[ \nabla_{xy}^2\Bar{f}^{\xi^i}] \mathbb{E}[[ \nabla_{yy}^2 \Bar{f}^{\xi^i} ]^{-1}] \| \nonumber\\
			&\leq L_{f_1}L_{\nabla f_3}\left( \frac{T_{xy} W_{zz}}{\mu_y} + b_{zz} \left(b_{yy} U_{xy} + T_{xy} W_{yy}\right)\right)\theta_i, \label{eq:TSG_bound_on_T_2^{(1)}}
		\end{alignat}
		where the last inequality follows from applying Assumptions~\ref{as:tri_lip_cont}, \ref{as:strong_conv_fbar_y}, \ref{as:bound_on_hess_inv_true_MLP}, and~\ref{as:TSG_bounded_var}, the bound $\| \nabla_{xy}^2\Bar{f} \|\leq T_{xy}$ we derived in~\eqref{eq:TSG_UB_on_norm_hess_xy_f_bar}, and the bound we derived on the term $T_1^{(1)}$ in~\eqref{eq:TSG_bound_on_T_1^{(1)}}.
		
		Finally, substituting the bounds~\eqref{eq:TSG_bound_on_T_1^{(1)}} and~\eqref{eq:TSG_bound_on_T_2^{(1)}} on the terms $T_1^{(1)}$ and $T_2^{(1)}$, respectively, back into~\eqref{eq:TSG_g1_biasedness_eq_to_bound_1}, we obtain the desired bound on the biasedness as
		\begin{alignat}{2}
			&\|\nabla f (x^i,y^{i+1},z^{i+1}) - \bar g_{f_1}^{i} \|\nonumber\\
			& \leq L_{f_1}\left(L_{\nabla f_3} W_{zz} + b_{yy} U_{xy} + T_{xy} W_{yy} + L_{\nabla f_3}\left( \frac{T_{xy} W_{zz}}{\mu_y} + b_{zz} \left(b_{yy} U_{xy} + T_{xy} W_{yy}\right)\right)\right)\theta_i \;:=\; \omega \theta_i. \label{eq:0050}
		\end{alignat}
		
		Now, to bound the variance of $\tilde g_{f_1}^{i}$, we can begin by using the fact that $\|a + b + c + d\|^2\leq 4\left(\|a\|^2 + \|b\|^2 + \|c\|^2 + \|d\|^2\right)$, with~$a$, $b$, $c$, and~$d$ real-valued vectors, to obtain (it bears mentioning that for ease of notation, we will use $(\cdot)^{\xi^i}$ to denote that all terms in the parenthesis are random variables)
		\begin{alignat}{2}
			&\mathbb{E}[ \| \tilde g_{f_1}^{i} - \bar g_{f_1}^{i} \|^2 ] = \mathbb{E}[ \| \tilde g_{f_1}^{i} - \mathbb{E}[ \tilde g_{f_1}^{i} \vert \mathcal{F}_{i} ] \|^2 ]\nonumber\\
			& \leq \underbrace{4\mathbb{E}[\| \nabla_x f_1^{\xi^i} - \mathbb{E}[\nabla_x f_1^{\xi^i}] \|^2]}_{T^{(2)}_1} + \underbrace{4\mathbb{E}[\| \mathbb{E}[\left(\nabla_{xz}^2 f_3 [ \nabla_{zz}^2 f_3 ]^{-1} \nabla_z f_1\right)^{\xi^i} ] - \left(\nabla_{xz}^2 f_3 [ \nabla_{zz}^2 f_3 ]^{-1} \nabla_z f_1\right)^{\xi^i}\|^2]}_{T^{(2)}_2}\nonumber\\
			&\quad + \underbrace{4\mathbb{E}[\| \mathbb{E}[ \left( \nabla_{xy}^2\Bar{f} [ \nabla_{yy}^2 \Bar{f} ]^{-1} \nabla_y f_1\right)^{\xi^i} ]  - \left( \nabla_{xy}^2\Bar{f} [ \nabla_{yy}^2 \Bar{f} ]^{-1} \nabla_y f_1\right)^{\xi^i} \|^2]}_{T^{(2)}_3}\nonumber\\
			&\quad+ \underbrace{4\mathbb{E}[\| \left(\nabla_{xy}^2\Bar{f} [ \nabla_{yy}^2 \Bar{f} ]^{-1} \nabla_{yz}^2 f_3 [ \nabla_{zz}^2 f_3 ]^{-1} \nabla_z f_1 \right)^{\xi^i} - \mathbb{E}[ \left(\nabla_{xy}^2\Bar{f} [ \nabla_{yy}^2 \Bar{f} ]^{-1} \nabla_{yz}^2 f_3 [ \nabla_{zz}^2 f_3 ]^{-1} \nabla_z f_1 \right)^{\xi^i} ] \|^2]}_{T^{(2)}_4}.\label{eq:TSG_g1_biasedness_eq_to_bound_2}
		\end{alignat}
		
		\textbf{$( \text{Analysis of } T^{(2)}_1 )$:} Notice that the term $T^{(2)}_1$ in~\eqref{eq:TSG_g1_biasedness_eq_to_bound_2} can be bounded by Assumption~\ref{as:TSG_unbiased_estimators}
		\begin{equation}\label{eq:TSG_bound_on_T_1^{(2)}}
			4\mathbb{E}[\| \nabla_x f_1^{\xi^i} - \mathbb{E}[\nabla_x f_1^{\xi^i}] \|^2] \;\leq\; 4\sigma^2_{\nabla f_1} := \tau_1.
		\end{equation}

		\textbf{$( \text{Analysis of } T^{(2)}_2 )$:} Now dealing with the contents of the term $T_2^{(2)}$, we can apply Assumption~\ref{as:TSG_unbiased_estimators} and re-factorize to obtain
		\begin{alignat*}{2}
			&\mathbb{E}[\left(\nabla_{xz}^2 f_3 [ \nabla_{zz}^2 f_3 ]^{-1} \nabla_z f_1\right)^{\xi^i} ] - \left(\nabla_{xz}^2 f_3 [ \nabla_{zz}^2 f_3 ]^{-1} \nabla_z f_1\right)^{\xi^i} \\
			&=  \nabla_{xz}^2 f_3 \mathbb{E}[[ \nabla_{zz}^2 f_3^{\xi^i} ]^{-1}] \nabla_z f_1 - \nabla_{xz}^2 f_3^{\xi^i} [ \nabla_{zz}^2 f_3^{\xi^i} ]^{-1} \nabla_z f_1^{\xi^i}\\
			& = ( \nabla_{xz}^2 f_3 - \nabla_{xz}^2 f_3^{\xi^i} ) \mathbb{E}[[ \nabla_{zz}^2 f_3^{\xi^i} ]^{-1}] \nabla_z f_1\\
			&\quad+ \nabla_{xz}^2 f_3^{\xi^i} ( \mathbb{E}[[ \nabla_{zz}^2 f_3^{\xi^i} ]^{-1}] - [ \nabla_{zz}^2 f_3^{\xi^i} ]^{-1} ) \nabla_z f_1\\
			&\quad+ \nabla_{xz}^2 f_3^{\xi^i} [ \nabla_{zz}^2 f_3^{\xi^i} ]^{-1} ( \nabla_z f_1 - \nabla_z f_1^{\xi^i} ).
		\end{alignat*}
		By using this, the fact that $\|a+b+c\|^2 \leq 3\left( \|a\|^2+\|b\|^2+\|c\|^2 \right)$, with~$a$, $b$, and~$c$ real-valued vectors, along with the consistency of matrix norms, and Assumptions~\ref{as:tri_lip_cont}, \ref{as:bound_on_hess_inv_true_MLP}, and~\ref{as:TSG_unbiased_estimators}, we can see that the term $T^{(2)}_2$ is upper-bounded by
		\begin{alignat}{2}
			& 4\mathbb{E}[\| \mathbb{E}[\left(\nabla_{xz}^2 f_3 [ \nabla_{zz}^2 f_3 ]^{-1} \nabla_z f_1\right)^{\xi^i} ] - \left(\nabla_{xz}^2 f_3 [ \nabla_{zz}^2 f_3 ]^{-1} \nabla_z f_1\right)^{\xi^i}\|^2]\nonumber\\
			%& \leq 12\mathbb{E}[\| \nabla_{xz}^2 f_3 - \nabla_{xz}^2 f_3^{\xi^i} \|^2] \cdot \mathbb{E}[\| \mathbb{E}[[ \nabla_{zz}^2 f_3^{\xi^i} ]^{-1}] \|^2] \cdot \mathbb{E}[\| \nabla_z f_1\|^2]\nonumber\\
			%&\quad+ 12\mathbb{E}[\|\nabla_{xz}^2 f_3^{\xi^i} \|^2] \cdot \mathbb{E}[\| \mathbb{E}[[ \nabla_{zz}^2 f_3^{\xi^i} ]^{-1}] - [ \nabla_{zz}^2 f_3^{\xi^i} ]^{-1} \|^2] \cdot \mathbb{E}[\| \nabla_z f_1\|^2]\nonumber\\
			%&\quad+ 12\mathbb{E}[\|\nabla_{xz}^2 f_3^{\xi^i} \|^2] \cdot \mathbb{E}[\| [ \nabla_{zz}^2 f_3^{\xi^i} ]^{-1} \|^2] \cdot \mathbb{E}[\| \nabla_z f_1 - \nabla_z f_1^{\xi^i} \|^2]\nonumber\\
			&\leq 12 \sigma_{\nabla^2 f_3}^2 b_{zz}^2 L_{f_1}^2 + 12\mathbb{E}[\|\nabla_{xz}^2 f_3^{\xi^i} \|^2] \mathbb{E}[\| \mathbb{E}[[ \nabla_{zz}^2 f_3^{\xi^i} ]^{-1}] - [ \nabla_{zz}^2 f_3^{\xi^i} ]^{-1} \|^2] L_{f_1}^2\nonumber\\
			&\quad+ 12\mathbb{E}[\|\nabla_{xz}^2 f_3^{\xi^i} \|^2] b_{zz}^2 \sigma_{\nabla f_1}^2.\label{eq:TSG_g1_biasedness_eq_to_bound_1_for_T_2}
		\end{alignat}
		Consider the term $\mathbb{E}[\|\nabla_{xz}^2 f_3^{\xi^i} \|^2]$ in~\eqref{eq:TSG_g1_biasedness_eq_to_bound_1_for_T_2}. Using the definition of variance (i.e., $\mathbb{E}[ X^2 ] = \text{Var}[ X ] + \mathbb{E}[ X ]^2$) along with Assumptions~\ref{as:TSG_unbiased_estimators} and~\ref{as:tri_lip_cont}, we have
		\begin{equation}\label{eq:TSG_g1_biasedness_aux_bound_1_for_T_2}
			\mathbb{E}[\| \nabla_{xz}^2 f_3^{\xi^i} \|^2] =  \mathbb{E}[\| \nabla_{xz}^2 f_3^{\xi^i} - \mathbb{E}[ \nabla_{xz}^2 f_3^{\xi^i} ]\|^2] + \mathbb{E}[\| \mathbb{E}[ \nabla_{xz}^2 f_3^{\xi^i} ] \|^2] \leq \sigma^2_{\nabla^2 f_3} + L_{\nabla f_3}^2.
		\end{equation}
		Consider the term $\mathbb{E}[\| [ \nabla_{zz}^2 f_3^{\xi^i} ]^{-1} - \mathbb{E}[[ \nabla_{zz}^2 f_3^{\xi^i} ]^{-1}] \|^2]$ in~\eqref{eq:TSG_g1_biasedness_eq_to_bound_1_for_T_2}. Using the fact that $\|a+b\|^2 \leq 2\left(\|a\|^2 + \|b\|^2\right)$, with~$a$ and~$b$ real-valued vectors, and applying Assumption~\ref{as:bound_on_hess_inv_true_MLP}, we have
		\begin{alignat}{2}
			\mathbb{E}[\| [ \nabla_{zz}^2 f_3^{\xi^i} ]^{-1} - \mathbb{E}[[ \nabla_{zz}^2 f_3^{\xi^i} ]^{-1}] \|^2] &\leq 2\mathbb{E}[\| [ \nabla_{zz}^2 f_3^{\xi^i} ]^{-1}\|^2] + 2\mathbb{E}[\|\mathbb{E}[[ \nabla_{zz}^2 f_3^{\xi^i} ]^{-1}] \|^2] \leq 4b_{zz}^2.\label{eq:TSG_g1_biasedness_aux_bound_2_for_T_2}
		\end{alignat}
		Now substituting the bounds~\eqref{eq:TSG_g1_biasedness_aux_bound_1_for_T_2} and~\eqref{eq:TSG_g1_biasedness_aux_bound_2_for_T_2} back into~\eqref{eq:TSG_g1_biasedness_eq_to_bound_1_for_T_2}, we obtain the bound on the term $T_2^{(2)}$ as
		\begin{alignat}{2}
			& 4\mathbb{E}[\| \mathbb{E}[\left(\nabla_{xz}^2 f_3 [ \nabla_{zz}^2 f_3 ]^{-1} \nabla_z f_1\right)^{\xi^i} ] - \left(\nabla_{xz}^2 f_3 [ \nabla_{zz}^2 f_3 ]^{-1} \nabla_z f_1\right)^{\xi^i}\|^2]\nonumber\\
			& \leq 12 \sigma_{\nabla^2 f_3}^2 b_{zz}^2 L_{f_1}^2 + 48( \sigma^2_{\nabla^2 f_3} + L_{\nabla f_3}^2 ) b_{zz}^2 L_{f_1}^2 + 12( \sigma^2_{\nabla^2 f_3} + L_{\nabla f_3}^2 ) b_{zz}^2 \sigma_{\nabla f_1}^2 \;:=\; \tau_2. \label{eq:TSG_bound_on_T_2^{(2)}}
		\end{alignat}
		
		\textbf{$( \text{Analysis of } T^{(2)}_3 )$:} Applying similar reasoning that was used in bounding the term $T^{(2)}_2$, along with utilizing Lemma~\ref{lem:TSG_derived_variance_bounds_of_f_bar} and Assumptions~\ref{as:tri_lip_cont}, \ref{as:bound_on_hess_inv_true_MLP}, and~\ref{as:TSG_unbiased_estimators}, we have
		\begin{alignat}{2}
			&4\mathbb{E}[\| \mathbb{E}[ \left( \nabla_{xy}^2\Bar{f} [ \nabla_{yy}^2 \Bar{f} ]^{-1} \nabla_y f_1\right)^{\xi^i} ]  - \left( \nabla_{xy}^2\Bar{f} [ \nabla_{yy}^2 \Bar{f} ]^{-1} \nabla_y f_1\right)^{\xi^i} \|^2]\nonumber\\
			%& \leq 12\mathbb{E}[\| \mathbb{E}[ \nabla_{xy}^2\Bar{f}^{\xi^i}] -  \nabla_{xy}^2\Bar{f}^{\xi^i}\|^2]\cdot\mathbb{E}[\| \mathbb{E}[[ \nabla_{yy}^2 \Bar{f}^{\xi^i} ]^{-1}] \|^2]\cdot\mathbb{E}[\| \nabla_y f_1\|^2]\nonumber\\
			%&\quad+ 12\mathbb{E}[\|\nabla_{xy}^2\Bar{f}^{\xi^i} \|^2]\cdot\mathbb{E}[\| \mathbb{E}[[ \nabla_{yy}^2 \Bar{f}^{\xi^i} ]^{-1}] - [ \nabla_{yy}^2 \Bar{f}^{\xi^i} ]^{-1} \|^2]\cdot\mathbb{E}[\| \nabla_y f_1\|^2]\nonumber\\
			%&\quad+ 12\mathbb{E}[\|\nabla_{xy}^2\Bar{f}^{\xi^i} \|^2]\cdot\mathbb{E}[\| [ \nabla_{yy}^2 \Bar{f}^{\xi^i} ]^{-1} \|^2]\cdot\mathbb{E}[\| \nabla_y f_1 - \nabla_y f_1^{\xi^i}\|^2]\nonumber\\
			&\leq 12V_{xy} b_{yy}^2 L_{f_1}^2 + 12\mathbb{E}[\|\nabla_{xy}^2\Bar{f}^{\xi^i} \|^2] \mathbb{E}[\| \mathbb{E}[[ \nabla_{yy}^2 \Bar{f}^{\xi^i} ]^{-1}] - [ \nabla_{yy}^2 \Bar{f}^{\xi^i} ]^{-1} \|^2] L_{f_1}^2\nonumber\\
			&\quad+ 12 \mathbb{E}[\|\nabla_{xy}^2\Bar{f}^{\xi^i} \|^2] b_{yy}^2 \sigma_{\nabla f_1}^2. \label{eq:TSG_g1_biasedness_eq_to_bound_1_for_T_3}
		\end{alignat}
		
		Consider the term $\mathbb{E}[\| [ \nabla_{yy}^2 \Bar{f}^{\xi^i} ]^{-1} - \mathbb{E}[[ \nabla_{yy}^2 \Bar{f}^{\xi^i} ]^{-1}] \|^2]$ in~\eqref{eq:TSG_g1_biasedness_eq_to_bound_1_for_T_3}. Applying the same reasoning that was used to derive~\eqref{eq:TSG_g1_biasedness_aux_bound_2_for_T_2}, we have
		\begin{equation}\label{eq:TSG_g1_biasedness_aux_bound_1_for_T_3}
			\mathbb{E}[\| [ \nabla_{yy}^2 \Bar{f}^{\xi^i} ]^{-1} - \mathbb{E}[[ \nabla_{yy}^2 \Bar{f}^{\xi^i} ]^{-1}] \|^2] \;\leq\; 4 b_{yy}^2.
		\end{equation}
		
		Consider the term $\mathbb{E}[\|\nabla_{xy}^2\Bar{f}^{\xi^i} \|^2]$ in~\eqref{eq:TSG_g1_biasedness_eq_to_bound_1_for_T_3}. Using the definition of variance (i.e., $\mathbb{E}[ X^2 ] = \text{Var}[ X ] + \mathbb{E}[ X ]^2$) and applying Lemma~\ref{lem:TSG_derived_variance_bounds_of_f_bar}, we have
		\begin{equation}\label{eq:TSG_g1_biasedness_eq_to_bound_2_for_T_3}
			\mathbb{E}[\|\nabla_{xy}^2\Bar{f}^{\xi^i} \|^2] \;\leq\; V_{xy} + \|\mathbb{E}[\nabla_{xy}^2\Bar{f}^{\xi^i}] \|^2.
		\end{equation}
		Now, consider the $\|\mathbb{E}[\nabla_{xy}^2\Bar{f}^{\xi^i}] \|^2$ term in~\eqref{eq:TSG_g1_biasedness_eq_to_bound_2_for_T_3}. Noticing that $\mathbb{E}[ \nabla_x z^{\xi^i\top}] = \mathbb{E}[ [ \nabla_{zz}^2 f_3^{\xi^i}  ]^{-1} ] \nabla_{zx}^2 f_3$ (from Assumption~\ref{as:TSG_unbiased_estimators}), we can apply the triangle inequality along with the consistency of matrix norms and Assumptions~\ref{as:tri_lip_cont}, \ref{as:TSG_unbiased_estimators}, and~\ref{as:bound_on_hess_inv_true_MLP} to obtain
		\begin{alignat}{2}
			%&\leq \| \nabla_{yzx}^3 f_3 \mathbb{E}[[\nabla_{zz}^2 f_3^{\xi^i}]^{-1}] \nabla_z f_2 \| + \| \nabla_{yzz}^3 f_3 \mathbb{E}[ [ \nabla_{zz}^2 f_3^{\xi^i}  ]^{-1} ] \nabla_{zx}^2 f_3 \mathbb{E}[[\nabla_{zz}^2 f_3^{\xi^i}]^{-1}] \nabla_z f_2 \|\nonumber\\
			%&\quad+ \| \nabla_{yz}^2f_3 \mathbb{E}[[ \nabla_{zz}^2 f_3^{\xi^i} ]^{-1}]  \nabla_{zzx}^3 f_3  \mathbb{E}[[ \nabla_{zz}^2 f_3^{\xi^i} ]^{-1}] \nabla_z f_2 \|\nonumber\\
			%&\quad+ \| \nabla_{yz}^2f_3 \mathbb{E}[[ \nabla_{zz}^2 f_3^{\xi^i} ]^{-1}] \nabla_{zzz}^3 f_3 \mathbb{E}[ [ \nabla_{zz}^2 f_3^{\xi^i}  ]^{-1} ] \nabla_{zx}^2 f_3 \mathbb{E}[[ \nabla_{zz}^2 f_3^{\xi^i} ]^{-1}] \nabla_z f_2 \|\nonumber\\
			%&\quad+ \| \nabla_{yz}^2 f_3 \mathbb{E}[[ \nabla_{zz}^2f_3^{\xi^i} ]^{-1}] \nabla_{zx}^2 f_2\|\nonumber\\
			%&\quad+ \| \nabla_{yz}^2 f_3 \mathbb{E}[[ \nabla_{zz}^2f_3^{\xi^i} ]^{-1}] \nabla_{zz}^2 f_2 \mathbb{E}[ [ \nabla_{zz}^2 f_3^{\xi^i}  ]^{-1} ] \nabla_{zx}^2 f_3 \|\nonumber\\
			\| \mathbb{E}[ \nabla_{xy}^2\Bar{f}^{\xi^i} ] \| &\leq L_{\nabla f_2} + b_{zz} ( L_{\nabla f_2}^2 + L_{f_2} L_{\nabla^2 f_3} ( 1 + 2 b_{zz} L_{\nabla f_3} + b_{zz}^2 L_{\nabla f_3}^2 ) + L_{\nabla f_3} L_{\nabla f_2} ( 1 + b_{zz} L_{\nabla f_3} ) ) \nonumber\\
			&:= \tilde T_{xy}.\label{eq:TSG_derived_bound_tilde_T_xy}
		\end{alignat}
		Finally, squaring both sides of this inequality, we have
		\begin{equation}\label{eq:TSG_g1_biasedness_aux_bound_2_for_T_3}
			\| \mathbb{E}[ \nabla_{xy}^2\Bar{f}^{\xi^i} ] \|^2 \;\leq\; \tilde T_{xy}^2.
		\end{equation}
		Thus, substituting~\eqref{eq:TSG_g1_biasedness_aux_bound_2_for_T_3} back into~\eqref{eq:TSG_g1_biasedness_eq_to_bound_2_for_T_3}, we have $\mathbb{E}[\|\nabla_{xy}^2\Bar{f}^{\xi^i} \|^2] \leq V_{xy} + \tilde T_{xy}^2$. Finally, substituting this and bound~\eqref{eq:TSG_g1_biasedness_aux_bound_1_for_T_3} back into~\eqref{eq:TSG_g1_biasedness_eq_to_bound_1_for_T_3}, we obtain the bound on the term $T_3^{(2)}$ as
		\begin{alignat}{2}
			&4\mathbb{E}[\| \mathbb{E}[ \left( \nabla_{xy}^2\Bar{f} [ \nabla_{yy}^2 \Bar{f} ]^{-1} \nabla_y f_1\right)^{\xi^i} ]  - \left( \nabla_{xy}^2\Bar{f} [ \nabla_{yy}^2 \Bar{f} ]^{-1} \nabla_y f_1\right)^{\xi^i} \|^2]\nonumber\\
			&\leq 12V_{xy} b_{yy}^2 L_{f_1}^2 + 48( V_{xy} + \tilde T_{xy}^2 )b_{yy}^2 L_{f_1}^2 + 12 ( V_{xy} + \tilde T_{xy}^2 ) b_{yy}^2 \sigma_{\nabla f_1}^2 \;:=\; \tau_3. \label{eq:TSG_bound_on_T_3^{(2)}}
		\end{alignat}

		\textbf{$( \text{Analysis of } T^{(2)}_4 )$:} Now, dealing with the contents of the norm in term $T^{(2)}_4$, we can apply Assumption~\ref{as:TSG_unbiased_estimators} and re-factorize to obtain
		\begin{alignat}{2}
			& \left(\nabla_{xy}^2\Bar{f} [ \nabla_{yy}^2 \Bar{f} ]^{-1} \nabla_{yz}^2 f_3 [ \nabla_{zz}^2 f_3 ]^{-1} \nabla_z f_1 \right)^{\xi^i} - \mathbb{E}[ \left(\nabla_{xy}^2\Bar{f} [ \nabla_{yy}^2 \Bar{f} ]^{-1} \nabla_{yz}^2 f_3 [ \nabla_{zz}^2 f_3 ]^{-1} \nabla_z f_1 \right)^{\xi^i} ]\nonumber\\
			%&= \nabla_{xy}^2\Bar{f}^{\xi^i} [ \nabla_{yy}^2 \Bar{f}^{\xi^i} ]^{-1} \nabla_{yz}^2 f_3^{\xi^i} [ \nabla_{zz}^2 f_3^{\xi^i} ]^{-1} \nabla_z f_1^{\xi^i} - \mathbb{E}[ \nabla_{xy}^2\Bar{f}^{\xi^i} ] \mathbb{E}[[ \nabla_{yy}^2 \Bar{f}^{\xi^i} ]^{-1}] \nabla_{yz}^2 f_3 \mathbb{E}[[ \nabla_{zz}^2 f_3^{\xi^i} ]^{-1}] \nabla_z f_1\nonumber\\
			& = ( \nabla_{xy}^2\Bar{f}^{\xi^i} - \mathbb{E}[ \nabla_{xy}^2\Bar{f}^{\xi^i} ] ) [ \nabla_{yy}^2 \Bar{f}^{\xi^i} ]^{-1} \nabla_{yz}^2 f_3^{\xi^i} [ \nabla_{zz}^2 f_3^{\xi^i} ]^{-1} \nabla_z f_1^{\xi^i}\nonumber\\
			&\quad+ \mathbb{E}[ \nabla_{xy}^2\Bar{f}^{\xi^i} ] \underbrace{( [ \nabla_{yy}^2 \Bar{f}^{\xi^i} ]^{-1} \nabla_{yz}^2 f_3^{\xi^i} [ \nabla_{zz}^2 f_3^{\xi^i} ]^{-1} - \mathbb{E}[[ \nabla_{yy}^2 \Bar{f}^{\xi^i} ]^{-1}] \nabla_{yz}^2 f_3 \mathbb{E}[[ \nabla_{zz}^2 f_3^{\xi^i} ]^{-1}] )}_{\hat T^{(2)}_4} \nabla_z f_1^{\xi^i}\nonumber\\
			&\quad+ \mathbb{E}[ \nabla_{xy}^2\Bar{f}^{\xi^i} ] \mathbb{E}[[ \nabla_{yy}^2 \Bar{f}^{\xi^i} ]^{-1}] \nabla_{yz}^2 f_3 \mathbb{E}[[ \nabla_{zz}^2 f_3^{\xi^i} ]^{-1}] ( \nabla_z f_1^{\xi^i} - \nabla_z f_1 ).\label{eq:TSG_g1_biasedness_eq_to_bound_3}
		\end{alignat}
		We can further re-factorize the term $\hat T^{(2)}_4$ in~\eqref{eq:TSG_g1_biasedness_eq_to_bound_3} to obtain
		\begin{alignat}{2}
			\hat T^{(2)}_4 & = ( [ \nabla_{yy}^2 \Bar{f}^{\xi^i} ]^{-1} - \mathbb{E}[[ \nabla_{yy}^2 \Bar{f}^{\xi^i} ]^{-1}] ) \nabla_{yz}^2 f_3^{\xi^i} [ \nabla_{zz}^2 f_3^{\xi^i} ]^{-1} + \mathbb{E}[[ \nabla_{yy}^2 \Bar{f}^{\xi^i} ]^{-1}] ( \nabla_{yz}^2 f_3^{\xi^i} - \nabla_{yz}^2 f_3 ) [ \nabla_{zz}^2 f_3^{\xi^i} ]^{-1}\nonumber\\
			&\quad+ \mathbb{E}[[ \nabla_{yy}^2 \Bar{f}^{\xi^i} ]^{-1}] \nabla_{yz}^2 f_3 ( [ \nabla_{zz}^2 f_3^{\xi^i} ]^{-1} - \mathbb{E}[[ \nabla_{zz}^2 f_3^{\xi^i} ]^{-1}] ).\label{eq:TSG_bound_on_hat_T_4^{(2)}}
		\end{alignat}
		Substituting~\eqref{eq:TSG_bound_on_hat_T_4^{(2)}} for $\hat T^{(2)}_4$ in~\eqref{eq:TSG_g1_biasedness_eq_to_bound_3}, we have
		\begin{alignat}{2}
			& (\nabla_{xy}^2\Bar{f} [ \nabla_{yy}^2 \Bar{f} ]^{-1} \nabla_{yz}^2 f_3 [ \nabla_{zz}^2 f_3 ]^{-1} \nabla_z f_1 )^{\xi^i} - \mathbb{E}[ (\nabla_{xy}^2\Bar{f} [ \nabla_{yy}^2 \Bar{f} ]^{-1} \nabla_{yz}^2 f_3 [ \nabla_{zz}^2 f_3 ]^{-1} \nabla_z f_1 )^{\xi^i} ]\nonumber\\
			& = ( \nabla_{xy}^2\Bar{f}^{\xi^i} - \mathbb{E}[ \nabla_{xy}^2\Bar{f}^{\xi^i} ] ) [ \nabla_{yy}^2 \Bar{f}^{\xi^i} ]^{-1} \nabla_{yz}^2 f_3^{\xi^i} [ \nabla_{zz}^2 f_3^{\xi^i} ]^{-1} \nabla_z f_1^{\xi^i}\nonumber\\
			&\quad+ \mathbb{E}[ \nabla_{xy}^2\Bar{f}^{\xi^i} ] ( [ \nabla_{yy}^2 \Bar{f}^{\xi^i} ]^{-1} - \mathbb{E}[[ \nabla_{yy}^2 \Bar{f}^{\xi^i} ]^{-1}] ) \nabla_{yz}^2 f_3^{\xi^i} [ \nabla_{zz}^2 f_3^{\xi^i} ]^{-1} \nabla_z f_1^{\xi^i} \nonumber\\
			&\quad+ \mathbb{E}[ \nabla_{xy}^2\Bar{f}^{\xi^i} ] \mathbb{E}[[ \nabla_{yy}^2 \Bar{f}^{\xi^i} ]^{-1}] ( \nabla_{yz}^2 f_3^{\xi^i} - \nabla_{yz}^2 f_3 ) [ \nabla_{zz}^2 f_3^{\xi^i} ]^{-1} \nabla_z f_1^{\xi^i}\nonumber\\
			&\quad+ \mathbb{E}[ \nabla_{xy}^2\Bar{f}^{\xi^i} ] \mathbb{E}[[ \nabla_{yy}^2 \Bar{f}^{\xi^i} ]^{-1}] \nabla_{yz}^2 f_3 ( [ \nabla_{zz}^2 f_3^{\xi^i} ]^{-1} - \mathbb{E}[[ \nabla_{zz}^2 f_3^{\xi^i} ]^{-1}] ) \nabla_z f_1^{\xi^i}\nonumber\\
			&\quad+ \mathbb{E}[ \nabla_{xy}^2\Bar{f}^{\xi^i} ] \mathbb{E}[[ \nabla_{yy}^2 \Bar{f}^{\xi^i} ]^{-1}] \nabla_{yz}^2 f_3 \mathbb{E}[[ \nabla_{zz}^2 f_3^{\xi^i} ]^{-1}] ( \nabla_z f_1^{\xi^i} - \nabla_z f_1 ). \label{eq:TSG_g1_biasedness_inside_of_norm_of_T_4^{(2)}}
		\end{alignat}
		Finally, substituting~\eqref{eq:TSG_g1_biasedness_inside_of_norm_of_T_4^{(2)}} back into the norm for $T_4^{(2)}$ in~\eqref{eq:TSG_g1_biasedness_eq_to_bound_2} and using the fact that $\|a+b+c+d+e\|^2 \leq 5\left(\|a\|^2+\|b\|^2+\|c\|^2+\|d\|^2+\|e\|^2\right)$, with~$a$, $b$, $c$, $d$, and~$e$ real-valued vectors, along with the consistency of matrix norms, and applying Assumptions~\ref{as:tri_lip_cont}, \ref{as:TSG_unbiased_estimators}, \ref{as:bound_on_hess_inv_true_MLP}, along with Lemma~\ref{lem:TSG_derived_variance_bounds_of_f_bar} and bounds~\eqref{eq:TSG_g1_biasedness_aux_bound_2_for_T_2}, \eqref{eq:TSG_g1_biasedness_aux_bound_1_for_T_3}, and~\eqref{eq:TSG_g1_biasedness_aux_bound_2_for_T_3}, to obtain
		\begin{alignat}{2}
			T_4^{(2)} &\leq 20 b_{yy}^2 b_{zz}^2 V_{xy} \mathbb{E}[\| \nabla_{yz}^2 f_3^{\xi^i} \|^2]  \mathbb{E}[\| \nabla_z f_1^{\xi^i}\|^2] + 80 b_{zz}^2 \tilde T_{xy}^2 b_{yy}^2 \mathbb{E}[\| \nabla_{yz}^2 f_3^{\xi^i} \|^2]  \mathbb{E}[\| \nabla_z f_1^{\xi^i} \|^2]\nonumber\\
			&\quad+ 20 b_{yy}^2 \sigma^2_{\nabla^2 f_3} b_{zz}^2 \tilde T_{xy}^2 \mathbb{E}[\| \nabla_z f_1^{\xi^i} \|^2] + 80 b_{yy}^2 L_{\nabla f_3}^2 b_{zz}^2 \tilde T_{xy}^2 \mathbb{E}[\| \nabla_z f_1^{\xi^i} \|^2] + 20 b_{yy}^2 L_{\nabla f_3}^2 b_{zz}^2 \sigma_{\nabla f_1}^2 \tilde T_{xy}^2. \label{eq:TSG_g1_biasedness_eq_to_bound_4}
		\end{alignat}
		
		Consider the terms $\mathbb{E}[\| \nabla_{yz}^2 f_3^{\xi^i} \|^2]$ and $\mathbb{E}[\| \nabla_z f_1^{\xi^i}\|^2]$ in~\eqref{eq:TSG_g1_biasedness_eq_to_bound_4}. Applying nearly identical reasoning that was used to derive~\eqref{eq:TSG_g1_biasedness_aux_bound_1_for_T_2}, we have
		\begin{alignat}{3}
			&\mathbb{E}[\| \nabla_{yz}^2 f_3^{\xi^i} \|^2] \;&&\leq\; \sigma^2_{\nabla^2 f_3} + L_{\nabla f_3}^2,\label{eq:TSG_g1_biasedness_aux_bound_1}\\
			&\mathbb{E}[\| \nabla_z f_1^{\xi^i}\|^2] \;&&\leq\;  \sigma^2_{\nabla f_1} + L_{f_1}^2.\label{eq:TSG_g1_biasedness_aux_bound_2}
		\end{alignat}
		
		Now substituting the bounds~\eqref{eq:TSG_g1_biasedness_aux_bound_1} and \eqref{eq:TSG_g1_biasedness_aux_bound_2} back into~\eqref{eq:TSG_g1_biasedness_eq_to_bound_4}, we obtain the bound on the term $T_4^{(2)}$ as
		\begin{alignat}{2}
			& 4\mathbb{E}[\| (\nabla_{xy}^2\Bar{f} [ \nabla_{yy}^2 \Bar{f} ]^{-1} \nabla_{yz}^2 f_3 [ \nabla_{zz}^2 f_3 ]^{-1} \nabla_z f_1 )^{\xi^i} - \mathbb{E}[ (\nabla_{xy}^2\Bar{f} [ \nabla_{yy}^2 \Bar{f} ]^{-1} \nabla_{yz}^2 f_3 [ \nabla_{zz}^2 f_3 ]^{-1} \nabla_z f_1 )^{\xi^i} ] \|^2]\nonumber\\
			&\leq 20 b_{yy}^2 b_{zz}^2 V_{xy} ( \sigma^2_{\nabla^2 f_3} + L_{\nabla f_3}^2 ) ( \sigma^2_{\nabla f_1} + L_{f_1}^2 ) + 80 b_{zz}^2 \tilde T_{xy}^2 b_{yy}^2 ( \sigma^2_{\nabla^2 f_3} + L_{\nabla f_3}^2 ) ( \sigma^2_{\nabla f_1} + L_{f_1}^2 )\nonumber\\
			&\quad+ 20 b_{yy}^2 \sigma^2_{\nabla^2 f_3} b_{zz}^2 \tilde T_{xy}^2 ( \sigma^2_{\nabla f_1} + L_{f_1}^2 ) + 80 b_{yy}^2 L_{\nabla f_3}^2 \tilde T_{xy}^2 b_{zz}^2 ( \sigma^2_{\nabla f_1} + L_{f_1}^2 ) + 20 b_{yy}^2 L_{\nabla f_3}^2 b_{zz}^2 \sigma_{\nabla f_1}^2 \tilde T_{xy}^2\nonumber\\
			& := \tau_4. \label{eq:TSG_bound_on_T_4^{(2)}}
		\end{alignat}
		
		The proof is completed by substituting the derived bounds for $T_1^{(2)}$, $T_2^{(2)}$, $T_3^{(2)}$, and $T_4^{(2)}$ (bounds~\eqref{eq:TSG_bound_on_T_1^{(2)}}, \eqref{eq:TSG_bound_on_T_2^{(2)}}, \eqref{eq:TSG_bound_on_T_3^{(2)}}, and~\eqref{eq:TSG_bound_on_T_4^{(2)}}, respectively) back into~\eqref{eq:TSG_g1_biasedness_eq_to_bound_2}, yielding the desired variance bound (including the omitted $\sigma$-algebra $\mathcal{F}_i$ that the expectation is conditioned on):
		\begin{equation}\label{eq:0060}
			\mathbb{E}[ \| \tilde g_{f_1}^{i} - \bar g_{f_1}^{i} \|^2 \vert \mathcal{F}_i ] \;\leq\; \tau, \quad \text{where} \quad \tau\;:=\;\tau_1+\tau_2+\tau_3+\tau_4.
		\end{equation}
	\end{proof}

	\begin{lemma}[Boundedness of UL direction]\label{lem:TSG_size_inexact_g}
		Under Assumptions~\ref{as:tri_lip_cont}, \ref{as:strong_conv_f3_z}, \ref{as:TSG_unbiased_estimators}, \ref{as:bound_on_hess_inv_true_MLP}, and~\ref{as:TSG_bounded_var}, there exists a positive constant~$\zeta$ such that
		\[
		\mathbb{E}[ \| \tilde g^i_{f_1} \|^2 \vert \mathcal{F}_i ] \;\leq\; \zeta.
		\]
		%Specifically, $\zeta$ is given by~\eqref{eq:0070} in Appendix~\ref{app:lem:TSG_size_inexact_g}.
	\end{lemma}
	\begin{proof}
		For this proof, we may omit the point $(x^i,y^{i+1},z^{i+1})$ that the terms are evaluated at; we will simply use a $\xi^i$-superscript as short-hand to indicate any random terms. From the definition of variance along with using Lemma~\ref{lem:TSG_new_inexact_variance_bound}, we have
		\begin{equation}\label{eq:TSG_size_inexact_g_aux_eq_to_bound_1}
			\mathbb{E}[ \| \tilde g^i_{f_1} \|^2 | \mathcal{F}_i ] = \| \bar g_{f_1}^{i} \|^2 + \mathbb{E}[ \| \tilde g_{f_1}^{i} - \bar g_{f_1}^{i} \|^2 | \mathcal{F}_i ] \leq \| \bar g_{f_1}^{i} \|^2 + \tau.
		\end{equation}
		Now, considering the $\| \bar g_{f_1}^{i} \|$ term in~\eqref{eq:TSG_size_inexact_g_aux_eq_to_bound_1}, we can apply the triangle inequality, Assumption~\ref{as:TSG_unbiased_estimators}, along with the consistency of matrix norms, to obtain
		\begin{alignat*}{2}
			\| \bar g_{f_1}^{i} \|& \leq \| \nabla_x f_1 \| + \| \nabla_{xz}^2 f_3 \mathbb{E}[ [ \nabla_{zz}^2 f_3^{\xi^i} ]^{-1} \vert \mathcal{F}_i ] \nabla_z f_1 \| + \| \mathbb{E}[ \nabla_{xy}^2\Bar{f}^{\xi^i} \vert \mathcal{F}_i] \mathbb{E}[ [ \nabla_{yy}^2 \Bar{f}^{\xi^i} ]^{-1} \vert \mathcal{F}_i ] \nabla_y f_1 \|\\
			&\quad+ \| \mathbb{E}[ \nabla_{xy}^2\Bar{f}^{\xi^i} \vert \mathcal{F}_i ] \mathbb{E}[[ \nabla_{yy}^2 \Bar{f}^{\xi^i} ]^{-1} \vert \mathcal{F}_i] \nabla_{yz}^2 f_3 \mathbb{E}[[ \nabla_{zz}^2 f_3^{\xi^i} ]^{-1} \vert \mathcal{F}_i] \nabla_z f_1 \|\\
			&\leq L_{f_1} + L_{\nabla f_3} L_{f_1}\| \mathbb{E}[ [ \nabla_{zz}^2 f_3^{\xi^i} ]^{-1} \vert \mathcal{F}_i ] \| + L_{f_1}\| \mathbb{E}[ \nabla_{xy}^2\Bar{f}^{\xi^i} \vert \mathcal{F}_i]\| \| \mathbb{E}[ [ \nabla_{yy}^2 \Bar{f}^{\xi^i} ]^{-1} \vert \mathcal{F}_i ] \|\\
			&\quad+ L_{\nabla f_3}L_{f_1}\| \mathbb{E}[ \nabla_{xy}^2\Bar{f}^{\xi^i} \vert \mathcal{F}_i]\|  \| \mathbb{E}[[ \nabla_{yy}^2 \Bar{f}^{\xi^i} ]^{-1} \vert \mathcal{F}_i]\| \| \mathbb{E}[[ \nabla_{zz}^2 f_3^{\xi^i} ]^{-1} \vert \mathcal{F}_i] \|\\
			&\leq L_{f_1} + L_{\nabla f_3} L_{f_1} b_{zz} + L_{f_1} \tilde T_{xy} b_{yy} + L_{\nabla f_3}L_{f_1} \tilde T_{xy} b_{yy} b_{zz},
		\end{alignat*}
		where the second inequality follows from applying Assumption~\ref{as:tri_lip_cont} and the last inequality follows from applying Assumption~\ref{as:bound_on_hess_inv_true_MLP} along with the derived bound~\eqref{eq:TSG_derived_bound_tilde_T_xy} from Lemma~\ref{lem:TSG_new_inexact_variance_bound}. Further, squaring both sides, we have the bound $\| \bar g_{f_1}^{i} \|^2 \;\leq\; (L_{f_1} + L_{\nabla f_3} L_{f_1} b_{zz} + L_{f_1} \tilde T_{xy} b_{yy} + L_{\nabla f_3}L_{f_1} \tilde T_{xy} b_{yy} b_{zz})^2$.
		Substituting this back into~\eqref{eq:TSG_size_inexact_g_aux_eq_to_bound_1}, we obtain the bound
		\begin{equation}\label{eq:0070}
			\mathbb{E}[ \| \tilde g^i_{f_1} \|^2 | \mathcal{F}_i ] \;\leq\; \zeta, \quad\text{where} \quad \zeta \;:=\; (L_{f_1} + L_{\nabla f_3} L_{f_1} b_{zz} + L_{f_1} \tilde T_{xy} b_{yy} + L_{\nabla f_3}L_{f_1} \tilde T_{xy} b_{yy} b_{zz})^2 + \tau.
		\end{equation}
	\end{proof}

	\begin{lemma}[Bounds on bias and variance of ML direction]\label{lem:TSG_var_bound_ML_dir}
		Recalling the definition of $\tilde g_{f_2}^{i,j}$ in equation~\eqref{eq:TSG_ML_stoch_dir}, define $\bar g_{f_2}^{i,j} = \mathbb{E}[ \tilde g_{f_2}^{i,j} \vert \mathcal{F}_{i,j} ]$. Then, under Assumptions~\ref{as:tri_lip_cont}, \ref{as:TSG_unbiased_estimators}, \ref{as:bound_on_hess_inv_true_MLP}, and~\ref{as:TSG_bounded_var}, there exist positive constants~$\hat\omega$ and~$\hat\tau$ such that 
		\begin{alignat}{2}
			\| \nabla_y \Bar{f}(x^i,y^{i,j},z^{i,j+1}) - \bar g_{f_2}^{i,j} \| \;&\leq\; \hat\omega\theta_i \quad\quad \text{and} \quad\quad \mathbb{E}[ \| \tilde g_{f_2}^{i,j} - \bar g_{f_2}^{i,j} \|^2 \vert \mathcal{F}_{i,j} ] \;&\leq\; \hat \tau.\nonumber
		\end{alignat}
		%Specifically, $\hat\omega$ and~$\hat{\tau}$ are given by~\eqref{eq:1050} and~\eqref{eq:0080}, respectively, in Appendix~\ref{app:lem:TSG_var_bound_ML_dir}.
	\end{lemma}
	\begin{proof}
		% For the remainder of this proof, we will omit the point that each of the gradient and Hessians are evaluated at for ease of notation. However, it should be noted that all gradient and Hessian terms are evaluated at the point $(x^i,y^{i,j},z^{i,j+1})$.
		For this proof, we may omit the point $(x^i,y^{i,j},z^{i,j+1})$ that the terms are evaluated at; we will simply use a $\xi^{i,j}$-superscript as short-hand to indicate any random terms.
		From Assumption~\ref{as:TSG_unbiased_estimators}, we have
		%\begin{alignat*}{2}
		%    \bar g_{f_2}^{i,j} & = \mathbb{E}[ \tilde g_{f_2}^{i,j} \vert \mathcal{F}_{i,j} ]\\
		%    & = \mathbb{E}[ \nabla_y f_2^{\xi^{i,j}} - \nabla_{yz}^2 f_3^{\xi^{i,j}} [ \nabla_{zz}^2 f_3^{\xi^{i,j}} ]^{-1}  \nabla_z f_2^{\xi^{i,j}} \vert \mathcal{F}_{i,j} ]\\
		%    & = \nabla_y f_2 - \nabla_{yz}^2 f_3 \mathbb{E}[ [ \nabla_{zz}^2 f_3^{\xi^{i,j}} ]^{-1} \vert \mathcal{F}_{i,j} ]  \nabla_z f_2.
		%\end{alignat*}
		%Now, we have
		\begin{alignat*}{2}
			\| \nabla_y \Bar{f} - \bar g_{f_2}^{i,j} \| & = \| \nabla_y f_2 - \nabla_{yz}^2 f_3 [ \nabla_{zz}^2 f_3 ]^{-1}  \nabla_z f_2 - ( \nabla_y f_2 - \nabla_{yz}^2 f_3 \mathbb{E}[ [ \nabla_{zz}^2 f_3^{\xi^{i,j}} ]^{-1} \vert \mathcal{F}_{i,j} ]  \nabla_z f_2 ) \|\\
			& = \| \nabla_{yz}^2 f_3 ( \mathbb{E}[ [ \nabla_{zz}^2 f_3^{\xi^{i,j}} ]^{-1} \vert \mathcal{F}_{i,j} ] - [ \nabla_{zz}^2 f_3 ]^{-1} )  \nabla_z f_2 \|\\
			& \leq L_{\nabla f_3}L_{f_2}\| \mathbb{E}[ [ \nabla_{zz}^2 f_3^{\xi^{i,j}} ]^{-1} \vert \mathcal{F}_{i,j} ] - [ \nabla_{zz}^2 f_3 ]^{-1} \|,
		\end{alignat*}
		where the inequality follows from Assumption~\ref{as:tri_lip_cont} along with the consistency of matrix norms. Utilizing Assumption~\ref{as:TSG_bounded_var}, we obtain the desired first result of
		\begin{equation}\label{eq:1050}
			\| \bar g_{f_2}^{i,j} - \nabla_y \Bar{f} \| \leq \hat\omega \theta_i, \quad\text{where}\quad \hat\omega \; := \; L_{\nabla f_3}L_{f_2}W_{zz}.
		\end{equation}
		
		Now, to estimate the variance of $\tilde g_{f_2}^{i,j}$, we can apply Assumption~\ref{as:TSG_unbiased_estimators} and the fact that $\|a + b\|^2 \leq 2\|a\|^2 + 2\|b\|^2$, with~$a$ and~$b$ real-valued vectors, yielding
		\begin{alignat}{2}
			& \mathbb{E}[ \| \tilde g_{f_2}^{i,j} - \bar g_{f_2}^{i,j} \|^2 \vert \mathcal{F}_{i,j} ] = \mathbb{E}[ \| \tilde g_{f_2}^{i,j} - \mathbb{E}[ \tilde g_{f_2}^{i,j} \vert \mathcal{F}_{i,j} ] \|^2 \vert \mathcal{F}_{i,j} ]\nonumber\\
			%& = \mathbb{E}[ \| \nabla_y f_2^{\xi^{i,j}} - \nabla_{yz}^2 f_3^{\xi^{i,j}} [ \nabla_{zz}^2 f_3^{\xi^{i,j}} ]^{-1}  \nabla_z f_2^{\xi^{i,j}} - \left( \nabla_y f_2 - \nabla_{yz}^2 f_3 \mathbb{E}[ [ \nabla_{zz}^2 f_3^{\xi^{i,j}} ]^{-1} \vert \mathcal{F}_{i,j} ]  \nabla_z f_2 \right) \|^2 \vert \mathcal{F}_{i,j} ]\\
			& = \mathbb{E}[ \| \nabla_y f_2^{\xi^{i,j}} - \nabla_y f_2 + \nabla_{yz}^2 f_3 \mathbb{E}[ [ \nabla_{zz}^2 f_3^{\xi^{i,j}} ]^{-1} \vert \mathcal{F}_{i,j} ]  \nabla_z f_2 - \nabla_{yz}^2 f_3^{\xi^{i,j}} [ \nabla_{zz}^2 f_3^{\xi^{i,j}} ]^{-1}  \nabla_z f_2^{\xi^{i,j}} \|^2 \vert \mathcal{F}_{i,j} ]\nonumber\\
			&\leq 2\sigma^2_{\nabla f_2} + 2\mathbb{E}[ \| \nabla_{yz}^2 f_3 \mathbb{E}[ [ \nabla_{zz}^2 f_3^{\xi^{i,j}} ]^{-1} \vert \mathcal{F}_{i,j} ]  \nabla_z f_2 - \nabla_{yz}^2 f_3^{\xi^{i,j}} [ \nabla_{zz}^2 f_3^{\xi^{i,j}} ]^{-1}  \nabla_z f_2^{\xi^{i,j}} \|^2 \vert \mathcal{F}_{i,j} ],\label{eq:TSG_lem_bias_f2_intermediate_1}
		\end{alignat}
		%where the second equality follows from Assumption~\ref{as:TSG_unbiased_estimators}. Now, using the fact that $\|a + b\|^2 \leq 2\|a\|^2 + 2\|b\|^2$, with~$a$ and~$b$ real-valued vectors, we have
		%\begin{alignat}{2}
		%    &\mathbb{E}[ \| \tilde g_{f_2}^{i,j} - \bar g_{f_2}^{i,j} \|^2 \vert \mathcal{F}_{i,j} ]\nonumber\\
		%    &\leq 2\mathbb{E}[ \| \nabla_y f_2^{\xi^{i,j}} - \nabla_y f_2  \|^2 \vert \mathcal{F}_{i,j} ]\nonumber\\
		%    &\quad+ 2\mathbb{E}[ \| \nabla_{yz}^2 f_3 \mathbb{E}[ [ \nabla_{zz}^2 f_3^{\xi^{i,j}} ]^{-1} \vert \mathcal{F}_{i,j} ]  \nabla_z f_2 - \nabla_{yz}^2 f_3^{\xi^{i,j}} [ \nabla_{zz}^2 f_3^{\xi^{i,j}} ]^{-1}  \nabla_z f_2^{\xi^{i,j}}  \|^2 \vert \mathcal{F}_{i,j} ]\nonumber\\
		%    &\leq 2\sigma^2_{\nabla f_2} + 2\mathbb{E}[ \| \nabla_{yz}^2 f_3 \mathbb{E}[ [ \nabla_{zz}^2 f_3^{\xi^{i,j}} ]^{-1} \vert \mathcal{F}_{i,j} ]  \nabla_z f_2 - \nabla_{yz}^2 f_3^{\xi^{i,j}} [ \nabla_{zz}^2 f_3^{\xi^{i,j}} ]^{-1}  \nabla_z f_2^{\xi^{i,j}} \|^2 \vert \mathcal{F}_{i,j} ],\label{eq:TSG_lem_bias_f2_intermediate_1}
		%\end{alignat}
		Now, dealing with the contents of the norm in the right-most term of~\eqref{eq:TSG_lem_bias_f2_intermediate_1}, we have
		\begin{alignat*}{2}
			&\nabla_{yz}^2 f_3 \mathbb{E}[ [ \nabla_{zz}^2 f_3^{\xi^{i,j}} ]^{-1} \vert \mathcal{F}_{i,j} ]  \nabla_z f_2 - \nabla_{yz}^2 f_3^{\xi^{i,j}} [ \nabla_{zz}^2 f_3^{\xi^{i,j}} ]^{-1}  \nabla_z f_2^{\xi^{i,j}}\\
			& = ( \nabla_{yz}^2 f_3 - \nabla_{yz}^2 f_3^{\xi^{i,j}} ) \mathbb{E}[ [ \nabla_{zz}^2 f_3^{\xi^{i,j}} ]^{-1} \vert \mathcal{F}_{i,j}] \nabla_z f_2 \;+\; \nabla_{yz}^2 f_3^{\xi^{i,j}} ( \mathbb{E}[ [ \nabla_{zz}^2 f_3^{\xi^{i,j}} ]^{-1} \vert \mathcal{F}_{i,j} ] - [ \nabla_{zz}^2 f_3^{\xi^{i,j}} ]^{-1} ) \nabla_z f_2\\
			&\quad+ \nabla_{yz}^2 f_3^{\xi^{i,j}} [ \nabla_{zz}^2 f_3^{\xi^{i,j}} ]^{-1} ( \nabla_z f_2 - \nabla_z f_2^{\xi^{i,j}} ).
		\end{alignat*}
		Using this, the fact that $\|a+b+c\|^2 \leq 3\left(\|a\|^2 + \|b\|^2 + \|c\|^2\right)$, with~$a$, $b$, and~$c$ real-valued vectors, along with the consistency of matrix norms, and applying Assumptions~\ref{as:tri_lip_cont}, \ref{as:bound_on_hess_inv_true_MLP}, \ref{as:TSG_unbiased_estimators}, and~\ref{as:TSG_bounded_var}, we can see that the norm term in~\eqref{eq:TSG_lem_bias_f2_intermediate_1} can be bounded as
		\begin{alignat*}{2}
			&\mathbb{E}[ \| \nabla_{yz}^2 f_3 \mathbb{E}[ [ \nabla_{zz}^2 f_3^{\xi^{i,j}} ]^{-1} \vert \mathcal{F}_{i,j} ]  \nabla_z f_2 - \nabla_{yz}^2 f_3^{\xi^{i,j}} [ \nabla_{zz}^2 f_3^{\xi^{i,j}} ]^{-1}  \nabla_z f_2^{\xi^{i,j}}  \|^2 \vert \mathcal{F}_{i,j} ]\\
			%& \leq 3 \mathbb{E}[ \| \nabla_{yz}^2 f_3 - \nabla_{yz}^2 f_3^{\xi^{i,j}} \|^2 \cdot \|\mathbb{E}[ [ \nabla_{zz}^2 f_3^{\xi^{i,j}} ]^{-1} \vert \mathcal{F}_{i,j}] \|^2 \cdot \| \nabla_z f_2 \|^2 \vert \mathcal{F}_{i,j} ]\\
			%&\quad+ 3 \mathbb{E}[ \|\nabla_{yz}^2 f_3^{\xi^{i,j}}\|^2 \cdot \| \mathbb{E}[ [ \nabla_{zz}^2 f_3^{\xi^{i,j}} ]^{-1} \vert \mathcal{F}_{i,j} ] - [ \nabla_{zz}^2 f_3^{\xi^{i,j}} ]^{-1} \|^2 \cdot \| \nabla_z f_2 \|^2 \vert \mathcal{F}_{i,j} ]\\
			%&\quad+ 3\mathbb{E}[ \|\nabla_{yz}^2 f_3^{\xi^{i,j}} \|^2 \cdot \| [ \nabla_{zz}^2 f_3^{\xi^{i,j}} ]^{-1} \|^2 \cdot \| \nabla_z f_2 - \nabla_z f_2^{\xi^{i,j}} \|^2 \vert \mathcal{F}_{i,j} ]\\
			&\leq 3 \sigma^2_{\nabla^2 f_3} b_{zz}^2 L_{f_2}^2 + 3 \mathbb{E}[ \|\nabla_{yz}^2 f_3^{\xi^{i,j}} \|^2 \vert \mathcal{F}_{i,j} ] W_{zz}^2 \theta_i^2 L_{f_2}^2 + 3 \mathbb{E}[ \|\nabla_{yz}^2 f_3^{\xi^{i,j}} \|^2 \vert \mathcal{F}_{i,j} ] \mathbb{E}[ \| [ \nabla_{zz}^2 f_3^{\xi^{i,j}} ]^{-1} \|^2 \vert \mathcal{F}_{i,j} ] \sigma^2_{\nabla f_2} \\
			& \leq 3 \sigma^2_{\nabla^2 f_3} b_{zz}^2 L_{f_2}^2 + 3 ( \sigma^2_{\nabla^2 f_3} + L^2_{\nabla f_3} ) W_{zz}^2 \theta_i^2 L_{f_2}^2 + 3 ( \sigma^2_{\nabla^2 f_3} + L^2_{\nabla f_3} ) ( W_{zz}^2\theta_i^2 + b_{zz}^2 ) \sigma^2_{\nabla f_2},
		\end{alignat*}
		where the last inequality follows from using
		$\mathbb{E}[ \|\nabla_{yz}^2 f_3^{\xi^{i,j}} \|^2 \vert \mathcal{F}_{i,j}] = \text{Var}[ \nabla_{yz}^2 f_3^{\xi^{i,j}} \vert \mathcal{F}_{i,j} ] \linebreak+ \|\mathbb{E}[\nabla_{yz}^2 f_3^{\xi^{i,j}} \vert \mathcal{F}_{i,j}] \|^2 \leq \sigma^2_{\nabla^2 f_3} + L^2_{\nabla f_3}$ (by the definition of variance along with Assumptions~\ref{as:tri_lip_cont} and~\ref{as:TSG_unbiased_estimators}) and by using $\mathbb{E}[ \| [ \nabla_{zz}^2 f_3^{\xi^{i,j}} ]^{-1} \|^2 \vert \mathcal{F}_{i,j} ] = \text{Var}[ [ \nabla_{zz}^2 f_3^{\xi^{i,j}} ]^{-1} \vert \mathcal{F}_{i,j} ] + \|\mathbb{E}[[ \nabla_{zz}^2 f_3^{\xi^{i,j}} ]^{-1} \vert \mathcal{F}_{i,j}] \|^2 \leq W_{zz}^2\theta_i^2 + b_{zz}^2$ (by the definition of variance along with Assumptions~\ref{as:bound_on_hess_inv_true_MLP} and~\ref{as:TSG_bounded_var}). Plugging this expression back into~\eqref{eq:TSG_lem_bias_f2_intermediate_1} and using the fact that $0\leq\theta_i^2\leq 1$, we obtain the desired result
		\begin{alignat}{2}
			&\mathbb{E}[ \| \tilde g_{f_2}^{i,j} - \bar g_{f_2}^{i,j} \|^2 \vert \mathcal{F}_{i,j} ]\nonumber\\
			%&\leq 2\sigma^2_{\nabla f_2} + 6 \sigma^2_{\nabla^2 f_3} b_{zz}^2 L_{f_2}^2 + 6 \left( \sigma^2_{\nabla^2 f_3} + L^2_{\nabla f_3} \right) W_{zz}^2 \theta_i^2 L_{f_2}^2 + 6 \left( \sigma^2_{\nabla^2 f_3} + L^2_{\nabla f_3} \right)\left( W_{zz}^2\theta_i^2 + b_{zz}^2 \right) \sigma^2_{\nabla f_2}\nonumber\\
			& \leq 2\sigma^2_{\nabla f_2} + 6 \sigma^2_{\nabla^2 f_3} b_{zz}^2 L_{f_2}^2 + 6 ( \sigma^2_{\nabla^2 f_3} + L^2_{\nabla f_3} ) W_{zz}^2 L_{f_2}^2 + 6 ( \sigma^2_{\nabla^2 f_3} + L^2_{\nabla f_3} ) ( W_{zz}^2 + b_{zz}^2 ) \sigma^2_{\nabla f_2}\nonumber\\
			&:= \hat\tau.\label{eq:0080}
		\end{alignat}
	\end{proof}

	\begin{lemma}[Boundedness of ML direction]\label{lem:TSG_size_inexact_g_ML}
		Under Assumptions~\ref{as:tri_lip_cont}, \ref{as:strong_conv_f3_z}, \ref{as:TSG_unbiased_estimators}, \ref{as:bound_on_hess_inv_true_MLP}, and~\ref{as:TSG_bounded_var}, there exists the positive constant~$\Upsilon$ such that
		\begin{alignat*}{2}
			\mathbb{E}[ \| \tilde g^{i,j}_{f_2} \|^2 \vert \mathcal{F}_{i,j} ] \;\leq\; \Upsilon \quad\quad \text{and} \quad\quad \| \bar g^{i,j}_{f_2} \|^2 \;\leq\; \Upsilon.
		\end{alignat*}
		%Specifically, $\Upsilon$ is given by~\eqref{eq:00110} in Appendix~\ref{app:lem:TSG_size_inexact_g_ML}.
	\end{lemma}
	
	\begin{proof}
		From the definition of variance, we have $\mathbb{E}[ \| \tilde g^{i,j}_{f_2} \|^2 \vert \mathcal{F}_{i,j} ] = \| \bar g^{i,j}_{f_2} \|^2 + \mathbb{E}[ \| \tilde g^{i,j}_{f_2} - \bar g^{i,j}_{f_2} \|^2 \vert \mathcal{F}_{i,j} ]$. Now, adding and subtracting $\nabla_y \Bar{f}(x^i,y^{i,j},z^{i,j+1})$ to the first term, followed by utilizing the fact that $\|a+b\|^2 \leq2\|a\|^2+2\|b\|^2$, with~$a$ and~$b$ real-valued vectors, along with Lemma~\ref{lem:TSG_var_bound_ML_dir}, we have
		\begin{alignat}{2}
			\mathbb{E}[ \| \tilde g^{i,j}_{f_2} \|^2 \vert \mathcal{F}_{i,j} ] & \leq 2\| \bar g^{i,j}_{f_2} - \nabla_y \Bar{f}(x^i,y^{i,j},z^{i,j+1}) \|^2 + 2\|\nabla_y \Bar{f}(x^i,y^{i,j},z^{i,j+1}) \|^2 + \hat\tau. \label{eq:TSG_size_inexact_g_ML_aux_eq_1}
		\end{alignat}
		Referencing~\eqref{eq:jacobian_z_ret_y} and~\eqref{eq:exact_grad_MLP_y}, the $\|\nabla_y \Bar{f}(x^i,y^{i,j},z^{i,j+1}) \|$ term can be bounded by applying the triangle inequality, the consistency of matrix norms, and Assumptions~\ref{as:tri_lip_cont} and~\ref{as:strong_conv_f3_z}, yielding
		\begin{alignat*}{2}
			\|\nabla_y \Bar{f}(x^i,y^{i,j},z^{i,j+1}) \| %& \leq \|\nabla_y f_2(x^i,y^{i,j},z^{i,j+1})\| \\
			%& \quad+ \|\nabla_{yz}^2 f_3(x^i,y^{i,j},z^{i,j+1})\| \| [\nabla_{zz}^2 f_3(x^i,y^{i,j},z^{i,j+1})]^{-1}\| \|\nabla_z f_2(x^i,y^{i,j},z^{i,j+1})\|\\
			\leq L_{f_2} + \frac{L_{\nabla f_3}L_{f_2}}{\mu_z} \quad \Longrightarrow \quad \|\nabla_y \Bar{f}(x^i,y^{i,j},z^{i,j+1}) \|^2 \leq \left( L_{f_2} + \frac{L_{\nabla f_3}L_{f_2}}{\mu_z} \right)^2.
		\end{alignat*}
		%Further, since the left and right-hand side of the inequality are non-negative, it follows that
		%\[\|\nabla_y \Bar{f}(x^i,y^{i,j},z^{i,j+1}) \|^2 \leq \left( L_{f_2} + \frac{L_{\nabla f_3}L_{f_2}}{\mu_z} \right)^2.\]
		Substituting this back into~\eqref{eq:TSG_size_inexact_g_ML_aux_eq_1}, utilizing Lemma~\ref{lem:TSG_var_bound_ML_dir}, and letting $W:=L_{f_2} + \frac{L_{\nabla f_3}L_{f_2}}{\mu_z}$, yields
		\begin{equation}\label{eq:00100}
			\mathbb{E}[ \| \tilde g^{i,j}_{f_2} \|^2 \vert \mathcal{F}_{i,j} ] %& \leq 2\hat\omega^2\theta_i^2 + 2W^2 + \hat\tau\\
			= 2\hat\omega^2\theta_i^2 + \hat\phi, \quad \text{where} \quad \hat\phi:= 2W^2 + \hat\tau.
		\end{equation}
		Finally, using the fact that~$0<\theta_i^2\leq 1$, it follows that
		\begin{equation}\label{eq:00110}
			\mathbb{E}[ \| \tilde g^{i,j}_{f_2} \|^2 \vert \mathcal{F}_{i,j} ] \leq \Upsilon, \quad\text{where}\quad \Upsilon:= 2\hat\omega^2 + \hat\phi. 
		\end{equation}
		
		The second result follows from the definition of variance and applying bound~\eqref{eq:00110}, yielding
		\begin{alignat*}{2}
			\| \bar g^{i,j}_{f_2} \|^2 \;=\; \mathbb{E}[ \| \tilde g^{i,j}_{f_2} \|^2 \vert \mathcal{F}_{i,j} ] - \mathbb{E}[ \| \tilde g^{i,j}_{f_2} - \bar g^{i,j}_{f_2} \|^2 \vert \mathcal{F}_{i,j} ] \;\leq\; \mathbb{E}[ \| \tilde g^{i,j}_{f_2} \|^2 \vert \mathcal{F}_{i,j} ] \;\leq\; \Upsilon.
		\end{alignat*}
	\end{proof}

	\section{Lipschitz continuity properties}\label{app:lipschitz_properties}
	
	This appendix contains all of the statements of derived Lipschitz continuity properties of the functions, gradients, Hessians, and Jacobians involved in the trilevel adjoint gradient~\eqref{adjoint}. All of their corresponding proofs are provided in Appendix~B.5 of the PhD thesis~\cite{GKent_2025} .
	
	\begin{proposition}\label{rem:lip_cont_prop_1}
		Under Assumptions~\ref{as:tri_lip_cont}--\ref{as:strong_conv_f3_z}, there exist positive constants $L_{z}$, $L_{z_{xy}}$, and $L_{z_y}$, such that the following Lipschitz continuity properties hold:
		\begin{alignat}{2}
			\| z(x_1) - z(x_2) \| &\;\leq\; L_{z}\| x_1 - x_2 \|,\label{eq:lip_prop_1}\\
			\| z(x_1,y_1) - z(x_2,y_2) \| &\;\leq\; L_{z_{xy}} \| (x_1, y_1) - (x_2, y_2) \|,\label{eq:lip_prop_2}\\
			\| z(x,y_1) - z(x,y_2) \| &\;\leq\; L_{z_y}\| y_1 - y_2 \|.\label{eq:lip_prop_3}
		\end{alignat}
	\end{proposition}
	
	\begin{proposition}\label{rem:lip_cont_prop_2}
		Under Assumptions~\ref{as:tri_lip_cont}--\ref{as:strong_conv_fbar_y}, there exist positive constants $L_{y}$, $L_{\nabla z}$, $L_{\bar F}$, $L_{\bar F_y}$, $L_{\bar F_z}$, $L_{\nabla_{yx}^2\Bar{f}}$, $L_{\nabla_{yy}^2\Bar{f}}$, $L_F$, $L_{F_{yz}}$, and $L_{\nabla y}$, such that the following Lipschitz properties hold:
		\begin{alignat}{2}
			\| y(x_1) - y(x_2) \| &\;\leq\; L_{y}\| x_1 - x_2 \|,\label{eq:lip_prop_4}\\
			\| \nabla z(x_1) - \nabla z(x_2) \| &\;\leq\; L_{\nabla z}\| x_1 - x_2 \|,\label{eq:lip_prop_5}\\
			\| \nabla_y \bar f(x_1) - \nabla_y \bar f(x_2) \| &\;\leq\; L_{\bar F}\| x_1 - x_2 \|,\label{eq:lip_prop_6}\\
			\| \nabla_y \bar f(x,y_1) - \nabla_y \bar f(x,y_2) \| &\;\leq\; L_{\bar F_y}\| y_1 - y_2 \|,\label{eq:lip_prop_7}\\
			\| \nabla_y \bar f(x,y,z_1) - \nabla_y \bar f(x,y,z_2) \| &\;\leq\; L_{\bar F_z}\| z_1 - z_2 \|,\label{eq:lip_prop_8}\\
			\|  \nabla_{yx}^2\Bar{f}(x_1,y(x_1))  - \nabla_{yx}^2\Bar{f}(x_2,y(x_2)) \| &\;\leq\; L_{\nabla_{yx}^2\Bar{f}}\| x_1 - x_2 \|,\label{eq:lip_prop_9}\\
			\|  \nabla_{yy}^2\Bar{f}(x_1,y(x_1))  - \nabla_{yy}^2\Bar{f}(x_2,y(x_2)) \| &\;\leq\; L_{\nabla_{yy}^2\Bar{f}}\| x_1 - x_2 \|,\label{eq:lip_prop_10}\\
			\| \nabla f(x_1) - \nabla f(x_2) \| &\;\leq\; L_F\| x_1 - x_2 \|,\label{eq:lip_prop_11}\\
			\| \nabla f(x,y_1,z_1) - \nabla f(x, y_2, z_2) \| &\;\leq\; L_{F_{yz}}\| (y_1,z_1) - (y_2,z_2) \|,\label{eq:lip_prop_12}\\
			\| \nabla y(x_1) - \nabla y(x_2) \| &\;\leq\; L_{\nabla y}\| x_1 - x_2 \|.\label{eq:lip_prop_13}
		\end{alignat}
	\end{proposition}
	A useful intermediary result of Proposition~\ref{rem:lip_cont_prop_1} is the following:
	\begin{equation}\label{eq:aux_lip_z_result}
		\|\nabla_x z(x,y(x))\| \;\leq\; \frac{L_{\nabla f_3}}{\mu_z} \text{ and } \|\nabla_y z(x,y(x))\| \;\leq\; \frac{L_{\nabla f_3}}{\mu_z},
	\end{equation}
	where $\mu_z$ is the constant of the strong convexity of $f_3$ (Assumption~\ref{as:strong_conv_f3_z}).

	\section{Numerical experimental setup}\label{app:numerical_experimental_setup}
	
	\subsection{Computing the TSG adjoint gradient inexactly}\label{subsec:computing_TSG_adj_grad_inexactly}
	Let us rewrite the adjoint gradient~\eqref{adjoint} in~$x$ as follows: 
	\begin{equation}\label{eq:adjoint_blo}
		\nabla f \; = \; a - A B^{-1} b,
	\end{equation}
	where~$a =  \nabla_x f_1 - \nabla_{xz}^2 f_3 \nabla_{zz}^2 f_3^{-1} \nabla_z f_1$, $A = \nabla_{xy}^2\Bar{f}$, $B = \nabla_{yy}^2\Bar{f}$, and~$b = \nabla_y f_1  - \nabla_{yz}^2 f_3 \nabla_{zz}^2 f_3^{-1} \nabla_z f_1$. Note that this is the same structure arising in the adjoint gradient of a~BLO problem. Two approaches have been proposed in the~BLO literature to deal with~$B^{-1}$. One option is to compute the adjoint gradient by first solving the linear system given by the adjoint equation~$B \, \lambda = b$ for the adjoint variables~$\lambda$, and then calculating
	$a - A \, \lambda$. The second option is to truncate the Neumann series given by~$B^{-1} = \sum_{h=0}^{\infty} (I - B)^h$, which requires the assumption of~$\Vert B \Vert_2 < 1$ to guarantee the convergence of the series. Note that the same two approaches can be used to deal with~$\nabla_{zz}^2 f_3^{-1}$ in~$a$ and~$b$ in~\eqref{eq:adjoint_blo}, as well as in the expression for~$\nabla_y \bar{f}$, given in~\eqref{eq:nablay_barf} below. The expression for the adjoint gradient~$\nabla_y \bar{f}$ follows from~\eqref{eq:exact_grad_MLP_y} in Appendix~\ref{app:prop:adjoint_grad}, together with~\eqref{eq:jacobian_z_ret_y}:
	\begin{equation}\label{eq:nablay_barf}
		\nabla_y \Bar{f}(x,y) = \nabla_y f_2 - \nabla_{yz}^2 f_3 \nabla_{zz}^2 f_3^{-1} \nabla_z f_2,
	\end{equation}
	where all gradients and Hessians on the right-hand side are evaluated at~$(x,y,z(x,y))$.

	\subsection{TSG-N-FD}\label{subsec:TSG-N-FD}
	Our first proposed method, TSG-N-FD, solves the adjoint systems in~\eqref{adjoint} and~\eqref{eq:nablay_barf} by using an iterative method where each Hessian-vector product is approximated with an~FD scheme. In particular, let us rewrite~\eqref{adjoint} and~\eqref{eq:nablay_barf} by highlighting the adjoint systems as follows:
	\begin{equation}\label{adjoint_NFD}
		\begin{split}
			\nabla f \; = \; ( \nabla_x f_1 - \nabla_{xz}^2 f_3 \underbrace{\nabla_{zz}^2 f_3^{-1} \nabla_z f_1}_{\lambda_z} ) - \nabla_{xy}^2\Bar{f} \underbrace{\nabla_{yy}^2 \Bar{f}^{-1} ( \nabla_y f_1  - \nabla_{yz}^2 f_3 \underbrace{\nabla_{zz}^2 f_3^{-1} \nabla_z f_1}_{\lambda_z} )}_{\lambda_y},
		\end{split}
	\end{equation}
	\begin{equation}\label{eq:nablay_barf_NFD}
		\nabla_y \Bar{f} = \nabla_y f_2 - \nabla_{yz}^2 f_3 \underbrace{\nabla_{zz}^2 f_3^{-1} \nabla_z f_2}_{\bar{\lambda}_z}.
	\end{equation}
	Specifically, the adjoint systems in~\eqref{adjoint_NFD} are~$\nabla_{zz}^2 f_3 \lambda_z = \nabla_z f_1$ and~$\nabla_{yy}^2 \bar{f} \lambda_y = \nabla_y f_1  - \nabla_{yz}^2 f_3 \lambda_z$. The adjoint system in~\eqref{eq:nablay_barf_NFD} is~$\nabla_{zz}^2 f_3 \bar{\lambda}_z = \nabla_z f_2$.
	% Note that $\nabla_{zz}^2 f_3 \lambda_z = \nabla_z f_1$ appears twice in~\eqref{adjoint_NFD}. 
	% Due to Assumption~\ref{as:TSG_unbiased_estimators}, which implies that stochastic estimates of gradients, Hessians, and third-order derivative tensors are~i.i.d., resulting in two different systems in the stochastic case. 
	
	First, we focus on~\eqref{adjoint_NFD}. In~TSG-N-FD, the adjoint system~$\nabla_{zz}^2 f_3 \lambda_z = \nabla_z f_1$ is solved for the adjoint variables~$\lambda_z$ by using the linear~CG method, with~$\nabla_{zz}^2 f_3 \lambda_z$ being approximated as follows: 
	\begin{equation}\label{eq:TSGNFD_FD0}
		\nabla_{zz}^2 f_3(x^i,y^{i,j},z^{i,j,k}; \xi^{i,j,k}) \lambda_z \; \approx \; \frac{\nabla_{z} f_3(x^i,y^{i,j},z^{i,j,k}_+; \xi^{i,j,k}) - \nabla_{z} f_3(x^i,y^{i,j},z^{i,j,k}_-; \xi^{i,j,k})}{2 \varepsilon},
	\end{equation}
	where $z^{i,j,k}_\pm \;=\; z^{i,j,k} \pm \varepsilon \lambda_z, \mbox{ with } \varepsilon > 0$.
	Then, the adjoint equation~$\nabla_{yy}^2 \bar{f} \lambda_y = \nabla_y f_1  - \nabla_{yz}^2 f_3 \lambda_z$ is solved for the adjoint variables~$\lambda_y$ by using the linear~CG method again, with~$\nabla_{yz}^2 f_3 \lambda_z$ being approximated via an~FD scheme similar to~\eqref{eq:TSGNFD_FD0}, and~$\nabla_{yy}^2 \bar{f} \lambda_y$ being approximated as follows:
	\begin{equation}\label{eq:TSGNFD_FD1}
		\nabla_{yy}^2 \bar{f}(x^i,y^{i,j},z^{i,j+1}; \xi^{i,j}) \lambda_y \; \approx \; \frac{\nabla_{y} \bar{f}(x^i,y^{i,j}_+,z^{i,j+1}; \xi^{i,j}) - \nabla_{y} \bar{f}(x^i,y^{i,j}_-,z^{i,j+1}; \xi^{i,j})}{2 \varepsilon},
	\end{equation}
	where $y^{i,j}_\pm \;=\; y^{i,j} \pm \varepsilon \lambda_y, \mbox{ with } \varepsilon > 0$.
	Then, the adjoint gradient is calculated from
	\begin{equation}\label{eq:TSGNFD_FD2}
		\nabla f \; \approx \; ( \nabla_x f_1 - \nabla_{xz}^2 f_3 \lambda_z ) - \nabla_{xy}^2\Bar{f} \lambda_y,
	\end{equation}
	where~$\nabla_{xz}^2 f_3 \, \lambda_z$ and~$\nabla_{xy}^2 \bar{f} \, \lambda_y$ are approximated via~FD schemes similar to~\eqref{eq:TSGNFD_FD0} and~\eqref{eq:TSGNFD_FD1}, respectively.  
	
	Let us now focus on~\eqref{eq:nablay_barf_NFD}. The adjoint system~$\nabla_{zz}^2 f_3 \bar{\lambda}_z = \nabla_z f_2$ is solved for the adjoint variables~$\bar{\lambda}_z$ by using the linear~CG method, with~$\nabla_{zz}^2 f_3 \bar{\lambda}_z$ being approximated as in~\eqref{eq:TSGNFD_FD0}. Then, the adjoint gradient is calculated from 
	\begin{equation}\label{eq:TSGNFD_FD3}
		\nabla_y \bar{f} \; \approx \; \nabla_y f_2 - \nabla_{yz}^2 f_3 \bar{\lambda}_z,
	\end{equation}
	where~$\nabla_{yz}^2 f_3 \, \bar{\lambda}_z$ is approximated via an~FD scheme similar to~\eqref{eq:TSGNFD_FD0}.
	
	The schema of~BSG-N-FD is included in Algorithm~\ref{alg:TSG-N-FD}. The~``N'' in the algorithm name refers to the Newton-type system defined by the adjoint equation, while the~``FD'' refers to the finite-difference approximations we use. We set the~FD parameter value to~$\varepsilon = 0.1$.
	
	\begin{algorithm}[H]
		\caption{TSG-N-FD}\label{alg:TSG-N-FD}
		\begin{algorithmic}[1]
			% \medskip        
			\item[] TSG-N-FD is obtained from Algorithm~\ref{alg:TSG} with the following modifications: 
			\medskip
			\item[] \qquad In Step~1, replace Step~2 of Algorithm~\ref{alg:TSG_MLP} with the following:
			\item[] \qquad\qquad {\bf Step 2.} Compute an approximation $\tilde{g}_{f_2}^{i,j}$, using~\eqref{eq:TSGNFD_FD3}. 
			\medskip
			\item[] \qquad In Step~3, replace the content with the following:
			\item[] \qquad\qquad {\bf Step 3.} Compute an approximation $\tilde{g}_{f_1}^{i}$, using~\eqref{eq:TSGNFD_FD2}.
			% \par\bigskip\noindent
		\end{algorithmic}
	\end{algorithm}

	\subsection{TSG-AD}\label{subsec:TSG-AD}
	
	Our second proposed method, TSG-AD, is based on the truncated Neumann series approach.  
	We will illustrate such an approach by applying it to the two terms from the adjoint gradient~\eqref{adjoint} that require it, i.e.,~$\nabla_{xz}^2 f_3 \nabla_{zz}^2 f_3^{-1} \nabla_z f_1$ and~$\nabla_{xy}^2\Bar{f} \nabla_{yy}^2 \Bar{f}^{-1} b$, where~$b = \nabla_y f_1  - \nabla_{yz}^2 f_3 \nabla_{zz}^2 f_3^{-1} \nabla_z f_1$. A similar approach can be applied to handle the term~$\nabla_{yz}^2 f_3 \nabla_{zz}^2 f_3^{-1} \nabla_z f_2$ in~\eqref{eq:nablay_barf}. 
	
	Let us start with~$\nabla_{xz}^2 f_3 \nabla_{zz}^2 f_3^{-1} \nabla_z f_1$ from~\eqref{adjoint}. 
	Approximating~$\nabla_{zz}^2 f_3^{-1}$ using a Neumann series (i.e.,~$B^{-1} = \sum_{h=0}^{\infty} (I - B)^h$, where~$B$ plays the role of~$\nabla_{zz}^2 f_3$) requires~$\Vert \nabla_{zz}^2 f_3\Vert_2 < 1$, which is a strong assumption in practice. However, recall that~$f_3$ is thrice continuously differentiable and~$\nabla_z f_3$ is Lipschitz continuous in~$z$ with some constant~$C_0 > 0$ by Assumption~\ref{as:tri_lip_cont}, implying that~$\Vert \nabla_{zz}^2 f_3 \Vert < C_0$~\cite[Theorem~5.12]{ABeck_2017}. Therefore, following a common approach in the~BLO literature~\cite{KJi_JYang_YLiang_2020}, we apply the truncated Neumann series to approximate~$[(1/C_0)\nabla_{zz}^2 f_3]^{-1}$.
	% As a consequence of Assumption~\ref{as:tri_lip_cont}, $C_0$ is equal to~$L_{\nabla f_3}$, but we prefer to use~$C_0$ for generality. 
	% Although this requires the knowledge of~$C_0$, its value can be found by fine-tuning.
	
	Given an accuracy level~$Q > 0$, we can write the truncated Neumann series as~$B^{-1} \approx \sum_{h=0}^{Q} (I - B)^h = \sum_{h=0}^{Q} \prod_{\ell = Q - h + 1}^{Q} (I - B)$, where we define~$\prod_{\ell = Q + 1}^Q (\cdot) = I$ for simplicity.
	Therefore, we can approximate~$\nabla_{zz}^2 f_3^{-1} \nabla_z f_1$ as follows
	\begin{equation}\label{eq:autodiff0}
		\nabla_{zz}^2 f_3^{-1} \nabla_z f_1 \; \approx \; (1/C_0) \left(\sum_{h=0}^{Q} \prod_{\ell = Q - h + 1}^Q (I - (1/C_0) \nabla_{zz}^2 f_3(x^i,y^{i,j},z^{i,j,k}; \xi^{i,j,k}_{\ell})) \right) \nabla_z f_1,
	\end{equation}
	with~$\xi^{i,j,k}_{\ell}$ representing the~$\ell$-th sample (or batch of samples) from the sequence of random variables~$\{\xi^{i,j,k}\}$. The expression on the right-hand side of~\eqref{eq:autodiff0} can be efficiently computed using the~AD procedure detailed in Algorithm~\ref{alg:autodiff10}. Then, given~$v_z$ returned by Algorithm~\ref{alg:autodiff10}, we can compute the desired term as follows 
	\begin{equation}\label{eq:AD10}
		\nabla_{xz}^2 f_3 \nabla_{zz}^2 f_3^{-1} \nabla_z f_1 \; \approx \; \frac{d}{dx} (\nabla_z f_3(x^i,y^{i,j},z^{i,j,k}; \xi^{i,j,k})^{\top} \, v_z),
	\end{equation}
	where differentiation with respect to~$x$ is performed using~AD (note that~$\nabla_z f_3$ is a function of~$x$ and~$v_z$ is fixed).
	
	\begin{algorithm}[H]
		\caption{Automatic differentiation procedure to compute~$\nabla_{zz}^2 f_3^{-1} \nabla_z f_1$}\label{alg:autodiff10}
		\begin{algorithmic}[1]
			% \medskip
			\item[] {\bf Input:} $(x^i, y^{i,j}, z^{i,j,k})$.
			\medskip
			\item[] {\bf For $\ell = 1, 2, \ldots, Q$ \bf do}
			\item[] \quad\quad $G_{\ell}(z^{i,j,k}) \; = \; z^{i,j,k} - (1/C_0) \nabla_z f_3(x^i,y^{i,j},z^{i,j,k}; \xi^{i,j,k}_{\ell}).$
			\item[] {\bf End}
			\item[] Set~$r_{0} = \nabla_z f_1(x^i, y^{i,j}, z^{i,j,k}; \xi^{i,j,k})$.
			\item[] {\bf For $h = 0, 1, \ldots, Q-1$ \bf do}
			\item[] \quad\quad Calculate~$r_{h+1} \; = \; \frac{d}{dz}(G_{h+1}(z^{i,j,k})^{\top} r_h) = (I - (1/C_0) \nabla_{zz}^2 f_3 (x^i,y^{i,j},z^{i,j,k}; \xi^{i,j,k}_{h+1})) \, r_h$, where differentiation with respect to~$z$ is performed using~AD (note that~$G_{h+1}$ is a function of~$z$ and~$r_h$ is fixed).
			\item[] {\bf End} 
			\medskip
			\item[] {\bf Output:} $v_z = (1/C_0) \sum_{h=0}^Q r_h$.
			% \par\bigskip\noindent
		\end{algorithmic}
	\end{algorithm}  
	
	Let us now focus on~$\nabla_{xy}^2\Bar{f} \nabla_{yy}^2 \Bar{f}^{-1} b$ from~\eqref{adjoint}. Recall that~$f_2$ is twice continuously differentiable and~$\nabla_y \bar{f}$ is Lipschitz continuous in~$y$ with some constant~$C_1 > 0$ as a consequence of~\eqref{eq:lip_prop_7} in Proposition~\ref{rem:lip_cont_prop_2} of Appendix~\ref{app:lipschitz_properties} (such a proposition implies that~$C_1$ is equal to~$L_{\bar{F}_y}$, but we prefer to use~$C_1$ for generality). Similar to~\eqref{eq:autodiff0}, we apply the truncated Neumann series to~$[(1/C_1)\nabla_{yy}^2\bar{f}]^{-1}$, which allows us to approximate~$\nabla_{yy}^2 \bar{f}^{-1} b$ as follows:
	\begin{equation}\label{eq:autodiff50}
		\nabla_{yy}^2 \bar{f}^{-1} b \; \approx \; (1/C_1) \left(\sum_{h=0}^{Q} \prod_{\ell = Q - h + 1}^Q (I - (1/C_1) \nabla_{yy}^2 \bar{f}(x^i,y^{i,j},z^{i,j+1}; \xi^{i,j}_{\ell}))\right) \, b,
	\end{equation}
	where~$\xi^{i,j}_{\ell}$ represents the $\ell$-th sample (or batch of samples) from the sequence of random variables~$\{\xi^{i,j}\}$. The expression on the right-hand side of~\eqref{eq:autodiff0} can be efficiently computed using the~AD procedure detailed in Algorithm~\ref{alg:autodiff11}. Then, given~$v_y$ returned by Algorithm~\ref{alg:autodiff11}, we can compute the desired term as follows 
	\begin{equation}\label{eq:AD11}
		\nabla_{xy}^2 \bar{f} \nabla_{yy}^2 \bar{f}^{-1} b \; \approx \; \frac{d}{dx} (\nabla_y \bar{f}(x^i,y^{i,j},z^{i,j+1}; \xi^{i,j})^{\top} \, v_y),
	\end{equation}
	where differentiation with respect to~$x$ is performed using~AD (note that~$\nabla_y \bar{f}$ is a function of~$x$ and~$v_y$ is fixed).
	
	\begin{algorithm}[H]
		\caption{Automatic differentiation procedure to compute~$\nabla_{yy}^2 \bar{f}^{-1} b$}\label{alg:autodiff11}
		\begin{algorithmic}[1]
			% \medskip
			\item[] {\bf Input:} $(x^i, y^{i,j}, z^{i,j+1})$.
			\medskip
			\item[] {\bf For $\ell = 1, 2, \ldots, Q$ \bf do}
			\item[] \quad\quad $G_{\ell}(y^{i,j}) \; = \; y^{i,j} - (1/C_1) \nabla_y \bar{f}(x^i,y^{i,j},z^{i,j+1}; \xi^{i,j}_{\ell}).$
			\item[] {\bf End}
			\item[] Set~$r_{0} = b$.
			\item[] {\bf For $h = 0, 1, \ldots, Q-1$ \bf do}
			\item[] \quad\quad Calculate~$r_{h+1} \; = \; \frac{d}{dy}(G_{h+1}(y^{i,j})^{\top} r_h) = (I - (1/C_1) \nabla_{yy}^2 \bar{f} (x^i,y^{i,j},z^{i,j+1}; \xi^{i,j}_{h+1})) \, r_h$, where differentiation with respect to~$y$ is performed using~AD (note that~$G_{h+1}$ is a function of~$y$ and~$r_h$ is fixed).
			\item[] {\bf End} 
			\medskip
			\item[] {\bf Output:} $v_y = (1/C_1) \sum_{h=0}^Q r_h$.
			% \par\bigskip\noindent
		\end{algorithmic}
	\end{algorithm}
	
	The schema of~TSG-AD is included in Algorithm~\ref{alg:TSG-AD}. 
	
	\begin{algorithm}[H]
		\caption{TSG-AD}\label{alg:TSG-AD}
		\begin{algorithmic}[1]
			% \medskip                
			\item[] TSG-AD is obtained from Algorithm~\ref{alg:TSG} with the following modifications: 
			\medskip
			\item[] \qquad In Step~1, replace Step~2 of Algorithm~\ref{alg:TSG_MLP} with the following:
			\item[] \qquad\qquad {\bf Step 2.} Compute an approximation $\tilde{g}_{f_2}^{i,j}$ by applying to~$\nabla_{yz}^2 f_3 \nabla_{zz}^2 f_3^{-1} \nabla_z f_2$ the same approach that was used to compute~$\nabla_{xz}^2 f_3 \nabla_{zz}^2 f_3^{-1} \nabla_z f_1$ in~\eqref{eq:AD10}. 
			\medskip
			\item[] \qquad In Step~3, replace the content with the following:
			\item[] \qquad\qquad {\bf Step 3.} Compute an approximation $\tilde{g}_{f_1}^{i}$, using~\eqref{eq:AD10} and~\eqref{eq:AD11}.
			% \par\bigskip\noindent
		\end{algorithmic}
	\end{algorithm}

	\subsection{Synthetic trilevel problems}\label{subsec:synthetic_trilevel_problems}
	
	Given~$h_x \in \mathbb{R}^n$, $h_y \in \mathbb{R}^m$, and~$h_z \in \mathbb{R}^t$, the~UL and~ML objective functions for both the quadratic and quartic synthetic trilevel problems considered in the experiments are respectively given by 
	\begin{alignat}{2}
		f_1(x,y,z) \; &= \;  h_x^\top x + h_y^\top y + h_z^\top z + 0.5 \, x^\top H_{xx} x + x^\top H_{xy} y + x^\top H_{xz} z, \label{prob:trilevel_synthetic_ul} \\
		f_2(x,y,z) \; &= \; 0.5 \, y^\top H_{yy} y - y^\top H_{yx} x - y^\top H_{yz} z, \label{prob:trilevel_synthetic_ml} 
	\end{alignat}
	where~$H_{xx} \in \mathbb{R}^{n \times n}$ and~$H_{yy} \in \mathbb{R}^{m \times m}$ are symmetric positive definite matrices, and~$H_{xy} \in \mathbb{R}^{n \times m}$, $H_{xz} \in \mathbb{R}^{n \times t}$, $H_{yx} = H_{xy}^\top$, and~$H_{yz} \in \mathbb{R}^{m \times t}$ are arbitrary matrices.
	The~LL objective functions of the two problems are respectively defined as follows
	\begin{alignat}{2}
		f_3(x,y,z) \; &= \;  0.5 \, z^\top H_{zz} z - z^\top H_{zx} x - z^\top H_{zy} y, \label{prob:trilevel_synthetic_ll_1} \\
		f_3(x,y,z) \; &= \; 0.5 \| z^\top H_{zz} z - z^\top H_{zx} x - z^\top H_{zy} y \|^2, \label{prob:trilevel_synthetic_ll_2}
	\end{alignat}
	where~$H_{zz} \in \mathbb{R}^{t \times t}$ is a symmetric positive definite matrix, and~$H_{zx} = H_{xz}^\top$ and~$H_{zy} = H_{yz}^\top$ are arbitrary matrices.
	
	In all the numerical experiments, we considered the same dimension at all levels (i.e.,~$n = m = t = 50$) for the quadratic problem, and varying dimensions (i.e.,~$n = m = 5$ and~$t = 1$) for the quartic problem. In~\eqref{prob:trilevel_synthetic_ul}, the components of the vectors~$h_x$, $h_y$, and~$h_z$ were randomly generated from a uniform distribution between~$0$ and~10 for the quadratic problem, and between~$0$ and~0.1 for the quartic problem. We set all matrices in~\eqref{prob:trilevel_synthetic_ul}--\eqref{prob:trilevel_synthetic_ll_2} equal to identity matrices, except for~$H_{yy}$ in~\eqref{prob:trilevel_synthetic_ml}, which was set to four times the identity matrix. 
	
	When using~\eqref{prob:trilevel_synthetic_ll_1}, our choices for the matrices in~\eqref{prob:trilevel_synthetic_ul}--\eqref{prob:trilevel_synthetic_ll_1} ensure that~$f_3$, $\bar{f}$, and~$f$ have unique solutions.\footnote{We have~$z(x,y) = H_{zz}^{-1}(H_{zx}x + H_{zy}y)$, $y(x) = (H_{yy}-2H_{yz}H_{zz}^{-1}H_{zy})^{-1}(H_{yx} + H_{yz}H_{zz}^{-1}H_{zx})$, $\nabla_y \bar{f}(x,y) = H_{yy}y - H_{yx}x - H_{yz}H_{zz}^{-1}(H_{zx}x + 2H_{zy}y)$, and~$\nabla_{yy}^2 \bar{f}(x,y) = H_{yy} - 2 H_{yz} H_{zz}^{-1} H_{zy}$. We omit the expressions of~$\nabla f(x)$ of~$\nabla^2 f(x)$ for brevity.} When using~\eqref{prob:trilevel_synthetic_ll_2}, 
	% and~$t = 1$, 
	the resulting~LL problem has two optimal solutions: $z(x,y) = 0$ and~$z(x,y) = H_{zx} x+ H_{zy} y$. Our choice for the initial points~$x^0$, $y^{0,0}$, and~$z^{0,0,0}$ ensures that the methods considered in the experiments converge to the~LL optimal solution~$z(x,y) = H_{zx} x+ H_{zy} y$. 
	% When~$t > 1$, determining an optimal solution for the~LL problem in closed-form is not possible. 
	% Note that our choices for the matrices in~\eqref{prob:trilevel_synthetic_ul}--\eqref{prob:trilevel_synthetic_ml} and~\eqref{prob:trilevel_synthetic_ll_2} ensure that the~third-order derivatives of~$f_3$ in~\eqref{prob:trilevel_synthetic_ll_2} are non-zero. 
	Specifically, the components of the initial points were randomly generated from a uniform distribution over the interval~[0, 20] when using~\eqref{prob:trilevel_synthetic_ll_1}, and over the intervals~[-0.4, 0], [-0.2, 0], and~[-0.6, 0] (for the~UL, ML, and~LL variables, respectively) when using~\eqref{prob:trilevel_synthetic_ll_2}.
	
	All algorithms (i.e., TSG-H, TSG-N-FD, and~TSG-AD) were compared using a decaying step size at each level. Specifically, we used~$\alpha_i = \bar{\alpha}/i$, $\beta_j = \bar{\beta}/j$, and~$\gamma_k = \bar{\gamma}/k$, where~$\bar{\alpha}$, $\bar{\beta}$, and~$\bar{\gamma}$ are positive scalars carefully chosen to ensure good performance for each algorithm (without conducting extensive, time-consuming grid searches at all levels, as our goal is not to compare our algorithms against others). The values of~$\bar{\alpha}$, $\bar{\beta}$, and~$\bar{\gamma}$ are provided in~Table~\ref{tab:stepsizes_synthetic}.
	
	\begin{table}[h!]
		\caption{Details of the stepsizes ($\alpha_i = \bar{\alpha}/i$, $\beta_j = \bar{\beta}/j$, $\gamma_k = \bar{\gamma}/k$) used across algorithms for the synthetic quadratic and quartic trilevel problems}
		\label{tab:stepsizes_synthetic}
		\centering
		\begin{tabular}{lllllll}
			\toprule
			Problem & Algorithm & Case & $\bar{\alpha}$ & $\bar{\beta}$ & $\bar{\gamma}$ \\
			\midrule
			\multirow{6}{*}{Quadratic} & TSG-H & Deterministic   & 0.3  & 0.2   & 0.1  \\
			& TSG-N-FD & Deterministic   & 0.01 & 0.1   & 0.05 \\
			& TSG-AD & Deterministic   & 0.01 & 0.1   & 0.1  \\
			& TSG-H & Stochastic      & 0.1  & 0.1   & 0.1  \\
			& TSG-N-FD & Stochastic     & 0.01 & 0.1   & 0.1  \\
			& TSG-AD & Stochastic     & 0.01 & 0.1   & 0.1  \\
			\midrule
			\multirow{6}{*}{Quartic} & TSG-H & Deterministic   & 0.3  & 0.2   & 0.1  \\
			& TSG-N-FD & Deterministic   & 0.3  & 0.2   & 0.0001 \\
			& TSG-AD & Deterministic   & 0.3  & 0.2   & 0.0001 \\
			& TSG-H & Stochastic      & 0.3  & 0.2   & 0.1  \\
			& TSG-N-FD & Stochastic     & 0.01 & 0.01  & 0.001 \\
			& TSG-AD & Stochastic    & 0.3  & 0.2   & 0.0001 \\
			\bottomrule
		\end{tabular}
	\end{table}
	
	% \tcb{In the deterministic case, for the synthetic quadratic trilevel problem, for~TSG-H, we used~($\bar{\alpha} = 0.3$, $\bar{\beta} = 0.2$, $\bar{\gamma} = 0.1$); for~TSG-N-FD, we used~($\bar{\alpha} = 0.01$, $\bar{\beta} = 0.1$, $\bar{\gamma} = 0.05$); for TSG-A, we used~($\bar{\alpha} = 0.01$, $\bar{\beta} = 0.1$, $\bar{\gamma} = 0.1$). In the stochastic case, for~TSG-H, we used~($\bar{\alpha} = 0.1$, $\bar{\beta} = 0.1$, $\bar{\gamma} = 0.1$); for~TSG-N-FD, we used~($\bar{\alpha} = 0.01$, $\bar{\beta} = 0.1$, $\bar{\gamma} = 0.1$); for TSG-A, we used~($\bar{\alpha} = 0.01$, $\bar{\beta} = 0.1$, $\bar{\gamma} = 0.1$).}
	
	% \tcb{In the deterministic case, for the synthetic quartic trilevel problem, for~TSG-H, we used~($\bar{\alpha} = 0.3$, $\bar{\beta} = 0.2$, $\bar{\gamma} = 0.1$); for~TSG-N-FD, we used~($\bar{\alpha} = 0.3$, $\bar{\beta} = 0.2$, $\bar{\gamma} = 0.0001$); for TSG-A, we used~($\bar{\alpha} = 0.3$, $\bar{\beta} = 0.2$, $\bar{\gamma} = 0.0001$). In the stochastic case, for~TSG-H, we used~($\bar{\alpha} = 0.3$, $\bar{\beta} = 0.2$, $\bar{\gamma} = 0.1$); for~TSG-N-FD, we used~($\bar{\alpha} = 0.01$, $\bar{\beta} = 0.01$, $\bar{\gamma} = 0.001$); for TSG-A, we used~($\bar{\alpha} = 0.3$, $\bar{\beta} = 0.2$, $\bar{\gamma} = 0.0001$).}

	\subsubsection{Additional figures and discussion for the synthetic trilevel problems }\label{subsubsec:synthetic_trilevel_problems_additional_figs}
	
	In the deterministic case, Figures~\ref{fig:quad_prob_det_breakdown} and~\ref{fig:quart_prob_det_breakdown} break down the behavior of~TSG-H, TSG-N-FD, and~TSG-AD at the~UL, ML, and~LL levels. Specifically, such figures plot the sequence of~$f(x^i)$ values (upper plot), $\bar{f}(x^i,y^{i,j})$ values (middle plot), and~$f_3(x^i,y^{i,j},z^{i,j,k})$ values (lower plot). They also include the values~$f(x_*)$ (only for the quadratic problem, where it can be computed analytically), with~$x_*$ denoting the optimal solution of the trilevel problem, as well as $\bar{f}(x^i, y(x^i))$ and~$f_3(x^i, y^{i,j}, z(x^i, y^{i,j}))$. The goal is for the sequences of~$f$, $\bar{f}$, and~$f_3$ values to converge to their respective dashed lines. In the middle- and lower-level plots, the horizontal axis represents cumulative~ML and~LL iterations, respectively. 
	
	As evident from Figure~\ref{fig:quad_prob_det_breakdown}, for the quadratic problem, the sequences of function values at the~UL and~ML problems converge when the function values at the~ML and~LL problems, respectively, also converge. As evident from Figure~\ref{fig:quart_prob_det_breakdown}, for the quartic problem, the sequences of function values at all levels converge after a few iterations.
	
	\begin{figure}[H]
		\centering
		\includegraphics[scale=0.25]{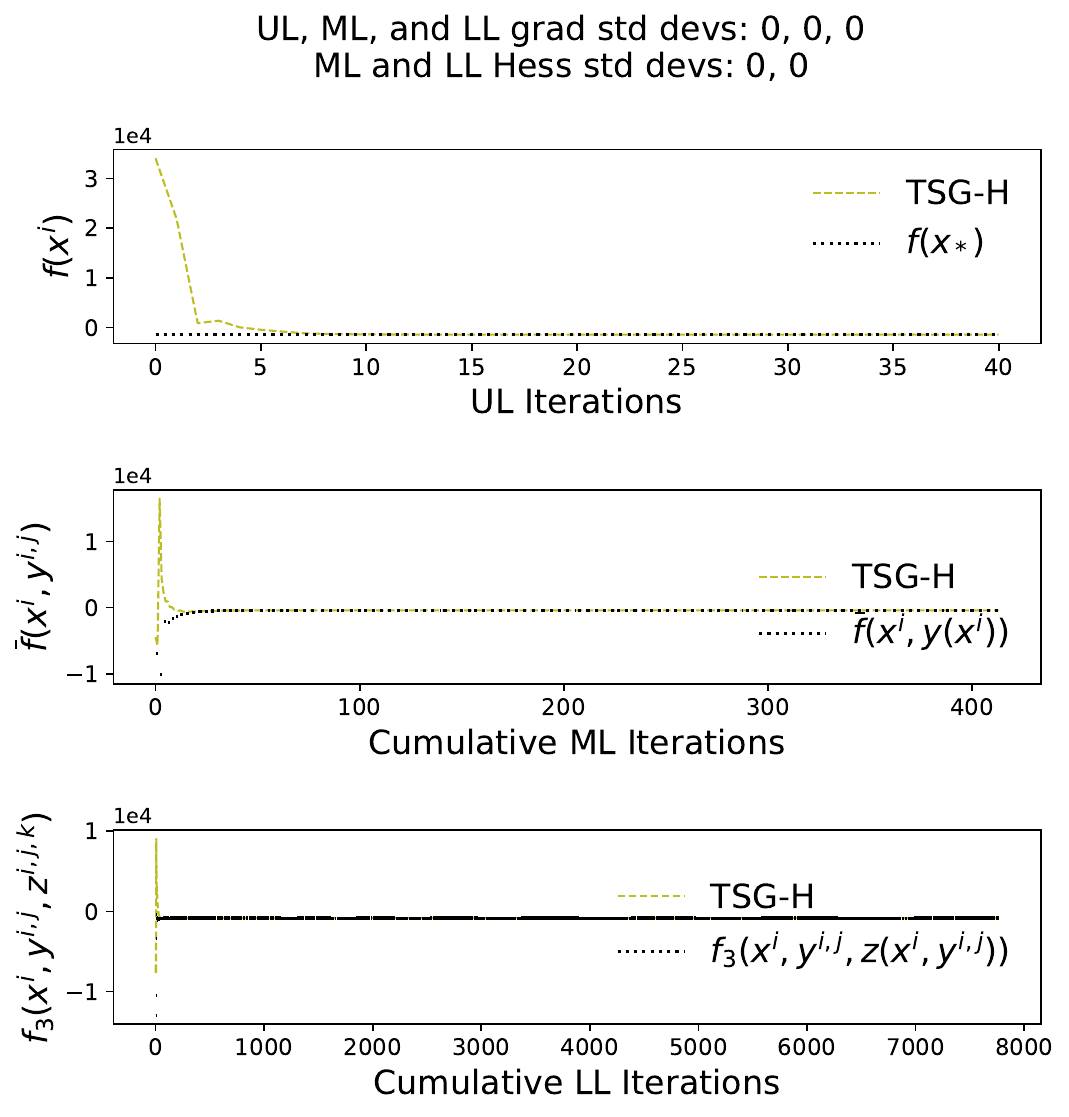}
		\includegraphics[scale=0.25]{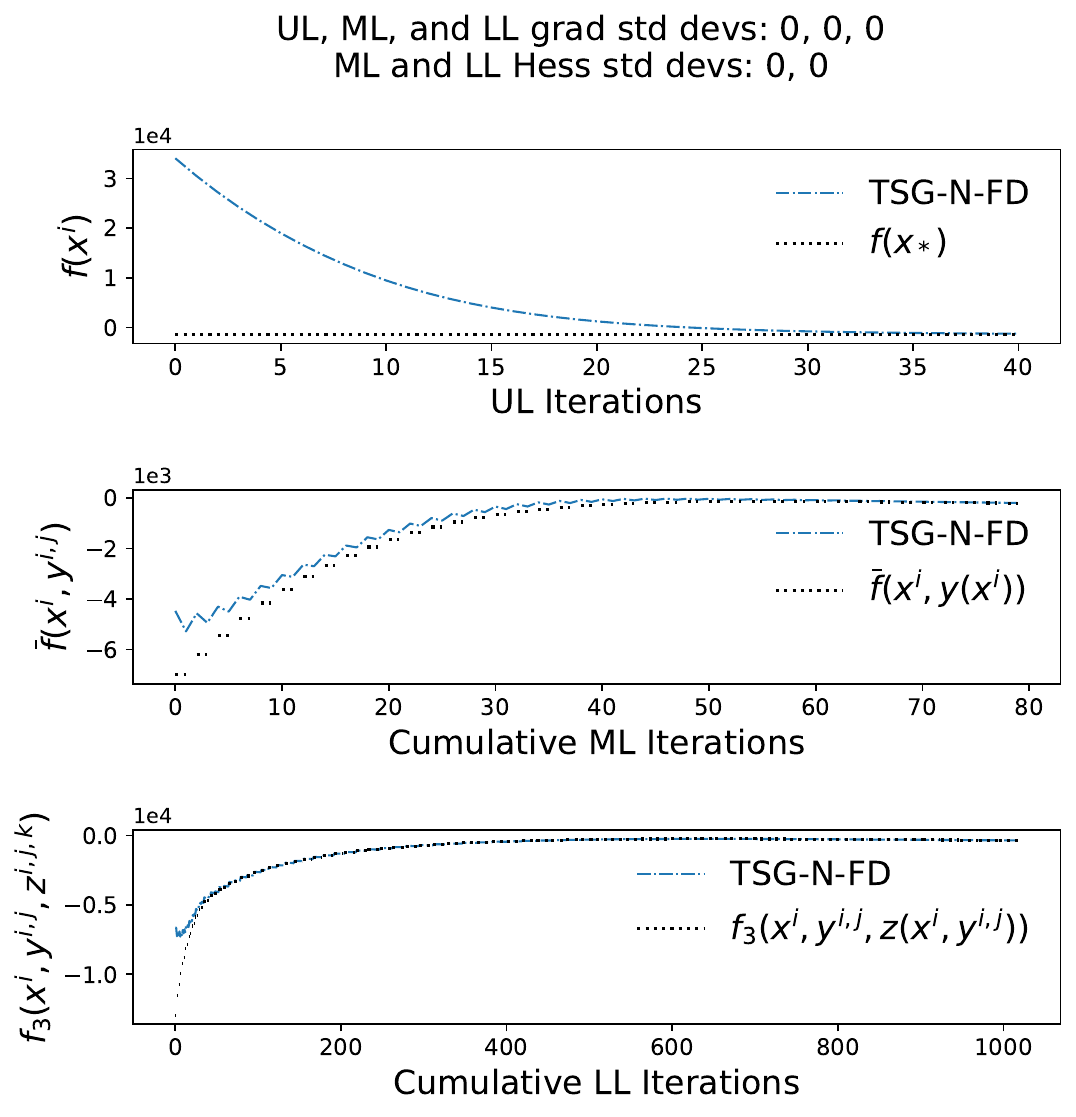}
		\includegraphics[scale=0.25]{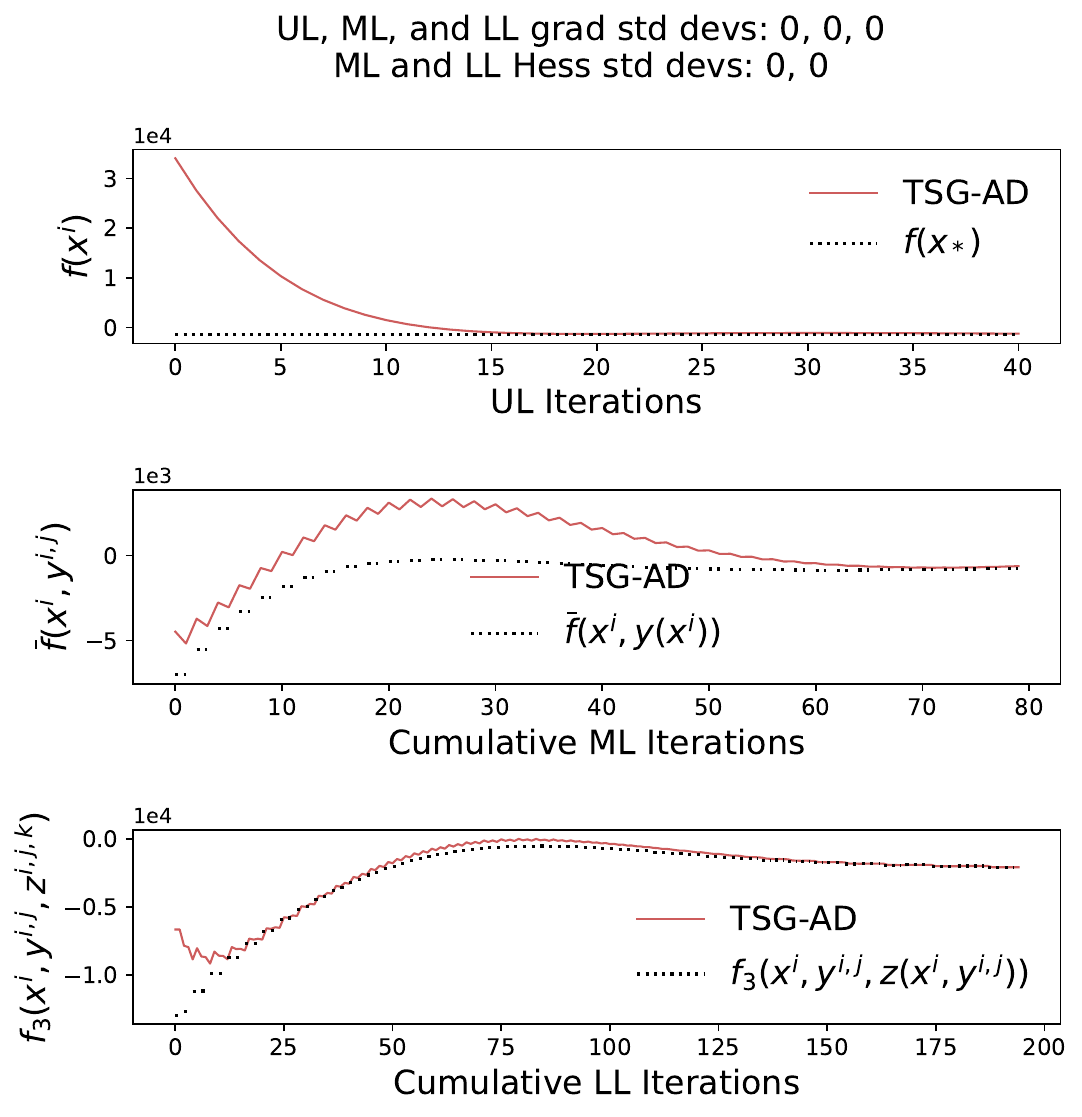}                    
		\caption{Breakdown of the algorithms, quadratic problem, deterministic case.}\label{fig:quad_prob_det_breakdown}
	\end{figure}
	
	\begin{figure}[H]
		\centering
		\includegraphics[scale=0.25]{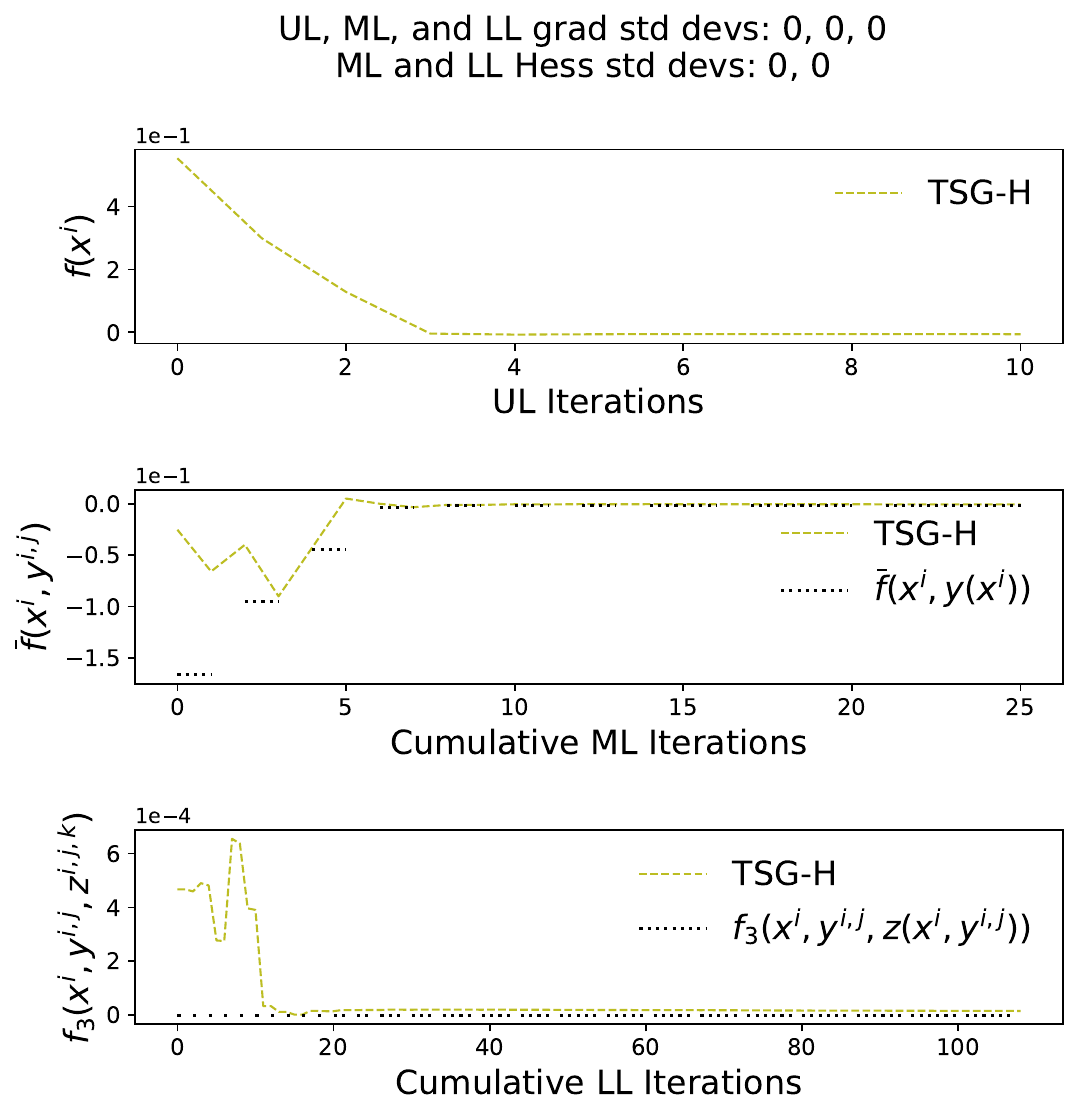}
		\includegraphics[scale=0.25]{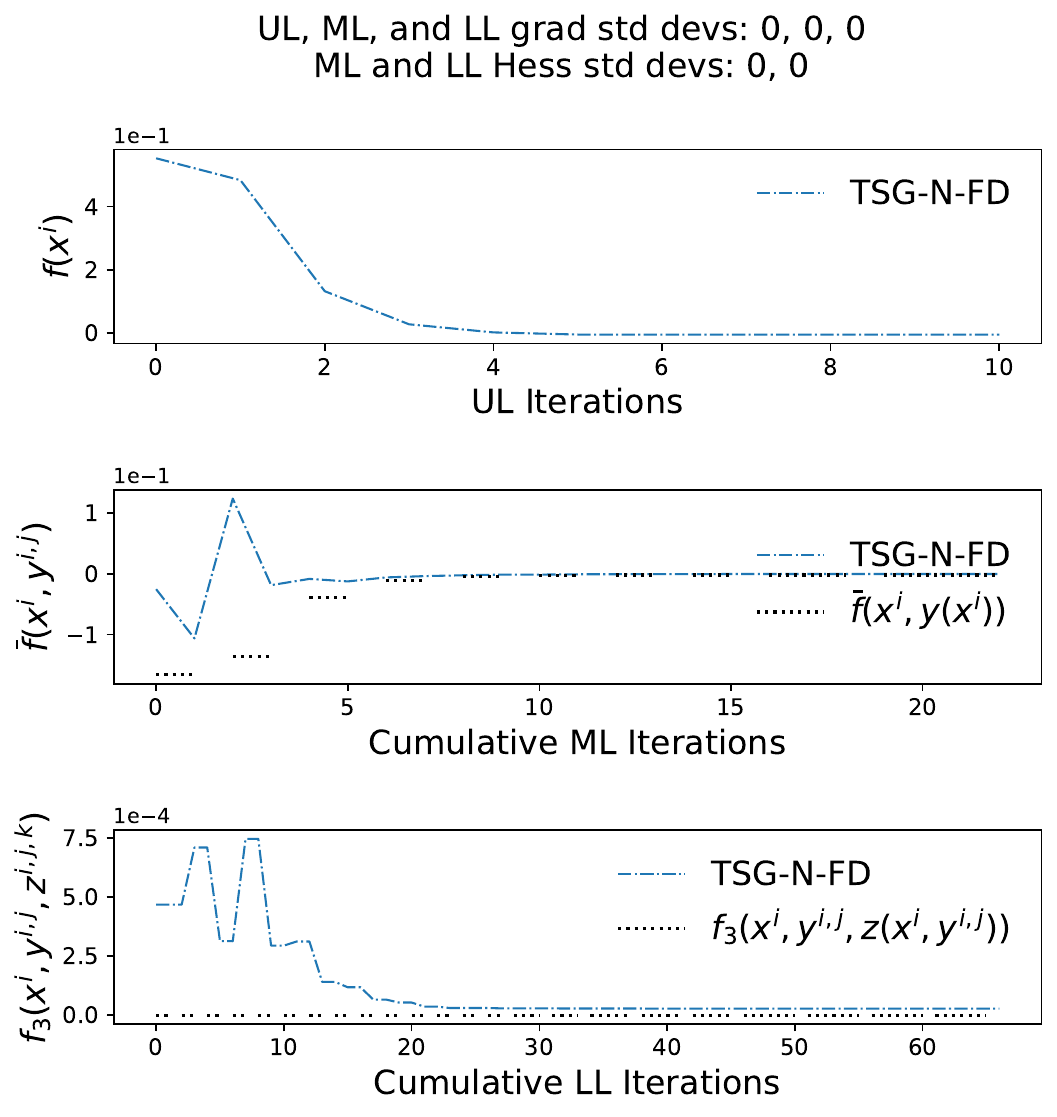}
		\includegraphics[scale=0.25]{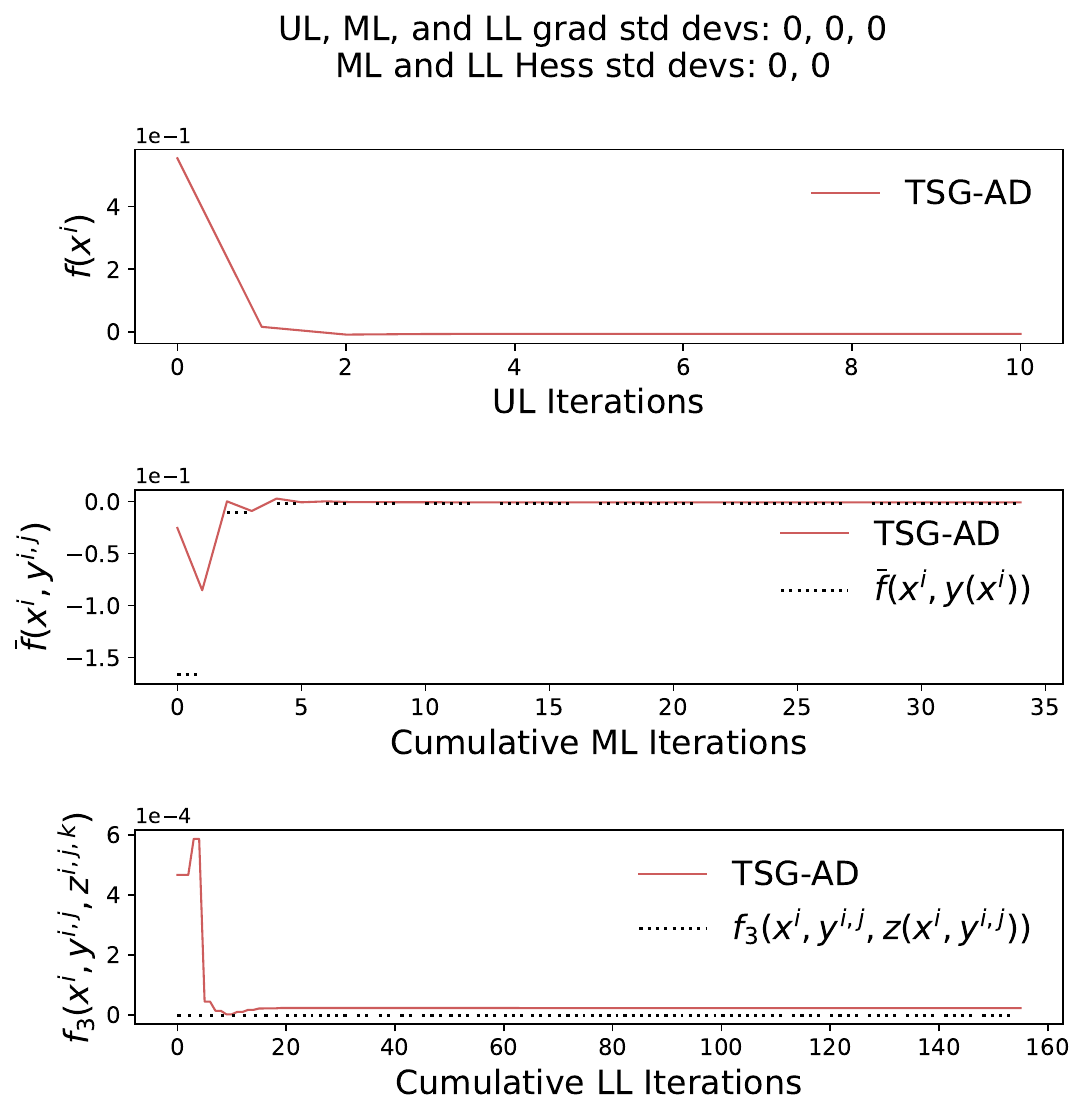}                    
		\caption{Breakdown of the algorithms, quartic problem, deterministic case.}\label{fig:quart_prob_det_breakdown}
	\end{figure}
	
	% In Subsection~\ref{sec:num_results_synth_trilevel_probs}, we observed that TSG-H is highly sensitive to the standard deviation of the Hessian of~$f_3$. Such behavior aligns with the well-known fact that stochastic Hessians require lower noise levels (i.e., larger mini-batch sizes when noise arises from sampling finite-sum Hessians in~SG contexts) than stochastic gradients to perform well~\cite[Section 6.1.1]{LBottou_FECurtis_JNocedal_2018}. We also observed that as noise levels increase, the performance of~TSG-N-FD deteriorates, whereas~TSG-AD remains more robust. The most critical source of noise for~TSG-N-FD is that added to~$\nabla f_3$, which is used to approximate the matrix-vector products involving the Hessian of~$f_3$ via the~FD scheme in~\eqref{eq:TSGNFD_FD0}. Note that such an~FD scheme affects the computation of both~\eqref{adjoint_NFD} and~\eqref{eq:nablay_barf_NFD}.

	\subsection{Trilevel adversarial hyperparameter tuning}\label{subsec:trilevel_adversarial_problems}
	
	Let us denote the whole learning dataset used in the experiments by~${\cal D} = \{ (u_j, v_j), \, j \in \{ 1, \ldots , N\} \}$, which consists of~$N$ pairs given by a feature vector~$u_j$ and the corresponding true label~$v_j$. 
	% For any instance $j$, the classification is deemed correct if the right label is predicted.
	We denote the datasets used for training and validation as~${D_{\train}}$ and~${D_{\val}}$, which respectively consist of~$N_{\train}$ and~$N_{\val}$ pairs extracted from the original dataset~$\cal D$ (with additional pairs set aside for testing). 
	Let~$\phi (u_j; \theta)$ be the prediction function, where~$\theta$ is a vector of parameters. 
	% which can take many different forms, such as a separating hyperplane and a Deep Neural Network (DNN).
	% To evaluate the prediction loss, we use a logistic loss function $\ell(u_j, v_j; \theta) = \log(1 + \exp(- v_j \, \phi (u_j; \theta)))$. 
	The adversarial training problem can be written according to the following minimax formulation (see, e.g., \cite{AMadry_etal_2017}):
	\begin{equation}\label{prob:adv_prob_minmax}
		\underset{\theta}{\operatorname{min}} \, \frac{1}{N_{\train}} \sum_{(u, v) \in D_{\train}} \max _{\|\delta_u\| \leq \epsilon} \ell(\phi (u + \delta_u; \ \theta), \ v),
	\end{equation}
	where~$\delta_u$ is a perturbation vector associated with each sample~$u$ in the training set, and~$\epsilon$ is a positive threshold. Introducing~$\delta = (\delta_u ~|~ (u,v) \in  D_{\train})$, we propose the following~TLO problem for adversarial hyperparameter tuning, inspired by~\cite{RSato_etal_2021}:
	\begin{equation}\label{prob:trilevel_adv_train}
		\begin{split}
			\min_{\lambda \in \mathbb{R}, \, \theta \in \mathbb{R}^m, \, \delta \in \mathbb{R}^t} ~~ & \frac{1}{N_{\val}} \sum_{(u, v) \in D_{\val}} \ell(\phi (u; \, \theta), \, v) \\
			\mbox{s.t.}~~ & \theta, \, \delta \in \argmin_{\theta \in \mathbb{R}^m, \, \delta\in\mathbb{R}^t} ~~ \frac{1}{N_{\train}} \sum_{(u, v) \in D_{\train}} \ell(\phi (u + \delta_u; \, \theta), \, v) + \Phi(\theta; \lambda)  \\
			& \quad\quad\mbox{s.t.}~~  \delta \in \argmax_{\delta \in \mathbb{R}^t} ~~ \frac{1}{N_{\train}} \sum_{(u, v) \in D_{\train}} \ell(\phi (u + \delta_u; \, \theta), \, v) - \Psi(\delta), \\
		\end{split}
	\end{equation}
	where~$\lambda$ is a penalty coefficient, and~$\Phi(\theta; \lambda) = (e^{\lambda}\|\theta\|_{1^\star})/m$ (with $\|\cdot\|_{1^\star}$ being a smooth approximation of the~$\ell_1$-norm~\cite[Eq.~(18) with~$\mu = 0.25$]{BSaheya_etal_2019}) and~$\Psi(\delta) = (c\|\delta\|^2)/(mN_{\train}$) (with~$c = 0.1$ being a penalty coefficient) are penalty terms that penalize large values of~$\theta$ and~$\delta$, respectively. To convert the~LL problem into a minimization problem, we switch to~$\argmin$ by multiplying the objective function by~$-1$.
	Following~\cite{RSato_etal_2021}, we use a linear prediction function and mean squared error~(MSE) as the loss function in our experiments.
	
	Regarding the datasets used in the experiments, the red and white wine quality datasets~\cite{PCortez_ALCerdeira_2009} contain~1,599 and~4,898 samples, respectively, each with~11 features, while the California housing dataset~\cite{RKPace_RBarry_1997} contains~20,640 samples and~8 features. Each dataset is split into training, validation, and test sets in proportions of~70\%, 15\%, and~15\%, respectively.
	
	For~TSG-N-FD and~TSG-AD, we use the same configuration described in Section~\ref{sec:num_results_synth_trilevel_probs}, including decaying stepsizes ($\alpha_i = \bar{\alpha}/i$, $\beta_j = \bar{\beta}/j$, and $\gamma_k = \bar{\gamma}/k$), 
	% and an increasing-accuracy strategy to determine~$J$ and~$K$.
	where the positive scalars $\bar{\alpha}$, $\bar{\beta}$, and $\bar{\gamma}$ are selected via grid search over the set~$\{0.1, 0.01, 0.001\}$. 
	For the~BSG-AD algorithms, which are derived from~TSG-AD to solve the~BLO problems obtained from~\eqref{prob:trilevel_adv_train}, we once again use decaying stepsizes selected via grid search over~$\{0.1, 0.01, 0.001\}$.
	Specifically, the values of~$\bar{\alpha}$, $\bar{\beta}$, and~$\bar{\gamma}$ are provided in~Table~\ref{tab:stepsizes}. In all experiments, the algorithms use a minibatch size of~64 for training, and the results presented in the figures are averaged over~10 runs.
	
	% \tcb{Specifically, for~TSG-N-FD, we used~($\bar{\alpha} = 0.1$, $\bar{\beta} = 0.1$, $\bar{\gamma} = 0.1$) (when run on the formulation proposed in~\cite{RSato_etal_2021}); for TSG-AD, we used~($\bar{\alpha} = 0.01$, $\bar{\beta} = 0.01$, $\bar{\gamma} = 0.01$) (when run on the formulation proposed in~\cite{RSato_etal_2021}), ($\bar{\alpha} = 0.1$, $\bar{\beta} = 0.01$, $\bar{\gamma} = 0.1$) (when run on our formulation~\eqref{prob:trilevel_adv_train} with the red and white wine quality datasets), and~($\bar{\alpha} = 0.01$, $\bar{\beta} = 0.001$, $\bar{\gamma} = 0.01$) (when run on our formulation~\eqref{prob:trilevel_adv_train} with the California housing dataset); for~BSG-AD (without~UL), we used~($\bar{\beta} = 0.01$, $\bar{\gamma} = 0.1$) (when run on our formulation~\eqref{prob:trilevel_adv_train} with the red and white wine quality datasets) and~($\bar{\beta} = 0.001$, $\bar{\gamma} = 0.1$) (when run on our formulation~\eqref{prob:trilevel_adv_train} with the California housing dataset); for~BSG-AD (without~LL), we used~($\bar{\alpha} = 0.1$, $\bar{\beta} = 0.01$) (when run on our formulation~\eqref{prob:trilevel_adv_train} with the red and white wine quality datasets) and~($\bar{\alpha} = 0.1$, $\bar{\beta} = 0.001$) (when run on our formulation~\eqref{prob:trilevel_adv_train} with the California housing dataset).}
	
	\begin{table}
		\caption{Details of the stepsizes ($\alpha_i = \bar{\alpha}/i$, $\beta_j = \bar{\beta}/j$, $\gamma_k = \bar{\gamma}/k$) used across algorithms, formulations, and datasets in the trilevel adversarial hyperparameter tuning experiments}
		\label{tab:stepsizes}
		\centering
		\begin{tabular}{lllllll}
			\toprule
			Algorithm & Formulation & Dataset & $\bar{\alpha}$ & $\bar{\beta}$ & $\bar{\gamma}$ \\
			\midrule
			TSG-N-FD & \cite{RSato_etal_2021} & Red Wine             & 0.1  & 0.1   & 0.1  \\
			TSG-AD   & \cite{RSato_etal_2021} & Red Wine             & 0.01 & 0.01  & 0.01 \\
			TSG-AD   & \eqref{prob:trilevel_adv_train} & Red \& White Wine   & 0.1  & 0.01  & 0.1  \\
			TSG-AD   & \eqref{prob:trilevel_adv_train} & California Housing  & 0.01 & 0.001 & 0.01 \\
			BSG-AD (without UL) & \eqref{prob:trilevel_adv_train} & Red \& White Wine  & --   & 0.01  & 0.1  \\
			BSG-AD (without UL) & \eqref{prob:trilevel_adv_train} & California Housing & --   & 0.001 & 0.1  \\
			BSG-AD (without LL) & \eqref{prob:trilevel_adv_train} & Red \& White Wine  & 0.1  & 0.01  & --   \\
			BSG-AD (without LL) & \eqref{prob:trilevel_adv_train} & California Housing & 0.1  & 0.001 & --   \\
			\bottomrule
		\end{tabular}
	\end{table}

	\subsubsection{Additional figures and discussion for trilevel adversarial hyperparameter tuning  }\label{subsubsec:trilevel_adversarial_additional_figs}
	
	% The~TLO problem for adversarial hyperparameter tuning proposed in~\cite{RSato_etal_2021} can be obtained by swapping the~ML and~LL problems in~\eqref{prob:trilevel_adv_train}. 
	In Figure~\ref{fig:Sato_redwine}, we assess the~TLO problem for adversarial hyperparameter tuning proposed in~\cite{RSato_etal_2021}, which can be obtained by swapping the~ML and~LL problems in~\eqref{prob:trilevel_adv_train}. The results on the red wine dataset demonstrate that both~TSG-N-FD and~TSG-AD exhibit essentially similar performance in terms of test~MSE. However, the test~MSE values are consistently worse or comparable to those obtained using the formulation in~\eqref{prob:trilevel_adv_train} (see Figure~\ref{fig:nonSato_redwine}), which is why we discontinued testing the formulation from~\cite{RSato_etal_2021}.  
	% This outcome is not surprising, as the results from the synthetic problems in Section~\ref{sec:num_results_synth_trilevel_probs} indicated that the most critical source of noise is the one affecting~$\nabla f_3$. In~\eqref{prob:trilevel_adv_train}, the~LL problem is the largest-scale problem among the three, as the size of~$\delta$ corresponds to the number of rows times the number of columns of the entire dataset, making~$\nabla f_3$ more susceptible to minibatch sampling.
	
	When using~\eqref{prob:trilevel_adv_train}, TSG-N-FD does not perform well and is therefore excluded from further analysis.
	This outcome is not surprising, as the results from the synthetic problems in Section~\ref{sec:num_results_synth_trilevel_probs} indicated that~TSG-N-FD is more affected by noise in~$\nabla f_3$ than~TSG-AD. In~\eqref{prob:trilevel_adv_train}, the noise is further amplified by the fact that the size of~$\delta$ corresponds to the number of rows times the number of columns of the entire dataset, making~$\nabla f_3$ more susceptible to minibatch sampling.

	\begin{figure}[H]
		\centering
		\includegraphics[scale=0.22]{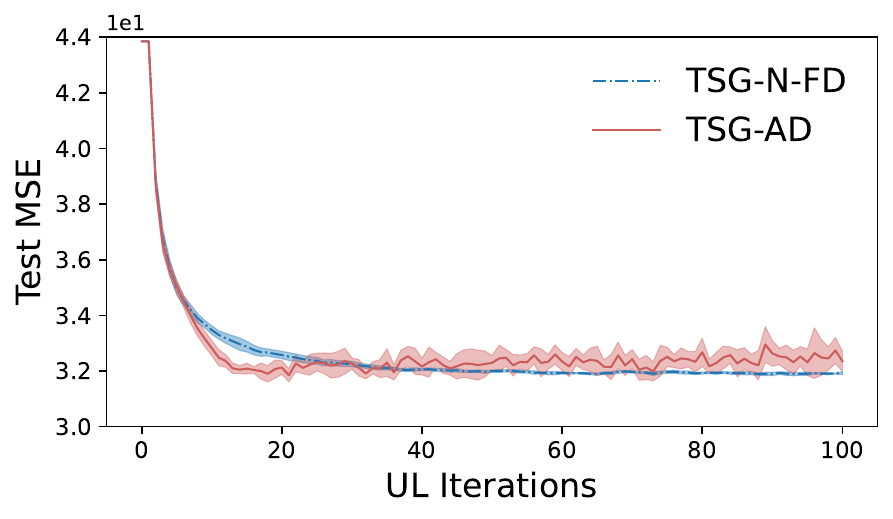}
		\includegraphics[scale=0.22]{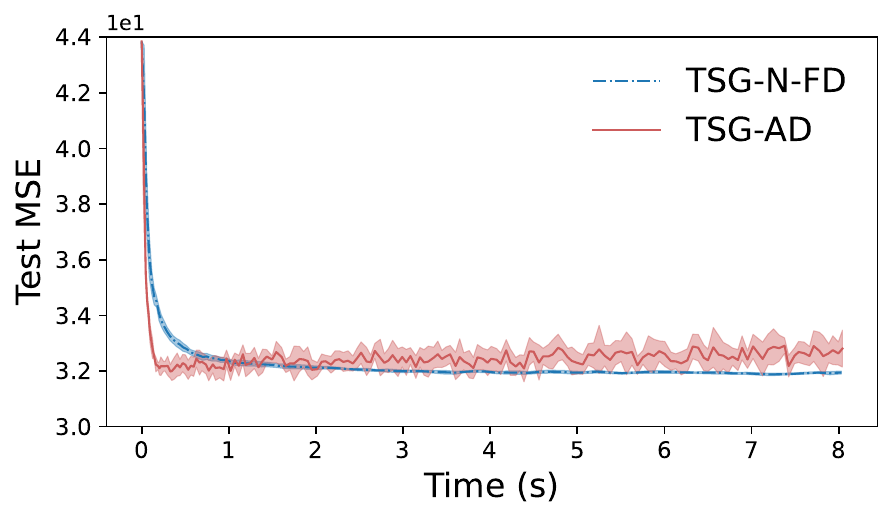}
		\vspace{0.2cm}
		\includegraphics[scale=0.22]{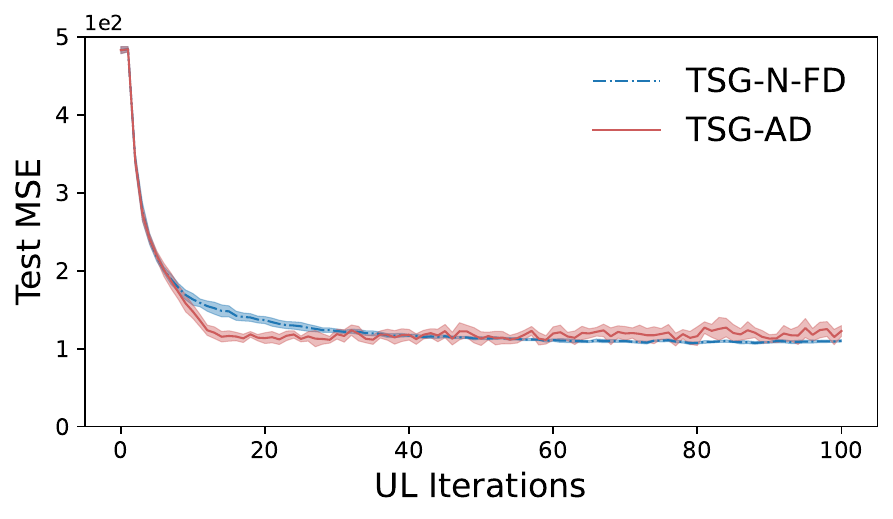}
		\includegraphics[scale=0.22]{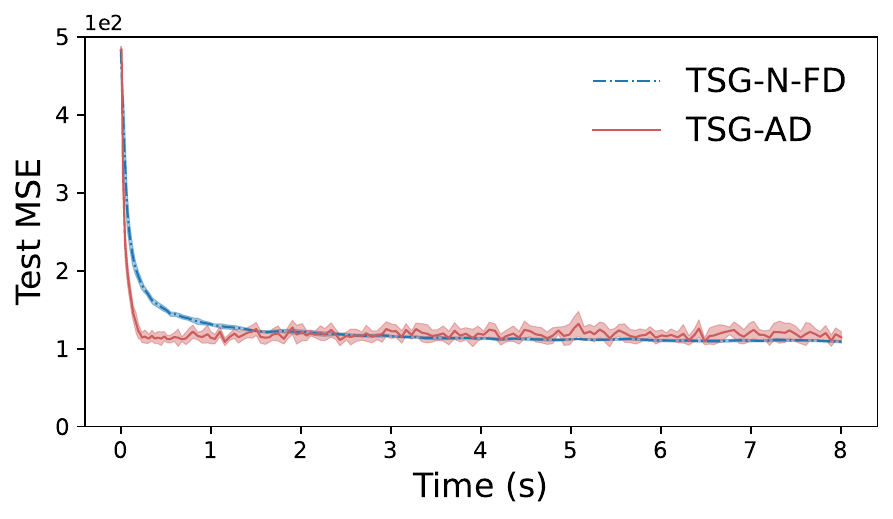}                   
		\caption{Trilevel adversarial learning formulation proposed in~\cite{RSato_etal_2021}, red wine quality dataset. The two left plots correspond to noise with standard deviation~0, and the two right plots to standard deviation~5.}\label{fig:Sato_redwine}
	\end{figure}


\begin{thebibliography}{45}
		\providecommand{\natexlab}[1]{#1}
		\providecommand{\url}[1]{\texttt{#1}}
		\expandafter\ifx\csname urlstyle\endcsname\relax
		\providecommand{\doi}[1]{doi: #1}\else
		\providecommand{\doi}{doi: \begingroup \urlstyle{rm}\Url}\fi
		
		\bibitem[Arguello et~al.(2023)Arguello, Johnson, and
		Gearhart]{BArguello_ESJohnson_JLGearhart_2023}
		B.~Arguello, E.~S. Johnson, and J.~L. Gearhart.
		\newblock A trilevel model for segmentation of the power transmission grid
		cyber network.
		\newblock \emph{IEEE Syst. J.}, 17:\penalty0 419--430, 2023.
		
		\bibitem[Bard(1984)]{JFBard_1984}
		J.~F. Bard.
		\newblock An investigation of the linear three level programming problem.
		\newblock \emph{IEEE Trans. Syst. Man Cybern.}, SMC-14:\penalty0 711--717,
		1984.
		
		\bibitem[Bard and Falk(1982)]{JFBard_JEFalk_1982}
		J.~F. Bard and J.~E. Falk.
		\newblock An explicit solution to the multi-level programming problem.
		\newblock \emph{Comput. Oper. Res.}, 9:\penalty0 77--100, 1982.
		
		\bibitem[Beck(2017)]{ABeck_2017}
		A.~Beck.
		\newblock \emph{First-Order Methods in Optimization}.
		\newblock SIAM-Society for Industrial and Applied Mathematics, Philadelphia,
		PA, 2017.
		
		\bibitem[Benson(1989)]{HPBenson_1989}
		H.~P. Benson.
		\newblock On the structure and properties of a linear multilevel programming
		problem.
		\newblock \emph{J. Optim. Theory Appl.}, 60:\penalty0 353--373, 1989.
		
		\bibitem[Blair(1992)]{CBlair_1992}
		C.~Blair.
		\newblock The computational complexity of multi-level linear programs.
		\newblock \emph{Ann. Oper. Res.}, 34:\penalty0 13--19, 1992.
		
		\bibitem[Bollapragada et~al.(2018)Bollapragada, Byrd, and
		Nocedal]{RBollapragada_etal_2018}
		R.~Bollapragada, R.~H. Byrd, and J.~Nocedal.
		\newblock {Exact and inexact subsampled Newton methods for optimization}.
		\newblock \emph{IMA Journal of Numerical Analysis}, 39:\penalty0 545--578, 04
		2018.
		
		\bibitem[Bottou et~al.(2018)Bottou, Curtis, and
		Nocedal]{LBottou_FECurtis_JNocedal_2018}
		L.~Bottou, F.~E. Curtis, and J.~Nocedal.
		\newblock Optimization methods for large-scale machine learning.
		\newblock \emph{SIAM Review}, 60:\penalty0 223--311, 2018.
		
		\bibitem[Cang and Petrusel(2010)]{LCCang_APetrusel_2010}
		L.~C. Cang and A.~Petrusel.
		\newblock Krasnoselski-{Mann} iterations for hierarchical fixed point problems
		for a finite family of nonself mappings in {Banach} spaces.
		\newblock \emph{J. Optim. Theory Appl.}, 146:\penalty0 617--639, 2010.
		
		\bibitem[Chen et~al.(2023)Chen, Chen, Ma, Liu, and Liu]{CChen_etal_2023}
		C.~Chen, X.~Chen, C.~Ma, Z.~Liu, and X~Liu.
		\newblock Gradient-based bi-level optimization for deep learning: A survey,
		2023.
		
		\bibitem[Chen et~al.(2021)Chen, Sun, and Yin]{TChen_YSun_WYin_2021_closingGap}
		T.~Chen, Y.~Sun, and W.~Yin.
		\newblock Closing the gap: Tighter analysis of alternating stochastic gradient
		methods for bilevel problems.
		\newblock In \emph{Advances in Neural Information Processing Systems},
		volume~34, pages 25294--25307. Curran Associates, Inc., 2021.
		
		\bibitem[Chen et~al.(2022)Chen, Sun, Xiao, and Yin]{TChen_YSun_WYin_2021}
		T.~Chen, Y.~Sun, Q.~Xiao, and W.~Yin.
		\newblock {A Single-Timescale Method for Stochastic Bilevel Optimization}.
		\newblock In \emph{Proceedings of The 25th International Conference on
			Artificial Intelligence and Statistics}, volume 151, pages 2466 -- 2488.
		PMLR, March 2022.
		
		\bibitem[Choe et~al.(2022)Choe, Neiswanger, Xie, and Xing]{SKeun_etal_2023}
		S.~K. Choe, W.~Neiswanger, P.~Xie, and E.~Xing.
		\newblock Betty: An automatic differentiation library for multilevel
		optimization.
		\newblock \emph{arXiv e-prints}, art. arXiv:2207.02849, July 2022.
		
		\bibitem[Cortez et~al.(2009)Cortez, Cerdeira, Almeida, Matos, and
		Reis]{PCortez_ALCerdeira_2009}
		P.~Cortez, Antonio~Lu{\'i}z Cerdeira, Fernando Almeida, Telmo Matos, and
		Jos{\'e} Reis.
		\newblock Modeling wine preferences by data mining from physicochemical
		properties.
		\newblock \emph{Decis. Support Syst.}, 47:\penalty0 547--553, 2009.
		
		\bibitem[Fathollahi-Fard et~al.(2018)Fathollahi-Fard, Hajiaghaei-Keshteli, and
		Mirjalili]{AFard_MKeshteli_SMirjalili_2018}
		A.~M. Fathollahi-Fard, M.~Hajiaghaei-Keshteli, and S.~Mirjalili.
		\newblock Hybrid optimizers to solve a tri-level programming model for a tire
		closed-loop supply chain network design problem.
		\newblock \emph{Applied Soft Computing}, 70:\penalty0 701--722, 2018.
		
		\bibitem[Franceschi et~al.(2017)Franceschi, Donini, Frasconi, and
		Pontil]{LFranceschi_etal_2017}
		L.~Franceschi, M.~Donini, P.~Frasconi, and M.~Pontil.
		\newblock Forward and reverse gradient-based hyperparameter optimization.
		\newblock In \emph{Proc. 34th Int. Conf. Mach. Learn. (ICML)}, volume~70 of
		\emph{Proceedings of Machine Learning Research}, pages 1165--1173. PMLR,
		06--11 Aug 2017.
		
		\bibitem[{Ghadimi} and {Wang}(2018)]{SGhadimi_MWang_2018}
		S.~{Ghadimi} and M.~{Wang}.
		\newblock {Approximation methods for bilevel programming}.
		\newblock \emph{arXiv e-prints}, art. arXiv:1802.02246, February 2018.
		
		\bibitem[{Giovannelli} et~al.(2022){Giovannelli}, {Kent}, and
		{Vicente}]{TGiovannelli_GKent_LNVicente_2022}
		T.~{Giovannelli}, G.~D. {Kent}, and L.~N. {Vicente}.
		\newblock {Inexact bilevel stochastic gradient methods for constrained and
			unconstrained lower-level problems}.
		\newblock \emph{ISE Technical Report 21T-025, Lehigh University}, December
		2022.
		
		\bibitem[{Giovannelli} et~al.(2024){Giovannelli}, {Kent}, and
		{Vicente}]{TGiovannelli_GKent_LNVicente_2024}
		T.~{Giovannelli}, G.~D. {Kent}, and L.~N. {Vicente}.
		\newblock {Bilevel optimization with a multi-objective lower-level problem:
			risk-neutral and risk-averse formulations}.
		\newblock \emph{Optim. Methods Softw.}, 39:\penalty0 756--778, 2024.
		
		\bibitem[Guo et~al.(2019)Guo, Yang, Xu, Liu, and Lin]{MGuo_etal_2020}
		M.~Guo, Y.~Yang, R.~Xu, Z.~Liu, and D.~Lin.
		\newblock When {NAS} meets robustness: In search of robust architectures
		against adversarial attacks.
		\newblock \emph{arXiv e-prints}, art. arXiv:1911.10695, November 2019.
		
		\bibitem[Guo et~al.(2023)Guo, Guo, and Yang]{YGuo_CGuo_JYang_2023}
		Y.~Guo, C.~Guo, and J.~Yang.
		\newblock A tri-level optimization model for power systems defense considering
		cyber-physical interdependence.
		\newblock \emph{IET Gener. Transm. Distrib.}, 17:\penalty0 1477--1490, 2023.
		
		\bibitem[{Ji} et~al.(2020){Ji}, {Yang}, and {Liang}]{KJi_JYang_YLiang_2020}
		K.~{Ji}, J.~{Yang}, and Y.~{Liang}.
		\newblock {Bilevel optimization: Convergence analysis and enhanced design}.
		\newblock \emph{arXiv e-prints}, art. arXiv:2010.07962, October 2020.
		
		\bibitem[Jiao et~al.(2023)Jiao, Yang, Wu, Jian, and Huang]{YJiao_etal_2024}
		Y.~Jiao, K.~Yang, T.~Wu, C.~Jian, and J.~Huang.
		\newblock Provably convergent federated trilevel learning.
		\newblock \emph{arXiv e-prints}, art. arXiv:2312.11835, December 2023.
		
		\bibitem[Jiao et~al.(2024)Jiao, Yang, and Jian]{YJiao_KYang_CJian_2024}
		Y.~Jiao, K.~Yang, and C.~Jian.
		\newblock Unlocking trilevel learning with level-wise zeroth order constraints:
		Distributed algorithms and provable non-asymptotic convergence.
		\newblock \emph{arXiv e-prints}, art. arXiv:2412.07138, December 2024.
		
		\bibitem[{Jin} et~al.(2019){Jin}, {Wang}, {Slocum}, {Yang}, {Dai}, {Yan}, and
		{Feng}]{XJin_etal_2019}
		X.~{Jin}, J.~{Wang}, J.~{Slocum}, M.~{Yang}, S.~{Dai}, S.~{Yan}, and J.~{Feng}.
		\newblock {{RC-DARTS}: Resource constrained differentiable architecture
			search}.
		\newblock \emph{arXiv e-prints}, art. arXiv:1912.12814, December 2019.
		
		\bibitem[Kent(2025, in preparation)]{GKent_2025}
		G.~D. Kent.
		\newblock \emph{Stochastic Methods for Multi-Level and Multi-Objective
			Optimization}.
		\newblock PhD thesis, Lehigh University, Department of Industrial and Systems
		Engineering, 2025, in preparation.
		
		\bibitem[Lai et~al.(2019)Lai, Illindala, and
		Subramaniam]{KLai_MIllindala_KSubraminiam_2019}
		K.~Lai, M.~Illindala, and K.~Subramaniam.
		\newblock A tri-level optimization model to mitigate coordinated attacks on
		electric power systems in a cyber-physical environment.
		\newblock \emph{Appl. Energy}, 235:\penalty0 204--218, 2019.
		
		\bibitem[Liduka(2011)]{HLiduka_2011}
		H.~Liduka.
		\newblock Iterative algorithm for solving triple-hierarchical constrained
		optimization problem.
		\newblock \emph{J. Optim. Theory Appl.}, 148:\penalty0 580--592, 2011.
		
		\bibitem[Liu et~al.(2019)Liu, Simonyan, and Yang]{HLiu_KSimonyan_YYang_2019}
		H.~Liu, K.~Simonyan, and Y.~Yang.
		\newblock {DARTS}: Differentiable architecture search.
		\newblock \emph{ArXiv}, arXiv:1806.09055, June 2019.
		
		\bibitem[{Liu} et~al.(2021){Liu}, {Gao}, {Zhang}, {Meng}, and
		{Lin}]{RLiu_JGao_etal_2021}
		R.~{Liu}, J.~{Gao}, J.~{Zhang}, D.~{Meng}, and Z.~{Lin}.
		\newblock {Investigating bi-Level optimization for learning and vision from a
			unified perspective: A survey and beyond}.
		\newblock \emph{arXiv e-prints}, art. arXiv:2101.11517, January 2021.
		
		\bibitem[Lu et~al.(2016)Lu, Han, Hu, and Zhang]{JLu_etal_2016}
		J.~Lu, J.~Han, Y.~Hu, and G.~Zhang.
		\newblock Multilevel decision-making: A survey.
		\newblock \emph{Information Sciences}, 346-347:\penalty0 463--487, 2016.
		
		\bibitem[{Madry} et~al.(2017){Madry}, {Makelov}, {Schmidt}, {Tsipras}, and
		{Vladu}]{AMadry_etal_2017}
		A.~{Madry}, A.~{Makelov}, L.~{Schmidt}, D.~{Tsipras}, and A.~{Vladu}.
		\newblock {Towards deep learning models resistant to adversarial attacks}.
		\newblock \emph{arXiv e-prints}, art. arXiv:1706.06083, June 2017.
		
		\bibitem[Nesterov(2018)]{YNesterov_2018}
		Y.~Nesterov.
		\newblock \emph{Lectures on Convex Optimization}.
		\newblock Springer Publishing Company, Incorporated, New York, 2nd edition,
		2018.
		
		\bibitem[Pace and Barry(1997)]{RKPace_RBarry_1997}
		R.~K. Pace and R.~Barry.
		\newblock Sparse spatial autoregressions.
		\newblock \emph{Stat. Probab. Lett.}, 33:\penalty0 291--297, 1997.
		
		\bibitem[Rahdar et~al.(2018)Rahdar, Wang, and Hu]{MRahdar_etal_2018}
		M.~Rahdar, L.~Wang, and G.~Hu.
		\newblock A tri-level optimization model for inventory control with uncertain
		demand and lead time.
		\newblock \emph{Int. J. Prod. Econ.}, 195:\penalty0 96--105, 2018.
		
		\bibitem[Rudin(1953)]{Ruding_1953}
		W.~Rudin.
		\newblock \emph{Principles of Mathematical Analysis}.
		\newblock McGraw-Hill Book Company, Inc., New York-Toronto-London, 1953.
		
		\bibitem[Saheya et~al.(2019)Saheya, Nguyen, and Chen]{BSaheya_etal_2019}
		B.~Saheya, C.~T. Nguyen, and J.-S. Chen.
		\newblock Neural network based on systematically generated smoothing functions
		for absolute value equation.
		\newblock \emph{J. Appl. Math. Comput.}, 61:\penalty0 533--558, 2019.
		
		\bibitem[Sato et~al.(2021)Sato, Tanaka, and Taked]{RSato_etal_2021}
		R.~Sato, M.~Tanaka, and A.~Taked.
		\newblock A gradient method for multilevel optimization.
		\newblock In \emph{Adv. Neural Inf. Process. Syst.}, volume~34, pages
		7522--7533. Curran Associates, Inc., 2021.
		
		\bibitem[Shafiei et~al.(2024)Shafiei, Kungurtsev, and
		Marecek]{AShafiei_etal_2024}
		A.~Shafiei, V.~Kungurtsev, and J.~Marecek.
		\newblock Trilevel and multilevel optimization using monotone operator theory.
		\newblock \emph{Math. Methods Oper. Res.}, 99:\penalty0 77--114, 2024.
		
		\bibitem[Tilahun et~al.(2012)Tilahun, Kassa, and Ong]{SLTilahun_etal_2012}
		S.~L. Tilahun, S.~M. Kassa, and H.~C. Ong.
		\newblock A new algorithm for multilevel optimization problems using
		evolutionary strategy, inspired by natural adaptation.
		\newblock In \emph{PRICAI 2012: Trends in Artificial Intelligence}, pages
		577--588. Springer Berlin Heidelberg, 2012.
		
		\bibitem[Ue-Pyng and Bialas(1986)]{WUepying_WFBialas_1986}
		W.~Ue-Pyng and W.~F. Bialas.
		\newblock The hybrid algorithm for solving the three-level linear programming
		problem.
		\newblock \emph{Comput. Oper. Res.}, 13:\penalty0 367--377, 1986.
		
		\bibitem[Vicente and Calamai(1994)]{LNVicente_PHCalamai_1994}
		L.~N. Vicente and P.~H. Calamai.
		\newblock Bilevel and multilevel programming: A bibliography review.
		\newblock \emph{J. Global Optim.}, 5:\penalty0 291--306, 1994.
		
		\bibitem[Wu and Conejo(2017)]{XWu_AJConejo_2017}
		X.~Wu and A.~J. Conejo.
		\newblock An efficient tri-level optimization model for electric grid defense
		planning.
		\newblock \emph{IEEE Trans. Power Syst.}, 32:\penalty0 2984--2994, 2017.
		
		\bibitem[Xu et~al.(2013)Xu, Meng, and Shen]{XXu_ZMeng_RShen_2013}
		X.~Xu, Z.~Meng, and R.~Shen.
		\newblock A tri-level programming model based on conditional value-at-risk for
		three-stage supply chain management.
		\newblock \emph{Comput. Ind. Eng.}, 66:\penalty0 470--475, 2013.
		
		\bibitem[Yao et~al.(2007)Yao, Edmunds, Papageorgiou, and
		Alvarez]{YYao_etal_2007}
		Y.~Yao, T.~Edmunds, D.~Papageorgiou, and R.~Alvarez.
		\newblock Trilevel optimization in power network defense.
		\newblock \emph{IEEE Trans. Syst. Man Cybern. C Appl. Rev.}, 37:\penalty0
		712--718, 2007.
		
	\end{thebibliography}
\end{document}